\documentclass[reqno]{amsart}
\usepackage{graphicx}
\usepackage{cite}
\usepackage{amssymb}
\usepackage{tikz}
\usepackage{color}
\newtheorem{thm}{Theorem}[section]
\newtheorem{cor}[thm]{Corollary}
\newtheorem{lem}[thm]{Lemma}
\newtheorem{prop}[thm]{Proposition}
\theoremstyle{definition}

\theoremstyle{conjecture}
\newtheorem{conj}[thm]{Conjecture}
\theoremstyle{remark}
\newtheorem{rem}[thm]{Remark}
\numberwithin{equation}{section}
\begin{document}

\title[Higher order Szeg\H{o} theorems]{Sum rules and Simon spectral gem problem on higher order Szeg\H{o} theorems}%

\author{Zhihua Du}%
%\address{Department of Mathematics, Jinan University, Guangzhou 510632, China}%
%\email{tzhdu@jnu.edu.cn}%

%\thanks{This work was initiated in Institute of
%Mathematics, Free University Berlin when the author visited and
%studied there from April, 2007 to April, 2008 by the support of
%State Scholarship Fund of China.
%The author appreciates Prof. Dr. Heinrich Begehr for his supervising and help and Prof. Dr. Hua Liu for his discussion.}%
\subjclass[2000]{42C05, 35P05}%
\keywords{Sum rules, Simon-Lukic conjecture, OPUC, Higher order Szeg\H o theorems}%

%\date{}%
%\dedicatory{}%
%\commby{}%
% ----------------------------------------------------------------
\begin{abstract}
In this work, we give a formula for coefficients of orthogonal polynomials on the unit circle. By using this formula, a new and computable approach is provided for sum rules which applying to a spectral gem problem proposed by Barry Simon on higher order Szeg\H{o} theorems.
\end{abstract}
\maketitle
% ----------------------------------------------------------------
\tableofcontents

\section{Introduction}

In this paper, we will study orthogonal polynomials on the unit circle (usually for short, OPUC). For the background of OPUC, we refer the readers to the references \cite{sim1,sim2,sze}. Let $\mathbb{D}=\{z: |z|<1\}$ be the open unit disc and $d\mu$ be a probability measure with infinite support on the unit circle $\partial \mathbb{D}$ in the complex plane $\mathbb{C}$, by Lebesque decomposition, we always write
\begin{equation}
d\mu(\theta)=w(\theta)\frac{d\theta}{2\pi}+d\mu_{s}(\theta),\,\,\,\theta\in [0,2\pi],
\end{equation}
where $d\mu_{ac}=w(\theta)\frac{d\theta}{2\pi}$ is the absolutely continuous part of $d\mu$, $d\mu_{s}$ is the singular part of $d\mu$. Here $\partial \mathbb{D}$ is identified with the interval $[0,2\pi]$ via the map $e^{i\theta}\rightarrow \theta$.

Using Gram-Schmidt procedure, the system of monic orthogonal polynomials on the unit circle, $\{\Phi_{n}(z)\}_{n=0}^{\infty}$, is obtained by orthogonalizing the system of monomials, $\{z^{n}\}_{n=0}^{\infty}$ with respect to $d\mu$ satisfying
\begin{equation}
\langle\Phi_{m}, \Phi_{n}\rangle=\int_{0}^{2\pi}\overline{\Phi_{m}(e^{i\theta})}\Phi_{n}(e^{i\theta})d\mu(\theta)=\kappa_{n}^{-2}\delta_{mn}
\end{equation}
with $\kappa_{n}>0$, where $\delta_{mn}$ is the Kronecker symbol, $\langle\cdot, \cdot\rangle$ is the inner product on $L^{2}(d\mu)$ with norm $\|\cdot\|=\sqrt{\langle\cdot, \cdot\rangle}$. Set $\varphi_{n}(z)=\kappa_{n}\Phi_{n}(z)$, then $\{\varphi_{n}\}_{n=0}^{\infty}$ is the system of orthonornmal polynomials on the unit circle fulfilling
\begin{equation}
\langle\varphi_{m}, \varphi_{n}\rangle=\int_{0}^{2\pi}\varphi_{m}(e^{i\theta})\varphi_{n}(e^{i\theta})d\mu(\theta)=\delta_{mn}.
\end{equation}

Up to now, OPUC has a rich theory since Szeg\H o introduced them  one hundred years ago in his papers \cite{sze1,sze2} in 1920-1921. Just as many other orthogonal polynomials, OPUC has a nice recursion relation called Szeg\H o recursion, that is,
\begin{equation}
\Phi_{n+1}(z)=z\Phi_{n}(z)-\overline{\alpha}_{n}\Phi_{n}^{*}(z),
\end{equation}
where $\alpha_{n}=-\overline{\Phi_{n+1}(0)}\in \mathbb{D}$ is called Verblunsky coefficient (cf. \cite{sim1}), and the reversed polynomial $\Phi_{n}^{*}(z)$ of $\Phi_{n}(z)$ is defined by
\begin{equation}
\Phi_{n}^{*}(z)=z^{n}\overline{\Phi_{n}(1/\overline{z})}.
\end{equation}
Presumably Szeg\H{o} theorem is the most remarkable result in the theory of OPUC (see \cite{sim3}). By Verblunsky's works \cite{v}, it can be stated as a sum rule as follows
\begin{equation}
\int_{0}^{2\pi}\log w(\theta)\frac{d\theta}{2\pi}=\sum_{n=0}^{\infty}\log(1-|\alpha_{n}|^{2}),
\end{equation}
which implies that
\begin{equation}
\int_{0}^{2\pi}\log w(\theta)\frac{d\theta}{2\pi}>-\infty\,\,\,\mbox{if and only if}\,\, \,\sum_{n=0}^{\infty}|\alpha_{n}|^{2}<\infty.
\end{equation}

In almost past twenty years, there was a great deal of investigation on sum rules for Jacobi matrices, orthogonal polynomials and Schr\"odinger operators, starting with the works of Deift-Killip \cite{dk} and Killip-Simon \cite{ks} and followed by many others (see \cite{bbz,bbz1,gz,sz,ks1,gnr,gnr1,gnr2,yan,lu,lu1} and therein).

The sum rules for OPUC are some analogues of the above Szeg\H{o} theorem in Verblunsky form and usually called higher order Szeg\H o theorems due to Simon \cite{sim1} (Nevertheless, throughout this paper, we only call some analogues of (1.9) below higher order Szeg\H{o} theorems). In his beautiful colloquium monograph \cite{sim1,sim2}, the following sum rule was established
\begin{align}
\int_{0}^{2\pi}(1-\cos\theta)\log w(\theta)\frac{d\theta}{2\pi}=&\frac{1}{2}-\frac{1}{2}|\alpha_{0}+1|^{2}+\sum_{n=0}^{\infty}\Big(\log(1-|\alpha_{n}|^{2})+|\alpha_{n}|^{2}\Big)\nonumber\\
&-\frac{1}{2}\sum_{n=0}^{\infty}|\alpha_{n+1}-\alpha_{n}|^{2}.
\end{align}
It leads to that
\begin{align}
\int_{0}^{2\pi}(1-\cos\theta)\log w(\theta)\frac{d\theta}{2\pi}>-\infty \,\,\Longleftrightarrow\,\, \sum_{n=0}^{\infty}\Big(|\alpha_{n+1}-\alpha_{n}|^{2}+|\alpha_{n}|^{4}\Big)<\infty,
\end{align}
where ``$\Longleftrightarrow$" means ``if and only if".

(1.7) and (1.9) are spectral theory results. Such results are called ``Gems" of spectral theory by Simon (see \cite{sim3}). In \cite{sz}, Simon and Zlato\v{s} obtained the following spectral results
\begin{align}
\int_{0}^{2\pi}(1-\cos\theta)^{2}\log w(\theta)\frac{d\theta}{2\pi}>-\infty \,\,\Longleftrightarrow\,\, \sum_{n=0}^{\infty}\Big(|\alpha_{n+2}-2\alpha_{n+1}+\alpha_{n}|^{2}+|\alpha_{n}|^{6}\Big)<\infty
\end{align}
and
\begin{align}
&\int_{0}^{2\pi}\big(1-\cos(\theta-\theta_{1})\big)\big(1-\cos(\theta-\theta_{2})\big)\log w(\theta)\frac{d\theta}{2\pi}>-\infty\nonumber\\
\Longleftrightarrow\,\, &\sum_{n=0}^{\infty}\left(\left|\left\{\prod_{j=1}^{2}\big(S-e^{-i\theta_{j}}\big)\alpha\right\}_{n}\right|^{2}+|\alpha_{n}|^{4}\right)<\infty
\end{align}
for $\theta_{1}\neq\theta_{2}$ in $[0,2\pi]$, where $S$ is the left shift operator given by
\begin{equation}
(S\alpha)_{n}=\alpha_{n+1}.
\end{equation}

Noting the forms of all the above results, it was natural for Simon to make a general conjecture as follows
\begin{conj}
For $\theta_{1},\theta_{2},\ldots,\theta_{k}$ distinct in $[0,2\pi)$ and $m_{1},m_{2},\ldots,m_{k}$ strictly positive integers,
\begin{align}
&\int_{0}^{2\pi}\prod_{j=1}^{k}\big(1-\cos(\theta-\theta_{j})\big)^{m_{j}}\log w(\theta)\frac{d\theta}{2\pi}>-\infty\nonumber\\
\Longleftrightarrow\,\, &\sum_{n=0}^{\infty}\left(\left|\left\{\prod_{j=1}^{k}\big(S-e^{-i\theta_{j}}\big)^{m_{j}}\alpha\right\}_{n}\right|^{2}+|\alpha_{n}|^{2m+2}\right)<\infty
\end{align}
in which $m=\max_{j=1,2,\ldots,k}m_{j}$.
\end{conj}
In other words, this is to say that
\begin{align}
&\int_{0}^{2\pi}\prod_{j=1}^{k}\big(1-\cos(\theta-\theta_{j})\big)^{m_{j}}\log w(\theta)\frac{d\theta}{2\pi}>-\infty
\end{align}
if and only if
\begin{equation}
(\mathrm{S1})\hspace{35mm}\prod_{j=1}^{k}\big(S-e^{-i\theta_{j}}\big)^{m_{j}}\alpha\in\ell^{2}\hspace{35mm}
\end{equation}
and
\begin{equation}
(\mathrm{S2})\hspace{40mm}\alpha\in\ell^{2\max_{j}(m_{j})+2}.\hspace{39mm}
\end{equation}

However, this conjecture is always not right. For the case of $\theta_{1}=0$, $\theta_{2}=\pi$, $m_{1}=1$ and $m_{2}=1$, in \cite{lu}, Lukic constructed a counterexample such that
$(S-1)^{2}(S+1)\alpha\in\ell^{2}$ and $\alpha\in \ell^{6}$ but
\begin{align}
\int_{0}^{2\pi}(1-\cos\theta)^{2}(1+\cos\theta)\log w(\theta)\frac{d\theta}{2\pi}=-\infty.
\end{align}
In the same paper, Lukic made an improved conjecture in place of (S1) and (S2) with the following conditions
\begin{equation}
(\mathrm{L1})\hspace{15mm}\alpha\,\, \mbox{can be expressed as}\,\, \alpha=\beta^{(1)}+\beta^{(2)}+\cdots\beta^{(k)},\hspace{15mm}
\end{equation}
\begin{equation}
(\mathrm{L2})\hspace{35mm}\big(S-e^{-i\theta_{j}}\big)^{m_{j}}\beta^{(j)}\in\ell^{2}\hspace{34mm}
\end{equation}
and
\begin{equation}
(\mathrm{L3})\hspace{41mm}\beta^{(j)}\in\ell^{2m_{j}+2}.\hspace{41mm}
\end{equation}

Lukic conjecture is
\begin{conj}
Let $\theta_{j}$ and $m_{j}$ be as in Conjecture 1.1, then (1.14) is equivalent to that there exists a sequence of $\beta^{(1)}, \beta^{(2)}, \ldots, \beta^{(k)}$ such that (1.18)-(1.20) hold.
\end{conj}

As usual, we call the above conjectures of type ($\theta_{1},\theta_{2},\ldots,\theta_{k}; m_{1},m_{2},\ldots,m_{k}$) for fixed $\theta_{1},\theta_{2},\ldots,\theta_{k}; m_{1},m_{2},\ldots,m_{k}$. The integral in (1.13) is called higher order Szeg\H o integral whereas the integral in (1.7) is called Szeg\H o integral. Under appropriate conditions, partial results for these conjectures of different types have been established. For instance, Lukic gave a ($0,\pi;2,1$) result in \cite{lu}, that is
\begin{align}
\int_{0}^{2\pi}(1-\cos\theta)^{2}(1+\cos\theta)\log w(\theta)\frac{d\theta}{2\pi}>-\infty\,\,\Longleftrightarrow\,\,(S-1)^{2}\alpha\in\ell^{4}
\end{align}
as $(S-1)(S+1)\alpha\in\ell^{2}$ and $\alpha\in\ell^{6}$.
In \cite{gz}, Golinskii and Zlato\v{s} obtained a general result of type ($\theta_{1},\theta_{2},\ldots,\theta_{k}; m_{1},m_{2},\ldots,m_{k}$) in $\ell^{4}$. More precisely, suppose $\alpha\in\ell^{4}$, then
\begin{align}
\int_{0}^{2\pi}|Q_{N}(e^{i\theta})|^{2}\log w(\theta)\frac{d\theta}{2\pi}>-\infty\,\,\Longleftrightarrow\,\,\overline{Q}_{N}(S)\alpha\in\ell^{2},
\end{align}
where $Q_{N}$ is any complex polynomial given by
\begin{equation}
Q_{N}(z)=\sum_{j=0}^{N}q_{j}z^{j}
\end{equation}
in which $q_{j}\in \mathbb{C}$, $j=0,1,\ldots,N$, and
\begin{equation}
\overline{Q}_{N}(z)=\sum_{j=0}^{N}\overline{q}_{j}z^{j}=z^{n}Q^{*}_{N}(1/z).
\end{equation}

For a long time, the main approach to attack these conjectures consists of three ingredients, step-by-step sum rule, constructing of positive terms and semi-continuity of entropy, which based on the relative Szeg\H o functions introduced by Simon \cite{sim1}. But there is no explanation for why this approach is valid other than good luck because it involved many clever algebraic and combinatorial manipulations for sums of positive terms. Until recent years, Gamboa, Nagel and Rouault partially changed this situation and gave partial explanation via their study on sum rules by using large deviations \cite{gnr,gnr1,gnr2}. In \cite{bbz}, Breuer, Simon and Zeitouni gave a pedagogical exposition on the approach of large deviations. By using this approach, in \cite{bbz1}, they recovered all the above known results due to Simon and Zlato\v{s}, some special results of Golinskii-Zlato\v{s} and a partial result of Lukic conjecture by showing Lukic's conditions on Verblunsky coefficients imply the finiteness of the relevant higher order Szeg\H o integral where Simon's do not. Since the side involving Verblunsky coefficents are rather complicated as the order goes to higher, the analysis about it is very difficult and full of challenge. To present, many results for spectral gems are focus on lower order cases or general cases but under some additional assumptions (for the latter, see \cite{lu1}).

In this paper, mainly motivated by the work \cite{gz} due to Golinskii and Zlato\v{s}, we will provide a new and computable approach to higher order Szeg\H o theorems and give an explanation why some of the known results should have suitable conditions in order to get the equivalence of the finiteness of higher order Szeg\H o integral and some series in terms of Verblunsky coefficients. As a conclusion, in general, one can not obtain the equivalence results as in Conjecture 1.1-1.2 under no additional conditions for Simon gems with some few exceptions such as (1.9)-(1.11).

One innovation of our method is to provide a unified approach to get sum rules on OPUC and higher order Szeg\H{o} theorems. This method can yield a lot of Simon spectral gems. However, these gems are conditional by strictly computing. Although the computaiton is involved, but the only thing is need to compute in order to obtain higher order Szeg\H{o} theorems. Another innovation is to give explicit expression of sum rules. By the explicit expression and some algebraic operations, one can get various Simon gems (or higher order Szeg\H{o} theorems) under different conditions. The main concrete results show that higher order Szeg\H{o} theorems are always not unconditional since the conditional part or positive part (for details, see Section 4) appears in the side of Verblunsky coefficients.

\section{Coeffecients of OPUC}

In \cite{gz}, Golinskii and Zlato\v{s} obtained an expression for coefficients of OPUC in terms of Verblundsky coefficients (maybe) for the first time. In this section, we will give an alternative expression in a new form for these coefficients. One will find that the new expression is very useful in this work. Throughout this paper, $\alpha_{-1}=-1$ as usual.

Let
\begin{equation}
\Phi_{n}(z)=z^{n}+a_{n,n-1}z^{n-1}+\cdots+a_{n,1}z+a_{n,0},
\end{equation}
by (1.4), \begin{equation}a_{n,0}=-\overline{\alpha}_{n-1}. \end{equation}

For the coefficients $a_{\cdot,\cdot}$, we have the following recursion relation.
\begin{prop} For $m,n\in \mathbb{N}$ with $1\leq m\leq n$,
\begin{align}
a_{n,n-m}&=-\overline{\alpha}_{m-1}+\sum_{j=0}^{n-m-1}\overline{\alpha}_{j+m}\alpha_{j}-\overline{\alpha}_{m}\sum_{j=1}^{m-1}\alpha_{j}\overline{\alpha}_{j-1}\nonumber\\
&+\sum_{k=1}^{n-m-1}\sum_{l=1}^{m-1}\beta_{k+m,k+l}a_{k+l,k},
\end{align}
in which $\beta_{s,t}=\overline{\alpha}_{s}\alpha_{t}$. As usual, the sum $\sum_{j=k}^{l}\cdot=0$ as $l<k$.
\end{prop}

\begin{proof}
By Szeg\H{o} recursion (1.4), we have
\begin{equation}
a_{n,m}=a_{n-1,m-1}-\overline{\alpha}_{n-1}\overline{a_{n-1,n-1-m}}
\end{equation}
for $1\leq m\leq n-1$. Replacing $n$ and $m$ by $n-1$ and $n-1-m$, then
\begin{equation}
a_{n-1,n-1-m}=a_{n-2,n-2-m}-\overline{\alpha}_{n-2}\overline{a_{n-2,m-1}}.
\end{equation}
Therefore, by substituting (2.5) into (2.4),
\begin{equation}
a_{n,m}=a_{n-1,m-1}+\overline{\alpha}_{n-1}\alpha_{n-2}a_{n-2,m-1}-\overline{\alpha}_{n-1}\overline{a_{n-2,n-2-m}}.
\end{equation}
By induction, we get
\begin{align}
a_{n,m}=&a_{n-1,m-1}+\overline{\alpha}_{n-1}\alpha_{n-2}a_{n-2,m-1}+\cdots+\overline{\alpha}_{n-1}\alpha_{m}a_{m,m-1}\nonumber\\
&-\overline{\alpha}_{n-1}\overline{a_{m,0}}.
\end{align}
So by (2.2),
\begin{align}
a_{n,m}-a_{n-1,m-1}=\overline{\alpha}_{n-1}\alpha_{m-1}+\sum_{j=m}^{n-2}\beta_{n-1,j}a_{j,m-1},
\end{align}
where $\beta_{n-1,j}=\overline{\alpha}_{n-1}\alpha_{j}$.

Repeating (2.8) in place of the subscript pair $(n,m)$ with $(n-1, m-1), (n-2,m-2),\ldots, (n-m+1, 1)$ respectively, and adding them together, we obtain
\begin{equation}
a_{n,m}=a_{n-m,0}+\sum_{k=0}^{m-1}\overline{\alpha}_{n-m+k}\alpha_{k}+\sum_{k=1}^{m}\sum_{j=k}^{n-m+k-2}\beta_{n-m+k-1,j}a_{j,k-1}.
\end{equation}
Thus
\begin{align*}
a_{n,n-m}=&a_{m,0}+\sum_{k=0}^{n-m-1}\overline{\alpha}_{m+k}\alpha_{k}+\sum_{k=1}^{n-m}\sum_{j=k}^{m+k-2}\beta_{m+k-1,j}a_{j,k-1}\\
=&a_{m,0}+\sum_{k=0}^{n-m-1}\overline{\alpha}_{m+k}\alpha_{k}+\sum_{j=1}^{m-1}\beta_{m,j}a_{j,0}+\sum_{k=2}^{n-m}\sum_{j=k}^{m+k-2}\beta_{m+k-1,j}a_{j,k-1}\\
=&-\overline{\alpha}_{m-1}+\sum_{j=0}^{n-m-1}\overline{\alpha}_{m+j}\alpha_{j}-\overline{\alpha}_{m}\sum_{j=1}^{m-1}\alpha_{j}\overline{\alpha}_{j-1}+\sum_{k=1}^{n-m-1}\sum_{j=k+1}^{m+k-1}\beta_{m+k,j}a_{j,k}\\
=&-\overline{\alpha}_{m-1}+\sum_{j=0}^{n-m-1}\overline{\alpha}_{j+m}\alpha_{j}-\overline{\alpha}_{m}\sum_{j=1}^{m-1}\alpha_{j}\overline{\alpha}_{j-1}\\
&+\sum_{k=1}^{n-m-1}\sum_{l=1}^{m-1}\beta_{k+m,k+l}a_{k+l,k}.\qedhere
\end{align*}
\end{proof}

\begin{rem}
It is convenient to rewrite (2.3) as in the following form
\begin{align}
a_{k_{0}+l_{0},k_{0}}=&-\overline{\alpha}_{l_{0}-1}+\sum_{j=0}^{k_{0}-1}\overline{\alpha}_{j+l_{0}}\alpha_{j}-\overline{\alpha}_{l_{0}}\sum_{j=1}^{l_{0}-1}\alpha_{j}\overline{\alpha}_{j-1}\nonumber\\
&+\sum_{k=1}^{k_{0}-1}\sum_{l=1}^{l_{0}-1}\beta_{k+l_{0},k+l}a_{k+l,k}
\end{align}
for $k_{0}\in \mathbb{N}_{0}$ and $l_{0}\in \mathbb{N}$ in which $\mathbb{N}_{0}=\mathbb{N}\cup\{0\}$.
\end{rem}

In particular, we have

\begin{cor}
For any $n\in \mathbb{N}$,
\begin{equation}
a_{n,n-1}=-\overline{\alpha}_{0}+\sum_{j=0}^{n-2}\overline{\alpha}_{j+1}\alpha_{j}=\sum_{j=0}^{n-1}\overline{\alpha}_{j}\alpha_{j-1}.
\end{equation}
\end{cor}

\begin{thm}
Let
\begin{equation}
G(k,l)=-\overline{\alpha}_{l-1}+\sum_{j=0}^{k-1}\overline{\alpha}_{j+l}\alpha_{j}-\overline{\alpha}_{l}\sum_{j=1}^{l-1}\alpha_{j}\overline{\alpha}_{j-1}, \,\,k,l\geq 1,
\end{equation}
then for any $k_{0}\in \mathbb{N}_{0}$ and $l_{0}\in \mathbb{N}$,
\begin{align}
&a_{k_{0}+l_{0},k_{0}}=G(k_{0},l_{0})\nonumber\\
&+\sum_{p=1}^{l_{0}-1}\sum_{l_{1}=p}^{l_{0}-1}\cdots\sum_{l_{p}=1}^{l_{p-1}-1}\sum_{k_{1}=p}^{k_{0}-1}\cdots\sum_{k_{p}=1}^{k_{p-1}-1}
\beta_{k_{1}+l_{0},k_{1}+l_{1}}\beta_{k_{2}+l_{1},k_{2}+l_{2}}\cdots\beta_{k_{p}+l_{p-1},k_{p}+l_{p}}G(k_{p},l_{p})\nonumber\\
&\triangleq G(k_{0},l_{0})+\sum_{p=1}^{l_{0}-1}\prod_{s=1}^{p}\left(\sum_{l_{s}=p-s+1}^{l_{s-1}-1}\sum_{k_{s}=p-s+1}^{k_{s-1}-1}\right)\prod_{s=1}^{p}\beta_{k_{s}+l_{s-1},k_{s}+l_{s}}G(k_{p},l_{p}).
\end{align}
\end{thm}

\begin{rem}
For convenient, in (2.13) and what follows, we always use the following notation of multiple sums
\begin{align*}
\prod_{s=1}^{p}\left(\sum_{t_{s}=p-s+1}^{t_{s-1}-1}\right)\prod_{s=1}^{p}c_{t_{s}}=\sum_{t_{1}=p}^{t_{0}-1}\sum_{t_{2}=p-1}^{t_{1}-1}\cdots\sum_{t_{p}=1}^{t_{p-1}-1}c_{t_{1}}c_{t_{2}}
\cdots c_{t_{p}}.
\end{align*}
\end{rem}
\begin{rem}
Interchanging the subscripts $k_{s}$ and $l_{s}$, we also have
\begin{align}
a_{k_{0}+l_{0},k_{0}}=G(k_{0},l_{0})+\sum_{p=1}^{l_{0}-1}\prod_{s=1}^{p}\left(\sum_{k_{s}=p-s+1}^{k_{s-1}-1}\sum_{l_{s}=p-s+1}^{l_{s-1}-1}\right)\prod_{s=1}^{p}\beta_{k_{s}+l_{s-1},k_{s}+l_{s}}G(k_{p},l_{p}).
\end{align}
By a convention, $\prod_{j=s}^{t}c_{j}=1$ as $t<s$, (2.13) and (2.14) can be rewritten as
\begin{align}
a_{k_{0}+l_{0},k_{0}}&=\sum_{p=0}^{l_{0}-1}\prod_{s=1}^{p}\left(\sum_{l_{s}=p-s+1}^{l_{s-1}-1}\sum_{k_{s}=p-s+1}^{k_{s-1}-1}\right)\prod_{s=1}^{p}\beta_{k_{s}+l_{s-1},k_{s}+l_{s}}G(k_{p},l_{p})\\
&=\sum_{p=0}^{l_{0}-1}\prod_{s=1}^{p}\left(\sum_{k_{s}=p-s+1}^{k_{s-1}-1}\sum_{l_{s}=p-s+1}^{l_{s-1}-1}\right)\prod_{s=1}^{p}\beta_{k_{s}+l_{s-1},k_{s}+l_{s}}G(k_{p},l_{p}).
\end{align}
Observing (2.13)-(2.16), one can see that $G(k,l)$ are bricks for these expressions of coefficients of OPUC in such form.
\end{rem}

\begin{proof} In order to get (2.13), we take the method of induction.

By (2.10) or (2.11), we have
\begin{equation}
a_{k+1,k}=-\overline{\alpha}_{0}+\sum_{j=0}^{k-1}\overline{\alpha}_{j+1}\alpha_{j}=\sum_{j=0}^{k}\overline{\alpha}_{j}\alpha_{j-1}=G(k,1)
\end{equation}
for any $k\in \mathbb{N}_{0}$. This is to say that (2.13) holds for $k\in \mathbb{N}_{0}$ and $l_{0}=1$.

Suppose that (2.13) holds for any $1\leq k< k_{0}$ and $1\leq l< l_{0}$. In what follows, we prove it for $k=k_{0}$ and $l=l_{0}$. By (2.10) and these assumptions, we have
\begin{align}
&a_{k_{0}+l_{0},k_{0}}=G(k_{0},l_{0})+\sum_{k_{1}=1}^{k_{0}-1}\sum_{l_{1}=1}^{l_{0}-1}\beta_{k_{1}+l_{0},k_{1}+l_{1}}a_{k_{1}+l_{1},k_{1}}\nonumber\\
=&G(k_{0},l_{0})+\sum_{k_{1}=1}^{k_{0}-1}\sum_{l_{1}=1}^{l_{0}-1}\beta_{k_{1}+l_{0},k_{1}+l_{1}}\Big(G(k_{1},l_{1})\nonumber\\
&+\sum_{p=1}^{l_{1}-1}\left(\prod_{s=1}^{p}\sum_{l_{s+1}=p-s+1}^{l_{s}-1}\right)\left(\prod_{s=1}^{p}\sum_{k_{s+1}=p-s+1}^{k_{s}-1}\right)
\prod_{s=1}^{p}\beta_{k_{s+1}+l_{s},k_{s+1}+l_{s+1}}G(k_{p+1},l_{p+1})\Big)\nonumber\\
=&G(k_{0},l_{0})+\sum_{l_{1}=1}^{l_{0}-1}\sum_{k_{1}=1}^{k_{0}-1}\beta_{k_{1}+l_{0},k_{1}+l_{1}}G(k_{1},l_{1})\nonumber\\
&+\sum_{l_{1}=1}^{l_{0}-1}\sum_{p=1}^{l_{1}-1}\left(\prod_{s=1}^{p}\sum_{l_{s+1}=p-s+1}^{l_{s}-1}\right)
\sum_{k_{1}=1}^{k_{0}-1}\left(\prod_{s=1}^{p}\sum_{k_{s+1}=p-s+1}^{k_{s}-1}\right)
\beta_{k_{1}+l_{0},k_{1}+l_{1}}\nonumber\\
&\prod_{s=1}^{p}\beta_{k_{s+1}+l_{s},k_{s+1}+l_{s+1}}G(k_{p+1},l_{p+1})\nonumber\\
=&G(k_{0},l_{0})+\sum_{l_{1}=1}^{l_{0}-1}\sum_{k_{1}=1}^{k_{0}-1}\beta_{k_{1}+l_{0},k_{1}+l_{1}}G(k_{1},l_{1})\nonumber\\
&+\sum_{l_{1}=2}^{l_{0}-1}\sum_{p=1}^{l_{1}-1}\left(\prod_{s=1}^{p}\sum_{l_{s+1}=p-s+1}^{l_{s}-1}\right)
\left(\prod_{s=1}^{p+1}\sum_{k_{s}=(p+1)-s+1}^{k_{s-1}-1}\right)
\prod_{s=1}^{p+1}\beta_{k_{s}+l_{s-1},k_{s}+l_{s}}G(k_{p+1},l_{p+1})\nonumber\\
=&G(k_{0},l_{0})+\sum_{l_{1}=1}^{l_{0}-1}\sum_{k_{1}=1}^{k_{0}-1}\beta_{k_{1}+l_{0},k_{1}+l_{1}}G(k_{1},l_{1})\nonumber\\
&+\sum_{p=1}^{l_{0}-2}\sum_{l_{1}=p+1}^{l_{0}-1}\left(\prod_{s=1}^{p}\sum_{l_{s+1}=p-s+1}^{l_{s}-1}\right)
\left(\prod_{s=1}^{p+1}\sum_{k_{s}=(p+1)-s+1}^{k_{s-1}-1}\right)
\prod_{s=1}^{p+1}\beta_{k_{s}+l_{s-1},k_{s}+l_{s}}G(k_{p+1},l_{p+1})\nonumber\\
=&G(k_{0},l_{0})+\sum_{l_{1}=1}^{l_{0}-1}\sum_{k_{1}=1}^{k_{0}-1}\beta_{k_{1}+l_{0},k_{1}+l_{1}}G(k_{1},l_{1})\nonumber\\
&+\sum_{p+1=2}^{l_{0}-1}\sum_{l_{1}=p+1}^{l_{0}-1}\left(\prod_{s+1=2}^{p+1}\sum_{l_{s+1}=(p+1)-(s+1)+1}^{l_{(s+1)-1}-1}\right)
\left(\prod_{s=1}^{p+1}\sum_{k_{s}=(p+1)-s+1}^{k_{s-1}-1}\right)\nonumber\\
&\prod_{s=1}^{p+1}\beta_{k_{s}+l_{s-1},k_{s}+l_{s}}G(k_{p+1},l_{p+1})\nonumber\\
=&G(k_{0},l_{0})+\sum_{l_{1}=1}^{l_{0}-1}\sum_{k_{1}=1}^{k_{0}-1}\beta_{k_{1}+l_{0},k_{1}+l_{1}}G(k_{1},l_{1})\nonumber\\
&+\sum_{p+1=2}^{l_{0}-1}\sum_{l_{1}=p+1}^{l_{0}-1}\left(\prod_{s=2}^{p+1}\sum_{l_{s}=(p+1)-s+1}^{l_{s-1}-1}\right)
\left(\prod_{s=1}^{p+1}\sum_{k_{s}=(p+1)-s+1}^{k_{s-1}-1}\right)\nonumber\\
&\prod_{s=1}^{p+1}\beta_{k_{s}+l_{s-1},k_{s}+l_{s}}G(k_{p+1},l_{p+1})\nonumber
\end{align}
\begin{align}
=&G(k_{0},l_{0})+\sum_{l_{1}=1}^{l_{0}-1}\sum_{k_{1}=1}^{k_{0}-1}\beta_{k_{1}+l_{0},k_{1}+l_{1}}G(k_{1},l_{1})\nonumber\\
&+\sum_{p+1=2}^{l_{0}-1}\left(\prod_{s=1}^{p+1}\sum_{l_{s}=(p+1)-s+1}^{l_{s-1}-1}\right)
\left(\prod_{s=1}^{p+1}\sum_{k_{s}=(p+1)-s+1}^{k_{s-1}-1}\right)
\prod_{s=1}^{p+1}\beta_{k_{s}+l_{s-1},k_{s}+l_{s}}G(k_{p+1},l_{p+1})\nonumber\\
=&G(k_{0},l_{0})+\sum_{l_{1}=1}^{l_{0}-1}\sum_{k_{1}=1}^{k_{0}-1}\beta_{k_{1}+l_{0},k_{1}+l_{1}}G(k_{1},l_{1})\nonumber\\
&+\sum_{p^{\prime}=2}^{l_{0}-1}\left(\prod_{s=1}^{p^{\prime}}\sum_{l_{s}=p^{\prime}-s+1}^{l_{s-1}-1}\right)
\left(\prod_{s=1}^{p^{\prime}}\sum_{k_{s}=p^{\prime}-s+1}^{k_{s-1}-1}\right)
\prod_{s=1}^{p^{\prime}}\beta_{k_{s}+l_{s-1},k_{s}+l_{s}}G(k_{p^{\prime}},l_{p^{\prime}})\nonumber\\
=&G(k_{0},l_{0})+\sum_{p^{\prime}=1}^{l_{0}-1}\left(\prod_{s=1}^{p^{\prime}}\sum_{l_{s}=p^{\prime}-s+1}^{l_{s-1}-1}\right)
\left(\prod_{s=1}^{p^{\prime}}\sum_{k_{s}=p^{\prime}-s+1}^{k_{s-1}-1}\right)
\prod_{s=1}^{p^{\prime}}\beta_{k_{s}+l_{s-1},k_{s}+l_{s}}G(k_{p^{\prime}},l_{p^{\prime}})\nonumber\\
=&G(k_{0},l_{0})+\sum_{p^{\prime}=1}^{l_{0}-1}\prod_{s=1}^{p^{\prime}}\left(\sum_{l_{s}=p^{\prime}-s+1}^{l_{s-1}-1}\sum_{k_{s}=p^{\prime}-s+1}^{k_{s-1}-1}\right)
\prod_{s=1}^{p^{\prime}}\beta_{k_{s}+l_{s-1},k_{s}+l_{s}}G(k_{p^{\prime}},l_{p^{\prime}}).\nonumber
\end{align}
In the fourth equality, we have used the usual convention that $\sum_{j=s}^{t}c_{j}=0$ as $t<s$ for the sums involving $l_{1}$ and $k_{1}$.
\end{proof}

\begin{rem}
Set $k_{s+1}=k_{s}^{\prime}$ and $l_{s+1}=l_{s}^{\prime}$, by assumption of induction, in the above second equality, we have
\begin{align*}
&a_{k_{1}+l_{1},k_{1}}=a_{k_{0}^{\prime}+l_{0}^{\prime},k_{0}^{\prime}}\nonumber\\
=&G(k_{0}^{\prime},l_{0}^{\prime})+\sum_{p=1}^{l_{0}^{\prime}-1}\prod_{s=1}^{p}\left(\sum_{k_{s}^{\prime}=p-s+1}^{k_{s-1}^{\prime}-1}\sum_{l_{s}^{\prime}=p-s+1}^{l_{s-1}^{\prime}-1}\right)
\prod_{s=1}^{p}\beta_{k_{s}^{\prime}+l_{s-1}^{\prime},k_{s}^{\prime}+l_{s}^{\prime}}G(k_{p}^{\prime},l_{p}^{\prime})\nonumber\\
=&G(k_{1},l_{1})+\sum_{p=1}^{l_{1}-1}\prod_{s=1}^{p}\left(\sum_{l_{s+1}=p-s+1}^{l_{s}-1}\sum_{k_{s+1}=p-s+1}^{k_{s}-1}\right)\prod_{s=1}^{p}\beta_{k_{s+1}+l_{s},k_{s+1}+l_{s+1}}G(k_{p+1},l_{p+1}).
\end{align*}
\end{rem}

By (2.12) and (2.14) (or (2.16)), we further have the following expression of coefficients of OPUC in another form.

\begin{thm}
For any $k_{0}\in \mathbb{N}_{0}$ and $l_{0}\in \mathbb{N}$,
\begin{align}
a_{k_{0}+l_{0}, k_{0}}=\sum_{p=0}^{l_{0}-1}\prod_{s=1}^{p}\left(\sum_{k_{s}=p-s}^{k_{s-1}-1}\sum_{l_{s}=p-s+1}^{l_{s-1}-1}\right)\prod_{s=1}^{p}\beta_{k_{s}+l_{s-1},k_{s}+l_{s}}
\sum_{j=0}^{k_{p}}\beta_{j+l_{p}-1,j-1},
\end{align}
where $\beta_{s,t}=\overline{\alpha}_{s}\alpha_{t}$.
\end{thm}

\begin{proof}
Rewrite (2.12) by
\begin{equation}
G(k,l)=\sum_{j=0}^{k}\overline{\alpha}_{j+l-1}\alpha_{j-1}-\overline{\alpha}_{l}\sum_{j=1}^{l-1}\alpha_{j}\overline{\alpha}_{j-1}.
\end{equation}
In order to get (2.18), we consider two consecutive terms indexed by $p$ in (2.16) and split them by using (2.19). More precisely, denote
\begin{align}
\mathrm{Term}_{p}=&\prod_{s=1}^{p}\left(\sum_{k_{s}=p-s+1}^{k_{s-1}-1}\sum_{l_{s}=p-s+1}^{l_{s-1}-1}\right)\prod_{s=1}^{p}\beta_{k_{s}+l_{s-1},k_{s}+l_{s}}G(k_{p},l_{p})\nonumber\\
=&\prod_{s=1}^{p}\left(\sum_{k_{s}=p-s+1}^{k_{s-1}-1}\sum_{l_{s}=p-s+1}^{l_{s-1}-1}\right)\prod_{s=1}^{p}\beta_{k_{s}+l_{s-1},k_{s}+l_{s}}\sum_{j=0}^{k_{p}}\beta_{j+l_{p}-1,j-1}\nonumber\\
&-\prod_{s=1}^{p}\left(\sum_{k_{s}=p-s+1}^{k_{s-1}-1}\sum_{l_{s}=p-s+1}^{l_{s-1}-1}\right)\prod_{s=1}^{p}\beta_{k_{s}+l_{s-1},k_{s}+l_{s}}\overline{\alpha}_{l_{p}}\sum_{j=1}^{l_{p}-1}\alpha_{j}\overline{\alpha}_{j-1}\nonumber\\
=&\prod_{s=1}^{p}\left(\sum_{k_{s}=p-s+1}^{k_{s-1}-1}\sum_{l_{s}=p-s+1}^{l_{s-1}-1}\right)\prod_{s=1}^{p}\beta_{k_{s}+l_{s-1},k_{s}+l_{s}}\sum_{j=0}^{k_{p}}\beta_{j+l_{p}-1,j-1}\nonumber\\
&-\prod_{s=1}^{p}\left(\sum_{k_{s}=p-s+1}^{k_{s-1}-1}\sum_{l_{s}=(p+1)-s+1}^{l_{s-1}-1}\right)\prod_{s=1}^{p}\beta_{k_{s}+l_{s-1},k_{s}+l_{s}}\overline{\alpha}_{l_{p}}\sum_{j=1}^{l_{p}-1}\alpha_{j}\overline{\alpha}_{j-1}\\
\triangleq&\mathrm{Term}_{p,1}+\mathrm{Term}_{p,2}.
\end{align}
Then
\begin{align}
&\mathrm{Term}_{p,2}+\mathrm{Term}_{p+1,1}\nonumber\\
=&-\prod_{s=1}^{p}\left(\sum_{k_{s}=p-s+1}^{k_{s-1}-1}\sum_{l_{s}=(p+1)-s+1}^{l_{s-1}-1}\right)\prod_{s=1}^{p}\beta_{k_{s}+l_{s-1},k_{s}+l_{s}}\overline{\alpha}_{l_{p}}\sum_{j=1}^{l_{p}-1}\alpha_{j}\overline{\alpha}_{j-1}\nonumber\\
&+\prod_{s=1}^{p+1}\left(\sum_{k_{s}=(p+1)-s+1}^{k_{s-1}-1}\sum_{l_{s}=(p+1)-s+1}^{l_{s-1}-1}\right)\prod_{s=1}^{p+1}\beta_{k_{s}+l_{s-1},k_{s}+l_{s}}\sum_{j=0}^{k_{p+1}}\beta_{j+l_{p+1}-1,j-1}\nonumber\\
=&-\prod_{s=1}^{p}\left(\sum_{k_{s}=p-s+1}^{k_{s-1}-1}\sum_{l_{s}=(p+1)-s+1}^{l_{s-1}-1}\right)\prod_{s=1}^{p}\beta_{k_{s}+l_{s-1},k_{s}+l_{s}}\overline{\alpha}_{l_{p}}\sum_{j=1}^{l_{p}-1}\alpha_{j}\overline{\alpha}_{j-1}\nonumber\\
&+\prod_{s=1}^{p}\left(\sum_{k_{s}=(p+1)-s}^{k_{s-1}-1}\sum_{l_{s}=(p+1)-s+1}^{l_{s-1}-1}\right)\prod_{s=1}^{p}\beta_{k_{s}+l_{s-1},k_{s}+l_{s}}\nonumber\\
&\,\,\,\,\,\,\,\sum_{k_{p+1}=1}^{k_{p}-1}\sum_{l_{p+1}=1}^{l_{p}-1}\beta_{k_{p+1}+l_{p},k_{p+1}+l_{p+1}}\sum_{j=0}^{k_{p+1}}\beta_{j+l_{p+1}-1,j-1}\nonumber\\
=&-\prod_{s=1}^{p}\left(\sum_{k_{s}=p-s+1}^{k_{s-1}-1}\sum_{l_{s}=(p+1)-s+1}^{l_{s-1}-1}\right)\prod_{s=1}^{p}\beta_{k_{s}+l_{s-1},k_{s}+l_{s}}\overline{\alpha}_{l_{p}}\sum_{j=1}^{l_{p}-1}\alpha_{j}\overline{\alpha}_{j-1}\nonumber\\
&+\prod_{s=1}^{p}\left(\sum_{k_{s}=(p+1)-s}^{k_{s-1}-1}\sum_{l_{s}=(p+1)-s+1}^{l_{s-1}-1}\right)\prod_{s=1}^{p}\beta_{k_{s}+l_{s-1},k_{s}+l_{s}}\nonumber
\end{align}
\begin{align}
&\,\,\,\,\,\,\,\sum_{k_{p+1}=0}^{k_{p}-1}\sum_{l_{p+1}=1}^{l_{p}-1}\beta_{k_{p+1}+l_{p},k_{p+1}+l_{p+1}}\sum_{j=0}^{k_{p+1}}\beta_{j+l_{p+1}-1,j-1}\nonumber\\
&-\prod_{s=1}^{p}\left(\sum_{k_{s}=(p+1)-s}^{k_{s-1}-1}\sum_{l_{s}=(p+1)-s+1}^{l_{s-1}-1}\right)\prod_{s=1}^{p}\beta_{k_{s}+l_{s-1},k_{s}+l_{s}}
\sum_{l_{p+1}=1}^{l_{p}-1}\beta_{l_{p},l_{p+1}}\beta_{l_{p+1}-1,-1}\nonumber\\
=&-\prod_{s=1}^{p}\left(\sum_{k_{s}=p-s+1}^{k_{s-1}-1}\sum_{l_{s}=(p+1)-s+1}^{l_{s-1}-1}\right)\prod_{s=1}^{p}\beta_{k_{s}+l_{s-1},k_{s}+l_{s}}\overline{\alpha}_{l_{p}}\sum_{j=1}^{l_{p}-1}\alpha_{j}\overline{\alpha}_{j-1}\nonumber\\
&+\prod_{s=1}^{p}\left(\sum_{k_{s}=(p+1)-s}^{k_{s-1}-1}\sum_{l_{s}=(p+1)-s+1}^{l_{s-1}-1}\right)\prod_{s=1}^{p}\beta_{k_{s}+l_{s-1},k_{s}+l_{s}}\nonumber\\
&\,\,\,\,\,\,\,\sum_{k_{p+1}=0}^{k_{p}-1}\sum_{l_{p+1}=1}^{l_{p}-1}\beta_{k_{p+1}+l_{p},k_{p+1}+l_{p+1}}\sum_{j=0}^{k_{p+1}}\beta_{j+l_{p+1}-1,j-1}\nonumber\\
&+\prod_{s=1}^{p}\left(\sum_{k_{s}=(p+1)-s}^{k_{s-1}-1}\sum_{l_{s}=(p+1)-s+1}^{l_{s-1}-1}\right)\prod_{s=1}^{p}\beta_{k_{s}+l_{s-1},k_{s}+l_{s}}
\overline{\alpha}_{l_{p}}\sum_{l_{p+1}=1}^{l_{p}-1}\alpha_{l_{p+1}}\overline{\alpha}_{l_{p+1}-1}\nonumber\\
=&\prod_{s=1}^{p}\left(\sum_{k_{s}=(p+1)-s}^{k_{s-1}-1}\sum_{l_{s}=(p+1)-s+1}^{l_{s-1}-1}\right)\prod_{s=1}^{p}\beta_{k_{s}+l_{s-1},k_{s}+l_{s}}\nonumber\\
&\,\,\,\,\,\,\,\sum_{k_{p+1}=0}^{k_{p}-1}\sum_{l_{p+1}=1}^{l_{p}-1}\beta_{k_{p+1}+l_{p},k_{p+1}+l_{p+1}}\sum_{j=0}^{k_{p+1}}\beta_{j+l_{p+1}-1,j-1}\nonumber\\
=&\prod_{s=1}^{p+1}\left(\sum_{k_{s}=(p+1)-s}^{k_{s-1}-1}\sum_{l_{s}=(p+1)-s+1}^{l_{s-1}-1}\right)\prod_{s=1}^{p+1}\beta_{k_{s}+l_{s-1},k_{s}+l_{s}}\sum_{j=0}^{k_{p+1}}\beta_{j+l_{p+1}-1,j-1}.
\end{align}
In the third equality of (2.20), we have used the usual convention that $\sum_{j=s}^{t}c_{j}=0$ as $t<s$ for the sums involving $l_{s}$. It is similar to the second equality of (2.22) for the sums involving $k_{s}$

Finally, $G(k_{p},l_{p})$ in the last term of (2.16) has only the first term as in (2.19). That is, for this case of $p=l_{0}-1$,
\begin{align}
G(k_{l_{0}-1},l_{l_{0}-1})=\sum_{j=0}^{k_{l_{0}-1}}\overline{\alpha}_{j}\alpha_{j-1}=\sum_{j=0}^{k_{l_{0}-1}}\overline{\alpha}_{j+l_{l_{0}-1}-1}\alpha_{j-1}
\end{align}
since $l_{l_{0}-1}=1$. Thus by (2.21)-(2.23), (2.18) immediately follows from (2.16).
\end{proof}

From (2.18), we obtain the following expression for coefficients of OPUC due to Golinskii and Zlato\v{s} (see \cite{gz}).

\begin{cor}
For any $k_{0}\in \mathbb{N}_{0}$ and $l_{0}\in \mathbb{N}$,
\begin{align}
a_{k_{0}+l_{0},k_{0}}=\sum_{\substack{\sum_{1}^{j}a_{l}=l_{0}\\j,a_{l}\geq 1}}\sum_{\substack{k_{1}<k_{0}+l_{0}\\
k_{2}<k_{1}-a_{1}\\\cdots\\k_{j}<k_{j-1}-a_{j-1}}}\overline{\alpha}_{k_{1}}\alpha_{k_{1}-a_{1}}\cdots\overline{\alpha}_{k_{j}}\alpha_{k_{j}-a_{j}}.
\end{align}
\end{cor}

\begin{proof}
Noting that
\begin{equation}
\beta_{k_{s}+l_{s-1},k_{s}+l_{s}}=\overline{\alpha}_{k_{s}+l_{s-1}}\alpha_{k_{s}+l_{s}}=\overline{\alpha}_{k_{s}+l_{s-1}}\alpha_{(k_{s}+l_{s-1})-(l_{s-1}-l_{s})}, \,\,1\leq s\leq p,
\end{equation}
\begin{equation}
\beta_{j+l_{p}-1,j-1}=\overline{\alpha}_{j+l_{p}-1}\alpha_{j-1}=\overline{\alpha}_{j+l_{p}-1}\alpha_{(j+l_{p}-1)-l_{p}}, \,\,0\leq p\leq l_{0}-1
\end{equation}
and
\begin{equation}
\sum_{s=1}^{p}(l_{s-1}-l_{s})+l_{p}=l_{0}
\end{equation}
as well as
\begin{equation}
l_{s}\leq l_{s-1}-1, \,\,\,1\leq s\leq p,
\end{equation}
thus (2.24) follows from (2.18).
\end{proof}

\begin{rem}
Theoretically, one can deduce (2.18) from (2.24). In fact, they are same but in different forms. Nevertheless, the derivation from (2.24) to (2.18) may not always be easy. However, in this paper, (2.18) is more useful than (2.24) to our approach.
\end{rem}
In the end of this section, we give some interesting consequences of Theorem 2.8 about the coefficients of OPUC, Verblunsky coefficients and the moments of $\mu$ although these results will not be used in this paper.

\begin{thm} For $m\in \mathbb{N}_{0}$ and $n\in \mathbb{N}$ with $0\leq m\leq n$,
\begin{align}
\alpha_{m-1}\prod_{j=m}^{n-1}\big(1-|\alpha_{j}|^{2}\big)=-\sum_{k=m}^{n}b_{k,m}\overline{a}_{n,n-k},
\end{align}
where
\begin{equation}
b_{k,l}=(-1)^{k-l}\left|\begin{array}{cccccc}
                    a_{k,k-1} & 1 & 0 & \cdots & 0 & 0 \\
                    a_{k,k-2} & a_{k-1,k-2}& 1 & \cdots & 0 & 0 \\
                    a_{k,k-3} & a_{k-1,k-3} & a_{k-2,k-3} &\cdots & 0 & 0 \\
                    \cdots & \cdots & \cdots &\cdots & \cdots & \cdots \\
                    a_{k,l+2} & a_{k-1,l+2}& a_{k-2,l+2} & \cdots & 1 & 0 \\
                    a_{k,l+1} & a_{k-1,l+1}& a_{k-2,l+1} & \cdots & a_{l+2,l+1} & 1 \\
                    a_{k,l} & a_{k-1,l} & a_{k-2,l} & \cdots & a_{l+2,l} & a_{l+1,l}
                  \end{array}\right|
\end{equation}
for $0\leq l<k\leq n$, and \begin{equation}b_{k,k}=1\end{equation} for $k\in \mathbb{N}_{0}$.
\end{thm}

\begin{proof}
Let
\begin{equation}
z\Phi_{n}(z)=\Phi_{n+1}(z)-\lambda_{n}\Phi_{n}(z)-\lambda_{n-1}\Phi_{n-1}(z)-\cdots-\lambda_{1}\Phi_{1}(z)-\lambda_{0}\Phi_{0}(z),
\end{equation}
by comparing with the coefficients, we have
\begin{align}
\begin{cases}
\lambda_{n}=a_{n+1,n}-a_{n,n-1},\\
\lambda_{n-1}=a_{n+1,n-1}-a_{n,n-2}-\lambda_{n}a_{n,n-1},\\
\lambda_{n-2}=a_{n+1,n-2}-a_{n,n-3}-\lambda_{n}a_{n,n-2}-\lambda_{n-1}a_{n-1,n-2},\\
\hspace{7.5mm}\cdots\\
\lambda_{1}=a_{n+1,1}-a_{n,0}-\lambda_{n}a_{n,1}-\lambda_{n-1}a_{n-1,1}-\cdots-\lambda_{2}a_{2,1},\\
\lambda_{0}=a_{n+1,0}-\lambda_{n}a_{n,0}-\lambda_{n-1}a_{n-1,0}-\cdots-\lambda_{2}a_{2,0}-\lambda_{1}a_{1,0}.
\end{cases}
\end{align}
Denote
\begin{equation*}
A_{n+1}=\left(\begin{array}{ccccccc}
            0 & 0 & 0 & \cdots & 0 & 0 & 0 \\
            a_{n,n-1} & 0 & 0 & \cdots & 0 & 0 & 0 \\
            a_{n,n-2} & a_{n-1,n-2} & 0 & \cdots & 0 & 0 & 0 \\
            \cdots & \cdots & \cdots & \cdots & \cdots & \cdots & \cdots \\
            a_{n,1} & a_{n-1,1} & a_{n-2,1} & \cdots & a_{2,1} & 0 & 0 \\
            a_{n,0} & a_{n-1,0} & a_{n-2,0} & \cdots & a_{2,0} & a_{1,0} & 0 \\
          \end{array}
       \right),
\end{equation*}
\begin{equation*}
\Lambda_{n+1}=\left(
                   \begin{array}{c}
                     \lambda_{n} \\
                     \lambda_{n-1} \\
                     \lambda_{n-2} \\
                     \vdots \\
                     \lambda_{1} \\
                     \lambda_{0} \\
                   \end{array}
                 \right)
\end{equation*}
and
\begin{equation*}
\mathfrak{a}_{n+1}=\left(
                   \begin{array}{c}
                     a_{n+1,n}-a_{n,n-1} \\
                     a_{n+1,n-1}-a_{n,n-2} \\
                     a_{n+1,n-2}-a_{n,n-3} \\
                     \vdots \\
                     a_{n+1,1}-a_{n,0} \\
                     a_{n+1,0} \\
                   \end{array}
                 \right),
\end{equation*}
then (2.33) is equivalent to
\begin{equation}
(I_{n+1}+A_{n+1})\Lambda_{n+1}=\mathfrak{a}_{n+1},
\end{equation}
where $I_{n+1}$ is the identity matrix of order $n+1$.
Therefore
\begin{equation}
\Lambda_{n+1}=(I_{n+1}+A_{n+1})^{-1}\mathfrak{a}_{n+1},
\end{equation}
where (by direct calculations)
\begin{align}
&(I_{n+1}+A_{n+1})^{-1}\nonumber\\
=&\left(\begin{array}{ccccccc}
            1 & 0 & 0 & \cdots & 0 & 0 & 0 \\
            b_{n,n-1} & 1 & 0 & \cdots & 0 & 0 & 0 \\
           b_{n,n-2} & b_{n-1,n-2} & 1 & \cdots & 0 & 0 & 0 \\
            \cdots & \cdots & \cdots & \cdots & \cdots & \cdots & \cdots \\
           b_{n,1} & b_{n-1,1} & b_{n-2,1} & \cdots & b_{2,1} & 1 & 0 \\
           b_{n,0} & b_{n-1,0} & b_{n-2,0} & \cdots & b_{2,0} & b_{1,0} & 1 \\
          \end{array}
       \right)
\end{align}
with $b_{k,l}$ given by (2.30) for $0\leq l<k\leq n$.

Applying Szeg\H{o} recursion (1.4), we have
\begin{equation*}
\mathfrak{a}_{n+1}=-\overline{\alpha}_{n}\left(
                   \begin{array}{c}
                     \overline{a}_{n,0} \\
                     \overline{a}_{n,1} \\
                     \overline{a}_{n,2} \\
                     \vdots \\
                     \overline{a}_{n,n-1} \\
                     1 \\
                   \end{array}
                 \right).
\end{equation*}

Thus the coefficients $\lambda_{m}$ in (2.32) can be expressed by the coefficients of OPUC, $\Phi_{n}$, as follows
\begin{equation}
\lambda_{m}=-\overline{\alpha}_{n}\sum_{k=m}^{n}b_{k,m}\overline{a}_{n,n-k}, \,\,\,0\leq m\leq n,
\end{equation}
where $b_{k,m}$ are given by (2.30) and (2.31).

Since $\kappa_{n}=\prod_{j=0}^{n-1}\rho_{j}^{-1}$ with $\rho_{j}=\big(1-|\alpha_{j}|^{2}\big)^{\frac{1}{2}}$, then
$$\lambda_{m}=-\kappa_{m}^{2}\langle \Phi_{m},z\Phi_{n}\rangle=-\frac{\kappa_{m}}{\kappa_{n}}\langle \varphi_{m},z\varphi_{n}\rangle
=-\prod_{j=m}^{n-1}\rho_{j}\langle \varphi_{m},z\varphi_{n}\rangle.$$

Therefore, by Proposition 1.5.9 in \cite{sim1},
\begin{align}
\lambda_{m}=\overline{\alpha}_{n}\alpha_{m-1}\prod_{j=m}^{n-1}\rho_{j}^{2}=\overline{\alpha}_{n}\alpha_{m-1}\prod_{j=m}^{n-1}\big(1-|\alpha_{j}|^{2}\big)^{2}.
\end{align}

Hence, by (2.37) and (2.38), we have that
\begin{align*}(**)\hspace{4mm} &\mathrm{If}\,\, \alpha_{n}\neq 0\,\, \mbox{for some}\,\, n\in \mathbb{N},\,\, \mathrm{then}\,\,\alpha_{m-1}\prod_{j=m}^{n-1}\big(1-|\alpha_{j}|^{2}\big)=-\sum_{k=m}^{n}b_{k,m}\overline{a}_{n,n-k}\\
&\mathrm{for}\,\, \mathrm{any}\,\, 0\leq m\leq n.
\end{align*}

For the case $\alpha_{n}=0$, we denote $\alpha_{n}=\alpha_{n}(d\mu)$, $a_{n,k}=a_{n,k}(d\mu)$ and $b_{n,k}=b_{n,k}(d\mu)$. By Verblunsky theorem, we introduce another probability measure $d\mu^{\prime}$ whose Verblunsky coefficitients defined by
\begin{equation}
\alpha_{j}(d\mu^{\prime})=\begin{cases}
\alpha_{j}(d\mu),\,\,j\neq n,\vspace{1mm}\\
1,\,\,\,j=n.
\end{cases}
\end{equation}
Thus $d\mu$ and $d\mu^{\prime}$ have the same Bernstein-Szeg\H{o} measure $d\mu_{n-1}$ such that
\begin{equation}
\alpha_{j}(d\mu_{n-1})=\alpha_{j}(d\mu^{\prime})=\alpha_{j}(d\mu)=\alpha_{j},\,\,\,0\leq j\leq n-1.
\end{equation}

By the fact $(**)$, we have
\begin{equation}
\alpha_{m-1}(d\mu^{\prime})\prod_{j=m}^{n-1}\big(1-|\alpha_{j}(d\mu^{\prime})|^{2}\big)=-\sum_{k=m}^{n}b_{k,m}(d\mu^{\prime})\overline{a}_{n,n-k}(d\mu^{\prime}),\,\,0\leq m\leq n.
\end{equation}

By (2.18) (or (2.24)), (2.40) and the definitions of $a_{n,k}$ and $b_{n,k}$, we obtain
\begin{equation*}
\alpha_{m-1}(d\mu)\prod_{j=m}^{n-1}\big(1-|\alpha_{j}(d\mu)|^{2}\big)=-\sum_{k=m}^{n}b_{k,m}(d\mu)\overline{a}_{n,n-k}(d\mu),\,\,0\leq m\leq n.
\end{equation*}
 That is, (2.29) also holds as $\alpha_{n}=0$.
\end{proof}

\begin{rem} By Szeg\H{o} recursion and (2.32), we have
\begin{equation}
\overline{\alpha}_{n}\Phi_{n}^{*}(z)=-\lambda_{n}\Phi_{n}(z)-\lambda_{n-1}\Phi_{n-1}(z)+\cdots-\lambda_{1}\Phi_{1}(z)-\lambda_{0}\Phi_{0}(z).
\end{equation}
By comparing the coefficients, one can also get (2.29) from (2.42) and Proposition 1.5.8 in \cite{sim1} when $\alpha_{n}\neq 0$.
\end{rem}

In particular, setting $m=0$, we have the following

\begin{cor}
\begin{equation}
\kappa_{n}^{-2}=b_{n,0}\overline{a}_{n,0}+b_{n-1,0}\overline{a}_{n,1}+\cdots+b_{1,0}\overline{a}_{n,n-1}+b_{0,0}\overline{a}_{n,n}=\sum_{k=0}^{n}b_{k,0}\overline{a}_{n,n-k},
\end{equation}
where
\begin{equation}
b_{k,0}=(-1)^{k}\left|\begin{array}{cccccc}
                    a_{k,k-1} & 1 & 0 & \cdots & 0 & 0 \\
                    a_{k,k-2} & a_{k-1,k-2}& 1 & \cdots & 0 & 0 \\
                    a_{k,k-3} & a_{k-1,k-3} & a_{k-2,k-3} &\cdots & 0 & 0 \\
                    \cdots & \cdots & \cdots &\cdots & \cdots & \cdots \\
                    a_{k,2} & a_{k-1,2}& a_{k-2,2} & \cdots & 1 & 0 \\
                    a_{k,1} & a_{k-1,1}& a_{k-2,1} & \cdots & a_{2,1} & 1 \\
                    a_{k,0} & a_{k-1,0} & a_{k-2,0} & \cdots & a_{2,0} & a_{1,0}
                  \end{array}\right|
\end{equation}
for $1\leq k\leq n$ and $b_{0,0}=1$.
\end{cor}

\begin{thm}
Assume that $\alpha\in \ell^{2}$, or equivalently $$\int_{0}^{2\pi}\log w(\theta)\frac{d\theta}{2\pi}>-\infty,$$ then

\begin{itemize}
         \item [(1)] \begin{equation} \lim_{n\rightarrow \infty}\sum_{k=0}^{n}b_{k,0}\overline{a}_{n,n-k}=\int_{0}^{2\pi}\log w(\theta)\frac{d\theta}{2\pi};
\end{equation}
         \item [(2)] \begin{align}
&\lim_{n\rightarrow \infty}\sum_{k=m}^{n}b_{k,m}\overline{a}_{n,n-k}=-\alpha_{m-1}\prod_{j=m}^{\infty}
\big(1-|\alpha_{j}|^{2}\big)\nonumber\\
=&-\alpha_{m-1}\prod_{j=0}^{m-1}\big(1-|\alpha_{j}|^{2}\big)^{-1}\int_{0}^{2\pi}\log w(\theta)\frac{d\theta}{2\pi};
\end{align}
         \item [(3)] \begin{align}
\lim_{n\rightarrow \infty}\frac{\displaystyle\sum_{k=m}^{n}b_{k,m}\overline{a}_{n,n-k}}{\displaystyle\sum_{k=0}^{n}b_{k,0}\overline{a}_{n,n-k}}=-\alpha_{m-1}\prod_{j=0}^{m-1}
\big(1-|\alpha_{j}|^{2}\big)^{-1},
\end{align}
\end{itemize}
where $b_{\cdot,\cdot}$ are given by (2.30) and (2.31) in which $a_{\cdot,\cdot}$ is given by (2.1).
\end{thm}

\begin{proof}
By the assumption and Szeg\H{o} theorem, (2.45) follows from (2.43), (2.46) follows from (2.29), (2.47) follows from (2.45) and (2.46).
\end{proof}

\begin{thm}
Let $c_{n}$ be the moments of $\mu$ defined by $$c_{n}=\int_{0}^{2\pi}e^{-in\theta}d\mu(\theta),$$ then
\begin{align}
c_{n}=\sum_{l=0}^{n-1}\overline{b}_{n-1,l}\alpha_{l}\kappa_{l}^{-2}=\alpha_{n-1}\prod_{k=0}^{n-2}\big(1-|\alpha_{k}|^{2}\big)
+\sum_{l=0}^{n-2}\overline{b}_{n-1,l}\alpha_{l}\kappa_{l}^{-2},\,\,n\in \mathbb{N},
\end{align}
where $b_{\cdot,\cdot}$ are given by (2.30) and (2.31) in which $a_{\cdot,\cdot}$ are given by (2.18) (or (2.24)).
\end{thm}

\begin{proof}
Let
\begin{align}
C_{n+1}=I_{n+1}+A_{n+1}&=\left(\begin{array}{ccccccc}
            1 & 0 & 0 & \cdots & 0 & 0 & 0 \\
            a_{n,n-1} & 1 & 0 & \cdots & 0 & 0 & 0 \\
            a_{n,n-2} & a_{n-1,n-2} & 1 & \cdots & 0 & 0 & 0 \\
            \cdots & \cdots & \cdots & \cdots & \cdots & \cdots & \cdots \\
            a_{n,1} & a_{n-1,1} & a_{n-2,1} & \cdots & a_{2,1} & 1 & 0 \\
            a_{n,0} & a_{n-1,0} & a_{n-2,0} & \cdots & a_{2,0} & a_{1,0} & 1 \\
          \end{array}
       \right),
\end{align}
then the columns of $C_{n+1}$ are the coefficients of $\Phi_{k}$, $0\leq k\leq n$, when $\Phi_{k}$ is viewed as a polynomial of order $n$ with the coefficients of $z^{l}$, $k+1\leq l\leq n$ setting to be $0$. That is,
\begin{equation}
\left(
z^{n}\,\, z^{n-1}\,\, \cdots\,\,z\,\, 1
\right)C_{n+1}=\left(
                 \Phi_{n}(z)\,\, \Phi_{n-1}(z)\,\, \cdots\,\, \Phi_{1}(z) \,\,\Phi_{0}(z)
               \right)
\end{equation}
or
\begin{equation}
C_{n+1}^{T}\left(
\begin{array}{c}
z^{n} \\
 z^{n-1} \\
 \vdots \\
 z \\
 1 \\
 \end{array}
 \right)
=\left(
\begin{array}{c}
\Phi_{n}(z) \\
\Phi_{n-1}(z) \\
\vdots \\
\Phi_{1}(z) \\
\Phi_{0}(z) \\
\end{array}
\right).
\end{equation}
By (2.51) and $\overline{\alpha}_{n}\kappa_{n}^{-2}=\int_{0}^{2\pi}z\Phi_{n}(z) d\mu$ (see (1.5.95) in \cite{sim1}), we have
\begin{equation}
C_{n+1}^{T}\left(
\begin{array}{c}
\overline{c}_{n+1} \\
 \overline{c}_{n} \\
 \vdots \\
 \overline{c}_{2} \\
 \overline{c}_{1} \\
 \end{array}
 \right)
=\left(
\begin{array}{c}
\overline{\alpha}_{n}\kappa_{n}^{-2} \\
\overline{\alpha}_{n-1}\kappa_{n-1}^{-2} \\
\vdots \\
\overline{\alpha}_{1}\kappa_{1}^{-2} \\
\overline{\alpha}_{0} \\
\end{array}
\right).
\end{equation}

Therefore, by (2.36),
\begin{align}
&\left(
\begin{array}{c}
c_{n+1} \\
 c_{n} \\
 \vdots \\
 c_{2} \\
 c_{1} \\
 \end{array}
 \right)
=\left(\overline{C_{n+1}^{T}}\right)^{-1}\left(
\begin{array}{c}
\alpha_{n}\kappa_{n}^{-2} \\
\alpha_{n-1}\kappa_{n-1}^{-2} \\
\vdots \\
\alpha_{1}\kappa_{1}^{-2} \\
\alpha_{0} \\
\end{array}
\right)\nonumber\\
=&\left(\begin{array}{ccccccc}
            1 & \overline{b}_{n,n-1} & \overline{b}_{n,n-2} & \cdots  & \overline{b}_{n,1} & \overline{b}_{n,0} \\
            0 & 1 & \overline{b}_{n-1,n-2} & \cdots  & \overline{b}_{n-1,1} & \overline{b}_{n-1,0} \\
          0 & 0 & 1 & \cdots &  0 & 0 \\
            \cdots & \cdots & \cdots & \cdots  & \cdots & \cdots \\
           0 & 0 & 0 & \cdots  & 1 & \overline{b}_{1,0} \\
           0 & 0 & 0 & \cdots  & 0 & 1 \\
          \end{array}
       \right)\left(
\begin{array}{c}
\alpha_{n}\kappa_{n}^{-2} \\
\alpha_{n-1}\kappa_{n-1}^{-2} \\
\vdots \\
\alpha_{1}\kappa_{1}^{-2} \\
\alpha_{0} \\
\end{array}
\right).
\end{align}
Thus
\begin{align*}
c_{n}=\sum_{l=0}^{n-1}\overline{b}_{n-1,l}\alpha_{l}\kappa_{l}^{-2}=\alpha_{n-1}\prod_{k=0}^{n-2}\big(1-|\alpha_{k}|^{2}\big)
+\sum_{l=0}^{n-2}\overline{b}_{n-1,l}\alpha_{l}\kappa_{l}^{-2},\,\,n\in \mathbb{N}.
\end{align*}
\end{proof}

\begin{rem}
Noting (2.18) (or (2.24)), (2.48) is Verblunsky's formula in a new form (see (1.5.53) in \cite{sim1}).
\end{rem}

\section{Logarithmic moments and an algorithm for them}

In this section, we consider the moments of $\log w(\theta)\frac{d\theta}{2\pi}$ when $\alpha\in \ell^{2}$ which we call them logarithmic moments of $d\mu$. We will provide an algorithm to calculate them.

By Szeg\H{o} theorem, $\log w\in L^{1}(d\mu)$ is equivalent to $\alpha\in \ell^{2}$. Therefore, given $\alpha\in \ell^{2}$, one can define
\begin{equation}
D(z)=\exp\left(\int_{0}^{2\pi}\frac{e^{i\theta}+z}{e^{i\theta}-z}\log w(\theta)\frac{d\theta}{4\pi}\right),
\end{equation}
which is called Szeg\H{o} function. Moreover, we define the logarithm moments $w_{k}$ of $\mu$ by
\begin{equation}
w_{k}=\int_{0}^{2\pi}e^{-ik\theta}\log w(\theta)\frac{d\theta}{2\pi},\,\,\,k\in \mathbb{N}_{0}
\end{equation}
when $\log w\in L^{1}(d\mu)$. Obviously, $w_{k}$ are Fourier coefficients of $\log w$. That is, $w_{k}=\widehat{\log w}(k)$ in which \,$\widehat{}$ \,\,is Fourier transform defined by
\begin{equation*}
\widehat{f}(k)=\int_{0}^{2\pi}e^{-ik\theta}f(\theta)\frac{d\theta}{2\pi}.
\end{equation*}

In fact, $w_{k}$ are also Taylor coefficients of the logarithm of Szeg\H{o} function, $\log D(z)$, as $z\in \mathbb{D}$ because
\begin{equation}
\log D(z)=\frac{1}{2}w_{0}+\sum_{k=1}^{\infty}w_{k}z^{k},\,\,\,z\in \mathbb{D}.
\end{equation}

Besides (1.6) whose original form is as follows (see (12.3.15) in \cite{sze})
\begin{equation}
\lim_{n\rightarrow\infty}\kappa_{n}=\exp\left(-\int_{0}^{2\pi}\log w(\theta)\frac{d\theta}{4\pi}\right)=\frac{1}{\sqrt{\Theta(w)}}
\end{equation}
in which
\begin{equation}
\Theta(w)=\exp\left(\int_{0}^{2\pi}\log w(\theta)\frac{d\theta}{2\pi}\right),
\end{equation}
Szeg\H{o} actually obtained the following complete asymptotic result (see (12.3.16) in \cite{sze} or (2.4.5) in \cite{sim1})
\begin{equation}
\lim_{n\rightarrow\infty}\varphi_{n}^{*}(z)=\lim_{n\rightarrow\infty}\frac{\Phi_{n}^{*}(z)}{\|\Phi^{*}_{n}\|}=D(z)^{-1}
\end{equation}
uniformly on compact sets in $\mathbb{D}$ when $\log w\in L^{1}(d\mu)$.

In what follows, we always assume that $\log w\in L^{1}(d\mu)$, equivalently $\alpha\in \ell^{2}$.
As in \cite{gz}, if we expand $\sqrt{\Theta(w)}D^{-1}$ as
\begin{equation}
\sqrt{\Theta(w)}D(z)^{-1}=1+d_{1}z+d_{2}z^{2}+\cdots,
\end{equation}
then by (3.4), (3.6) and $\|\Phi^{*}_{n}\|=\kappa_{n}^{-1}$,
\begin{equation}
d_{m}=\lim_{n\rightarrow\infty}\overline{a}_{n,n-m}.
\end{equation}

Set $k_{0}=n-m$, $l_{0}=m$ and let $n\rightarrow\infty$, by Theorem 2.8 and Corollary 2.9, we have
\begin{thm}
Suppose that $\alpha\in\ell^{2}$, then
\begin{align}
&d_{m}=\sum_{p=0}^{m-1}\sum_{k_{1}=p-1}^{\infty}\sum_{l_{1}=p}^{m-1}\prod_{s=2}^{p}\left(\sum_{k_{s}=p-s}^{k_{s-1}-1}\sum_{l_{s}=p-s+1}^{l_{s-1}-1}\right)\prod_{s=1}^{p}\overline{\beta}_{k_{s}+l_{s-1},k_{s}+l_{s}}
\sum_{j=0}^{k_{p}}\overline{\beta}_{j+l_{p}-1,j-1}\\
=&\sum_{p=0}^{m-1}\sum_{k_{1}=p-1}^{\infty}\sum_{l_{1}=p}^{m-1}\alpha_{k_{1}+m}\overline{\alpha}_{k_{1}+l_{1}}\cdots\sum_{k_{p}=0}^{k_{p-1}-1}\sum_{l_{p}=1}^{l_{p-1}-1}\alpha_{k_{p}+l_{p-1}}\overline{\alpha}_{k_{p}+l_{p}}
\sum_{j=0}^{k_{p}}\alpha_{j+l_{p}-1}\overline{\alpha}_{j-1}
\end{align}
for any $m\in \mathbb{N}$, where $k_{0}=\infty$, $l_{0}=m$ and $\beta_{s,t}=\overline{\alpha}_{s}\alpha_{t}$.
\end{thm}

\begin{cor}
Suppose that $\alpha\in\ell^{2}$, then
\begin{align}
d_{m}=\sum_{\substack{\sum_{1}^{j}a_{l}=m\\j,a_{l}\geq 1}}\sum_{\substack{
k_{2}<k_{1}-a_{1}\\\cdots\\k_{j}<k_{j-1}-a_{j-1}}}\alpha_{k_{1}}\overline{\alpha}_{k_{1}-a_{1}}\cdots\alpha_{k_{j}}\overline{\alpha}_{k_{j}-a_{j}}.
\end{align}
for any $m\in \mathbb{N}$.
\end{cor}

\begin{rem}
If fact, by using the convention of $\sum_{j=s}^{t}c_{j}=0$, (3.9) and (3.10) can be rewritten as
\begin{align}
&d_{m}=\sum_{p=0}^{m-1}\sum_{k_{1}=0}^{\infty}\sum_{l_{1}=p}^{m-1}\prod_{s=2}^{p}\left(\sum_{k_{s}=0}^{k_{s-1}-1}\sum_{l_{s}=p-s+1}^{l_{s-1}-1}\right)\prod_{s=1}^{p}\overline{\beta}_{k_{s}+l_{s-1},k_{s}+l_{s}}
\sum_{j=0}^{k_{p}}\overline{\beta}_{j+l_{p}-1,j-1}\tag{3.9a}\\
=&\sum_{p=0}^{m-1}\sum_{k_{1}=0}^{\infty}\sum_{l_{1}=p}^{m-1}\alpha_{k_{1}+m}\overline{\alpha}_{k_{1}+l_{1}}\cdots\sum_{k_{p}=0}^{k_{p-1}-1}\sum_{l_{p}=1}^{l_{p-1}-1}\alpha_{k_{p}+l_{p-1}}\overline{\alpha}_{k_{p}+l_{p}}
\sum_{j=0}^{k_{p}}\alpha_{j+l_{p}-1}\overline{\alpha}_{j-1},\tag{3.10a}
\end{align}
where $k_{0}=\infty$, $l_{0}=m$ and $\beta_{s,t}=\overline{\alpha}_{s}\alpha_{t}$.
\end{rem}

From (3.3) and (3.7), we can easily get
\begin{equation}
w_{k}=\sum_{\substack{\sum_{1}^{j}b_{l}=k\\j,l\geq 1}}\frac{(-1)^{j}}{j}\prod_{l=1}^{j}d_{b_{l}}
\end{equation}
and
\begin{equation}
d_{k}=\sum_{\substack{\sum_{1}^{j}b_{l}=k\\j,l\geq 1}}\frac{(-1)^{j}}{j!}\prod_{l=1}^{j}w_{b_{l}}
\end{equation}
for any $k\in \mathbb{N}$.

By using (3.11) and (3.12) together with some combinatorial manipulations (for instance, Lerch's identity), Golinskii and Zlato\v{s} obtained a striking result about  the expression for $w_{k}$ as follows
\begin{equation}
w_{k}=\sum_{P\in M_{k}^{0}}N(P)\sum_{l=0}^{\infty}\beta(P+l),
\end{equation}
where collection $P=\{(k_{j},a_{j})\}_{j=1}^{i}$ with $i\in \mathbb{N}$, $(k_{j}, a_{j})\in \mathbb{N}_{0}\times\mathbb{N}$ such that $\sum_{j=1}^{i}a_{j}=k$, $P+l=\{(k_{j}+l,a_{j})\}_{j=1}^{i}$, $\beta(P)=\alpha_{k_{1}}\overline{\alpha}_{k_{1}-a_{1}}\cdots\alpha_{k_{i}}\overline{\alpha}_{k_{i}-a_{i}}$, $N(P)$ is a constant dependent on $P$, and $M_{k}^{0}$ is a set of all different collections with certain properties of points in $\mathbb{N}_{0}\times\mathbb{N}$. The detailed information about $N(P)$ and $M_{k}^{0}$ can be seen in \cite{gz} which we don't need in this work.

The thrust of (3.14) is that $w_{k}$ can be expressed in term of sums with a single infinite index although $d_{k}$ is expressed in term of sums with multi-fold infinite indices. Based on this result, we give a direct and computable algorithm to get $w_{k}$ by using (3.10) (or (3.10a)) in Theorem 3.1. For simplicity, we only calculate $w_{1}$, $w_{2}$, $w_{3}$ and $w_{4}$ in this section. General results about $w_{m}$ for any $m\in \mathbb{N}$ are given in Section 5.

To do so, we need the following basic facts.
\begin{prop}
Suppose that $\{a_{n}\}_{n\in \mathbb{N}_{0}}, \{b_{n}\}_{n\in \mathbb{N}_{0}}\in \ell^{1}$, then
\begin{equation}
\sum_{n=0}^{\infty}a_{n}\sum_{k=0}^{n}b_{k}+\sum_{n=0}^{\infty}b_{n}\sum_{k=0}^{n}a_{k}=\sum_{n=0}^{\infty}a_{n}\sum_{n=0}^{\infty}b_{n}+\sum_{n=0}^{\infty}a_{n}b_{n}.
\end{equation}
\end{prop}
\begin{proof}
Note that
\begin{align}
\sum_{k=0}^{n}a_{k}\sum_{k=0}^{n}b_{k}&=\sum_{\substack{k\neq l\\0\leq k,l\leq n}}a_{k}b_{l}+\sum_{k=0}^{n}a_{k}b_{k}\nonumber\\
&=\sum_{k=0}^{n}a_{k}\sum_{l=0}^{k-1}b_{l}+\sum_{k=0}^{n}b_{k}\sum_{l=0}^{k-1}a_{l}+\sum_{k=0}^{n}a_{k}b_{k}\\
&=\sum_{k=0}^{n}a_{k}\sum_{l=0}^{k}b_{l}+\sum_{k=0}^{n}b_{k}\sum_{l=0}^{k}a_{l}-\sum_{k=0}^{n}a_{k}b_{k},
\end{align}
thus (3.15) follows from (3.17) by taking $n\rightarrow \infty$ since $\{a_{n}\}, \{b_{n}\}\in \ell^{1}$.
\end{proof}

\begin{cor}
Suppose that $\{a_{n}\}_{n\in \mathbb{N}_{0}}\in \ell^{1}$, then
\begin{equation}
\sum_{n=0}^{\infty}a_{n}\sum_{k=0}^{n}a_{k}=\frac{1}{2}\left(\sum_{n=0}^{\infty}a_{n}\right)^{2}+\frac{1}{2}\sum_{n=0}^{\infty}a_{n}^{2}.
\end{equation}
\end{cor}

With these preliminaries, we calculate $w_{j}$ for $1\leq j\leq 4$ as $\alpha\in\ell^{2}$ in what follows.\\

\emph{Calculation for $w_{1}$}:  By (3.10),
\begin{equation}
d_{1}=\sum_{j=0}^{\infty}\alpha_{j}\overline{\alpha}_{j-1}.
\end{equation}
Therefore, by (3.12),
\begin{equation}
w_{1}=-d_{1}=-\sum_{j=0}^{\infty}\alpha_{j}\overline{\alpha}_{j-1}.
\end{equation}

\emph{Calculation for $w_{2}$}:  By (3.10a),
\begin{align}
d_{2}=\sum_{j=0}^{\infty}\alpha_{j+1}\overline{\alpha}_{j-1}+\sum_{k=0}^{\infty}\alpha_{k+2}\overline{\alpha}_{k+1}\sum_{j=0}^{k}\alpha_{j}\overline{\alpha}_{j-1}.
\end{align}
By interchanging order of sums, we have
\begin{align}
&\sum_{k=0}^{\infty}\alpha_{k+2}\overline{\alpha}_{k+1}\sum_{j=0}^{k}\alpha_{j}\overline{\alpha}_{j-1}=\sum_{j=0}^{\infty}\alpha_{j}\overline{\alpha}_{j-1}
\sum_{k=j}^{\infty}\alpha_{k+2}\overline{\alpha}_{k+1}\nonumber\\
=&\sum_{j=0}^{\infty}\alpha_{j}\overline{\alpha}_{j-1}\sum_{k=j+2}^{\infty}\alpha_{k}\overline{\alpha}_{k-1}=\left(\sum_{j=0}^{\infty}\alpha_{j}\overline{\alpha}_{j-1}\right)^{2}
-\sum_{j=0}^{\infty}\alpha_{j}\overline{\alpha}_{j-1}\sum_{k=0}^{j+1}\alpha_{k}\overline{\alpha}_{k-1}\nonumber\\
=&d_{1}^{2}-\sum_{j=0}^{\infty}\alpha_{j}\overline{\alpha}_{j-1}\sum_{k=0}^{j+1}\alpha_{k}\overline{\alpha}_{k-1}.
\end{align}

By Corollary 3.5,
\begin{align}
&\sum_{j=0}^{\infty}\alpha_{j}\overline{\alpha}_{j-1}\sum_{k=0}^{j+1}\alpha_{k}\overline{\alpha}_{k-1}=\sum_{j=0}^{\infty}\alpha_{j+1}|\alpha_{j}|^{2}\overline{\alpha}_{j-1}
+\sum_{j=0}^{\infty}\alpha_{j}\overline{\alpha}_{j-1}\sum_{k=0}^{j}\alpha_{k}\overline{\alpha}_{k-1}\nonumber\\
=&\sum_{j=0}^{\infty}\alpha_{j+1}|\alpha_{j}|^{2}\overline{\alpha}_{j-1}+\frac{1}{2}\left(\sum_{j=0}^{\infty}\alpha_{j}\overline{\alpha}_{j-1}\right)^{2}
+\frac{1}{2}\sum_{j=0}^{\infty}\alpha_{j}^{2}\overline{\alpha}_{j-1}^{2}\nonumber\\
=&\sum_{j=0}^{\infty}\alpha_{j+1}|\alpha_{j}|^{2}\overline{\alpha}_{j-1}+\frac{1}{2}d_{1}^{2}
+\frac{1}{2}\sum_{j=0}^{\infty}\alpha_{j}^{2}\overline{\alpha}_{j-1}^{2}
\end{align}

By (3.21)-(3.23), we obtain
\begin{align}
d_{2}=&\sum_{j=0}^{\infty}\alpha_{j+1}\overline{\alpha}_{j-1}-\sum_{j=0}^{\infty}\alpha_{j+1}|\alpha_{j}|^{2}\overline{\alpha}_{j-1}
-\frac{1}{2}\sum_{j=0}^{\infty}\alpha_{j}^{2}\overline{\alpha}_{j-1}^{2}+\frac{1}{2}d_{1}^{2}\nonumber\\
=&\sum_{j=0}^{\infty}\alpha_{j+1}\rho_{j}^{2}\overline{\alpha}_{j-1}-\frac{1}{2}\sum_{j=0}^{\infty}\alpha_{j}^{2}\overline{\alpha}_{j-1}^{2}+\frac{1}{2}d_{1}^{2}.
\end{align}

By (3.12) and (3.24),
\begin{align}
w_{2}=-d_{2}+\frac{1}{2}d_{1}^{2}=-\sum_{j=0}^{\infty}\alpha_{j+1}\rho_{j}^{2}\overline{\alpha}_{j-1}+\frac{1}{2}\sum_{j=0}^{\infty}\alpha_{j}^{2}\overline{\alpha}_{j-1}^{2}.
\end{align}
\begin{rem}
Observing (3.24) and (3.25), we find
\begin{align}
w_{2}=-\mbox{sums with a single infinite index in}\,\, d_{2}.
\end{align}

In particular, these sums come from the only sum with a single infinite index (i.e., $\sum_{j=0}^{\infty}\alpha_{j+1}\overline{\alpha}_{j-1}$ in this case) and the ones (which are not products of some sums, such as $\frac{1}{2}d_{1}^{2}$ in (3.24)) obtained after by interchanging order of sums for general sums with multi-fold infinite indices in (3.10) (in this case, that is $\sum_{j=0}^{\infty}\alpha_{j}\overline{\alpha}_{j-1}\sum_{k=0}^{j+1}\alpha_{k}\overline{\alpha}_{k-1}$). By the result (3.14) of Golinskii and Zlato\v{s} (for short, we call this result Golinskii-Zlato\v{s} single index theorem later), such fact exhibiting here is exact for any $w_{m}$, $m\geq 2$. In deed, from (3.10) and (3.14), we have
\begin{prop}
For any $m\in \mathbb{N}$,
\begin{align*}
w_{m}=-\mbox{sums with a single infinite index in}\,\, d_{m}.
\end{align*}
\end{prop}
So we will deduce $w_{3}$ and $w_{4}$ by taking partly critical computing but not by completely calculating (similar to $w_{2}$) in what follows. The same strategy can be also applied to get all other $w_{m}$, $m\geq 5$ (for the latter, see Section 5 below).
\end{rem}

\emph{Calculation for $w_{3}$}: By (3.10a),
\begin{align}
d_{3}=&\sum_{j=0}^{\infty}\alpha_{j+2}\overline{\alpha}_{j-1}+\sum_{k=0}^{\infty}\alpha_{k+3}\overline{\alpha}_{k+1}\sum_{j=0}^{k}\alpha_{j}\overline{\alpha}_{j-1}+\sum_{k=0}^{\infty}\alpha_{k+3}\overline{\alpha}_{k+2}\sum_{j=0}^{k}\alpha_{j+1}\overline{\alpha}_{j-1}\nonumber\\
&+\sum_{k=0}^{\infty}\alpha_{k+3}\overline{\alpha}_{k+2}\sum_{l=0}^{k-1}\alpha_{l+2}\overline{\alpha}_{l+1}\sum_{j=0}^{l}\alpha_{j}\overline{\alpha}_{j-1}.
\end{align}
By interchanging order of sums, we have
\begin{align}
&\sum_{k=0}^{\infty}\alpha_{k+3}\overline{\alpha}_{k+1}\sum_{j=0}^{k}\alpha_{j}\overline{\alpha}_{j-1}+\sum_{k=0}^{\infty}\alpha_{k+3}\overline{\alpha}_{k+2}\sum_{j=0}^{k}\alpha_{j+1}\overline{\alpha}_{j-1}\nonumber\\
=&\sum_{j=0}^{\infty}\alpha_{j}\overline{\alpha}_{j-1}\sum_{k=j+2}^{\infty}\alpha_{k+1}\overline{\alpha}_{k-1}+\sum_{j=0}^{\infty}\alpha_{j+1}\overline{\alpha}_{j-1}\sum_{k=j+3}^{\infty}\alpha_{k}\overline{\alpha}_{k-1}\nonumber\\
=&2d_{1}\sum_{j=0}^{\infty}\alpha_{j+1}\overline{\alpha}_{j-1}-\sum_{j=0}^{\infty}\alpha_{j}\overline{\alpha}_{j-1}\sum_{k=0}^{j+1}\alpha_{k+1}\overline{\alpha}_{k-1}-\sum_{j=0}^{\infty}\alpha_{j+1}\overline{\alpha}_{j-1}\sum_{k=0}^{j+2}\alpha_{k}\overline{\alpha}_{k-1}
\end{align}
and
\begin{align}
&\sum_{k=0}^{\infty}\alpha_{k+3}\overline{\alpha}_{k+2}\sum_{l=0}^{k-1}\alpha_{l+2}\overline{\alpha}_{l+1}\sum_{j=0}^{l}\alpha_{j}\overline{\alpha}_{j-1}
=\sum_{j=0}^{\infty}\alpha_{j}\overline{\alpha}_{j-1}\sum_{l=j+2}^{\infty}\alpha_{l}\overline{\alpha}_{l-1}\sum_{k=l+2}^{\infty}\alpha_{k}\overline{\alpha}_{k-1}\nonumber\\
=&\sum_{j=0}^{\infty}\alpha_{j}\overline{\alpha}_{j-1}\sum_{l=0}^{\infty}\alpha_{l}\overline{\alpha}_{l-1}\sum_{k=0}^{\infty}\alpha_{k}\overline{\alpha}_{k-1}-\sum_{j=0}^{\infty}\alpha_{j}\overline{\alpha}_{j-1}\sum_{l=0}^{\infty}\alpha_{l}\overline{\alpha}_{l-1}\sum_{k=0}^{l+1}\alpha_{k}\overline{\alpha}_{k-1}\nonumber\\
&-\sum_{j=0}^{\infty}\alpha_{j}\overline{\alpha}_{j-1}\sum_{l=0}^{j+1}\alpha_{l}\overline{\alpha}_{l-1}\sum_{k=0}^{\infty}\alpha_{k}\overline{\alpha}_{k-1}+\sum_{j=0}^{\infty}\alpha_{j}\overline{\alpha}_{j-1}\sum_{l=0}^{j+1}\alpha_{l}\overline{\alpha}_{l-1}\sum_{k=0}^{l+1}\alpha_{k}\overline{\alpha}_{k-1}\nonumber\\
=&d_{1}^{3}-d_{1}\left(\sum_{j=0}^{\infty}\alpha_{j}\overline{\alpha}_{j-1}\sum_{l=0}^{j+1}\alpha_{l}\overline{\alpha}_{l-1}
+\sum_{l=0}^{\infty}\alpha_{l}\overline{\alpha}_{l-1}\sum_{j=0}^{l+1}\alpha_{j}\overline{\alpha}_{j-1}\right)\nonumber\\
&+\sum_{j=0}^{\infty}\alpha_{j}\overline{\alpha}_{j-1}\sum_{l=0}^{j+1}\alpha_{l}\overline{\alpha}_{l-1}\sum_{k=0}^{l+1}\alpha_{k}\overline{\alpha}_{k-1}.
\end{align}

By Proposition 3.7, to get the sums with a single infinite index in $d_{3}$ then give $w_{3}$, we only need consider the following sums in (3.28) and (3.29):

\begin{itemize}
  \item [(3-1)] \hspace{4mm}$\displaystyle-\sum_{j=0}^{\infty}\alpha_{j}\overline{\alpha}_{j-1}\sum_{k=0}^{j+1}\alpha_{k+1}\overline{\alpha}_{k-1}$,
  \item [(3-2)] \hspace{4mm}$\displaystyle-\sum_{j=0}^{\infty}\alpha_{j+1}\overline{\alpha}_{j-1}\sum_{k=0}^{j+2}\alpha_{k}\overline{\alpha}_{k-1}$,
  \item [(3-3)] \hspace{4mm}$\displaystyle\sum_{j=0}^{\infty}\alpha_{j}\overline{\alpha}_{j-1}\sum_{l=0}^{j+1}\alpha_{l}\overline{\alpha}_{l-1}\sum_{k=0}^{l+1}\alpha_{k}\overline{\alpha}_{k-1}$.
\end{itemize}

By Proposition 3.4,
\begin{align}
&\sum_{j=0}^{\infty}\alpha_{j}\overline{\alpha}_{j-1}\sum_{k=0}^{j+1}\alpha_{k+1}\overline{\alpha}_{k-1}+\sum_{j=0}^{\infty}\alpha_{j+1}\overline{\alpha}_{j-1}\sum_{k=0}^{j+2}\alpha_{k}\overline{\alpha}_{k-1}\nonumber\\
=&\sum_{j=0}^{\infty}\alpha_{j+2}|\alpha_{j}|^{2}\overline{\alpha}_{j-1}+\sum_{j=0}^{\infty}\alpha_{j+2}|\alpha_{j+1}|^{2}\overline{\alpha}_{j-1}+\sum_{j=0}^{\infty}\alpha_{j+1}^{2}\overline{\alpha}_{j}\overline{\alpha}_{j-1}\nonumber\\
&+\sum_{j=0}^{\infty}\alpha_{j+1}\alpha_{j}\overline{\alpha}_{j-1}^{2}+\sum_{j=0}^{\infty}\alpha_{j+1}\overline{\alpha}_{j-1}\sum_{k=0}^{\infty}\alpha_{k}\overline{\alpha}_{k-1}.
\end{align}

Since (by Proposition 3.4 again)
\begin{align}
&\sum_{j=0}^{\infty}\alpha_{j}\overline{\alpha}_{j-1}\sum_{l=0}^{\infty}\alpha_{l}\overline{\alpha}_{l-1}\sum_{k=0}^{l+1}\alpha_{k}\overline{\alpha}_{k-1}\nonumber\\
=&\sum_{j=0}^{\infty}\alpha_{j}\overline{\alpha}_{j-1}\sum_{l=0}^{j}\alpha_{l}\overline{\alpha}_{l-1}\sum_{k=0}^{l+1}\alpha_{k}\overline{\alpha}_{k-1}+
\sum_{l=0}^{\infty}\alpha_{l}\overline{\alpha}_{l-1}\sum_{k=0}^{l+1}\alpha_{k}\overline{\alpha}_{k-1}\sum_{j=0}^{l}\alpha_{j}\overline{\alpha}_{j-1}\nonumber\\
&-\sum_{l=0}^{\infty}\alpha_{l}^{2}\overline{\alpha}_{l-1}^{2}\sum_{k=0}^{l+1}\alpha_{k}\overline{\alpha}_{k-1}\nonumber\\
=&\sum_{j=0}^{\infty}\alpha_{j}\overline{\alpha}_{j-1}\sum_{l=0}^{j+1}\alpha_{l}\overline{\alpha}_{l-1}\sum_{k=0}^{l+1}\alpha_{k}\overline{\alpha}_{k-1}
-\sum_{j=0}^{\infty}\alpha_{j+1}|\alpha_{j}|^{2}\overline{\alpha}_{j-1}\sum_{k=0}^{j+2}\alpha_{k}\overline{\alpha}_{k-1}\nonumber\\
&+\sum_{l=0}^{\infty}\alpha_{l}\overline{\alpha}_{l-1}\sum_{k=0}^{l+1}\alpha_{k}\overline{\alpha}_{k-1}\sum_{j=0}^{l+1}\alpha_{j}\overline{\alpha}_{j-1}
-\sum_{l=0}^{\infty}\alpha_{l+1}|\alpha_{l}|^{2}\overline{\alpha}_{l-1}\sum_{k=0}^{l+1}\alpha_{k}\overline{\alpha}_{k-1}\nonumber\\
&-\sum_{l=0}^{\infty}\alpha_{l}^{2}\overline{\alpha}_{l-1}^{2}\sum_{k=0}^{l+1}\alpha_{k}\overline{\alpha}_{k-1}\nonumber\\
=&\sum_{j=0}^{\infty}\alpha_{j}\overline{\alpha}_{j-1}\sum_{l=0}^{j+1}\alpha_{l}\overline{\alpha}_{l-1}\sum_{k=0}^{l+1}\alpha_{k}\overline{\alpha}_{k-1}
-\sum_{j=0}^{\infty}\alpha_{j+1}|\alpha_{j}|^{2}\overline{\alpha}_{j-1}\sum_{k=0}^{j+2}\alpha_{k}\overline{\alpha}_{k-1}\nonumber\\
&+\sum_{l=0}^{\infty}\alpha_{l}\overline{\alpha}_{l-1}\sum_{k=0}^{l+1}\alpha_{k}\overline{\alpha}_{k-1}\sum_{j=0}^{k}\alpha_{j}\overline{\alpha}_{j-1}
+\sum_{l=0}^{\infty}\alpha_{l}\overline{\alpha}_{l-1}\sum_{j=0}^{l+1}\alpha_{j}\overline{\alpha}_{j-1}\sum_{k=0}^{j}\alpha_{k}\overline{\alpha}_{k-1}\nonumber\\
&-\sum_{l=0}^{\infty}\alpha_{l}\overline{\alpha}_{l-1}\sum_{k=0}^{l+1}\alpha_{k}^{2}\overline{\alpha}_{k-1}^{2}-\sum_{l=0}^{\infty}\alpha_{l+1}|\alpha_{l}|^{2}\overline{\alpha}_{l-1}\sum_{k=0}^{l+1}\alpha_{k}\overline{\alpha}_{k-1}\nonumber\\
&-\sum_{l=0}^{\infty}\alpha_{l}^{2}\overline{\alpha}_{l-1}^{2}\sum_{k=0}^{l+1}\alpha_{k}\overline{\alpha}_{k-1}\nonumber\\
=&\sum_{j=0}^{\infty}\alpha_{j}\overline{\alpha}_{j-1}\sum_{l=0}^{j+1}\alpha_{l}\overline{\alpha}_{l-1}\sum_{k=0}^{l+1}\alpha_{k}\overline{\alpha}_{k-1}
-\sum_{j=0}^{\infty}\alpha_{j+1}|\alpha_{j}|^{2}\overline{\alpha}_{j-1}\sum_{k=0}^{j+2}\alpha_{k}\overline{\alpha}_{k-1}\nonumber
\end{align}
\begin{align}
&+\sum_{l=0}^{\infty}\alpha_{l}\overline{\alpha}_{l-1}\sum_{k=0}^{l+1}\alpha_{k}\overline{\alpha}_{k-1}\sum_{j=0}^{k+1}\alpha_{j}\overline{\alpha}_{j-1}
-\sum_{l=0}^{\infty}\alpha_{l}\overline{\alpha}_{l-1}\sum_{k=0}^{l+1}\alpha_{k+1}|\alpha_{k}|^{2}\overline{\alpha}_{k-1}\nonumber\\
&+\sum_{l=0}^{\infty}\alpha_{l}\overline{\alpha}_{l-1}\sum_{j=0}^{l+1}\alpha_{j}\overline{\alpha}_{j-1}\sum_{k=0}^{j+1}\alpha_{k}\overline{\alpha}_{k-1}
-\sum_{l=0}^{\infty}\alpha_{l}\overline{\alpha}_{l-1}\sum_{j=0}^{l+1}\alpha_{j+1}|\alpha_{j}|^{2}\overline{\alpha}_{j-1}\nonumber\\
&-\sum_{l=0}^{\infty}\alpha_{l}\overline{\alpha}_{l-1}\sum_{k=0}^{l+1}\alpha_{k}^{2}\overline{\alpha}_{k-1}^{2}-\sum_{l=0}^{\infty}\alpha_{l+1}|\alpha_{l}|^{2}\overline{\alpha}_{l-1}\sum_{k=0}^{l+1}\alpha_{k}\overline{\alpha}_{k-1}\nonumber\\
&-\sum_{l=0}^{\infty}\alpha_{l}^{2}\overline{\alpha}_{l-1}^{2}\sum_{k=0}^{l+1}\alpha_{k}\overline{\alpha}_{k-1},\nonumber
\end{align}
where (3.17) is used in the third equality,
then
\begin{align}
&\sum_{j=0}^{\infty}\alpha_{j}\overline{\alpha}_{j-1}\sum_{l=0}^{j+1}\alpha_{l}\overline{\alpha}_{l-1}\sum_{k=0}^{l+1}\alpha_{k}\overline{\alpha}_{k-1}\nonumber\\
=&\frac{1}{3}\sum_{j=0}^{\infty}\alpha_{j}\overline{\alpha}_{j-1}\sum_{l=0}^{\infty}\alpha_{l}\overline{\alpha}_{l-1}\sum_{k=0}^{l+1}\alpha_{k}\overline{\alpha}_{k-1}
+\frac{1}{3}\sum_{j=0}^{\infty}\alpha_{j+2}|\alpha_{j+1}|^{2}|\alpha_{j}|^{2}\overline{\alpha}_{j-1}\nonumber\\
&+\frac{2}{3}\Big(\sum_{j=0}^{\infty}\alpha_{j+1}|\alpha_{j}|^{2}\overline{\alpha}_{j-1}\sum_{k=0}^{j+1}\alpha_{k}\overline{\alpha}_{k-1}
+\sum_{k=0}^{\infty}\alpha_{k}\overline{\alpha}_{k-1}\sum_{j=0}^{k+1}\alpha_{j+1}|\alpha_{j}|^{2}\overline{\alpha}_{j-1}\Big)\nonumber\\
&+\frac{1}{3}\Big(\sum_{j=0}^{\infty}\alpha_{j}^{2}\overline{\alpha}_{j-1}^{2}\sum_{k=0}^{j+1}\alpha_{k}\overline{\alpha}_{k-1}
+\sum_{k=0}^{\infty}\alpha_{k}\overline{\alpha}_{k-1}\sum_{j=0}^{k+1}\alpha_{j}^{2}\overline{\alpha}_{j-1}^{2}\Big)\nonumber\\
=&\frac{1}{3}\sum_{j=0}^{\infty}\alpha_{j}\overline{\alpha}_{j-1}\sum_{l=0}^{\infty}\alpha_{l}\overline{\alpha}_{l-1}\sum_{k=0}^{l+1}\alpha_{k}\overline{\alpha}_{k-1}
+\frac{1}{3}\sum_{j=0}^{\infty}\alpha_{j+2}|\alpha_{j+1}|^{2}|\alpha_{j}|^{2}\overline{\alpha}_{j-1}\nonumber\\
&+\frac{2}{3}\Big(\sum_{j=0}^{\infty}\alpha_{j+1}^{2}|\alpha_{j}|^{2}\overline{\alpha}_{j}\overline{\alpha}_{j-1}
+\sum_{k=0}^{\infty}\alpha_{k+2}|\alpha_{k+1}|^{2}|\alpha_{k}|^{2}\overline{\alpha}_{k-1}\Big)\nonumber\\
&+\frac{2}{3}\Big(\sum_{j=0}^{\infty}\alpha_{j+1}|\alpha_{j}|^{2}\overline{\alpha}_{j-1}\sum_{k=0}^{\infty}\alpha_{k}\overline{\alpha}_{k-1}
+\sum_{j=0}^{\infty}\alpha_{j+1}\alpha_{j}|\alpha_{j}|^{2}\overline{\alpha}_{j-1}^{2}\Big)\nonumber\\
&+\frac{1}{3}\Big(\sum_{j=0}^{\infty}\alpha_{j+1}\alpha_{j}|\alpha_{j}|^{2}\overline{\alpha}_{j-1}^{2}
+\sum_{j=0}^{\infty}\alpha_{j+1}^{2}|\alpha_{j}|^{2}\overline{\alpha}_{j}\overline{\alpha}_{j-1}\Big)\nonumber\\
&+\frac{1}{3}\Big(\sum_{j=0}^{\infty}\alpha_{j}^{3}\overline{\alpha}_{j-1}^{3}
+\sum_{k=0}^{\infty}\alpha_{k}\overline{\alpha}_{k-1}\sum_{j=0}^{\infty}\alpha_{j}^{2}\overline{\alpha}_{j-1}^{2}\Big)\nonumber\\
=&\frac{1}{3}\sum_{j=0}^{\infty}\alpha_{j}\overline{\alpha}_{j-1}\sum_{l=0}^{\infty}\alpha_{l}\overline{\alpha}_{l-1}\sum_{k=0}^{l+1}\alpha_{k}\overline{\alpha}_{k-1}
+\sum_{j=0}^{\infty}\alpha_{j+2}|\alpha_{j+1}|^{2}|\alpha_{j}|^{2}\overline{\alpha}_{j-1}\nonumber\\
&+\sum_{j=0}^{\infty}\alpha_{j+1}^{2}|\alpha_{j}|^{2}\overline{\alpha}_{j}\overline{\alpha}_{j-1}
+\sum_{j=0}^{\infty}\alpha_{j+1}\alpha_{j}|\alpha_{j}|^{2}\overline{\alpha}_{j-1}^{2}+\frac{1}{3}\sum_{j=0}^{\infty}\alpha_{j}^{3}\overline{\alpha}_{j-1}^{3}\nonumber
\end{align}
\begin{align}
&+\frac{2}{3}\sum_{j=0}^{\infty}\alpha_{j+1}|\alpha_{j}|^{2}\overline{\alpha}_{j-1}\sum_{k=0}^{\infty}\alpha_{k}\overline{\alpha}_{k-1}
+\frac{1}{3}\sum_{k=0}^{\infty}\alpha_{k}\overline{\alpha}_{k-1}\sum_{j=0}^{\infty}\alpha_{j}^{2}\overline{\alpha}_{j-1}^{2}.
\end{align}
Thus
\begin{align}
w_{3}=&-\sum_{j=0}^{\infty}\alpha_{j+2}\overline{\alpha}_{j-1}+\sum_{j=0}^{\infty}\alpha_{j+2}|\alpha_{j}|^{2}\overline{\alpha}_{j-1}+\sum_{j=0}^{\infty}\alpha_{j+2}|\alpha_{j+1}|^{2}\overline{\alpha}_{j-1}\nonumber\\
&+\sum_{j=0}^{\infty}\alpha_{j+1}^{2}\overline{\alpha}_{j}\overline{\alpha}_{j-1}+\sum_{j=0}^{\infty}\alpha_{j+1}\alpha_{j}\overline{\alpha}_{j-1}^{2}-\sum_{j=0}^{\infty}\alpha_{j+2}|\alpha_{j+1}|^{2}|\alpha_{j}|^{2}\overline{\alpha}_{j-1}\nonumber\\
&-\sum_{j=0}^{\infty}\alpha_{j+1}^{2}|\alpha_{j}|^{2}\overline{\alpha}_{j}\overline{\alpha}_{j-1}
-\sum_{j=0}^{\infty}\alpha_{j+1}\alpha_{j}|\alpha_{j}|^{2}\overline{\alpha}_{j-1}^{2}-\frac{1}{3}\sum_{j=0}^{\infty}\alpha_{j}^{3}\overline{\alpha}_{j-1}^{3}\nonumber\\
=&-\sum_{j=0}^{\infty}\alpha_{j+2}\rho_{j+1}^{2}\rho_{j}^{2}\overline{\alpha}_{j-1}
+\sum_{j=0}^{\infty}\alpha_{j+1}^{2}\rho_{j}^{2}\overline{\alpha}_{j}\overline{\alpha}_{j-1}+\sum_{j=0}^{\infty}\alpha_{j+1}\alpha_{j}\rho_{j}^{2}\overline{\alpha}_{j-1}^{2}\nonumber\\
&-\frac{1}{3}\sum_{j=0}^{\infty}\alpha_{j}^{3}\overline{\alpha}_{j-1}^{3}.
\end{align}

\begin{rem}
$w_{1}, w_{2}$ and $w_{3}$ were given in \cite{gz} due to Golinskii and Zlato\v{s}.
\end{rem}

\begin{rem}
In (3.31), to get the sums with a single infinite index from (3-3), we mainly expand
$$\sum_{j=0}^{\infty}\alpha_{j}\overline{\alpha}_{j-1}\sum_{l=0}^{\infty}\alpha_{l}\overline{\alpha}_{l-1}\sum_{k=0}^{l+1}\alpha_{k}\overline{\alpha}_{k-1}.$$
In fact, by Proposition 3.4. we can also obtain these sums by expanding
$$\sum_{j=0}^{\infty}\alpha_{j}\overline{\alpha}_{j-1}\sum_{l=0}^{\infty}\alpha_{l}\overline{\alpha}_{l-1}\sum_{k=0}^{\infty}\alpha_{k}\overline{\alpha}_{k-1}.$$
\end{rem}

\emph{Calculation for $w_{4}$}: As stated in Remark 3.6 and the calculation for $w_{3}$, to get $w_{4}$, we only need consider the sums deduced from $d_{4}$ as follows

\begin{itemize}
  \item [(4-1)] \hspace{4mm} $\displaystyle\sum_{j=0}^{\infty}\alpha_{j+3}\overline{\alpha}_{j-1},$
  \item [(4-2)] \hspace{4mm} $\displaystyle-\sum_{j=0}^{\infty}\alpha_{j}\overline{\alpha}_{j-1}\sum_{k=0}^{j+1}\alpha_{k+2}\overline{\alpha}_{k-1},\,\,\,-\sum_{k=0}^{\infty}\alpha_{k+2}\overline{\alpha}_{k-1}\sum_{j=0}^{k+3}\alpha_{j}\overline{\alpha}_{j-1},$
  \item [(4-2)$^{\prime}$] \hspace{4mm} $\displaystyle -\sum_{j=0}^{\infty}\alpha_{j+1}\overline{\alpha}_{j-1}\sum_{k=0}^{j+2}\alpha_{k+1}\overline{\alpha}_{k-1},$
  \item [(4-3)] \hspace{4mm} $\displaystyle \sum_{j=0}^{\infty}\alpha_{j+1}\overline{\alpha}_{j-1}\sum_{k=0}^{j+2}\alpha_{k}\overline{\alpha}_{k-1}\sum_{l=0}^{k+1}\alpha_{l}\overline{\alpha}_{l-1},\,\,\,
\sum_{j=0}^{\infty}\alpha_{j}\overline{\alpha}_{j-1}\sum_{k=0}^{j+1}\alpha_{k+1}\overline{\alpha}_{k-1}\sum_{l=0}^{k+2}\alpha_{l}\overline{\alpha}_{l-1},$  \\ $\displaystyle\sum_{j=0}^{\infty}\alpha_{j}\overline{\alpha}_{j-1}\sum_{k=0}^{j+1}\alpha_{k}\overline{\alpha}_{k-1}\sum_{l=0}^{k+1}\alpha_{l+1}\overline{\alpha}_{l-1}$
  \item [(4-4)]\hspace{4mm}$\displaystyle-\sum_{j=0}^{\infty}\alpha_{j}\overline{\alpha}_{j-1}\sum_{k=0}^{j+1}\alpha_{k}\overline{\alpha}_{k-1}\sum_{l=0}^{k+1}\alpha_{l}\overline{\alpha}_{l-1}\sum_{m=0}^{l+1}\alpha_{m}\overline{\alpha}_{m-1}.$
\end{itemize}

Similar to (3.30), from (4-2), we get the following sums with a single infinite index
\begin{align}
&-\sum_{j=0}^{\infty}\alpha_{j+3}|\alpha_{j}|^{2}\overline{\alpha}_{j-1}-\sum_{j=0}^{\infty}\alpha_{j+2}\alpha_{j}\overline{\alpha}_{j-1}^{2}-\sum_{j=0}^{\infty}\alpha_{j+3}|\alpha_{j+2}|^{2}\overline{\alpha}_{j-1}\nonumber\\
&-\sum_{j=0}^{\infty}\alpha_{j+2}^{2}\overline{\alpha}_{j+1}\overline{\alpha}_{j-1}-\sum_{j=0}^{\infty}\alpha_{j+2}\alpha_{j+1}\overline{\alpha}_{j}\overline{\alpha}_{j-1}.
\end{align}

From (4-2)$^{\prime}$, we obtain
\begin{align}
-\sum_{j=0}^{\infty}\alpha_{j+3}|\alpha_{j+1}|^{2}\overline{\alpha}_{j-1}-\sum_{j=0}^{\infty}\alpha_{j+2}\alpha_{j+1}\overline{\alpha}_{j}\overline{\alpha}_{j-1}
-\frac{1}{2}\sum_{j=0}^{\infty}\alpha_{j+1}^{2}\overline{\alpha}_{j-1}^{2}.
\end{align}

Similar to (3.31), by expanding \begin{equation}\sum_{j=0}^{\infty}\alpha_{j+1}\overline{\alpha}_{j-1}\sum_{k=0}^{\infty}\alpha_{k}\overline{\alpha}_{k-1}\sum_{l=0}^{k+1}\alpha_{l}\overline{\alpha}_{l-1},\end{equation} from (4-3), we obtain the sums with a single infinite index as follows
\begin{align}
&\sum_{j=0}^{\infty}\alpha_{j+3}|\alpha_{j+2}|^{2}|\alpha_{j+1}|^{2}\overline{\alpha}_{j-1}+\sum_{j=0}^{\infty}\alpha_{j+2}^{2}|\alpha_{j+1}|^{2}\overline{\alpha}_{j+1}\overline{\alpha}_{j-1}
+2\sum_{j=0}^{\infty}\alpha_{j+2}\alpha_{j+1}|\alpha_{j+1}|^{2}\overline{\alpha}_{j}\overline{\alpha}_{j-1}\nonumber\\
&+\sum_{j=0}^{\infty}\alpha_{j+3}|\alpha_{j+2}|^{2}|\alpha_{j}|^{2}\overline{\alpha}_{j-1}+\sum_{j=0}^{\infty}\alpha_{j+2}|\alpha_{j+1}|^{2}\alpha_{j}\overline{\alpha}_{j-1}^{2}
+\sum_{j=0}^{\infty}\alpha_{j+1}^{3}\overline{\alpha}_{j}^{2}\overline{\alpha}_{j-1}\nonumber\\
&+\sum_{j=0}^{\infty}\alpha_{j+2}^{2}|\alpha_{j}|^{2}\overline{\alpha}_{j+1}\overline{\alpha}_{j-1}+\sum_{j=0}^{\infty}\alpha_{j+1}^{2}|\alpha_{j}|^{2}\overline{\alpha}_{j-1}^{2}
+2\sum_{j=0}^{\infty}\alpha_{j+2}\alpha_{j+1}|\alpha_{j}|^{2}\overline{\alpha}_{j}\overline{\alpha}_{j-1}\nonumber\\
&+\sum_{j=0}^{\infty}\alpha_{j+3}|\alpha_{j+1}|^{2}|\alpha_{j}|^{2}\overline{\alpha}_{j-1}+\sum_{j=0}^{\infty}\alpha_{j+2}\alpha_{j}|\alpha_{j}|^{2}\overline{\alpha}_{j-1}^{2}
+\sum_{j=0}^{\infty}\alpha_{j+1}^{2}|\alpha_{j}|^{2}\overline{\alpha}_{j-1}^{2}\nonumber\\
&+\sum_{j=0}^{\infty}\alpha_{j+1}\alpha_{j}^{2}\overline{\alpha}_{j-1}^{3}.
\end{align}

In the same way, by expanding
\begin{align}\sum_{j=0}^{\infty}\alpha_{j}\overline{\alpha}_{j-1}\sum_{k=0}^{\infty}\alpha_{k}\overline{\alpha}_{k-1}\sum_{l=0}^{k+1}\alpha_{l}\overline{\alpha}_{l-1}\sum_{m=0}^{l+1}\alpha_{m}\overline{\alpha}_{m-1},\end{align}
from (4-4), we get the following sums with a single infinite index
\begin{align}
&-\frac{1}{4}\Big(2\sum_{j=0}^{\infty}\alpha_{j+3}|\alpha_{j+2}|^{2}|\alpha_{j+1}|^{2}|\alpha_{j}|^{2}\overline{\alpha}_{j-1}+\sum_{j=0}^{\infty}\alpha_{j+2}^{2}|\alpha_{j+1}|^{2}|\alpha_{j}|^{2}\overline{\alpha}_{j+1}\overline{\alpha}_{j-1}\nonumber\\
&+\sum_{j=0}^{\infty}\alpha_{j+2}\alpha_{j+1}|\alpha_{j+1}|^{2}|\alpha_{j}|^{2}\overline{\alpha}_{j}\overline{\alpha}_{j-1}+\sum_{j=0}^{\infty}\alpha_{j+2}\alpha_{j}|\alpha_{j+1}|^{2}|\alpha_{j}|^{2}\overline{\alpha}_{j-1}^{2}\Big)\nonumber\\
&-\frac{1}{2}\Big(3\sum_{j=0}^{\infty}\alpha_{j+2}\alpha_{j+1}|\alpha_{j+1}|^{2}\alpha_{j}|^{2}\overline{\alpha}_{j}\overline{\alpha}_{j-1}+\sum_{j=0}^{\infty}\alpha_{j+1}^{3}|\alpha_{j}|^{2}\overline{\alpha}_{j}^{2}\overline{\alpha}_{j-1}\nonumber\\
&+\sum_{j=0}^{\infty}\alpha_{j+2}^{2}|\alpha_{j+1}|^{2}|\alpha_{j}|^{2}\overline{\alpha}_{j+1}\overline{\alpha}_{j-1}+2\sum_{j=0}^{\infty}\alpha_{j+1}^{2}|\alpha_{j}|^{4}\overline{\alpha}_{j-1}^{2}\nonumber\\
&+\sum_{j=0}^{\infty}\alpha_{j+2}\alpha_{j}|\alpha_{j+1}|^{2}|\alpha_{j}|^{2}\overline{\alpha}_{j-1}^{2}+\sum_{j=0}^{\infty}\alpha_{j+3}|\alpha_{j+2}|^{2}|\alpha_{j+1}|^{2}|\alpha_{j}|^{2}\overline{\alpha}_{j-1}\nonumber\\
&+\sum_{j=0}^{\infty}\alpha_{j+1}\alpha_{j}^{2}|\alpha_{j}|^{2}\overline{\alpha}_{j-1}^{3}\Big)-\frac{1}{4}\Big(\sum_{j=0}^{\infty}\alpha_{j+2}\alpha_{j}|\alpha_{j+1}|^{2}|\alpha_{j}|^{2}\overline{\alpha}_{j-1}^{2}\nonumber\\
&+2\sum_{j=0}^{\infty}\alpha_{j+1}^{2}|\alpha_{j}|^{4}\overline{\alpha}_{j-1}^{2}+\sum_{j=0}^{\infty}\alpha_{j+2}\alpha_{j+1}|\alpha_{j+1}|^{2}|\alpha_{j}|^{2}\overline{\alpha}_{j}\overline{\alpha}_{j-1}\nonumber\\
&+2\sum_{j=0}^{\infty}\alpha_{j+1}\alpha_{j}^{2}|\alpha_{j}|^{2}\overline{\alpha}_{j-1}^{3}+2\sum_{j=0}^{\infty}\alpha_{j+1}^{3}|\alpha_{j}|^{2}\overline{\alpha}_{j}^{2}\overline{\alpha}_{j-1}\nonumber\\
&+\sum_{j=0}^{\infty}\alpha_{j+2}^{2}|\alpha_{j+1}|^{2}|\alpha_{j}|^{2}\overline{\alpha}_{j+1}\overline{\alpha}_{j-1}+\sum_{j=0}^{\infty}\alpha_{j}^{4}\overline{\alpha}_{j-1}^{4}\Big)\nonumber\\
=&-\Big(\sum_{j=0}^{\infty}\alpha_{j+3}|\alpha_{j+2}|^{2}|\alpha_{j+1}|^{2}|\alpha_{j}|^{2}\overline{\alpha}_{j-1}+\sum_{j=0}^{\infty}\alpha_{j+2}^{2}|\alpha_{j+1}|^{2}|\alpha_{j}|^{2}\overline{\alpha}_{j+1}\overline{\alpha}_{j-1}\nonumber\\
&+2\sum_{j=0}^{\infty}\alpha_{j+2}\alpha_{j+1}|\alpha_{j+1}|^{2}|\alpha_{j}|^{2}\overline{\alpha}_{j}\overline{\alpha}_{j-1}+\sum_{j=0}^{\infty}\alpha_{j+2}\alpha_{j}|\alpha_{j+1}|^{2}|\alpha_{j}|^{2}\overline{\alpha}_{j-1}^{2}\nonumber\\
&+\sum_{j=0}^{\infty}\alpha_{j+1}^{3}|\alpha_{j}|^{2}\overline{\alpha}_{j}^{2}\overline{\alpha}_{j-1}+\frac{3}{2}\sum_{j=0}^{\infty}\alpha_{j+1}^{2}|\alpha_{j}|^{4}\overline{\alpha}_{j-1}^{2}+\sum_{j=0}^{\infty}\alpha_{j+1}\alpha_{j}^{2}|\alpha_{j}|^{2}\overline{\alpha}_{j-1}^{3}\nonumber\\
&+\frac{1}{4}\sum_{j=0}^{\infty}\alpha_{j}^{4}\overline{\alpha}_{j-1}^{4}\Big).
\end{align}

By Proposition 3.7, (4-1), (3.33), (3.34), (3.36) and (3.38), we obtain
\begin{align}
w_{4}=&-\sum_{j=0}^{\infty}\alpha_{j+3}\overline{\alpha}_{j-1}+\sum_{j=0}^{\infty}\alpha_{j+3}|\alpha_{j}|^{2}\overline{\alpha}_{j-1}+\sum_{j=0}^{\infty}\alpha_{j+2}\alpha_{j}\overline{\alpha}_{j-1}^{2}+\sum_{j=0}^{\infty}\alpha_{j+3}|\alpha_{j+2}|^{2}\overline{\alpha}_{j-1}\nonumber\\
&+\sum_{j=0}^{\infty}\alpha_{j+2}^{2}\overline{\alpha}_{j+1}\overline{\alpha}_{j-1}+\sum_{j=0}^{\infty}\alpha_{j+3}|\alpha_{j+1}|^{2}\overline{\alpha}_{j-1}\nonumber\\
&+2\sum_{j=0}^{\infty}\alpha_{j+2}\alpha_{j+1}\overline{\alpha}_{j}\overline{\alpha}_{j-1}
+\frac{1}{2}\sum_{j=0}^{\infty}\alpha_{j+1}^{2}\overline{\alpha}_{j-1}^{2}\nonumber
\end{align}
\begin{align}
&-\Big(\sum_{j=0}^{\infty}\alpha_{j+3}|\alpha_{j+2}|^{2}|\alpha_{j+1}|^{2}\overline{\alpha}_{j-1}+\sum_{j=0}^{\infty}\alpha_{j+2}^{2}|\alpha_{j+1}|^{2}\overline{\alpha}_{j+1}\overline{\alpha}_{j-1}
\nonumber\\
&+\sum_{j=0}^{\infty}\alpha_{j+3}|\alpha_{j+2}|^{2}|\alpha_{j}|^{2}\overline{\alpha}_{j-1}+\sum_{j=0}^{\infty}\alpha_{j+2}|\alpha_{j+1}|^{2}\alpha_{j}\overline{\alpha}_{j-1}^{2}
+\sum_{j=0}^{\infty}\alpha_{j+1}^{3}\overline{\alpha}_{j}^{2}\overline{\alpha}_{j-1}\nonumber\\
&+\sum_{j=0}^{\infty}\alpha_{j+2}^{2}|\alpha_{j}|^{2}\overline{\alpha}_{j+1}\overline{\alpha}_{j-1}+2\sum_{j=0}^{\infty}\alpha_{j+1}^{2}|\alpha_{j}|^{2}\overline{\alpha}_{j-1}^{2}
+2\sum_{j=0}^{\infty}\alpha_{j+2}\alpha_{j+1}|\alpha_{j}|^{2}\overline{\alpha}_{j}\overline{\alpha}_{j-1}\nonumber\\
&+\sum_{j=0}^{\infty}\alpha_{j+3}|\alpha_{j+1}|^{2}|\alpha_{j}|^{2}\overline{\alpha}_{j-1}+\sum_{j=0}^{\infty}\alpha_{j+2}\alpha_{j}|\alpha_{j}|^{2}\overline{\alpha}_{j-1}^{2}
\nonumber\\
&+2\sum_{j=0}^{\infty}\alpha_{j+2}\alpha_{j+1}|\alpha_{j+1}|^{2}\overline{\alpha}_{j}\overline{\alpha}_{j-1}+\sum_{j=0}^{\infty}\alpha_{j+1}\alpha_{j}^{2}\overline{\alpha}_{j-1}^{3}\Big)\nonumber\\
&+\Big(\sum_{j=0}^{\infty}\alpha_{j+3}|\alpha_{j+2}|^{2}|\alpha_{j+1}|^{2}|\alpha_{j}|^{2}\overline{\alpha}_{j-1}+\sum_{j=0}^{\infty}\alpha_{j+2}^{2}|\alpha_{j+1}|^{2}|\alpha_{j}|^{2}\overline{\alpha}_{j+1}\overline{\alpha}_{j-1}\nonumber\\
&+2\sum_{j=0}^{\infty}\alpha_{j+2}\alpha_{j+1}|\alpha_{j+1}|^{2}|\alpha_{j}|^{2}\overline{\alpha}_{j}\overline{\alpha}_{j-1}+\sum_{j=0}^{\infty}\alpha_{j+2}\alpha_{j}|\alpha_{j+1}|^{2}|\alpha_{j}|^{2}\overline{\alpha}_{j-1}^{2}\nonumber\\
&+\sum_{j=0}^{\infty}\alpha_{j+1}^{3}|\alpha_{j}|^{2}\overline{\alpha}_{j}^{2}\overline{\alpha}_{j-1}+\frac{3}{2}\sum_{j=0}^{\infty}\alpha_{j+1}^{2}|\alpha_{j}|^{4}\overline{\alpha}_{j-1}^{2}+\sum_{j=0}^{\infty}\alpha_{j+1}\alpha_{j}^{2}|\alpha_{j}|^{2}\overline{\alpha}_{j-1}^{3}\nonumber\\
&+\frac{1}{4}\sum_{j=0}^{\infty}\alpha_{j}^{4}\overline{\alpha}_{j-1}^{4}\Big)\nonumber\\
=&-\sum_{j=0}^{\infty}\alpha_{j+3}\rho_{j+2}^{2}\rho_{j+1}^{2}\rho_{j}^{2}\overline{\alpha}_{j-1}+\sum_{j=0}^{\infty}\alpha_{j+2}^{2}\rho_{j+1}^{2}\rho_{j}^{2}\overline{\alpha}_{j+1}\overline{\alpha}_{j-1}\nonumber\\
&+2\sum_{j=0}^{\infty}\alpha_{j+2}\alpha_{j+1}\rho_{j+1}^{2}\rho_{j}^{2}\overline{\alpha}_{j}\overline{\alpha}_{j-1}+\sum_{j=0}^{\infty}\alpha_{j+2}\alpha_{j}\rho_{j+1}^{2}\rho_{j}^{2}\overline{\alpha}_{j-1}^{2}\nonumber\\
&-\sum_{j=0}^{\infty}\alpha_{j+1}^{3}\rho_{j}^{2}\overline{\alpha}_{j}^{2}\overline{\alpha}_{j-1}-\sum_{j=0}^{\infty}\alpha_{j+1}\alpha_{j}^{2}\rho_{j}^{2}\overline{\alpha}_{j-1}^{3}
-\sum_{j=0}^{\infty}\alpha_{j+1}^{2}\rho_{j}^{2}\overline{\alpha}_{j-1}^{2}\nonumber\\
&+\frac{3}{2}\sum_{j=0}^{\infty}\alpha_{j+1}^{2}\rho_{j}^{4}\overline{\alpha}_{j-1}^{2}+\frac{1}{4}\sum_{j=0}^{\infty}\alpha_{j}^{4}\overline{\alpha}_{j-1}^{4}.
\end{align}

\begin{rem}
If $m=r_{1}+r_{2}+\cdots+r_{s}$ with $r_{j}\geq 1$, $1\leq j\leq s$, then $(r_{1},r_{2},\ldots,r_{s})$ is called an $s$-decomposition of $m$, $1\leq s\leq m$. Similar to the calculation of $w_{4}$, all the sums with a single infinite index in $d_{m}$ come from the sums with multi-fold infinite indices in it which determined by different $s$-decompositions of $m$, $(r_{1},r_{2},\ldots,r_{s})$, $1\leq s\leq m$, as follows
\begin{align}
(-1)^{s}\sum_{k_{1}=0}^{\infty}\alpha_{k_{1}+r_{1}-1}\overline{\alpha}_{k_{1}-1}\sum_{k_{2}=0}^{k_{1}+r_{1}}\alpha_{k_{2}+r_{2}-1}\overline{\alpha}_{k_{2}-1}\cdots
\sum_{k_{s}=0}^{k_{s-1}+r_{s-1}}\alpha_{k_{s}+r_{s}-1}\overline{\alpha}_{k_{s}-1}.
\end{align}
The same strategy of calculations is applicable to get these sums with a single infinite index for general $w_{m}$ as for $w_{j}$, $1\leq j\leq 4$.
\end{rem}

As in Remark 3.9, to get (3.36) and (3.38), we can replace (3.35) and (3.37) by expanding
$$\sum_{j=0}^{\infty}\alpha_{j+1}\overline{\alpha}_{j-1}\sum_{k=0}^{\infty}\alpha_{k}\overline{\alpha}_{k-1}\sum_{l=0}^{\infty}\alpha_{l}\overline{\alpha}_{l-1}$$
and
$$\sum_{j=0}^{\infty}\alpha_{j}\overline{\alpha}_{j-1}\sum_{k=0}^{\infty}\alpha_{k}\overline{\alpha}_{k-1}\sum_{l=0}^{\infty}\alpha_{l}\overline{\alpha}_{l-1}\sum_{m=0}^{\infty}\alpha_{m}\overline{\alpha}_{m-1}$$
respectively. In order to get the sums with a single infinite index from (3.40), we can apply general expansion for product of several series below.

In the end of this section and as a preliminary for some general results in Section 5, we give this general expansion. To do so, we introduce the following notions and notations.

For any two series $a=\sum_{k=1}^{\infty}a_{k}$ and $b=\sum_{k=1}^{\infty}b_{k}$, a contractive product of $a$ and $b$ is defined by
\begin{equation}
a\odot b=\sum_{k=1}^{\infty}a_{k}b_{k}.
\end{equation}
Obviously, $a\odot b$ is the diagonal part of the usual product of $a$ and $b$ (i.e., $\sum_{k=1}^{\infty}a_{k}\sum_{k=1}^{\infty}b_{k}$).
Set $c^{(m)}=\sum_{k=1}^{\infty}c^{(m)}_{k}$, $1\leq m\leq n$, $\prod_{m=1}^{n}\odot_{j}\big(c^{(m)}\big)$ is defined to be the sum of products of all distinct contractive products of $c^{(m)}$, $1\leq m\leq n$ (i.e., $\sum_{k=1}^{\infty}c^{(1)}_{k}\cdots\sum_{k=1}^{\infty}c^{(n)}_{k}$) for $0\leq j\leq n-1$, fulfilling that there exists at least a contractive product of $j$ times as a factor and other contractive product factors of at most $j$ times in its summands. For instance, as $n=3$,
\begin{equation}
\prod_{m=1}^{3}\odot_{0}\big(c^{(m)}\big)=c^{(1)}c^{(2)}c^{(3)}=\sum_{k=1}^{\infty}c^{(1)}_{k}\sum_{k=1}^{\infty}c^{(2)}_{k}\sum_{k=1}^{\infty} c^{(3)}_{k},
\end{equation}
\begin{align}
&\prod_{m=1}^{3}\odot_{1}\big(c^{(m)}\big)=\big(c^{(1)}\odot c^{(2)}\big)c^{(3)}+\big(c^{(2)}\odot c^{(3)}\big)c^{(1)}+\big(c^{(1)}\odot c^{(3)}\big)c^{(2)}\nonumber\\
=&\sum_{k=1}^{\infty}c^{(1)}_{k}c^{(2)}_{k}\sum_{k=1}^{\infty} c^{(3)}_{k}+\sum_{k=1}^{\infty}c^{(2)}_{k} c^{(3)}_{k}\sum_{k=1}^{\infty}c^{(1)}_{k}+\sum_{k=1}^{\infty}c^{(1)}_{k} c^{(3)}_{k}\sum_{k=1}^{\infty}c^{(2)}_{k}
\end{align}
and
\begin{equation}
\prod_{m=1}^{3}\odot_{2}\big(c^{(m)}\big)=c^{(1)}\odot c^{(2)}\odot c^{(3)}=\sum_{k=1}^{\infty}c^{(1)}_{k}c^{(2)}_{k} c^{(3)}_{k}.
\end{equation}
As $n=4$,
\begin{equation}
\prod_{m=1}^{4}\odot_{0}\big(c^{(m)}\big)=c^{(1)}c^{(2)}c^{(3)}c^{(4)}=\sum_{k=1}^{\infty}c^{(1)}_{k}\sum_{k=1}^{\infty}c^{(2)}_{k}\sum_{k=1}^{\infty} c^{(3)}_{k}\sum_{k=1}^{\infty} c^{(4)}_{k},
\end{equation}
\begin{align}
&\prod_{m=1}^{4}\odot_{1}\big(c^{(m)}\big)=\big(c^{(1)}\odot c^{(2)}\big)c^{(3)}c^{(4)}+\big(c^{(1)}\odot c^{(3)}\big)c^{(2)}c^{(4)}+\big(c^{(1)}\odot c^{(4)}\big)c^{(2)}c^{(3)}\nonumber\\
&+\big(c^{(2)}\odot c^{(3)}\big)c^{(1)}c^{(4)}+\big(c^{(2)}\odot c^{(4)}\big)c^{(1)}c^{(3)}+\big(c^{(3)}\odot c^{(4)}\big)c^{(1)}c^{(2)}\nonumber\\
&+\big(c^{(1)}\odot c^{(2)}\big)\big(c^{(3)}\odot c^{(4)}\big)+\big(c^{(1)}\odot c^{(3)}\big)\big(c^{(2)}\odot c^{(4)}\big)+\big(c^{(1)}\odot c^{(4)}\big)\big(c^{(2)}\odot c^{(3)}\big)\nonumber\\
=&\sum_{k=1}^{\infty}c^{(1)}_{k}c^{(2)}_{k}\sum_{k=1}^{\infty} c^{(3)}_{k}\sum_{k=1}^{\infty} c^{(4)}_{k}+\sum_{k=1}^{\infty}c^{(1)}_{k} c^{(3)}_{k}\sum_{k=1}^{\infty}c^{(2)}_{k}\sum_{k=1}^{\infty} c^{(4)}_{k}+\sum_{k=1}^{\infty}c^{(1)}_{k} c^{(4)}_{k}\sum_{k=1}^{\infty}c^{(2)}_{k}\sum_{k=1}^{\infty} c^{(3)}_{k}\nonumber\\
&+\sum_{k=1}^{\infty}c^{(2)}_{k} c^{(3)}_{k}\sum_{k=1}^{\infty}c^{(1)}_{k}\sum_{k=1}^{\infty} c^{(4)}_{k}+\sum_{k=1}^{\infty}c^{(2)}_{k} c^{(4)}_{k}\sum_{k=1}^{\infty}c^{(1)}_{k}\sum_{k=1}^{\infty} c^{(3)}_{k}+\sum_{k=1}^{\infty}c^{(3)}_{k} c^{(4)}_{k}\sum_{k=1}^{\infty}c^{(1)}_{k}\sum_{k=1}^{\infty} c^{(2)}_{k}\nonumber\\
&+\sum_{k=1}^{\infty}c^{(1)}_{k}c^{(2)}_{k}\sum_{k=1}^{\infty} c^{(3)}_{k} c^{(4)}_{k}+\sum_{k=1}^{\infty}c^{(1)}_{k} c^{(3)}_{k}\sum_{k=1}^{\infty}c^{(2)}_{k} c^{(4)}_{k}+\sum_{k=1}^{\infty}c^{(1)}_{k} c^{(4)}_{k}\sum_{k=1}^{\infty}c^{(2)}_{k} c^{(3)}_{k},
\end{align}
\begin{align}
&\prod_{m=1}^{4}\odot_{2}\big(c^{(m)}\big)=\big(c^{(1)}\odot c^{(2)}\odot c^{(3)}\big)c^{(4)}+\big(c^{(1)}\odot c^{(2)}\odot c^{(4)}\big)c^{(3)}\nonumber\\
&+\big(c^{(1)}\odot c^{(3)}\odot c^{(4)}\big)c^{(2)}+\big(c^{(2)}\odot c^{(3)}\odot c^{(4)}\big)c^{(1)}\nonumber\\
=&\sum_{k=1}^{\infty}c^{(1)}_{k}c^{(2)}_{k}c^{(3)}_{k}\sum_{k=1}^{\infty}  c^{(4)}_{k}+\sum_{k=1}^{\infty}c^{(1)}_{k} c^{(2)}_{k}c^{(4)}_{k}\sum_{k=1}^{\infty} c^{(3)}_{k}+\sum_{k=1}^{\infty}c^{(1)}_{k} c^{(3)}_{k}c^{(4)}_{k}\sum_{k=1}^{\infty}c^{(2)}_{k}\nonumber\\
&+\sum_{k=1}^{\infty}c^{(2)}_{k} c^{(3)}_{k}c^{(4)}_{k}\sum_{k=1}^{\infty} c^{(1)}_{k}
\end{align}
and
\begin{equation}
\prod_{m=1}^{4}\odot_{3}\big(c^{(m)}\big)=c^{(1)}\odot c^{(2)}\odot c^{(3)}\odot c^{(4)}=\sum_{k=1}^{\infty}c^{(1)}_{k}c^{(2)}_{k}c^{(3)}_{k}c^{(4)}_{k}.
\end{equation}

As a convention, the contractive product of a single series is itself.

\begin{thm}
Assume that $a^{(m)}=\Big\{a^{(m)}_{0},a^{(m)}_{1},a^{(m)}_{2},\ldots\Big\}\in\ell^{1}$ are distinct for $1\leq m\leq n$, $m,n\in \mathbb{N}$, then
\begin{align}
\prod_{m=1}^{n}\left(\sum_{k=0}^{\infty}a^{(m)}_{k}\right)=&\sum_{\sigma\in S_{n}}\left(\sum_{k_{1}=0}^{\infty}a^{(\sigma(1))}_{k_{1}}\sum_{k_{2}=0}^{k_{1}}a^{(\sigma(2))}_{k_{2}}\cdots\sum_{k_{n}=0}^{k_{n-1}}a^{(\sigma(n))}_{k_{n}}\right)\nonumber\\
&-\sum_{j=1}^{n-1}j!\prod_{m=1}^{n}\odot_{j}\left(\sum_{k=0}^{\infty}a^{(m)}_{k}\right)
\end{align}
or
\begin{align}
\sum_{\sigma\in S_{n}}\left(\sum_{k_{1}=0}^{\infty}a^{(\sigma(1))}_{k_{1}}\sum_{k_{2}=0}^{k_{1}}a^{(\sigma(2))}_{k_{2}}\cdots\sum_{k_{n}=0}^{k_{n-1}}a^{(\sigma(n))}_{k_{n}}\right)
=\sum_{j=0}^{n-1}j!\prod_{m=1}^{n}\odot_{j}\left(\sum_{k=0}^{\infty}a^{(m)}_{k}\right),
\end{align}
where $S_{n}$ is the standard permutation group and $\sigma=\big(\sigma(1),\sigma(2),\cdots,\sigma(n)\big)$ is a permutation of the set $A_{n}=\{1,2,\ldots,n\}$.
\end{thm}
\begin{rem}
It is easy to know that
\begin{align}
&\sum_{\sigma\in S_{n}}\left(\sum_{k_{1}=0}^{\infty}a^{(\sigma(1))}_{k_{1}}\sum_{k_{2}=0}^{k_{1}}a^{(\sigma(2))}_{k_{2}}\cdots\sum_{k_{n}=0}^{k_{n-1}}a^{(\sigma(n))}_{k_{n}}\right)\nonumber\\
=&\sum_{\sigma\in S_{n}}\sum_{k_{1}=0}^{\infty}\sum_{k_{2}=0}^{k_{1}}\cdots\sum_{k_{n}=0}^{k_{n-1}}a^{(\sigma(1))}_{k_{1}}a^{(\sigma(2))}_{k_{2}}\cdots a^{(\sigma(n))}_{k_{n}}\\
=&\sum_{\sigma\in S_{n}}\sum_{k_{\sigma(1)}=0}^{\infty}\sum_{k_{\sigma(2)}=0}^{k_{\sigma(1)}}\cdots\sum_{k_{\sigma(n)}=0}^{k_{\sigma(n-1)}}a^{(1)}_{k_{\sigma(1)}}a^{(2)}_{k_{\sigma(2)}}\cdots a^{(n)}_{k_{\sigma(n)}}.
\end{align}
\end{rem}

\begin{proof}
We take the method of induction. It is trivial for $n=1$. By Proposition 3.4, (3.50) is just (3.15) when $n=2$. Suppose that (3.49) or (3.50) holds for $n-1$ as $n>3$, then
\begin{align}
\sum_{\tau\in S_{n-1}}\left(\sum_{k_{1}=0}^{\infty}a^{(\tau(1))}_{k_{1}}\sum_{k_{2}=0}^{k_{1}}a^{(\tau(2))}_{k_{2}}\cdots\sum_{k_{n-1}=0}^{k_{n-2}}a^{(\tau(n-1))}_{k_{n-1}}\right)
=\sum_{j=0}^{n-2}j!\prod_{m=1}^{n-1}\odot_{j}\left(\sum_{k=0}^{\infty}a^{(m)}_{k}\right),
\end{align}
where $S_{n-1}$ is the standard permutation group and $\tau=\big(\tau(1),\tau(2),\cdots,\tau(n-1)\big)$ is a permutation of $\{1,2,\ldots,n-1\}$.

Thus by (3.15), (3.17), (3.51) and (3.52),
\begin{align}
&\prod_{m=1}^{n}\left(\sum_{k=0}^{\infty}a^{(m)}_{k}\right)=\left(\prod_{m=1}^{n-1}\left(\sum_{k=0}^{\infty}a^{(m)}_{k}\right)\right)\sum_{k=0}^{\infty}a^{(n)}_{k}\nonumber\\
=&\left[\sum_{\tau\in S_{n-1}}\left(\sum_{k_{1}=0}^{\infty}a^{(\tau(1))}_{k_{1}}\sum_{k_{2}=0}^{k_{1}}a^{(\tau(2))}_{k_{2}}\cdots\sum_{k_{n-1}=0}^{k_{n-2}}a^{(\tau(n-1))}_{k_{n-1}}\right)\right]\sum_{k=0}^{\infty}a^{(n)}_{k}\nonumber\\
&-\left[\sum_{j=1}^{n-2}j!\prod_{m=1}^{n-1}\odot_{j}\left(\sum_{k=0}^{\infty}a^{(m)}_{k}\right)\right]\sum_{k=0}^{\infty}a^{(n)}_{k}\nonumber\\
=&\sum_{\tau\in S_{n-1}}\left(\sum_{k=0}^{\infty}a^{(n)}_{k}\sum_{k_{1}=0}^{k}a^{(\tau(1))}_{k_{1}}\sum_{k_{2}=0}^{k_{1}}a^{(\tau(2))}_{k_{2}}\cdots\sum_{k_{n-1}=0}^{k_{n-2}}a^{(\tau(n-1))}_{k_{n-1}}\right)\nonumber\\
&+\sum_{\tau\in S_{n-1}}\left(\sum_{k_{1}=0}^{\infty}a^{(\tau(1))}_{k_{1}}\sum_{k_{2}=0}^{k_{1}}a^{(\tau(2))}_{k_{2}}\cdots\sum_{k_{n-1}=0}^{k_{n-2}}a^{(\tau(n-1))}_{k_{n-1}}\sum_{k=0}^{k_{1}}a^{(n)}_{k}\right)\nonumber\\
&-\sum_{\tau\in S_{n-1}}\left(\sum_{k_{1}=0}^{\infty}a^{(n)}_{k_{1}}a^{(\tau(1))}_{k_{1}}\sum_{k_{2}=0}^{k_{1}}a^{(\tau(2))}_{k_{2}}\cdots\sum_{k_{n-1}=0}^{k_{n-2}}a^{(\tau(n-1))}_{k_{n-1}}\right)\nonumber\\
&-\left[\sum_{j=1}^{n-2}j!\prod_{m=1}^{n-1}\odot_{j}\left(\sum_{k=0}^{\infty}a^{(m)}_{k}\right)\right]\sum_{k=0}^{\infty}a^{(n)}_{k}\nonumber
\end{align}
\begin{align}
=&\sum_{\tau\in S_{n-1}}\left(\sum_{k=0}^{\infty}a^{(n)}_{k}\sum_{k_{1}=0}^{k}a^{(\tau(1))}_{k_{1}}\sum_{k_{2}=0}^{k_{1}}a^{(\tau(2))}_{k_{2}}\cdots\sum_{k_{n-1}=0}^{k_{n-2}}a^{(\tau(n-1))}_{k_{n-1}}\right)\nonumber\\
&+\sum_{\tau\in S_{n-1}}\left(\sum_{k_{1}=0}^{\infty}a^{(\tau(1))}_{k_{1}}\sum_{k=0}^{k_{1}}a^{(n)}_{k}\sum_{k_{2}=0}^{k}a^{(\tau(2))}_{k_{2}}\cdots\sum_{k_{n-1}=0}^{k_{n-2}}a^{(\tau(n-1))}_{k_{n-1}}\right)\nonumber\\
&+\sum_{\tau\in S_{n-1}}\left(\sum_{k_{1}=0}^{\infty}a^{(\tau(1))}_{k_{1}}\sum_{k_{2}=0}^{k_{1}}a^{(\tau(2))}_{k_{2}}\cdots\sum_{k_{n-1}=0}^{k_{n-2}}a^{(\tau(n-1))}_{k_{n-1}}\sum_{k=0}^{k_{2}}a^{(n)}_{k}\right)\nonumber\\
&-\sum_{\tau\in S_{n-1}}\left(\sum_{k_{1}=0}^{\infty}a^{(n)}_{k_{1}}a^{(\tau(1))}_{k_{1}}\sum_{k_{2}=0}^{k_{1}}a^{(\tau(2))}_{k_{2}}\cdots\sum_{k_{n-1}=0}^{k_{n-2}}a^{(\tau(n-1))}_{k_{n-1}}\right)\nonumber\\
&-\sum_{\tau\in S_{n-1}}\left(\sum_{k_{1}=0}^{\infty}a^{(\tau(1))}_{k_{1}}\sum_{k_{2}=0}^{k_{1}}a^{(n)}_{k_{2}}a^{(\tau(2))}_{k_{2}}\cdots\sum_{k_{n-1}=0}^{k_{n-2}}a^{(\tau(n-1))}_{k_{n-1}}\right)\nonumber\\
&-\left[\sum_{j=1}^{n-2}j!\prod_{m=1}^{n-1}\odot_{j}\left(\sum_{k=0}^{\infty}a^{(m)}_{k}\right)\right]\sum_{k=0}^{\infty}a^{(n)}_{k}\nonumber\\
=&\sum_{\tau\in S_{n-1}}\left(\sum_{k=0}^{\infty}a^{(n)}_{k}\sum_{k_{1}=0}^{k}a^{(\tau(1))}_{k_{1}}\sum_{k_{2}=0}^{k_{1}}a^{(\tau(2))}_{k_{2}}\cdots\sum_{k_{n-1}=0}^{k_{n-2}}a^{(\tau(n-1))}_{k_{n-1}}\right)\nonumber\\
&+\sum_{\tau\in S_{n-1}}\left(\sum_{k_{1}=0}^{\infty}a^{(\tau(1))}_{k_{1}}\sum_{k=0}^{k_{1}}a^{(n)}_{k}\sum_{k_{2}=0}^{k}a^{(\tau(2))}_{k_{2}}\cdots\sum_{k_{n-1}=0}^{k_{n-2}}a^{(\tau(n-1))}_{k_{n-1}}\right)\nonumber\\
&+\cdots\nonumber\\
&+\sum_{\tau\in S_{n-1}}\left(\sum_{k_{1}=0}^{\infty}a^{(\tau(1))}_{k_{1}}\sum_{k_{2}=0}^{k_{1}}a^{(\tau(2))}_{k_{2}}\cdots\sum_{k_{n-1}=0}^{k_{n-2}}a^{(\tau(n-1))}_{k_{n-1}}\sum_{k=0}^{k_{n-1}}a^{(n)}_{k}\right)\nonumber\\
&-\sum_{\tau\in S_{n-1}}\left(\sum_{k_{1}=0}^{\infty}a^{(n)}_{k_{1}}a^{(\tau(1))}_{k_{1}}\sum_{k_{2}=0}^{k_{1}}a^{(\tau(2))}_{k_{2}}\cdots\sum_{k_{n-1}=0}^{k_{n-2}}a^{(\tau(n-1))}_{k_{n-1}}\right)\nonumber\\
&-\sum_{\tau\in S_{n-1}}\left(\sum_{k_{1}=0}^{\infty}a^{(\tau(1))}_{k_{1}}\sum_{k_{2}=0}^{k_{1}}a^{(n)}_{k_{2}}a^{(\tau(2))}_{k_{2}}\cdots\sum_{k_{n-1}=0}^{k_{n-2}}a^{(\tau(n-1))}_{k_{n-1}}\right)\nonumber\\
&-\cdots\nonumber\\
&-\sum_{\tau\in S_{n-1}}\left(\sum_{k_{1}=0}^{\infty}a^{(\tau(1))}_{k_{1}}\sum_{k_{2}=0}^{k_{1}}a^{(n)}_{k_{2}}a^{(\tau(2))}_{k_{2}}\cdots\sum_{k_{n-1}=0}^{k_{n-2}}a^{(n)}_{k_{n-1}}a^{(\tau(n-1))}_{k_{n-1}}\right)\nonumber\\
&-\left[\sum_{j=1}^{n-2}j!\prod_{m=1}^{n-1}\odot_{j}\left(\sum_{k=0}^{\infty}a^{(m)}_{k}\right)\right]\sum_{k=0}^{\infty}a^{(n)}_{k}\nonumber
\end{align}
\begin{align}
=&\sum_{\sigma\in S_{n}}\left(\sum_{k_{1}=0}^{\infty}a^{(\sigma(1))}_{k_{1}}\sum_{k_{2}=0}^{k_{1}}a^{(\sigma(2))}_{k_{2}}\cdots\sum_{k_{n}=0}^{k_{n-1}}a^{(\sigma(n))}_{k_{n}}\right)\nonumber\\
&-\sum_{s=1}^{n-1}\sum_{\tau\in S_{n-1}}\sum_{k_{\tau(1)}=0}^{\infty}\sum_{k_{\tau(2)}=0}^{k_{\tau(1)}}\cdots\sum_{k_{\tau(n-1)}=0}^{k_{\tau(n-2)}}a^{(1)}_{k_{\tau(1)}}a^{(2)}_{k_{\tau(2)}}\cdots \big(a^{(n)}_{k_{\tau(s)}}a^{(s)}_{k_{\tau(s)}}\big)\cdots a^{(n-1)}_{k_{\tau(n-1)}}\nonumber\\
&-\left[\sum_{j=1}^{n-2}j!\prod_{m=1}^{n-1}\odot_{j}\left(\sum_{k=0}^{\infty}a^{(m)}_{k}\right)\right]\sum_{k=0}^{\infty}a^{(n)}_{k}.
\end{align}
By the assumption of induction, for any $1\leq s\leq n-1$, we have
\begin{align}
&\sum_{\tau\in S_{n-1}}\sum_{k_{\tau(1)}=0}^{\infty}\sum_{k_{\tau(2)}=0}^{k_{\tau(1)}}\cdots\sum_{k_{\tau(n-1)}=0}^{k_{\tau(n-2)}}a^{(1)}_{k_{\tau(1)}}a^{(2)}_{k_{\tau(2)}}\cdots \big(a^{(n)}_{k_{\tau(s)}}a^{(s)}_{k_{\tau(s)}}\big)\cdots a^{(n-1)}_{k_{\tau(n-1)}}\nonumber\\
=&\sum_{j=0}^{n-2}j!\prod_{m=1}^{n-1}\odot_{j}\left(t_{m,s}\odot a^{(m)}\right),
\end{align}
where $t_{m,s}=\begin{cases}
a^{(n)}, \,\,m=s,\\
\mathbf{1}, \,\,\,\hspace{3.5mm}m\neq s
\end{cases}$
in which the sequence $\mathbf{1}=\{1,1,\ldots\}$.

Hence, in order to get (3.49), we only need justify the following equality
\begin{align}
&\left[\sum_{j=1}^{n-2}j!\prod_{m=1}^{n-1}\odot_{j}\left(\sum_{k=0}^{\infty}a^{(m)}_{k}\right)\right]\sum_{k=0}^{\infty}a^{(n)}_{k}
+\sum_{s=1}^{n-1}\sum_{j=0}^{n-2}j!\prod_{m=1}^{n-1}\odot_{j}\left(t_{m,s}\odot a^{(m)}\right)\nonumber\\
=&\sum_{j=1}^{n-1}j!\prod_{m=1}^{n}\odot_{j}\left(\sum_{k=0}^{\infty}a^{(m)}_{k}\right).
\end{align}

Denote
\begin{equation}
L_{\mathrm{I}}=\left[\sum_{j=1}^{n-2}j!\prod_{m=1}^{n-1}\odot_{j}\left(\sum_{k=0}^{\infty}a^{(m)}_{k}\right)\right]\sum_{k=0}^{\infty}a^{(n)}_{k}
\end{equation}
and
\begin{equation}
L_{\mathrm{II}}=\sum_{s=1}^{n-1}\sum_{j=0}^{n-2}j!\prod_{m=1}^{n-1}\odot_{j}\left(t_{m,s}\odot a^{(m)}\right).
\end{equation}

By the definition, for fixed $1\leq s\leq n-1$ and $1\leq j\leq n-2$, general terms in $L_{\mathrm{II}}$ can be classified into two kinds: one has the following form
\begin{equation}
j!\sum_{k=0}^{\infty}\Big(a^{(n)}_{k} a^{(s)}_{k}\Big)a^{(s_{1})}_{k}a^{(s_{2})}_{k}\cdots a^{(s_{j})}_{k}\times \prod_{l=1}^{n-2-j}\odot_{j}\Big(a^{(t_{l})}\Big),
\end{equation}
where distinct $s_{1}, \ldots,s_{j}\in A_{n}\setminus\{n,s\}$ and $t_{1},\ldots,t_{n-2-j}\in A_{n}\setminus\{n,s,s_{1},\ldots,s_{j}\}$;
The other has the forms as follows
\begin{equation}
j!\sum_{k=0}^{\infty}a^{(s^{\prime}_{1})}_{k}a^{(s^{\prime}_{2})}_{k}\cdots a^{(s^{\prime}_{j})}_{k}\times \prod_{l=1}^{n-j}\odot_{j}\Big(a^{(t^{\prime}_{l})}\Big),
\end{equation}
where distinct $s^{\prime}_{1}, \ldots,s^{\prime}_{j}\in A_{n}\setminus\{n,s\}$ and $t^{\prime}_{1},\ldots,t^{\prime}_{n-j}\in A_{n}\setminus\{s^{\prime}_{1},\ldots,s^{\prime}_{j}\}$.
Moreover, for $j=0$, general terms in $L_{\mathrm{II}}$ are of the form
\begin{equation}
\sum_{k=0}^{\infty}a^{(n)}_{k}a^{(s^{\prime})}_{k}\times \prod_{l=1}^{n-2}\odot_{0}\Big(a^{(t^{\prime\prime}_{l})}\Big),
\end{equation}
where $s^{\prime}\in A_{n}\setminus\{n\}$ and distinct $t^{\prime\prime}_{1},\ldots,t^{\prime\prime}_{n-2}\in A_{n}\setminus\{n, s^{\prime}\}$.

It is easy to note that the summands in (3.59) are ones in $\prod_{m=1}^{n}\odot_{j+1}\Big(a^{(m)}\Big)$ in RHS of (3.56). Fixing $s_{1},s_{2},\ldots,s_{j}$,
we repeatedly obtain (3.59) for $j+1$ times with changing $s$ and $s_{\mu}$ each other for $1\leq\mu\leq j$. So the coefficients of summands in (3.59) are just
$(j+1)\times j!=(j+1)!$ which coincided with the ones in $\prod_{m=1}^{n}\odot_{j+1}\Big(a^{(m)}\Big)$ in RHS of (3.56). Thus the sum of all summands in (3.59) for index $j$ and the ones in (3.60) for index $j+1$ together with the ones from $(j+1)!\prod_{m=1}^{n}\odot_{j+1}\Big(a^{(m)}\Big)$ in $L_{\mathrm{I}}$ equals to the corresponding sum for index $j+1$ in RHS of (3.56). This is to say that two sides of (3.56) are equal except the sum with only one contractive product in it (i.e., index $j=1$).

However, by the definition, it is immediate to know that the sum of the summands in (3.60) and $L_{\mathrm{I}}$ for index $j=1$ together with the ones in (3.61) is just the one in RHS of (3.56)  (viz., $\prod_{m=1}^{n}\odot_{1}(a^{(m)})$).
\end{proof}

\begin{rem}
As above, in our approach, it is important for us to get the sums with a single infinite index in (3.49) or (3.50) for $n$ different series. It is noteworthy that the number of such sums is only one and it has a nice and explicit expression from the original series $a^{(l)}=\sum_{k=0}^{\infty}a_{k}^{(l)}$, $1\leq l\leq n $.  That is,
$$(n-1)!a^{(1)}\odot a^{(2)}\cdots\odot a^{(n)}=(n-1)!\sum_{k=0}^{\infty}a_{k}^{(1)}a_{k}^{(2)}\cdots a_{k}^{(n)}.$$
\end{rem}

If $t_{1}+t_{2}+\cdots+t_{p}=n$, $t_{l}\in \mathbb{N}$, $1\leq l\leq p$, denote $S_{n,p}(t_{1},t_{2},\ldots,t_{p})$ be the set of all distinct permutations of $n$ numbers consisting of $1,2,\ldots,p$ with $l$ repeating $t_{l}$ times, $1\leq l\leq p$. In short, $S_{n,p}(t_{1},t_{2},\ldots,t_{p})$ is the set of all distinct permutations whose elements are permitted to be the same. For example, $11322$ and $13221$ are permutations belong to $S_{5,3}(2,2,1)$. Moreover, denote $S_{n,p}$ be the set of all permutations of all $n$ elements which forming $S_{n,p}(t_{1},t_{2},\ldots,t_{p})$ with permission of repetition. That is to say all elements forming $S_{n,p}$ are seen to be different each other although some of them are the same. In other words, $S_{n,p}$ is just $S_{n}$ with the permission of repeating some elements.

By Theorem 3.11, we have
\begin{cor}
Assume that $a^{(l)}=\Big\{a^{(l)}_{0},a^{(l)}_{1},a^{(l)}_{2},\ldots\Big\}\in\ell^{1}$ are distinct for $1\leq l\leq p$, $l,p\in \mathbb{N}$, then for any $\tau\in S_{n,p}(t_{1},t_{2},\cdots,t_{p})$ with $t_{1}+t_{2}+\cdots+t_{p}=n$, $t_{l}\in \mathbb{N}$, $1\leq l\leq p$,
\begin{align}
\prod_{l=1}^{p}\left(\sum_{k=0}^{\infty}a^{(l)}_{k}\right)^{t_{l}}=&\sum_{\sigma\in S_{n,p}}\left(\sum_{k_{1}=0}^{\infty}a^{(\sigma(1))}_{k_{1}}\sum_{k_{2}=0}^{k_{1}}a^{(\sigma(2))}_{k_{2}}\cdots\sum_{k_{n}=0}^{k_{n-1}}a^{(\sigma(n))}_{k_{n}}\right)\nonumber\\
&-\sum_{j=1}^{n-1}j!\prod_{m=1}^{n}\odot_{j}\left(\sum_{k=0}^{\infty}a^{(\tau(m))}_{k}\right)
\end{align}
or
\begin{align}
&\sum_{\sigma\in S_{n,p}(t_{1},t_{2},\cdots,t_{p})}\left(\sum_{k_{1}=0}^{\infty}a^{(\sigma(1))}_{k_{1}}\sum_{k_{2}=0}^{k_{1}}a^{(\sigma(2))}_{k_{2}}\cdots\sum_{k_{n}=0}^{k_{n-1}}a^{(\sigma(n))}_{k_{n}}\right)
\nonumber\\
=&\frac{1}{t_{1}!t_{2}!\cdots t_{p}!}\sum_{j=0}^{n-1}j!\prod_{m=1}^{n}\odot_{j}\left(\sum_{k=0}^{\infty}a^{(\tau(m))}_{k}\right).
\end{align}
\end{cor}

\begin{proof}
The key is to think $\underbrace{l,l,\ldots,l}_{t_{l}}$ as $t_{l}$ different elements for $1\leq l\leq p$. Then (3.62) follows from (3.49). (3.63) follows from (3.62) by the relation between $S_{n,p}$ and $S_{n,p}(t_{1},t_{2},\cdots,t_{p})$.
\end{proof}

\begin{rem}
As in Remark 3.13, from the expansion of the product of $n$ series among which some ones are same (exactly, from the LHS of (3.63)), the only sum with a single infinite index in it is the following
\begin{align*}&\frac{1}{t_{1}!t_{2}!\cdots t_{p}!}(n-1)!\underbrace{a^{(1)}\odot\cdots\odot a^{(1)}}_{t_{1}} \odot\cdots\odot \underbrace{a^{(p)}\odot\cdots\odot a^{(p)}}_{t_{p}}\nonumber\\
=&\frac{1}{t_{1}!t_{2}!\cdots t_{p}!}(n-1)!\sum_{k=0}^{\infty}\Big(a_{k}^{(1)}\Big)^{t_{1}}\Big(a_{k}^{(2)}\Big)^{t_{2}}\cdots \Big(a_{k}^{(p)}\Big)^{t_{p}}.\end{align*}
\end{rem}

\section{Sum rules and higher order Szeg\H{o} theorems: Special cases}

In this section, by using the expressions of $w_{m}$, $0\leq m\leq 4$, obtained in the former section, we will establish some sum rules and higher order Szeg\H{o} theorems for a few of special cases. Some of which were obtained by Simon, Zlato\v{s} and others \cite{bbz1,gz,lu,lu1,sim1,sz}. Nevertheless, the approach here is unified and computable. In particular, some of sum rules in what follows are new and explicit.

To do so, we need the following lemmas due to Golinskii and Zlato\v{s} as well as Breuer, Simon and Zeitouni respectively (for details, see \cite{bbz1,gz}).
\begin{lem}[\cite{gz}]
Let $Q$ be a complex polynomial and $$Z_{Q}(\mu)=\int_{0}^{2\pi}|Q(e^{i\theta})|^{2}\log w(\theta)\frac{d\theta}{2\pi},$$ then
$$Z_{Q}(\mu)=\lim_{n\rightarrow \infty}Z_{Q}(\mu_{n}),$$
where $\mu_{n}$ is the $n$th order Bernstein-Szeg\H{o} approximation of $\mu$.
\end{lem}

\begin{rem}
Assume that $Q(z)=-\frac{1}{\sqrt{2}}(z-1)$, then the above result (due to Simon) was implicitly obtained in the proof of Theorem 2.8.1 in \cite{sim1}. In fact, $Z_{Q}(\mu)$ in this case is just $Z_{1}(\mu)$ below.
\end{rem}

\begin{lem}[\cite{bbz1}]
For any $\alpha$,
$$-\sum_{j=0}^{\infty}\left[\log\big(1-|\alpha_{j}|^{2}\big)+\sum_{m=1}^{M}\frac{|\alpha_{j}|^{2m}}{m}\right]<+\infty$$
if and only if
$$\sum_{j=0}^{\infty}|\alpha_{j}|^{2M+2}<+\infty.$$
\end{lem}
\begin{rem}
As $M=1$, this result is also due to Simon (see Lemma 2.8.3 in \cite{sim1}).
\end{rem}

\subsection{First order case}

At first, by using the expressions of $w_{0}$ and $w_{1}$, we have
\begin{thm}
\begin{align}
\int_{0}^{2\pi}(1-\cos\theta)\log w(\theta)\frac{d\theta}{2\pi}>-\infty \,\,\Longleftrightarrow\,\, \sum_{n=0}^{\infty}\Big(|\alpha_{n+1}-\alpha_{n}|^{2}+|\alpha_{n}|^{4}\Big)<\infty.
\end{align}
More precisely, for any sequence $\alpha=\{\alpha_{n}\}_{n\in \mathbb{N}_{0}}$ of Verblunsky coefficients,
\begin{align}
\int_{0}^{2\pi}(1-\cos\theta)\log w(\theta)\frac{d\theta}{2\pi}=&\frac{1}{2}+\sum_{n=0}^{\infty}\Big(\log(1-|\alpha_{n}|^{2})+|\alpha_{n}|^{2}\Big)\nonumber\\
&-\frac{1}{2}\sum_{n=0}^{\infty}|\alpha_{n}-\alpha_{n-1}|^{2}.
\end{align}
\end{thm}

\begin{proof}
Let
\begin{equation}
Z_{1}(\mu)=\int_{0}^{2\pi}(1-\cos\theta)\log w(\theta)\frac{d\theta}{2\pi},
\end{equation}
then
\begin{equation}
Z_{1}(\mu)=w_{0}-\mathrm{Re}(w_{1}).
\end{equation}
Firstly, we suppose that $\alpha\in \ell^{2}$. Under this assumption, by the classical Szeg\H{o} theorem, we have
\begin{align}
w_{0}=\int_{0}^{2\pi}\log w(\theta)\frac{d\theta}{2\pi}=\sum_{j=0}^{\infty}\log(1-|\alpha_{j}|^{2}).
\end{align}
Moreover, by (3.20), we have
\begin{equation}
\mathrm{Re}(w_{1})=-\sum_{j=0}^{\infty}\mathrm{Re}(\alpha_{j}\overline{\alpha}_{j-1})
\end{equation}
in which
\begin{align}
\mathrm{Re}(\alpha_{j}\overline{\alpha}_{j-1})=&\frac{1}{2}(\alpha_{j}\overline{\alpha}_{j-1}+\overline{\alpha}_{j}\alpha_{j-1})\nonumber\\
=&\frac{1}{2}\big(|\alpha_{j}|^{2}+|\alpha_{j-1}|^{2}-|\alpha_{j}-\alpha_{j-1}|^{2}\big).
\end{align}
Therefore, by (4.4)-(4.7), we get
\begin{align}
Z_{1}(\mu)=&\sum_{j=0}^{\infty}\log(1-|\alpha_{j}|^{2})+\frac{1}{2}\sum_{j=0}^{\infty}\big(|\alpha_{j}|^{2}+|\alpha_{j-1}|^{2}-|\alpha_{j}-\alpha_{j-1}|^{2}\big)\nonumber\\
=&\frac{1}{2}+\sum_{j=0}^{\infty}\Big(\log(1-|\alpha_{j}|^{2})+|\alpha_{j}|^{2}\Big)-\frac{1}{2}\sum_{j=0}^{\infty}|\alpha_{j}-\alpha_{j-1}|^{2}.
\end{align}

Next, we consider general $\alpha$. By Lemma 4.1, we have
\begin{equation}
Z_{1}(\mu)=\lim_{n\rightarrow\infty}Z_{1}(\mu_{n}),
\end{equation}
It is well known that if the Verblunsky coefficients of $\mu$ are $\{-1,\alpha_{0},\alpha_{1},\ldots,\alpha_{n}, \alpha_{n+1},\ldots\}$, then the Verblunsky coefficients of $\mu_{n}$ are $\{-1,\alpha_{0},\alpha_{1},\ldots,\alpha_{n}, 0,\ldots\}$. For any $\mu$, the Verblunsky coefficients of $\mu_{n}$ is in $\ell^{2}$. By the above result of $\ell^{2}$ case, we have
\begin{align}
Z_{1}(\mu_{n})
=&\frac{1}{2}+\sum_{j=0}^{n}\Big(\log(1-|\alpha_{j}|^{2})+|\alpha_{j}|^{2}\Big)-\frac{1}{2}\sum_{j=0}^{n}|\alpha_{j}-\alpha_{j-1}|^{2}.
\end{align}
Thus by (4.9), (4.8) holds for any $\alpha$ (and the corresponding $\mu$) by letting $n\rightarrow\infty$ in both sides of (4.10) because the RHS of (4.10) is the partial sum of the RHS of (4.8).

By Lemma 4.3 and (4.8), we have
\begin{align}
Z_{1}(\mu)>-\infty &\Leftrightarrow -Z_{1}(\mu)<+\infty \nonumber\\
&\Leftrightarrow \sum_{j=0}^{\infty}\Big(-\log(1-|\alpha_{j}|^{2})-|\alpha_{j}|^{2}\Big)+\frac{1}{2}\sum_{j=0}^{\infty}|\alpha_{j}-\alpha_{j-1}|^{2}<+\infty\nonumber\\
&\Leftrightarrow \sum_{j=0}^{\infty}|\alpha_{j}|^{4}+\frac{1}{2}\sum_{j=0}^{\infty}|\alpha_{j}-\alpha_{j-1}|^{2}<+\infty\nonumber\\
&\Leftrightarrow \sum_{j=0}^{\infty}\Big(|\alpha_{j}|^{4}+|\alpha_{j}-\alpha_{j-1}|^{2}\Big)<+\infty.\nonumber
\end{align}
That is,
$$
\int_{0}^{2\pi}(1-\cos\theta)\log w(\theta)\frac{d\theta}{2\pi}>-\infty \,\,\Longleftrightarrow\,\, \sum_{n=0}^{\infty}\Big(|\alpha_{n+1}-\alpha_{n}|^{2}+|\alpha_{n}|^{4}\Big)<\infty.\qedhere
$$
\end{proof}

\begin{rem}
Observing the above argument, by Lemma 4.1, sum rule (4.8) holds for any $\alpha$ once it holds for $\alpha\in\ell^{2}$. Hence, in what follows, we only prove some sum rules for $\alpha \in\ell^{2}$ since they also hold for any $\alpha$ by a similar argument to (4.8) from the special case of $\alpha\in\ell^{2}$ to general cases at some time.
\end{rem}

\subsection{Second order case} Secondly, by using the expressions of $w_{0}$, $w_{1}$ and $w_{2}$, we have
\begin{thm}
\begin{align}
&\int_{0}^{2\pi}\big(1-\cos^{2}\theta\big)\log w(\theta)\frac{d\theta}{2\pi}>-\infty
\Leftrightarrow\sum_{n=0}^{\infty}\left(\left|\alpha_{n+2}-\alpha_{n}\right|^{2}+|\alpha_{n}|^{4}\right)<\infty.
\end{align}
More precisely, for any $\alpha$,
\begin{align}
&\int_{0}^{2\pi}\big(1-\cos^{2}\theta\big)\log w(\theta)\frac{d\theta}{2\pi}\nonumber\\
=&\frac{3}{8}+\frac{1}{2}\sum_{j=0}^{\infty}\Big[\log(1-|\alpha_{j}|^{2})+|\alpha_{j}|^{2}+\frac{1}{2}|\alpha_{j}|^{4}\Big]\nonumber\\
&-\frac{1}{2}\sum_{j=0}^{\infty}|\alpha_{j}\alpha_{j-1}|^{2}-\frac{1}{4}\sum_{j=0}^{\infty}\rho_{j}^{2}|\alpha_{j+1}-\alpha_{j-1}|^{2}\nonumber\\
&-\frac{1}{16}\sum_{j=0}^{\infty}\left[\big(2|\alpha_{j}|^{2}-|\alpha_{j}-\alpha_{j-1}|^{2}\big)^{2}+\big(2|\alpha_{j-1}|^{2}-|\alpha_{j}-\alpha_{j-1}|^{2}\big)^{2}\right].
\end{align}
\end{thm}

\begin{proof}
Denote
\begin{equation}
Z_{2,1}(\mu)=\int_{0}^{2\pi}(1-\cos^{2}\theta)\log w(\theta)\frac{d\theta}{2\pi},
\end{equation}
then
\begin{equation}
Z_{2,1}(\mu)=\frac{1}{2}w_{0}-\frac{1}{2}\mathrm{Re}(w_{2}).
\end{equation}
For $\alpha\in \ell^{2}$, by (3.25), we have
\begin{align*}
\mathrm{Re}(w_{2})=-\sum_{j=0}^{\infty}\mathrm{Re}(\alpha_{j+1}\overline{\alpha}_{j-1})\rho_{j}^{2}+\frac{1}{2}\sum_{j=0}^{\infty}\mathrm{Re}\big(\alpha_{j}^{2}\overline{\alpha}_{j-1}^{2}\big)
\end{align*}
in which
\begin{align*}
\mathrm{Re}(\alpha_{j+1}\overline{\alpha}_{j-1})=&\frac{1}{2}(\alpha_{j+1}\overline{\alpha}_{j-1}+\overline{\alpha}_{j+1}\alpha_{j-1})\nonumber\\
=&\frac{1}{2}(|\alpha_{j+1}|^{2}+|\alpha_{j-1}|^{2}-|\alpha_{j+1}-\alpha_{j-1}|^{2})
\end{align*}
and
\begin{align*}
\mathrm{Re}\big(\alpha_{j}^{2}\overline{\alpha}_{j-1}^{2}\big)=&\frac{1}{2}\big(\alpha_{j}^{2}\overline{\alpha}_{j-1}^{2}+\overline{\alpha}_{j}^{2}\alpha_{j-1}^{2}\big)\nonumber\\
=&\frac{1}{2}[(\alpha_{j}\overline{\alpha}_{j-1}+\overline{\alpha}_{j}\alpha_{j-1})^{2}-2|\alpha_{j}|^{2}|\alpha_{j-1}|^{2}]\nonumber\\
=&\frac{1}{2}[(|\alpha_{j}|^{2}+|\alpha_{j-1}|^{2}-|\alpha_{j}-\alpha_{j-1}|^{2})^{2}-2|\alpha_{j}|^{2}|\alpha_{j-1}|^{2}]\nonumber\\
=&\frac{1}{2}[|\alpha_{j}|^{4}+|\alpha_{j-1}|^{4}+|\alpha_{j}-\alpha_{j-1}|^{4}-2|\alpha_{j}|^{2}|\alpha_{j}-\alpha_{j-1}|^{2}\nonumber\\
&-2|\alpha_{j-1}|^{2}|\alpha_{j}-\alpha_{j-1}|^{2}].
\end{align*}

More precisely,
\begin{align}
\mathrm{Re}(w_{2})=&-\frac{1}{2}\sum_{j=0}^{\infty}\rho_{j}^{2}(|\alpha_{j+1}|^{2}+|\alpha_{j-1}|^{2}-|\alpha_{j+1}-\alpha_{j-1}|^{2})\nonumber\\
&+\frac{1}{4}\sum_{j=0}^{\infty}\big(|\alpha_{j}|^{4}+|\alpha_{j-1}|^{4}+|\alpha_{j}-\alpha_{j-1}|^{4}-2|\alpha_{j}|^{2}|\alpha_{j}-\alpha_{j-1}|^{2}\nonumber\\
&-2|\alpha_{j-1}|^{2}|\alpha_{j}-\alpha_{j-1}|^{2}\big).
\end{align}

By (4.5), (4.14) and (4.15), we have
\begin{align*}
2Z_{2,1}(\mu)=&\sum_{j=0}^{\infty}\log(1-|\alpha_{j}|^{2})+\frac{1}{2}\sum_{j=0}^{\infty}\rho_{j}^{2}(|\alpha_{j+1}|^{2}+|\alpha_{j-1}|^{2}-|\alpha_{j+1}-\alpha_{j-1}|^{2})\nonumber\\
&-\frac{1}{4}\sum_{j=0}^{\infty}\big(|\alpha_{j}|^{4}+|\alpha_{j-1}|^{4}+|\alpha_{j}-\alpha_{j-1}|^{4}-2|\alpha_{j}|^{2}|\alpha_{j}-\alpha_{j-1}|^{2}\nonumber\\
&-2|\alpha_{j-1}|^{2}|\alpha_{j}-\alpha_{j-1}|^{2}\big)\nonumber\\
=&\frac{1}{2}|\alpha_{-1}|^{2}-\frac{1}{2}|\alpha_{0}|^{2}+\frac{1}{4}|\alpha_{-1}|^{4}+\frac{1}{2}|\alpha_{-1}|^{2}|\alpha_{0}|^{2}+\sum_{j=0}^{\infty}\Big[\log(1-|\alpha_{j}|^{2})+|\alpha_{j}|^{2}+\frac{1}{2}|\alpha_{j}|^{4}\Big]\nonumber\\
&-\sum_{j=0}^{\infty}|\alpha_{j}|^{2}|\alpha_{j-1}|^{2}-\frac{1}{2}\sum_{j=0}^{\infty}\rho_{j}^{2}|\alpha_{j+1}-\alpha_{j-1}|^{2}\nonumber\\
&-\frac{1}{4}\sum_{j=0}^{\infty}\big(2|\alpha_{j}|^{4}+2|\alpha_{j-1}|^{4}+|\alpha_{j}-\alpha_{j-1}|^{4}-2|\alpha_{j}|^{2}|\alpha_{j}-\alpha_{j-1}|^{2}\nonumber\\
&-2|\alpha_{j-1}|^{2}|\alpha_{j}-\alpha_{j-1}|^{2}\big)\nonumber\\
=&\frac{3}{4}-\sum_{j=0}^{\infty}\Big[-\log(1-|\alpha_{j}|^{2})-|\alpha_{j}|^{2}-\frac{1}{2}|\alpha_{j}|^{4}\Big]\nonumber\\
&-\sum_{j=0}^{\infty}|\alpha_{j}\alpha_{j-1}|^{2}-\frac{1}{2}\sum_{j=0}^{\infty}\rho_{j}^{2}|\alpha_{j+1}-\alpha_{j-1}|^{2}\nonumber\\
&-\frac{1}{8}\sum_{j=0}^{\infty}\left[\big(2|\alpha_{j}|^{2}-|\alpha_{j}-\alpha_{j-1}|^{2}\big)^{2}+\big(2|\alpha_{j-1}|^{2}-|\alpha_{j}-\alpha_{j-1}|^{2}\big)^{2}\right].
\end{align*}
So
\begin{align}
Z_{2,1}(\mu)=&\frac{3}{8}-\frac{1}{2}\sum_{j=0}^{\infty}\Big[-\log(1-|\alpha_{j}|^{2})-|\alpha_{j}|^{2}-\frac{1}{2}|\alpha_{j}|^{4}\Big]\nonumber\\
&-\frac{1}{2}\sum_{j=0}^{\infty}|\alpha_{j}\alpha_{j-1}|^{2}-\frac{1}{4}\sum_{j=0}^{\infty}\rho_{j}^{2}|\alpha_{j+1}-\alpha_{j-1}|^{2}\nonumber\\
&-\frac{1}{16}\sum_{j=0}^{\infty}\left[\big(2|\alpha_{j}|^{2}-|\alpha_{j}-\alpha_{j-1}|^{2}\big)^{2}+\big(2|\alpha_{j-1}|^{2}-|\alpha_{j}-\alpha_{j-1}|^{2}\big)^{2}\right].
\end{align}

For general $\alpha$, noting $Z_{2,1}(\mu)=Z_{Q}(\mu)$ with $Q(z)=-\frac{1}{2}(z^{2}-1)$, by a similar argument to (4.8) for any $\alpha$, we get that (4.16) also holds in this case. So (4.12) holds for any $\alpha$.

Now turn to (4.11). Suppose $Z_{2,1}(\mu)>-\infty$, by (4.16), then
\begin{align}
&\sum_{j=0}^{\infty}\Big[-\log(1-|\alpha_{j}|^{2})-|\alpha_{j}|^{2}-\frac{1}{2}|\alpha_{j}|^{4}\Big]<\infty;\\
&\sum_{j=0}^{\infty}|\alpha_{j}\alpha_{j-1}|^{2}<\infty;
\end{align}
\begin{align}
&\sum_{j=0}^{\infty}\rho_{j}^{2}|\alpha_{j+1}-\alpha_{j-1}|^{2}<\infty;\\
&\sum_{j=0}^{\infty}\big(2|\alpha_{j}|^{2}-|\alpha_{j}-\alpha_{j-1}|^{2}\big)^{2}<\infty
\end{align}
and
\begin{equation}
\hspace{-13mm}\sum_{j=0}^{\infty}\big(2|\alpha_{j-1}|^{2}-|\alpha_{j}-\alpha_{j-1}|^{2}\big)^{2}<\infty.
\end{equation}
Note that for $a,b\in \mathbb{C}$,
\begin{equation}
|a-b|^{2}\leq 2(|a|^{2}+|b|^{2}),
\end{equation}
setting $a=2|\alpha_{j}|^{2}-|\alpha_{j}-\alpha_{j-1}|^{2}$ and $b=2|\alpha_{j-1}|^{2}-|\alpha_{j}-\alpha_{j-1}|^{2}$, by (4.20) and (4.21), we have
\begin{equation}
\sum_{j=0}^{\infty}\big(|\alpha_{j}|^{2}-|\alpha_{j-1}|^{2}\big)^{2}<\infty.
\end{equation}

Since $\big(|\alpha_{j}|^{2}-|\alpha_{j-1}|^{2}\big)^{2}=|\alpha_{j}|^{4}+|\alpha_{j-1}|^{4}-2|\alpha_{j}\alpha_{j-1}|^{2}$, then
\begin{equation}
\sum_{j=0}^{\infty}\big(|\alpha_{j}|^{2}-|\alpha_{j-1}|^{2}\big)^{2}=1+2\sum_{j=0}^{\infty}|\alpha_{j}|^{4}-2\sum_{j=0}^{\infty}|\alpha_{j}\alpha_{j-1}|^{2}.
\end{equation}
Thus by (4.18) and (4.23),
\begin{equation}
\sum_{j=0}^{\infty}|\alpha_{j}|^{4}<\infty.
\end{equation}
Furthermore, since $\rho_{j}^{2}=1-|\alpha_{j}|^{2}$, by (4.19) and (4.25) (or (4.17)), we get
\begin{equation}
\sum_{j=0}^{\infty}|\alpha_{j+1}-\alpha_{j-1}|^{2}<\infty.
\end{equation}
That is, the $\Rightarrow$ direction in (4.11) is exact.

For the opposite direction, it is easy to know that (4.17)-(4.21) are consequences of (4.25) and (4.26). Thus (4.25) and (4.26) imply $Z_{2,1}(\mu)>-\infty$.
\end{proof}

\begin{thm}
\begin{align}
\int_{0}^{2\pi}(1-\cos\theta)^{2}\log w(\theta)\frac{d\theta}{2\pi}>-\infty \Longleftrightarrow \sum_{n=0}^{\infty}\Big(|\alpha_{n+2}-2\alpha_{n+1}+\alpha_{n}|^{2}+|\alpha_{n}|^{6}\Big)<\infty.
\end{align}
More precisely, for any $\alpha$,
\begin{align}
&\int_{0}^{2\pi}(1-\cos\theta)^{2}\log w(\theta)\frac{d\theta}{2\pi}\nonumber\\
=&\frac{9}{8}+\frac{3}{2}\sum_{j=0}^{\infty}\Big[\log(1-|\alpha_{j}|^{2})+|\alpha_{j}|^{2}+\frac{1}{2}|\alpha_{j}|^{4}\Big]-\frac{1}{4}\sum_{j=0}^{\infty}\Big(|\alpha_{j}|^{2}-|\alpha_{j-1}|^{2}\Big)^{2}\nonumber\\
&-\frac{1}{4}\sum_{j=0}^{\infty}\rho_{j}^{2}|\alpha_{j+1}-2\alpha_{j}+\alpha_{j-1}|^{2}\nonumber\\
&-\frac{1}{8}\sum_{j=0}^{\infty}\Big(6|\alpha_{j}|^{2}+6|\alpha_{j-1}|^{2}-|\alpha_{j}-\alpha_{j-1}|^{2}\Big)|\alpha_{j}-\alpha_{j-1}|^{2}\\
=&\frac{9}{8}+\frac{3}{2}\sum_{j=0}^{\infty}\Big[\log(1-|\alpha_{j}|^{2})+|\alpha_{j}|^{2}+\frac{1}{2}|\alpha_{j}|^{4}\Big]-\frac{1}{4}\sum_{j=0}^{\infty}\Big(|\alpha_{j}|^{2}-|\alpha_{j-1}|^{2}\Big)^{2}\nonumber\\
&-\frac{1}{4}\sum_{j=0}^{\infty}\rho_{j}^{2}|\alpha_{j+1}-2\alpha_{j}+\alpha_{j-1}|^{2}\nonumber\\
&-\frac{1}{8}\sum_{j=0}^{\infty}\Big(4|\alpha_{j}|^{2}+4|\alpha_{j-1}|^{2}+|\alpha_{j}+\alpha_{j-1}|^{2}\Big)|\alpha_{j}-\alpha_{j-1}|^{2}.
\end{align}
\end{thm}
\begin{proof}
Denote
\begin{equation}
Z_{2,2}(\mu)=\int_{0}^{2\pi}(1-\cos\theta)^{2}\log w(\theta)\frac{d\theta}{2\pi},
\end{equation}
then
\begin{equation}
Z_{2,2}(\mu)=\frac{3}{2}w_{0}-2\mathrm{Re}(w_{1})+\frac{1}{2}\mathrm{Re}(w_{2})=2Z_{1}(\mu)-Z_{2,1}(\mu).
\end{equation}
At first, assume that $\alpha\in\ell^{2}$. By (4.8), (4.16) and (4.31), we have
\begin{align}
Z_{2,2}(\mu)=&\frac{5}{8}+\frac{3}{2}\sum_{j=0}^{\infty}\Big[\log(1-|\alpha_{j}|^{2})+|\alpha_{j}|^{2}+\frac{1}{2}|\alpha_{j}|^{4}\Big]-\sum_{j=0}^{\infty}|\alpha_{j}|^{4}-\sum_{j=0}^{\infty}|\alpha_{j}-\alpha_{j-1}|^{2}\nonumber\\
&+\frac{1}{2}\sum_{j=0}^{\infty}|\alpha_{j}\alpha_{j-1}|^{2}+\frac{1}{4}\sum_{j=0}^{\infty}\rho_{j}^{2}|\alpha_{j+1}-\alpha_{j-1}|^{2}\nonumber\\
&+\frac{1}{16}\sum_{j=0}^{\infty}\left[\big(2|\alpha_{j}|^{2}-|\alpha_{j}-\alpha_{j-1}|^{2}\big)^{2}+\big(2|\alpha_{j-1}|^{2}-|\alpha_{j}-\alpha_{j-1}|^{2}\big)^{2}\right]
\end{align}
in which
\begin{align}
&\frac{1}{4}\sum_{j=0}^{\infty}\rho_{j}^{2}|\alpha_{j+1}-\alpha_{j-1}|^{2}=\frac{1}{4}\sum_{j=0}^{\infty}\rho_{j}^{2}\Big[2|\alpha_{j+1}-\alpha_{j}|^{2}+2|\alpha_{j}-\alpha_{j-1}|^{2}\nonumber\\
&-|\alpha_{j+1}-2\alpha_{j}+\alpha_{j-1}|^{2}\Big]\nonumber\\
=&\sum_{j=0}^{\infty}|\alpha_{j}-\alpha_{j-1}|^{2}-\frac{1}{2}\sum_{j=0}^{\infty}\Big(|\alpha_{j}|^{2}+|\alpha_{j-1}|^{2}\Big)|\alpha_{j}-\alpha_{j-1}|^{2}\nonumber\\
&-\frac{1}{4}\sum_{j=0}^{\infty}\rho_{j}^{2}|\alpha_{j+1}-2\alpha_{j}+\alpha_{j-1}|^{2}
\end{align}
and
\begin{align}
&\frac{1}{16}\sum_{j=0}^{\infty}\left[\big(2|\alpha_{j}|^{2}-|\alpha_{j}-\alpha_{j-1}|^{2}\big)^{2}+\big(2|\alpha_{j-1}|^{2}-|\alpha_{j}-\alpha_{j-1}|^{2}\big)^{2}\right]\nonumber\\
=&\frac{1}{4}\sum_{j=0}^{\infty}\Big(|\alpha_{j}|^{4}+|\alpha_{j-1}|^{4}\Big)+\frac{1}{8}\sum_{j=0}^{\infty}|\alpha_{j}-\alpha_{j-1}|^{4}\nonumber\\
&-\frac{1}{4}\sum_{j=0}^{\infty}\Big(|\alpha_{j}|^{2}+|\alpha_{j-1}|^{2}\Big)|\alpha_{j}-\alpha_{j-1}|^{2}.
\end{align}
Note that
\begin{equation}
\sum_{j=0}^{\infty}|\alpha_{j}|^{4}=-\frac{1}{2}+\frac{1}{2}\sum_{j=0}^{\infty}\Big(|\alpha_{j}|^{4}+|\alpha_{j-1}|^{4}\Big),
\end{equation}
then by (4.33)-(4.35), we obtain
\begin{align}
Z_{2,2}(\mu)=&\frac{9}{8}+\frac{3}{2}\sum_{j=0}^{\infty}\Big[\log(1-|\alpha_{j}|^{2})+|\alpha_{j}|^{2}+\frac{1}{2}|\alpha_{j}|^{4}\Big]-\frac{1}{4}\sum_{j=0}^{\infty}\Big(|\alpha_{j}|^{2}-|\alpha_{j-1}|^{2}\Big)^{2}\nonumber\\
&-\frac{1}{4}\sum_{j=0}^{\infty}\rho_{j}^{2}|\alpha_{j+1}-2\alpha_{j}+\alpha_{j-1}|^{2}\nonumber\\
&-\frac{1}{8}\sum_{j=0}^{\infty}\Big(6|\alpha_{j}|^{2}+6|\alpha_{j-1}|^{2}-|\alpha_{j}-\alpha_{j-1}|^{2}\Big)|\alpha_{j}-\alpha_{j-1}|^{2}\\
=&\frac{9}{8}+\frac{3}{2}\sum_{j=0}^{\infty}\Big[\log(1-|\alpha_{j}|^{2})+|\alpha_{j}|^{2}+\frac{1}{2}|\alpha_{j}|^{4}\Big]-\frac{1}{4}\sum_{j=0}^{\infty}\Big(|\alpha_{j}|^{2}-|\alpha_{j-1}|^{2}\Big)^{2}\nonumber\\
&-\frac{1}{4}\sum_{j=0}^{\infty}\rho_{j}^{2}|\alpha_{j+1}-2\alpha_{j}+\alpha_{j-1}|^{2}\nonumber\\
&-\frac{1}{8}\sum_{j=0}^{\infty}\Big(4|\alpha_{j}|^{2}+4|\alpha_{j-1}|^{2}+|\alpha_{j}+\alpha_{j-1}|^{2}\Big)|\alpha_{j}-\alpha_{j-1}|^{2}.
\end{align}
By using Lemma 4.1 and a similar argument to (4.8) for any $\alpha$, (4.36) and (4.37) also hold for any $\alpha$. That is, (4.28) and (4.29) hold for any any $\alpha$.

By Lemma 4.3, (4.30) and (4.37), we have that
\begin{equation}
\int_{0}^{2\pi}(1-\cos\theta)^{2}\log w(\theta)\frac{d\theta}{2\pi}>-\infty
\end{equation}
is equivalent to all of the following conditions
\begin{align}
&\sum_{j=0}^{\infty}|\alpha_{j}|^{6}<+\infty,\\
&\sum_{j=0}^{\infty}\big(|\alpha_{j}|^{2}-|\alpha_{j-1}|^{2}\big)^{2}<+\infty,\\
&\sum_{j=0}^{\infty}\rho_{j}^{2}|\alpha_{j+1}-2\alpha_{j}+\alpha_{j-1}|^{2}<+\infty
\end{align}
and
\begin{equation}
\sum_{j=0}^{\infty}\Big(4|\alpha_{j}|^{2}+4|\alpha_{j-1}|^{2}+|\alpha_{j}+\alpha_{j-1}|^{2}\Big)|\alpha_{j}-\alpha_{j-1}|^{2}<+\infty.
\end{equation}

On one hand, it is easy to find that (4.39) and (4.41) imply
\begin{equation}
\sum_{j=0}^{\infty}|\alpha_{j+1}-2\alpha_{j}+\alpha_{j-1}|^{2}<+\infty
\end{equation}
since $\rho_{j}\geq \frac{1}{2}$ for sufficiently large $j$ following from (4.39).

On the other hand, by the following discrete Gagliardo-Nirenberg inequality (see \cite{bbz1} or \cite{sz})
\begin{equation}
\|(S-1)\alpha\|_{3}^{2}\leq 2\|(S-1)^{2}\alpha\|_{2}\|\alpha\|_{6},
\end{equation}
we know that (4.39) and (4.43) imply
\begin{equation}
\sum_{j=0}^{\infty}|\alpha_{j}-\alpha_{j-1}|^{3}<+\infty.
\end{equation}
By (4.22), triangle inequality and H\"older inequality, we further get (4.40)-(4.42) from (4.39), (4.43) and (4.45).
That is to say that (4.39)-(4.42) is equivalent to (4.39) and (4.43). Thus (4.27) holds.
\end{proof}

\begin{rem}
Splitting
\begin{align}
\frac{1}{4}\sum_{j=0}^{\infty}\rho_{j}^{2}|\alpha_{j+1}-\alpha_{j-1}|^{2}=\frac{1}{4}\sum_{j=0}^{\infty}|\alpha_{j+1}-\alpha_{j-1}|^{2}
-\frac{1}{4}\sum_{j=0}^{\infty}|\alpha_{j}|^{2}|\alpha_{j+1}-\alpha_{j-1}|^{2}
\end{align}
and
noting
\begin{align}
&\sum_{j=0}^{\infty}|\alpha_{j+1}-\alpha_{j}|^{2}=\sum_{j=0}^{\infty}\Big(2|\alpha_{j+1}-\alpha_{j}|^{2}+2|\alpha_{j}-\alpha_{j-1}|^{2}-|\alpha_{j+1}-2\alpha_{j}+\alpha_{j-1}|^{2}\Big)\nonumber\\
=&-2|\alpha_{0}+1|^{2}+4\sum_{j=0}^{\infty}|\alpha_{j}-\alpha_{j-1}|^{2}
-\sum_{j=0}^{\infty}|\alpha_{j+1}-2\alpha_{j}+\alpha_{j-1}|^{2},
\end{align}
by (4.32), (4.34) and (4.35), we also obtain
\begin{align}
&\int_{0}^{2\pi}(1-\cos\theta)^{2}\log w(\theta)\frac{d\theta}{2\pi}\nonumber\\
=&\frac{9}{8}-\frac{1}{2}|\alpha_{0}+1|^{2}+\frac{3}{2}\sum_{j=0}^{\infty}\Big[\log(1-|\alpha_{j}|^{2})+|\alpha_{j}|^{2}+\frac{1}{2}|\alpha_{j}|^{4}\Big]-\frac{1}{4}\sum_{j=0}^{\infty}\Big(|\alpha_{j}|^{2}-|\alpha_{j-1}|^{2}\Big)^{2}\nonumber\\
&-\frac{1}{4}\sum_{j=0}^{\infty}|\alpha_{j+1}-2\alpha_{j}+\alpha_{j-1}|^{2}-\frac{1}{4}\sum_{j=0}^{\infty}|\alpha_{j}|^{2}|\alpha_{j+1}-\alpha_{j-1}|^{2}\nonumber\\
&-\frac{1}{8}\sum_{j=0}^{\infty}\Big(2|\alpha_{j}|^{2}+2|\alpha_{j-1}|^{2}-|\alpha_{j}-\alpha_{j-1}|^{2}\Big)|\alpha_{j}-\alpha_{j-1}|^{2}\\
=&\frac{9}{8}-\frac{1}{2}|\alpha_{0}+1|^{2}+\frac{3}{2}\sum_{j=0}^{\infty}\Big[\log(1-|\alpha_{j}|^{2})+|\alpha_{j}|^{2}+\frac{1}{2}|\alpha_{j}|^{4}\Big]-\frac{1}{4}\sum_{j=0}^{\infty}\Big(|\alpha_{j}|^{2}-|\alpha_{j-1}|^{2}\Big)^{2}\nonumber\\
&-\frac{1}{4}\sum_{j=0}^{\infty}|\alpha_{j+1}-2\alpha_{j}+\alpha_{j-1}|^{2}-\frac{1}{4}\sum_{j=0}^{\infty}|\alpha_{j}|^{2}|\alpha_{j+1}-\alpha_{j-1}|^{2}\nonumber\\
&-\frac{1}{8}\sum_{j=0}^{\infty}|\alpha_{j}+\alpha_{j-1}|^{2}|\alpha_{j}-\alpha_{j-1}|^{2}.
\end{align}
Similarly, we can deduce (4.27) from (4.48) and (4.49).
\end{rem}

\subsection{Third order case} Next, we apply the expressions of $w_{j}$, $0\leq j\leq 3$ to get some some rules, and further obtain some gems under appropriate conditions from these sum rules.
In order to do so, we need the following simple facts on real part of a product of two and three complex variables, some of which have been used in the above two cases such as (4.7).
\begin{prop}
For $A, B, C\in \mathbb{C}$,
\begin{align}
\mathrm{Re}(AB)=&\frac{1}{2}(AB+\overline{A}\,\overline{B})=\frac{1}{2}[(A+\overline{A})(B+\overline{B})-(A\overline{B}+\overline{A}B)]\nonumber\\
=&2\mathrm{Re}(A)\mathrm{Re}(B)-\mathrm{Re}(A\overline{B}),
\end{align}
\begin{align}
\mathrm{Re}(A\overline{B})=&\frac{1}{2}(A\overline{B}+\overline{A}B)=\frac{1}{2}[|A|^{2}+|B|^{2}-(A-B)(\overline{A}-\overline{B})]\nonumber\\
=&\frac{1}{2}(|A|^{2}+|B|^{2}-|A-B|^{2})
\end{align}
and
\begin{align}
\mathrm{Re}(ABC)=&\mathrm{Re}(A)\mathrm{Re}(BC)+\mathrm{Re}(B)\mathrm{Re}(AC)+\mathrm{Re}(C)\mathrm{Re}(AB)\nonumber\\
&-2\mathrm{Re}(A)\mathrm{Re}(B)\mathrm{Re}(C)\\
=&4\mathrm{Re}(A)\mathrm{Re}(B)\mathrm{Re}(C)-\mathrm{Re}(A)\mathrm{Re}(B\overline{C})-\mathrm{Re}(B)\mathrm{Re}(A\overline{C})\nonumber\\
&-\mathrm{Re}(C)\mathrm{Re}(A\overline{B}).
\end{align}
\end{prop}

\begin{proof}
(4.50) and (4.51) are trivial. (4.53) follows from (4.50) and (4.52). It is sufficient to get (4.52).
By (4.50), we have
\begin{align}
\mathrm{Re}(ABC)=2\mathrm{Re}(A)\mathrm{Re}(BC)-\mathrm{Re}(A\overline{B}\,\overline{C}),
\end{align}
\begin{align}
\mathrm{Re}(A\overline{B}\,\overline{C})=2\mathrm{Re}(\overline{B})\mathrm{Re}(A\overline{C})-\mathrm{Re}(\overline{A}\,\overline{B}C)
\end{align}
and
\begin{align}
\mathrm{Re}(\overline{A}\,\overline{B}C)=2\mathrm{Re}(C)\mathrm{Re}(\overline{A}\,\overline{B})-\mathrm{Re}(ABC).
\end{align}
Noting that $\mathrm{Re}(\overline{A})=\mathrm{Re}(A)$ and
\begin{align*}
\mathrm{Re}(A\overline{C})=2\mathrm{Re}(A)\mathrm{Re}(C)-\mathrm{Re}(AC),
\end{align*}
then by (4.54)-(4.56), we get
\begin{align*}
\mathrm{Re}(ABC)=&\mathrm{Re}(A)\mathrm{Re}(BC)-\mathrm{Re}(\overline{B})\mathrm{Re}(A\overline{C})+\mathrm{Re}(C)\mathrm{Re}(\overline{A}\,\overline{B})\\
=&\mathrm{Re}(A)\mathrm{Re}(BC)+\mathrm{Re}(B)\mathrm{Re}(AC)+\mathrm{Re}(C)\mathrm{Re}(AB)\nonumber\\
&-2\mathrm{Re}(A)\mathrm{Re}(B)\mathrm{Re}(C).\qedhere
\end{align*}
\end{proof}

Denote
\begin{equation}
Z_{3,1}(\mu)=\int_{0}^{2\pi}(1-\cos3\theta)\log w(\theta)\frac{d\theta}{2\pi},
\end{equation}
\begin{equation}
Z_{3,2}(\mu)=\int_{0}^{2\pi}(1-\cos\theta)^{2}(1+\cos\theta)\log w(\theta)\frac{d\theta}{2\pi}
\end{equation}
and
\begin{equation}
Z_{3,3}(\mu)=\int_{0}^{2\pi}(1-\cos\theta)^{3}\log w(\theta)\frac{d\theta}{2\pi}.
\end{equation}

\begin{thm}
For any $\alpha$,
\begin{align}
&\int_{0}^{2\pi}(1-\cos3\theta)\log w(\theta)\frac{d\theta}{2\pi}\nonumber\\
=&\frac{2}{3}-\frac{1}{2}(|\alpha_{0}|^{2}+|\alpha_{1}|^{2})-\frac{1}{2}|\alpha_{0}|^{2}|\alpha_{0}+1|^{2}\nonumber\\
&+\sum_{j=0}^{\infty}\Big[\log(1-|\alpha_{j}|^{2})+|\alpha_{j}|^{2}+\frac{1}{2}|\alpha_{j}|^{4}+\frac{1}{3}|\alpha_{j}|^{6}\Big]\nonumber\\
&-\frac{1}{2}\sum_{j=0}^{\infty}|\alpha_{j}|^{4}-\frac{1}{2}\sum_{j=0}^{\infty}(|\alpha_{j+2}|^{2}+|\alpha_{j-1}|^{2})|\alpha_{j}|^{2}\nonumber\\
&-\frac{1}{2}\sum_{j=0}^{\infty}(|\alpha_{j+2}|^{2}+|\alpha_{j-1}|^{2})|\alpha_{j+1}|^{2}\rho_{j}^{2}-\frac{1}{2}\sum_{j=0}^{\infty}|\alpha_{j+2}-\alpha_{j-1}|^{2}\rho_{j+1}^{2}\rho_{j}^{2}\nonumber\\
&-\frac{1}{2}\sum_{j=0}^{\infty}\Big[|\alpha_{j+1}|^{2}|\alpha_{j}|^{2}+|\alpha_{j-1}|^{4}+|\alpha_{j}|^{2}|\alpha_{j-1}|^{2}+|\alpha_{j+1}|^{4}\nonumber
\end{align}
\begin{align}
&+|\alpha_{j+1}-\alpha_{j-1}|^{2}\big(|\alpha_{j}-\alpha_{j-1}|^{2}+|\alpha_{j+1}-\alpha_{j}|^{2}\big)\Big]\rho_{j}^{2}\nonumber\\
&-\frac{1}{2}\sum_{j=0}^{\infty}\Big[(|\alpha_{j+1}|^{2}+2|\alpha_{j}|^{2}+|\alpha_{j-1}|^{2})|\alpha_{j+1}-\alpha_{j-1}|^{2}\nonumber\\
&+|\alpha_{j+1}|^{2}|\alpha_{j+1}-\alpha_{j}|^{2}\Big]|\alpha_{j}|^{2}\nonumber\\
&-\frac{1}{6}\sum_{j=0}^{\infty}|\alpha_{j}-\alpha_{j-1}|^{6}-\frac{1}{2}\sum_{j=0}^{\infty}\big(|\alpha_{j}|^{2}+|\alpha_{j-1}|^{2}\big)^{2}|\alpha_{j}-\alpha_{j-1}|^{2}\nonumber\\
&+\frac{1}{2}\sum_{j=0}^{\infty}\Big[(|\alpha_{j+1}|^{2}+2|\alpha_{j}|^{2}+|\alpha_{j-1}|^{2})|\alpha_{j+1}-\alpha_{j-1}|^{2}\nonumber\\
&+(|\alpha_{j}|^{2}+|\alpha_{j-1}|^{2})(|\alpha_{j}-\alpha_{j-1}|^{2}+|\alpha_{j}-\alpha_{j-1}|^{4})\Big].
\end{align}
\end{thm}

\begin{proof}
Since
$$1-\cos3\theta=1-\frac{e^{i3\theta}+e^{-i3\theta}}{2},$$
then
\begin{equation}
Z_{3,1}(\mu)=w_{0}-\mathrm{Re}(w_{3}).
\end{equation}
By (3.32), we have
\begin{align}
\mathrm{Re}(w_{3})=&-\sum_{j=0}^{\infty}\rho_{j+1}^{2}\rho_{j}^{2}\mathrm{Re}(\alpha_{j+2}\overline{\alpha}_{j-1})
+\sum_{j=0}^{\infty}\rho_{j}^{2}\mathrm{Re}(\alpha_{j+1}\alpha_{j}\overline{\alpha}_{j-1}^{2})\nonumber\\
&+\sum_{j=0}^{\infty}\rho_{j}^{2}\mathrm{Re}(\alpha_{j+1}^{2}\overline{\alpha}_{j}\overline{\alpha}_{j-1})
-\frac{1}{3}\sum_{j=0}^{\infty}\mathrm{Re}(\alpha_{j}^{3}\overline{\alpha}_{j-1}^{3}).
\end{align}
By (4.50), (4.51) and (4.53),
\begin{align}
\mathrm{Re}(\alpha_{j+2}\overline{\alpha}_{j-1})=\frac{1}{2}(|\alpha_{j+2}|^{2}+|\alpha_{j-1}|^{2}-|\alpha_{j+2}-\alpha_{j-1}|^{2}),
\end{align}
\begin{align}
&\mathrm{Re}(\alpha_{j+1}\alpha_{j}\overline{\alpha}_{j-1}^{2})=\mathrm{Re}\big((\alpha_{j+1}\overline{\alpha}_{j-1})(\alpha_{j}\overline{\alpha}_{j-1})\big)\nonumber\\
=&2\mathrm{Re}(\alpha_{j+1}\overline{\alpha}_{j-1})\mathrm{Re}(\alpha_{j}\overline{\alpha}_{j-1})-\mathrm{Re}(\alpha_{j+1}\overline{\alpha}_{j-1}\overline{\alpha_{j}\overline{\alpha}_{j-1}})\nonumber\\
=&\frac{1}{2}(|\alpha_{j+1}|^{2}+|\alpha_{j-1}|^{2}-|\alpha_{j+1}-\alpha_{j-1}|^{2})(|\alpha_{j}|^{2}+|\alpha_{j-1}|^{2}-|\alpha_{j}-\alpha_{j-1}|^{2})\nonumber\\
&-\frac{1}{2}|\alpha_{j-1}|^{2}(|\alpha_{j+1}|^{2}+|\alpha_{j}|^{2}-|\alpha_{j+1}-\alpha_{j}|^{2}),
\end{align}
\begin{align}
&\mathrm{Re}(\alpha_{j+1}^{2}\overline{\alpha}_{j}\overline{\alpha}_{j-1})=\mathrm{Re}\big((\alpha_{j+1}\overline{\alpha}_{j})(\alpha_{j+1}\overline{\alpha}_{j-1})\big)\nonumber\\
=&2\mathrm{Re}(\alpha_{j+1}\overline{\alpha}_{j})\mathrm{Re}(\alpha_{j+1}\overline{\alpha}_{j-1})-\mathrm{Re}(\alpha_{j+1}\overline{\alpha}_{j}\overline{\alpha_{j+1}\overline{\alpha}_{j-1}})\nonumber\\
=&\frac{1}{2}(|\alpha_{j+1}|^{2}+|\alpha_{j}|^{2}-|\alpha_{j+1}-\alpha_{j}|^{2})(|\alpha_{j+1}|^{2}+|\alpha_{j-1}|^{2}-|\alpha_{j+1}-\alpha_{j-1}|^{2})\nonumber\\
&-\frac{1}{2}|\alpha_{j+1}|^{2}(|\alpha_{j}|^{2}+|\alpha_{j-1}|^{2}-|\alpha_{j}-\alpha_{j-1}|^{2})
\end{align}
and
\begin{align}
&\mathrm{Re}(\alpha_{j}^{3}\overline{\alpha}_{j-1}^{3})=\mathrm{Re}\big((\alpha_{j}\overline{\alpha}_{j-1})(\alpha_{j}\overline{\alpha}_{j-1})(\alpha_{j}\overline{\alpha}_{j-1})\big)\nonumber\\
=&\mathrm{Re}(\alpha_{j}\overline{\alpha}_{j-1})\big(4(\mathrm{Re}(\alpha_{j}\overline{\alpha}_{j-1}))^{2}-3|\alpha_{j}|^{2}|\overline{\alpha}_{j-1}|^{2}\big)\nonumber\\
=&\frac{1}{2}(|\alpha_{j}|^{2}+|\alpha_{j-1}|^{2}-|\alpha_{j}-\alpha_{j-1}|^{2})^{3}\nonumber\\
&-\frac{3}{2}|\alpha_{j}|^{2}|\alpha_{j-1}|^{2}(|\alpha_{j}|^{2}+|\alpha_{j-1}|^{2}-|\alpha_{j}-\alpha_{j-1}|^{2}).
\end{align}
From (4.63),
\begin{align}
&\sum_{j=0}^{\infty}\mathrm{Re}(\alpha_{j+2}\overline{\alpha}_{j-1})\rho_{j+1}^{2}\rho_{j}^{2}=\frac{1}{2}\sum_{j=0}^{\infty}(|\alpha_{j+2}|^{2}+|\alpha_{j-1}|^{2}-|\alpha_{j+2}-\alpha_{j-1}|^{2})\rho_{j+1}^{2}\rho_{j}^{2}\nonumber\\
=&\frac{1}{2}\sum_{j=0}^{\infty}(|\alpha_{j+2}|^{2}+|\alpha_{j-1}|^{2})-\frac{1}{2}\sum_{j=0}^{\infty}(|\alpha_{j+2}|^{2}+|\alpha_{j-1}|^{2})|\alpha_{j}|^{2}\nonumber\\
&-\frac{1}{2}\sum_{j=0}^{\infty}(|\alpha_{j+2}|^{2}+|\alpha_{j-1}|^{2})|\alpha_{j+1}|^{2}\rho_{j}^{2}-\frac{1}{2}\sum_{j=0}^{\infty}|\alpha_{j+2}-\alpha_{j-1}|^{2}\rho_{j+1}^{2}\rho_{j}^{2}.
\end{align}
From (4.64),
\begin{align}
&\sum_{j=0}^{\infty}\mathrm{Re}(\alpha_{j+1}\alpha_{j}\overline{\alpha}_{j-1}^{2})\rho_{j}^{2}=\frac{1}{2}\sum_{j=0}^{\infty}\Big[(|\alpha_{j+1}|^{2}+|\alpha_{j-1}|^{2}-|\alpha_{j+1}-\alpha_{j-1}|^{2})\nonumber\\
&(|\alpha_{j}|^{2}+|\alpha_{j-1}|^{2}-|\alpha_{j}-\alpha_{j-1}|^{2})-|\alpha_{j-1}|^{2}(|\alpha_{j+1}|^{2}+|\alpha_{j}|^{2}-|\alpha_{j+1}-\alpha_{j}|^{2})\Big]\rho_{j}^{2}\nonumber\\
=&\frac{1}{2}\sum_{j=0}^{\infty}\Big[|\alpha_{j+1}|^{2}|\alpha_{j}|^{2}+|\alpha_{j-1}|^{4}-(|\alpha_{j}|^{2}+|\alpha_{j-1}|^{2})|\alpha_{j+1}-\alpha_{j-1}|^{2}\nonumber\\
&-(|\alpha_{j+1}|^{2}+|\alpha_{j-1}|^{2})|\alpha_{j}-\alpha_{j-1}|^{2}+|\alpha_{j-1}|^{2}|\alpha_{j+1}-\alpha_{j}|^{2}\nonumber\\
&+|\alpha_{j+1}-\alpha_{j-1}|^{2}|\alpha_{j}-\alpha_{j-1}|^{2}\Big]\rho_{j}^{2}.
\end{align}
From (4.65),
\begin{align}
&\sum_{j=0}^{\infty}\mathrm{Re}(\alpha_{j+1}^{2}\overline{\alpha}_{j}\overline{\alpha}_{j-1})\rho_{j}^{2}=\frac{1}{2}\sum_{j=0}^{\infty}\Big[(|\alpha_{j+1}|^{2}+|\alpha_{j}|^{2}-|\alpha_{j+1}-\alpha_{j}|^{2})\nonumber\\
&(|\alpha_{j+1}|^{2}+|\alpha_{j-1}|^{2}-|\alpha_{j+1}-\alpha_{j-1}|^{2})-|\alpha_{j+1}|^{2}(|\alpha_{j}|^{2}+|\alpha_{j-1}|^{2}-|\alpha_{j}-\alpha_{j-1}|^{2})\Big]\rho_{j}^{2}\nonumber\\
=&\frac{1}{2}\sum_{j=0}^{\infty}\Big[|\alpha_{j}|^{2}|\alpha_{j-1}|^{2}+|\alpha_{j+1}|^{4}-(|\alpha_{j+1}|^{2}+|\alpha_{j}|^{2})|\alpha_{j+1}-\alpha_{j-1}|^{2}\nonumber\\
&-(|\alpha_{j+1}|^{2}+|\alpha_{j-1}|^{2})|\alpha_{j+1}-\alpha_{j}|^{2}+|\alpha_{j+1}|^{2}|\alpha_{j}-\alpha_{j-1}|^{2}\nonumber\\
&+|\alpha_{j+1}-\alpha_{j-1}|^{2}|\alpha_{j+1}-\alpha_{j}|^{2}\Big]\rho_{j}^{2}.
\end{align}
From (4.66),
\begin{align}
&\sum_{j=0}^{\infty}\mathrm{Re}(\alpha_{j}^{3}\overline{\alpha}_{j-1}^{3})=\frac{1}{2}\sum_{j=0}^{\infty}\Big[(|\alpha_{j}|^{2}+|\alpha_{j-1}|^{2}-|\alpha_{j}-\alpha_{j-1}|^{2})^{3}\nonumber\\
&-3|\alpha_{j}|^{2}|\alpha_{j-1}|^{2}(|\alpha_{j}|^{2}+|\alpha_{j-1}|^{2}-|\alpha_{j}-\alpha_{j-1}|^{2})\Big]\nonumber\\
=&\frac{1}{2}\sum_{j=0}^{\infty}\Big[|\alpha_{j}|^{6}+|\alpha_{j-1}|^{6}-|\alpha_{j}-\alpha_{j-1}|^{6}-3\big(|\alpha_{j}|^{4}|\alpha_{j}-\alpha_{j-1}|^{2}\nonumber\\
&+|\alpha_{j-1}|^{4}|\alpha_{j}-\alpha_{j-1}|^{2}-|\alpha_{j}|^{2}|\alpha_{j}-\alpha_{j-1}|^{4}-|\alpha_{j-1}|^{2}|\alpha_{j}-\alpha_{j-1}|^{4}\nonumber\\
&+|\alpha_{j}|^{2}|\alpha_{j-1}|^{2}|\alpha_{j}-\alpha_{j-1}|^{2}\big)\Big].
\end{align}
By (4.61), (4.62), (4.67)-(4.70), we have
\begin{align}
Z_{3,1}(\mu)=&\sum_{j=0}^{\infty}\log(1-|\alpha_{j}|^{2})+\frac{1}{2}\sum_{j=0}^{\infty}(|\alpha_{j+2}|^{2}+|\alpha_{j-1}|^{2})-\frac{1}{2}\sum_{j=0}^{\infty}(|\alpha_{j+2}|^{2}+|\alpha_{j-1}|^{2})|\alpha_{j}|^{2}\nonumber\\
&-\frac{1}{2}\sum_{j=0}^{\infty}(|\alpha_{j+2}|^{2}+|\alpha_{j-1}|^{2})|\alpha_{j+1}|^{2}\rho_{j}^{2}-\frac{1}{2}\sum_{j=0}^{\infty}|\alpha_{j+2}-\alpha_{j-1}|^{2}\rho_{j+1}^{2}\rho_{j}^{2}\nonumber\\
&-\frac{1}{2}\sum_{j=0}^{\infty}\Big[|\alpha_{j+1}|^{2}|\alpha_{j}|^{2}+|\alpha_{j-1}|^{4}+|\alpha_{j}|^{2}|\alpha_{j-1}|^{2}+|\alpha_{j+1}|^{4}\nonumber\\
&+|\alpha_{j+1}-\alpha_{j-1}|^{2}\big(|\alpha_{j}-\alpha_{j-1}|^{2}+|\alpha_{j+1}-\alpha_{j}|^{2}\big)\nonumber\\&-(|\alpha_{j+1}|^{2}+2|\alpha_{j}|^{2}+|\alpha_{j-1}|^{2})|\alpha_{j+1}-\alpha_{j-1}|^{2}\nonumber\\
&-|\alpha_{j-1}|^{2}|\alpha_{j}-\alpha_{j-1}|^{2}-|\alpha_{j+1}|^{2}|\alpha_{j+1}-\alpha_{j}|^{2}\Big]\rho_{j}^{2}\nonumber\\
&+\frac{1}{6}\sum_{j=0}^{\infty}\Big[|\alpha_{j}|^{6}+|\alpha_{j-1}|^{6}-|\alpha_{j}-\alpha_{j-1}|^{6}-3\big(|\alpha_{j}|^{4}|\alpha_{j}-\alpha_{j-1}|^{2}\nonumber\\
&+|\alpha_{j-1}|^{4}|\alpha_{j}-\alpha_{j-1}|^{2}-|\alpha_{j}|^{2}|\alpha_{j}-\alpha_{j-1}|^{4}-|\alpha_{j-1}|^{2}|\alpha_{j}-\alpha_{j-1}|^{4}\nonumber\\
&+|\alpha_{j}|^{2}|\alpha_{j-1}|^{2}|\alpha_{j}-\alpha_{j-1}|^{2}\big)\Big]\nonumber\\
=&\frac{2}{3}-\frac{1}{2}(|\alpha_{0}|^{2}+|\alpha_{1}|^{2})-\frac{1}{2}|\alpha_{0}|^{2}|\alpha_{0}+1|^{2}\nonumber\\
&+\sum_{j=0}^{\infty}\Big[\log(1-|\alpha_{j}|^{2})+|\alpha_{j}|^{2}+\frac{1}{2}|\alpha_{j}|^{4}+\frac{1}{3}|\alpha_{j}|^{6}\Big]\nonumber\\
&-\frac{1}{2}\sum_{j=0}^{\infty}|\alpha_{j}|^{4}-\frac{1}{2}\sum_{j=0}^{\infty}(|\alpha_{j+2}|^{2}+|\alpha_{j-1}|^{2})|\alpha_{j}|^{2}\nonumber\\
&-\frac{1}{2}\sum_{j=0}^{\infty}(|\alpha_{j+2}|^{2}+|\alpha_{j-1}|^{2})|\alpha_{j+1}|^{2}\rho_{j}^{2}-\frac{1}{2}\sum_{j=0}^{\infty}|\alpha_{j+2}-\alpha_{j-1}|^{2}\rho_{j+1}^{2}\rho_{j}^{2}\nonumber\\
&-\frac{1}{2}\sum_{j=0}^{\infty}\Big[|\alpha_{j+1}|^{2}|\alpha_{j}|^{2}+|\alpha_{j-1}|^{4}+|\alpha_{j}|^{2}|\alpha_{j-1}|^{2}+|\alpha_{j+1}|^{4}\nonumber\\
&+|\alpha_{j+1}-\alpha_{j-1}|^{2}\big(|\alpha_{j}-\alpha_{j-1}|^{2}+|\alpha_{j+1}-\alpha_{j}|^{2}\big)\Big]\rho_{j}^{2}\nonumber
\end{align}
\begin{align}
&-\frac{1}{2}\sum_{j=0}^{\infty}\Big[(|\alpha_{j+1}|^{2}+2|\alpha_{j}|^{2}+|\alpha_{j-1}|^{2})|\alpha_{j+1}-\alpha_{j-1}|^{2}\nonumber\\
&+|\alpha_{j+1}|^{2}|\alpha_{j+1}-\alpha_{j}|^{2}\Big]|\alpha_{j}|^{2}\nonumber\\
&-\frac{1}{6}\sum_{j=0}^{\infty}|\alpha_{j}-\alpha_{j-1}|^{6}-\frac{1}{2}\sum_{j=0}^{\infty}\big(|\alpha_{j}|^{2}+|\alpha_{j-1}|^{2}\big)^{2}|\alpha_{j}-\alpha_{j-1}|^{2}\nonumber\\
&+\frac{1}{2}\sum_{j=0}^{\infty}\Big[(|\alpha_{j+1}|^{2}+2|\alpha_{j}|^{2}+|\alpha_{j-1}|^{2})|\alpha_{j+1}-\alpha_{j-1}|^{2}\nonumber\\
&+(|\alpha_{j}|^{2}+|\alpha_{j-1}|^{2})(|\alpha_{j}-\alpha_{j-1}|^{2}+|\alpha_{j}-\alpha_{j-1}|^{4})\Big].\qedhere
\end{align}
\end{proof}

Denote
\begin{align}
\mathrm{EP}=&\frac{2}{3}-\frac{1}{2}(|\alpha_{0}|^{2}+|\alpha_{1}|^{2})-\frac{1}{2}|\alpha_{0}|^{2}|\alpha_{0}+1|^{2}\nonumber\\
&+\sum_{j=0}^{\infty}\Big[\log(1-|\alpha_{j}|^{2})+|\alpha_{j}|^{2}+\frac{1}{2}|\alpha_{j}|^{4}+\frac{1}{3}|\alpha_{j}|^{6}\Big]\nonumber\\
&-\frac{1}{2}\sum_{j=0}^{\infty}|\alpha_{j}|^{4}-\frac{1}{2}\sum_{j=0}^{\infty}(|\alpha_{j+2}|^{2}+|\alpha_{j-1}|^{2})|\alpha_{j}|^{2}\nonumber\\
&-\frac{1}{2}\sum_{j=0}^{\infty}(|\alpha_{j+2}|^{2}+|\alpha_{j-1}|^{2})|\alpha_{j+1}|^{2}\rho_{j}^{2}-\frac{1}{2}\sum_{j=0}^{\infty}|\alpha_{j+2}-\alpha_{j-1}|^{2}\rho_{j+1}^{2}\rho_{j}^{2}\nonumber\\
&-\frac{1}{2}\sum_{j=0}^{\infty}\Big[|\alpha_{j+1}|^{2}|\alpha_{j}|^{2}+|\alpha_{j-1}|^{4}+|\alpha_{j}|^{2}|\alpha_{j-1}|^{2}+|\alpha_{j+1}|^{4}\nonumber\\
&+|\alpha_{j+1}-\alpha_{j-1}|^{2}\big(|\alpha_{j}-\alpha_{j-1}|^{2}+|\alpha_{j+1}-\alpha_{j}|^{2}\big)\Big]\rho_{j}^{2}\nonumber\\
&-\frac{1}{2}\sum_{j=0}^{\infty}\Big[(|\alpha_{j+1}|^{2}+2|\alpha_{j}|^{2}+|\alpha_{j-1}|^{2})|\alpha_{j+1}-\alpha_{j-1}|^{2}\nonumber\\
&+|\alpha_{j+1}|^{2}|\alpha_{j+1}-\alpha_{j}|^{2}\Big]|\alpha_{j}|^{2}\nonumber\\
&-\frac{1}{6}\sum_{j=0}^{\infty}|\alpha_{j}-\alpha_{j-1}|^{6}-\frac{1}{2}\sum_{j=0}^{\infty}\big(|\alpha_{j}|^{2}+|\alpha_{j-1}|^{2}\big)^{2}|\alpha_{j}-\alpha_{j-1}|^{2}
\end{align}
and
\begin{align}
\mathrm{CP}=&\frac{1}{2}\sum_{j=0}^{\infty}\Big[(|\alpha_{j+1}|^{2}+2|\alpha_{j}|^{2}+|\alpha_{j-1}|^{2})|\alpha_{j+1}-\alpha_{j-1}|^{2}\nonumber\\
&+(|\alpha_{j}|^{2}+|\alpha_{j-1}|^{2})(|\alpha_{j}-\alpha_{j-1}|^{2}+|\alpha_{j}-\alpha_{j-1}|^{4})\Big],
\end{align}
where \textrm{EP} and \textrm{CP} are the meaning that equivalent part and conditional part respectively, then the sum rule in the above theorem (and some similar theorems below) can be expressed as
\begin{center}
``Higher order Szeg\H{o} integral = Equivalent part + Conditional part".
\end{center}

Straightly speaking, the convergence of higher order Szeg\H{o} integrals in what follows is equivalent to the convergence of the equivalent part once the conditional part is finite. This is why we call them equivalent part and conditional part respectively. Obviously, the summands in \textrm{EP} are negative and the ones in \textrm{CP} are positive.

\begin{thm}
Assume $\alpha\in\ell^{4}$, then
\begin{align}
\int_{0}^{2\pi}(1-\cos3\theta)\log w(\theta)\frac{d\theta}{2\pi}>-\infty \Leftrightarrow (S^{3}-1)\alpha\in\ell^{2}.
\end{align}
\end{thm}
\begin{proof}
Since $\alpha\in\ell^{4}$, by Lemma 4.3, Theorem 4.11 and H\"older inequality, the series in \textrm{CP} and \textrm{EP} are finite with only one exception. This exceptional one is
\begin{align}
&-\frac{1}{2}\sum_{j=0}^{\infty}|\alpha_{j+2}-\alpha_{j-1}|^{2}\rho_{j+1}^{2}\rho_{j}^{2}\nonumber\\
=&-\frac{1}{2}\sum_{j=0}^{\infty}|\alpha_{j+2}-\alpha_{j-1}|^{2}
+\frac{1}{2}\sum_{j=0}^{\infty}(|\alpha_{j+1}|^{2}+|\alpha_{j}|^{2})|\alpha_{j+2}-\alpha_{j-1}|^{2}\nonumber\\
&-\frac{1}{2}\sum_{j=0}^{\infty}|\alpha_{j+2}-\alpha_{j-1}|^{2}|\alpha_{j+1}|^{2}|\alpha_{j}|^{2}.
\end{align}
By the assumption and H\"older inequality again, the second and third series in (4.74) are also finite. So (4.73) follows immediately.
\end{proof}

\begin{rem}
This is one special case of a result in \cite{gz} due to Golinskii and Zlato\v{s}.
\end{rem}

\begin{thm}
Assume $(S-1)\alpha\in\ell^{2}$, then
\begin{align}
\int_{0}^{2\pi}(1-\cos3\theta)\log w(\theta)\frac{d\theta}{2\pi}>-\infty \Leftrightarrow \alpha\in\ell^{4}.
\end{align}
\end{thm}

\begin{proof}
By triangle inequality, it is easy to know that $(S-1)\alpha\in \ell^{2}$ imply $(S^{2}-1)\alpha\in \ell^{2}$ and $(S^{3}-1)\alpha\in \ell^{2}$. From these facts and $\alpha\in \mathbb{D}^{\infty}$, we obtain
\begin{equation}
\mathrm{CP}<+\infty
\end{equation}
and
\begin{align}
\int_{0}^{2\pi}(1-\cos3\theta)\log w(\theta)\frac{d\theta}{2\pi}>-\infty \Leftrightarrow -\mathrm{EP}<\infty.
\end{align}
More precisely, $\int_{0}^{2\pi}(1-\cos3\theta)\log w(\theta)\frac{d\theta}{2\pi}>-\infty$ is equivalent to
\begin{equation}
-\sum_{j=0}^{\infty}\Big[\log(1-|\alpha_{j}|^{2})+|\alpha_{j}|^{2}+\frac{1}{2}|\alpha_{j}|^{4}+\frac{1}{3}|\alpha_{j}|^{6}\Big]<+\infty,
\end{equation}
\begin{equation}
\sum_{j=0}^{\infty}|\alpha_{j}|^{4}<+\infty,
\end{equation}
\begin{equation}
\sum_{j=0}^{\infty}(|\alpha_{j+2}|^{2}+|\alpha_{j-1}|^{2})|\alpha_{j}|^{2}<+\infty,
\end{equation}
\begin{equation}
\sum_{j=0}^{\infty}(|\alpha_{j+2}|^{2}+|\alpha_{j-1}|^{2})|\alpha_{j+1}|^{2}\rho_{j}^{2}<+\infty,
\end{equation}
and
\begin{equation}
\sum_{j=0}^{\infty}\Big[|\alpha_{j+1}|^{2}|\alpha_{j}|^{2}+|\alpha_{j-1}|^{4}+|\alpha_{j}|^{2}|\alpha_{j-1}|^{2}+|\alpha_{j+1}|^{4}\Big]\rho_{j}^{2}<+\infty.
\end{equation}
Noting $\rho_{j}<1$, by H\"older inequality and Lemma 4.3, we have that (4.79) implies (4.78) and (4.80)-(4.82). So (4.75) holds by this fact and (4.77) as $(S-1)\alpha\in\ell^{2}$.
\end{proof}

\begin{rem}
This is a result similar to one special case of the result in \cite{lu1} due to Lukic.
\end{rem}
\begin{prop}
For $a,b,c,d\in \mathbb{C}$,
\begin{align}
|a-b-c+d|^{2}=&|a-b|^{2}+|a-c|^{2}+|b-d|^{2}+|c-d|^{2}\nonumber\\
&-|a-d|^{2}-|b-c|^{2}.
\end{align}
\end{prop}

\begin{proof}
It follows from direct calculations or parallelogram law. For the latter, more precisely, we have
\begin{align}
|a-b-c+d|^{2}+|a-b+c-d|^{2}=2(|a-b|^{2}+|c-d|^{2}),
\end{align}
\begin{align}
|a-b+c-d|^{2}+|a+b-c-d|^{2}=2(|a-d|^{2}+|b-c|^{2})
\end{align}
and
\begin{align}
|a+b-c-d|^{2}+|a-b-c+d|^{2}=2(|a-c|^{2}+|b-d|^{2}).
\end{align}
Thus, (4.83) immediately follows from (4.84)-(4.86).
\end{proof}

\begin{prop}
For any $\alpha=\{\alpha_{j}\}_{0}^{\infty}\in \ell^{2}$ and $\alpha_{-1}=-1$,
\begin{align}
&\sum_{j=0}^{\infty}|\alpha_{j+2}-\alpha_{j-1}|^{2}=-(|\alpha_{1}-\alpha_{0}|^{2}+|\alpha_{1}+1|^{2})+\sum_{j=0}^{\infty}|\alpha_{j}-\alpha_{j-1}|^{2}\nonumber\\
&+2\sum_{j=0}^{\infty}|\alpha_{j+1}-\alpha_{j-1}|^{2}-\sum_{j=0}^{\infty}|\alpha_{j+2}-\alpha_{j+1}-\alpha_{j}+\alpha_{j-1}|^{2}.
\end{align}
\end{prop}
\begin{proof}
By Proposition 4.16, we have
\begin{align}
|\alpha_{j+2}-\alpha_{j-1}|^{2}=&|\alpha_{j+2}-\alpha_{j+1}|^{2}+|\alpha_{j}-\alpha_{j-1}|^{2}
+|\alpha_{j+1}-\alpha_{j-1}|^{2}+|\alpha_{j+2}-\alpha_{j}|^{2}\nonumber\\
&-|\alpha_{j+1}-\alpha_{j}|^{2}-|\alpha_{j+2}-\alpha_{j+1}-\alpha_{j}+\alpha_{j-1}|^{2}
\end{align}
for $j\geq 0$. Then for $\alpha\in\ell^{2}$ and $\alpha_{-1}=-1$,
\begin{align*}
&\sum_{j=0}^{\infty}|\alpha_{j+2}-\alpha_{j-1}|^{2}=\sum_{j=0}^{\infty}|\alpha_{j+2}-\alpha_{j+1}|^{2}+\sum_{j=0}^{\infty}|\alpha_{j}-\alpha_{j-1}|^{2}
+\sum_{j=0}^{\infty}|\alpha_{j+1}-\alpha_{j-1}|^{2}\nonumber\\
&+\sum_{j=0}^{\infty}|\alpha_{j+2}-\alpha_{j}|^{2}-\sum_{j=0}^{\infty}|\alpha_{j+1}-\alpha_{j}|^{2}-\sum_{j=0}^{\infty}|\alpha_{j+2}-\alpha_{j+1}-\alpha_{j}+\alpha_{j-1}|^{2}\nonumber\\
=&-(|\alpha_{1}-\alpha_{0}|^{2}+|\alpha_{1}+1|^{2})+\sum_{j=0}^{\infty}|\alpha_{j}-\alpha_{j-1}|^{2}+2\sum_{j=0}^{\infty}|\alpha_{j+1}-\alpha_{j-1}|^{2}\nonumber\\
&-\sum_{j=0}^{\infty}|\alpha_{j+2}-\alpha_{j+1}-\alpha_{j}+\alpha_{j-1}|^{2}.\qedhere
\end{align*}
\end{proof}

\begin{thm}
 For any $\alpha$,
\begin{align}
&\int_{0}^{2\pi}(1-\cos\theta)^{2}(1+\cos\theta)\log w(\theta)\frac{d\theta}{2\pi}\nonumber\\
=&\frac{1}{3}-\frac{1}{8}|\alpha_{0}|^{2}(1+|\alpha_{1}|^{2})+\frac{1}{8}|\alpha_{0}|^{2}|\alpha_{0}+1|^{2}-\frac{1}{4}|\alpha_{0}|^{4}-\frac{1}{8}|\alpha_{1}-\alpha_{0}|^{2}-\frac{1}{8}|\alpha_{1}+1|^{2}\nonumber\\
&+\frac{1}{2}\sum_{j=0}^{\infty}\Big[\log(1-|\alpha_{j}|^{2})+|\alpha_{j}|^{2}+\frac{1}{2}|\alpha_{j}|^{4}\Big]
-\frac{1}{8}\sum_{j=0}^{\infty}|\alpha_{j+2}-\alpha_{j+1}-\alpha_{j}+\alpha_{j-1}|^{2}\nonumber\\
&-\frac{1}{8}\sum_{j=0}^{\infty}(|\alpha_{j+1}|^{2}-|\alpha_{j-1}|^{2})^{2}-\frac{1}{8}\sum_{j=0}^{\infty}|\alpha_{j}-\alpha_{j-1}|^{4}\nonumber\\
&-\frac{1}{8}\sum_{j=0}^{\infty}(|\alpha_{j+1}|^{2}+|\alpha_{j}|^{2})|\alpha_{j+2}-\alpha_{j-1}|^{2}-\frac{1}{8}\sum_{j=0}^{\infty}(|\alpha_{j+1}|^{2}+|\alpha_{j-1}|^{2})|\alpha_{j+1}-\alpha_{j-1}|^{2}\nonumber\\
&-\frac{1}{12}\sum_{j=0}^{\infty}|\alpha_{j}|^{6}-\frac{1}{8}\sum_{j=0}^{\infty}(|\alpha_{j+2}|^{2}+|\alpha_{j-1}|^{2})|\alpha_{j+1}|^{2}|\alpha_{j}|^{2}\nonumber\\
&-\frac{1}{8}\sum_{j=0}^{\infty}(|\alpha_{j}|^{2}+|\alpha_{j-1}|^{2})|\alpha_{j}-\alpha_{j-1}|^{4}\nonumber\\
&-\frac{1}{8}\sum_{j=0}^{\infty}\Big[|\alpha_{j+1}|^{2}|\alpha_{j}|^{2}+|\alpha_{j-1}|^{4}+|\alpha_{j}|^{2}|\alpha_{j-1}|^{2}+|\alpha_{j+1}|^{4}\nonumber\\
&+|\alpha_{j+1}-\alpha_{j-1}|^{2}\big(|\alpha_{j}-\alpha_{j-1}|^{2}+|\alpha_{j+1}-\alpha_{j}|^{2}\big)\Big]|\alpha_{j}|^{2}\nonumber\\
&+\frac{1}{8}\sum_{j=0}^{\infty}(|\alpha_{j}|^{2}+|\alpha_{j-1}|^{2})|\alpha_{j}-\alpha_{j-1}|^{2}\nonumber\\
&+\frac{1}{8}\sum_{j=0}^{\infty}|\alpha_{j+1}-\alpha_{j-1}|^{2}\big(|\alpha_{j}-\alpha_{j-1}|^{2}+|\alpha_{j+1}-\alpha_{j}|^{2}\big)\nonumber\\
&+\frac{1}{24}\sum_{j=0}^{\infty}|\alpha_{j}-\alpha_{j-1}|^{6}+\frac{1}{8}\sum_{j=0}^{\infty}|\alpha_{j+2}-\alpha_{j-1}|^{2}|\alpha_{j+1}|^{2}|\alpha_{j}|^{2}\nonumber\\
&+\frac{1}{8}\sum_{j=0}^{\infty}\big(|\alpha_{j}|^{2}+|\alpha_{j-1}|^{2}\big)^{2}|\alpha_{j}-\alpha_{j-1}|^{2}\nonumber\\
&+\frac{1}{8}\sum_{j=0}^{\infty}\Big[(|\alpha_{j+1}|^{2}+2|\alpha_{j}|^{2}+|\alpha_{j-1}|^{2})|\alpha_{j+1}-\alpha_{j-1}|^{2}\nonumber\\
&+|\alpha_{j+1}|^{2}|\alpha_{j+1}-\alpha_{j}|^{2}\Big]|\alpha_{j}|^{2}.
\end{align}
\end{thm}
\begin{proof}
Since
\begin{align}
(1-\cos\theta)^{2}(1+\cos\theta)=&\frac{1}{2}-\frac{1}{4}\cos\theta-\frac{1}{2}\cos2\theta+\frac{1}{4}\cos3\theta\nonumber\\
=&\frac{1}{4}(1-\cos\theta)+\frac{1}{2}(1-\cos2\theta)-\frac{1}{4}(1-\cos3\theta),
\end{align}
then
\begin{align}
Z_{3,2}(\mu)=&\frac{1}{2}w_{0}-\frac{1}{4}\mathrm{Re}(w_{1})-\frac{1}{2}\mathrm{Re}(w_{2})+\frac{1}{4}\mathrm{Re}(w_{3})\nonumber\\
=&\frac{1}{4}Z_{1}(\mu)+Z_{2,1}(\mu)-\frac{1}{4}Z_{3,1}(\mu).
\end{align}
For $\alpha\in\ell^{2}$, by (4.2), (4.12), (4.60), (4.87) and (4.91), we have
\begin{align}
&Z_{3,2}(\mu)=\int_{0}^{2\pi}(1-\cos\theta)^{2}(1+\cos\theta)\log w(\theta)\frac{d\theta}{2\pi}\nonumber\\
=&\frac{1}{8}+\frac{1}{4}\sum_{j=0}^{\infty}\Big(\log(1-|\alpha_{j}|^{2})+|\alpha_{j}|^{2}\Big)-\frac{1}{8}\sum_{j=0}^{\infty}|\alpha_{j}-\alpha_{j-1}|^{2}\nonumber\\
&+\frac{3}{8}+\frac{1}{2}\sum_{j=0}^{\infty}\Big[\log(1-|\alpha_{j}|^{2})+|\alpha_{j}|^{2}+\frac{1}{2}|\alpha_{j}|^{4}\Big]\nonumber\\
&-\frac{1}{2}\sum_{j=0}^{\infty}|\alpha_{j}\alpha_{j-1}|^{2}-\frac{1}{4}\sum_{j=0}^{\infty}\rho_{j}^{2}|\alpha_{j+1}-\alpha_{j-1}|^{2}\nonumber\\
&-\frac{1}{16}\sum_{j=0}^{\infty}\left[\big(2|\alpha_{j}|^{2}-|\alpha_{j}-\alpha_{j-1}|^{2}\big)^{2}+\big(2|\alpha_{j-1}|^{2}-|\alpha_{j}-\alpha_{j-1}|^{2}\big)^{2}\right]\nonumber\\
&-\frac{1}{6}+\frac{1}{8}(|\alpha_{0}|^{2}+|\alpha_{1}|^{2})+\frac{1}{8}|\alpha_{0}|^{2}|\alpha_{0}+1|^{2}\nonumber\\
&-\Big\{\frac{1}{4}\sum_{j=0}^{\infty}\Big[\log(1-|\alpha_{j}|^{2})+|\alpha_{j}|^{2}+\frac{1}{2}|\alpha_{j}|^{4}+\frac{1}{3}|\alpha_{j}|^{6}\Big]\nonumber\\
&-\frac{1}{8}\sum_{j=0}^{\infty}|\alpha_{j}|^{4}-\frac{1}{8}\sum_{j=0}^{\infty}(|\alpha_{j+2}|^{2}+|\alpha_{j-1}|^{2})|\alpha_{j}|^{2}\nonumber\\
&-\frac{1}{8}\sum_{j=0}^{\infty}(|\alpha_{j+2}|^{2}+|\alpha_{j-1}|^{2})|\alpha_{j+1}|^{2}\rho_{j}^{2}-\frac{1}{8}\sum_{j=0}^{\infty}|\alpha_{j+2}-\alpha_{j-1}|^{2}\rho_{j+1}^{2}\rho_{j}^{2}\nonumber\\
&-\frac{1}{8}\sum_{j=0}^{\infty}\Big[|\alpha_{j+1}|^{2}|\alpha_{j}|^{2}+|\alpha_{j-1}|^{4}+|\alpha_{j}|^{2}|\alpha_{j-1}|^{2}+|\alpha_{j+1}|^{4}\nonumber\\
&+|\alpha_{j+1}-\alpha_{j-1}|^{2}\big(|\alpha_{j}-\alpha_{j-1}|^{2}+|\alpha_{j+1}-\alpha_{j}|^{2}\big)\Big]\rho_{j}^{2}\nonumber\\
&-\frac{1}{8}\sum_{j=0}^{\infty}\Big[(|\alpha_{j+1}|^{2}+2|\alpha_{j}|^{2}+|\alpha_{j-1}|^{2})|\alpha_{j+1}-\alpha_{j-1}|^{2}\nonumber\\
&+|\alpha_{j-1}|^{2}|\alpha_{j}-\alpha_{j-1}|^{2}+|\alpha_{j+1}|^{2}|\alpha_{j+1}-\alpha_{j}|^{2}\Big]|\alpha_{j}|^{2}\nonumber\\
&-\frac{1}{24}\sum_{j=0}^{\infty}|\alpha_{j}-\alpha_{j-1}|^{6}-\frac{1}{8}\sum_{j=0}^{\infty}\big(|\alpha_{j}|^{4}+|\alpha_{j}|^{2}|\alpha_{j-1}|^{2}+|\alpha_{j-1}|^{4})|\alpha_{j}-\alpha_{j-1}|^{2}\nonumber\\
&+\frac{1}{8}\sum_{j=0}^{\infty}\Big[(|\alpha_{j+1}|^{2}+2|\alpha_{j}|^{2}+|\alpha_{j-1}|^{2})|\alpha_{j+1}-\alpha_{j-1}|^{2}\nonumber\\
&+(|\alpha_{j}|^{2}+|\alpha_{j-1}|^{2})(|\alpha_{j}-\alpha_{j-1}|^{2}+|\alpha_{j}-\alpha_{j-1}|^{4})\Big]\Big\}\nonumber
\end{align}
\begin{align}
&=\frac{1}{3}-\frac{1}{8}|\alpha_{0}|^{2}(1+|\alpha_{1}|^{2})+\frac{1}{8}|\alpha_{0}|^{2}|\alpha_{0}+1|^{2}-\frac{1}{4}|\alpha_{0}|^{4}-\frac{1}{8}|\alpha_{1}-\alpha_{0}|^{2}-\frac{1}{8}|\alpha_{1}+1|^{2}\nonumber\\
&+\frac{1}{2}\sum_{j=0}^{\infty}\Big[\log(1-|\alpha_{j}|^{2})+|\alpha_{j}|^{2}+\frac{1}{2}|\alpha_{j}|^{4}\Big]
-\frac{1}{8}\sum_{j=0}^{\infty}|\alpha_{j+2}-\alpha_{j+1}-\alpha_{j}+\alpha_{j-1}|^{2}\nonumber\\
&-\frac{1}{8}\sum_{j=0}^{\infty}(|\alpha_{j+1}|^{2}-|\alpha_{j-1}|^{2})^{2}-\frac{1}{8}\sum_{j=0}^{\infty}|\alpha_{j}-\alpha_{j-1}|^{4}\nonumber\\
&-\frac{1}{8}\sum_{j=0}^{\infty}(|\alpha_{j+1}|^{2}+|\alpha_{j}|^{2})|\alpha_{j+2}-\alpha_{j-1}|^{2}-\frac{1}{8}\sum_{j=0}^{\infty}(|\alpha_{j+1}|^{2}+|\alpha_{j-1}|^{2})|\alpha_{j+1}-\alpha_{j-1}|^{2}\nonumber\\
&-\frac{1}{12}\sum_{j=0}^{\infty}|\alpha_{j}|^{6}-\frac{1}{8}\sum_{j=0}^{\infty}(|\alpha_{j+2}|^{2}+|\alpha_{j-1}|^{2})|\alpha_{j+1}|^{2}|\alpha_{j}|^{2}\nonumber\\
&-\frac{1}{8}\sum_{j=0}^{\infty}(|\alpha_{j}|^{2}+|\alpha_{j-1}|^{2})|\alpha_{j}-\alpha_{j-1}|^{4}\nonumber\\
&-\frac{1}{8}\sum_{j=0}^{\infty}\Big[|\alpha_{j+1}|^{2}|\alpha_{j}|^{2}+|\alpha_{j-1}|^{4}+|\alpha_{j}|^{2}|\alpha_{j-1}|^{2}+|\alpha_{j+1}|^{4}\nonumber\\
&+|\alpha_{j+1}-\alpha_{j-1}|^{2}\big(|\alpha_{j}-\alpha_{j-1}|^{2}+|\alpha_{j+1}-\alpha_{j}|^{2}\big)\Big]|\alpha_{j}|^{2}\nonumber\\
&+\frac{1}{8}\sum_{j=0}^{\infty}(|\alpha_{j}|^{2}+|\alpha_{j-1}|^{2})|\alpha_{j}-\alpha_{j-1}|^{2}\nonumber\\
&+\frac{1}{8}\sum_{j=0}^{\infty}|\alpha_{j+1}-\alpha_{j-1}|^{2}\big(|\alpha_{j}-\alpha_{j-1}|^{2}+|\alpha_{j+1}-\alpha_{j}|^{2}\big)\nonumber\\
&+\frac{1}{24}\sum_{j=0}^{\infty}|\alpha_{j}-\alpha_{j-1}|^{6}+\frac{1}{8}\sum_{j=0}^{\infty}|\alpha_{j+2}-\alpha_{j-1}|^{2}|\alpha_{j+1}|^{2}|\alpha_{j}|^{2}\nonumber\\
&+\frac{1}{8}\sum_{j=0}^{\infty}\big(|\alpha_{j}|^{2}+|\alpha_{j-1}|^{2}\big)^{2}|\alpha_{j}-\alpha_{j-1}|^{2}\nonumber\\
&+\frac{1}{8}\sum_{j=0}^{\infty}\Big[(|\alpha_{j+1}|^{2}+2|\alpha_{j}|^{2}+|\alpha_{j-1}|^{2})|\alpha_{j+1}-\alpha_{j-1}|^{2}\nonumber\\
&+|\alpha_{j+1}|^{2}|\alpha_{j+1}-\alpha_{j}|^{2}\Big]|\alpha_{j}|^{2}.\nonumber\qedhere
\end{align}
\end{proof}

As before, the equivalent part and conditional part in this case are as follows:
\begin{align}
\mathrm{EP}=&\frac{1}{3}-\frac{1}{8}|\alpha_{0}|^{2}(1+|\alpha_{1}|^{2})+\frac{1}{8}|\alpha_{0}|^{2}|\alpha_{0}+1|^{2}-\frac{1}{4}|\alpha_{0}|^{4}-\frac{1}{8}|\alpha_{1}-\alpha_{0}|^{2}-\frac{1}{8}|\alpha_{1}+1|^{2}\nonumber\\
&+\frac{1}{2}\sum_{j=0}^{\infty}\Big[\log(1-|\alpha_{j}|^{2})+|\alpha_{j}|^{2}+\frac{1}{2}|\alpha_{j}|^{4}\Big]
-\frac{1}{8}\sum_{j=0}^{\infty}|\alpha_{j+2}-\alpha_{j+1}-\alpha_{j}+\alpha_{j-1}|^{2}\nonumber\\
&-\frac{1}{8}\sum_{j=0}^{\infty}(|\alpha_{j+1}|^{2}-|\alpha_{j-1}|^{2})^{2}-\frac{1}{8}\sum_{j=0}^{\infty}|\alpha_{j}-\alpha_{j-1}|^{4}\nonumber
\end{align}
\begin{align}
&-\frac{1}{8}\sum_{j=0}^{\infty}(|\alpha_{j+1}|^{2}+|\alpha_{j}|^{2})|\alpha_{j+2}-\alpha_{j-1}|^{2}-\frac{1}{8}\sum_{j=0}^{\infty}(|\alpha_{j+1}|^{2}+|\alpha_{j-1}|^{2})|\alpha_{j+1}-\alpha_{j-1}|^{2}\nonumber\\
&-\frac{1}{12}\sum_{j=0}^{\infty}|\alpha_{j}|^{6}-\frac{1}{8}\sum_{j=0}^{\infty}(|\alpha_{j+2}|^{2}+|\alpha_{j-1}|^{2})|\alpha_{j+1}|^{2}|\alpha_{j}|^{2}\nonumber\\
&-\frac{1}{8}\sum_{j=0}^{\infty}(|\alpha_{j}|^{2}+|\alpha_{j-1}|^{2})|\alpha_{j}-\alpha_{j-1}|^{4}\nonumber\\
&-\frac{1}{8}\sum_{j=0}^{\infty}\Big[|\alpha_{j+1}|^{2}|\alpha_{j}|^{2}+|\alpha_{j-1}|^{4}+|\alpha_{j}|^{2}|\alpha_{j-1}|^{2}+|\alpha_{j+1}|^{4}\nonumber\\
&+|\alpha_{j+1}-\alpha_{j-1}|^{2}\big(|\alpha_{j}-\alpha_{j-1}|^{2}+|\alpha_{j+1}-\alpha_{j}|^{2}\big)\Big]|\alpha_{j}|^{2}
\end{align}
and
\begin{align}
\mathrm{CP}=&\frac{1}{8}\sum_{j=0}^{\infty}(|\alpha_{j}|^{2}+|\alpha_{j-1}|^{2})|\alpha_{j}-\alpha_{j-1}|^{2}\nonumber\\
&+\frac{1}{8}\sum_{j=0}^{\infty}|\alpha_{j+1}-\alpha_{j-1}|^{2}\big(|\alpha_{j}-\alpha_{j-1}|^{2}+|\alpha_{j+1}-\alpha_{j}|^{2}\big)\nonumber\\
&+\frac{1}{24}\sum_{j=0}^{\infty}|\alpha_{j}-\alpha_{j-1}|^{6}+\frac{1}{8}\sum_{j=0}^{\infty}|\alpha_{j+2}-\alpha_{j-1}|^{2}|\alpha_{j+1}|^{2}|\alpha_{j}|^{2}\nonumber\\
&+\frac{1}{8}\sum_{j=0}^{\infty}\big(|\alpha_{j}|^{2}+|\alpha_{j-1}|^{2}\big)^{2}|\alpha_{j}-\alpha_{j-1}|^{2}\nonumber\\
&+\frac{1}{8}\sum_{j=0}^{\infty}\Big[(|\alpha_{j+1}|^{2}+2|\alpha_{j}|^{2}+|\alpha_{j-1}|^{2})|\alpha_{j+1}-\alpha_{j-1}|^{2}\nonumber\\
&+|\alpha_{j+1}|^{2}|\alpha_{j+1}-\alpha_{j}|^{2}\Big]|\alpha_{j}|^{2}.
\end{align}

Applying sum rule (4.89), we obtain a series of higher order analogues of Szeg\H{o} theorem under appropriate conditions.

\begin{thm}
Assume $(S^{2}-1)\alpha\in \ell^{3}$ and $\alpha\in \ell^{6}$, then
\begin{align}
&\int_{0}^{2\pi}(1-\cos\theta)^{2}(1+\cos\theta)\log w(\theta)\frac{d\theta}{2\pi}>-\infty\nonumber\\
\Longleftrightarrow &\,\,(S-1)\alpha\in \ell^{4} \,\,and\,\,(S-1)^{2}(S+1)\alpha\in\ell^{2}.
\end{align}
\end{thm}

\begin{proof}
By (4.88) (or directly (4.87)), we have
\begin{align}
&-\frac{1}{8}\sum_{j=0}^{\infty}(|\alpha_{j+1}|^{2}+|\alpha_{j}|^{2})|\alpha_{j+2}-\alpha_{j-1}|^{2}+\frac{1}{8}\sum_{j=0}^{\infty}(|\alpha_{j}|^{2}+|\alpha_{j-1}|^{2})|\alpha_{j}-\alpha_{j-1}|^{2}\nonumber
\end{align}
\begin{align}
=&-\frac{1}{8}\sum_{j=0}^{\infty}(|\alpha_{j+1}|^{2}+|\alpha_{j}|^{2})|\alpha_{j+2}-\alpha_{j+1}|^{2}-\frac{1}{8}\sum_{j=0}^{\infty}(|\alpha_{j+1}|^{2}+|\alpha_{j}|^{2})|\alpha_{j}-\alpha_{j-1}|^{2}\nonumber\\
&-\frac{1}{8}\sum_{j=0}^{\infty}(|\alpha_{j+1}|^{2}+|\alpha_{j}|^{2})|\alpha_{j+1}-\alpha_{j-1}|^{2}-\frac{1}{8}\sum_{j=0}^{\infty}(|\alpha_{j+1}|^{2}+|\alpha_{j}|^{2})|\alpha_{j+2}-\alpha_{j}|^{2}\nonumber\\
&+\frac{1}{8}\sum_{j=0}^{\infty}(|\alpha_{j+1}|^{2}+|\alpha_{j}|^{2})|\alpha_{j+1}-\alpha_{j}|^{2}+\frac{1}{8}\sum_{j=0}^{\infty}(|\alpha_{j}|^{2}+|\alpha_{j-1}|^{2})|\alpha_{j}-\alpha_{j-1}|^{2}\nonumber\\
&+\frac{1}{8}\sum_{j=0}^{\infty}(|\alpha_{j+1}|^{2}+|\alpha_{j}|^{2})|\alpha_{j+2}-\alpha_{j+1}-\alpha_{j}+\alpha_{j-1}|^{2}\nonumber\\
=&\frac{1}{8}(|\alpha_{0}|^{2}+1)(|\alpha_{1}-\alpha_{0}|^{2}+|\alpha_{1}+1|^{2})\nonumber\\
&+\frac{1}{8}\sum_{j=0}^{\infty}(|\alpha_{j+1}|^{2}-|\alpha_{j-1}|^{2})(|\alpha_{j+1}-\alpha_{j}|^{2}-|\alpha_{j}-\alpha_{j-1}|^{2})\nonumber\\
&-\frac{1}{8}\sum_{j=0}^{\infty}(|\alpha_{j+1}|^{2}+2|\alpha_{j}|^{2}+|\alpha_{j-1}|^{2})|\alpha_{j+1}-\alpha_{j-1}|^{2}\nonumber\\
&+\frac{1}{8}\sum_{j=0}^{\infty}(|\alpha_{j+1}|^{2}+|\alpha_{j}|^{2})|\alpha_{j+2}-\alpha_{j+1}-\alpha_{j}+\alpha_{j-1}|^{2}.
\end{align}
Noting that
\begin{align}
&||\alpha_{j+1}|^{2}-|\alpha_{j-1}|^{2}\big)|\nonumber\\
\leq& |\alpha_{j+1}-\alpha_{j-1}|(|\alpha_{j+1}|+|\alpha_{j-1}|)
\end{align}
and
\begin{align}
&\big||\alpha_{j+1}-\alpha_{j}|^{2}-|\alpha_{j}-\alpha_{j-1}|^{2}\big|\nonumber\\
\leq& |\alpha_{j+1}-\alpha_{j-1}|(|\alpha_{j+1}-\alpha_{j}|+|\alpha_{j}-\alpha_{j-1}|),
\end{align}
where triangle inequality $\big||a|-|b|\big|\leq |a-b|$ is used, by the assumption, (4.95) and using H\"older inequality (or Young inequality), the series in (4.89) are convergent except two ones in \textrm{EP}, more precisely,
$$-\frac{1}{8}\sum_{j=0}^{\infty}|\alpha_{j+2}-\alpha_{j+1}-\alpha_{j}+\alpha_{j-1}|^{2}$$
and
$$-\frac{1}{8}\sum_{j=0}^{\infty}|\alpha_{j}-\alpha_{j-1}|^{4}.$$
That is to say that $Z_{3,2}(\mu)>-\infty$ is equivalent to $(S-1)\alpha\in \ell^{4}$ and $(S-1)^{2}(S+1)\alpha\in \ell^{2}$ as $(S^{2}-1)\alpha\in \ell^{3}$ and $\alpha\in \ell^{6}$.
\end{proof}

\begin{cor}
Assume $(S^{2}-1)\alpha\in \ell^{2}$ and $\alpha\in \ell^{6}$, then
\begin{align}
\int_{0}^{2\pi}(1-\cos\theta)^{2}(1+\cos\theta)\log w(\theta)\frac{d\theta}{2\pi}>-\infty
\Longleftrightarrow \,\,(S-1)\alpha\in \ell^{4}.
\end{align}
\end{cor}

\begin{proof}
It is easy to know that $(S^{2}-1)\alpha\in \ell^{2}$ implies $(S^{2}-1)\alpha\in \ell^{3}$ and $(S-1)^{2}(S+1)\alpha\in \ell^{2}$. Thus (4.98) follows from (4.94).
\end{proof}

\begin{rem}
This is the main result of Lukic in \cite{lu}, which disproves the original conjecture of Simon in \cite{sim1}.
\end{rem}

\begin{cor}
Assume $(S-1)\alpha\in \ell^{3}$ and $\alpha\in \ell^{6}$, then
\begin{align}
&\int_{0}^{2\pi}(1-\cos\theta)^{2}(1+\cos\theta)\log w(\theta)\frac{d\theta}{2\pi}>-\infty \,\,\Longleftrightarrow\,\,(S-1)^{2}(S+1)\alpha\in \ell^{2}.
\end{align}
\end{cor}

\begin{proof}
The key is to note that $(S-1)\alpha\in \ell^{3}$ implies $(S^{2}-1)\alpha\in \ell^{3}$ and $(S-1)\alpha\in \ell^{4}$.
\end{proof}

With respect to one-directional implication from $\alpha$ to $Z_{3,2}(\mu)$,  we have

\begin{cor}
Assume $(S-1)^{2}\alpha\in \ell^{2}$ and $\alpha\in \ell^{6}$, then
\begin{align*}
&\int_{0}^{2\pi}(1-\cos\theta)^{2}(1+\cos\theta)\log w(\theta)\frac{d\theta}{2\pi}>-\infty .
\end{align*}
\end{cor}

\begin{proof}
By (4.44), $(S-1)^{2}\alpha\in \ell^{2}$ and $\alpha\in \ell^{6}$ imply $(S-1)\alpha\in \ell^{3}$, further $(S^{2}-1)\alpha\in \ell^{3}$ and $(S-1)\alpha\in \ell^{4}$, whereas $(S-1)^{2}\alpha\in \ell^{2}$ implies $(S-1)^{2}(S+1)\alpha\in \ell^{2}$.
\end{proof}

\begin{cor}
Assume $(S-1)^{2}(S+1)\alpha\in \ell^{2}$, $(S^{2}-1)\alpha\in \ell^{3}$, $(S-1)\alpha\in \ell^{4}$ and $\alpha\in \ell^{6}$, then
\begin{align*}
&\int_{0}^{2\pi}(1-\cos\theta)^{2}(1+\cos\theta)\log w(\theta)\frac{d\theta}{2\pi}>-\infty .
\end{align*}
\end{cor}

\begin{proof}
This is an immediate consequence of Theorem 4.19.
\end{proof}

\begin{cor}
Assume $(S-1)^{2}(S+1)\alpha\in \ell^{2}$, $(S-1)\alpha\in \ell^{3}$ and $\alpha\in \ell^{6}$, then
\begin{align*}
&\int_{0}^{2\pi}(1-\cos\theta)^{2}(1+\cos\theta)\log w(\theta)\frac{d\theta}{2\pi}>-\infty.
\end{align*}
\end{cor}

\begin{proof}
Note that $(S-1)\alpha\in \ell^{3}$ implies $(S^{2}-1)\alpha\in \ell^{3}$ and $(S-1)\alpha\in \ell^{4}$.
\end{proof}

\begin{cor}
Assume $(S-1)^{2}(S+1)\alpha\in \ell^{2}$, $(S-1)\alpha\in \ell^{4}$ and $\alpha\in \ell^{6}$, then
\begin{align*}
&\int_{0}^{2\pi}(1-\cos\theta)^{2}(1+\cos\theta)\log w(\theta)\frac{d\theta}{2\pi}>-\infty.
\end{align*}
\end{cor}

\begin{proof}
Applying (4.44) to $(S+1)\alpha$, $(S-1)^{2}(S+1)\alpha\in \ell^{2}$ and $\alpha\in \ell^{6}$ imply $(S^{2}-1)\alpha\in \ell^{3}$ (also see Theorem 2.5 in \cite{bbz1}).
\end{proof}

\begin{rem}
Corollary 4.26 is an important result in \cite{bbz1} which provides one half support for the validity of Lukic conjecture according to the view of Breuer, Simon and Zeitouni.
\end{rem}

Moreover, we have

\begin{thm}
Assume $\alpha\in \ell^{4}$, then
\begin{align}
&\int_{0}^{2\pi}(1-\cos\theta)^{2}(1+\cos\theta)\log w(\theta)\frac{d\theta}{2\pi}>-\infty \,\,\Longleftrightarrow\,\,(S-1)^{2}(S+1)\alpha\in \ell^{2}.
\end{align}
\end{thm}

\begin{proof}
By the assumpation,  $\mathrm{CP}<+\infty$ and the series in \textrm{EP} are convergent except
$$-\frac{1}{8}\sum_{j=0}^{\infty}|\alpha_{j+2}-\alpha_{j+1}-\alpha_{j}+\alpha_{j-1}|^{2}.$$
Thus (4.100) holds.
\end{proof}

\begin{rem}
This is one special case of a result  in \cite{gz} due to Golinskii and Zlato\v{s}.
\end{rem}

\begin{thm}
Assume $(S-1)\alpha\in \ell^{2}$, then
\begin{align}
&\int_{0}^{2\pi}(1-\cos\theta)^{2}(1+\cos\theta)\log w(\theta)\frac{d\theta}{2\pi}>-\infty \,\,\Longleftrightarrow\,\,\alpha\in \ell^{6}.
\end{align}
\end{thm}

\begin{proof}
By the assumpation,  $\mathrm{CP}<+\infty$. Therefore, $Z_{3,2}(\mu)>-\infty$ if and only if $-\mathrm{EP}<+\infty$. By the assumption, Lemma 4.3 and H\"older inequality, it is easy to get that $-\mathrm{EP}<+\infty$ is equivalent to $\alpha\in\ell^{6}$.
\end{proof}

\begin{rem}
As Theorem 4.14, this is also a result similar to one special case of the result in \cite{lu1} due to Lukic .
\end{rem}

\begin{thm}
 For any $\alpha$,
\begin{align}
&\int_{0}^{2\pi}(1-\cos\theta)^{3}\log w(\theta)\frac{d\theta}{2\pi}\nonumber\\
=&\frac{23}{12}+\frac{3}{8}|\alpha_{0}|^{2}+\frac{1}{8}|\alpha_{0}|^{2}|\alpha_{1}|^{2}-\frac{1}{8}|\alpha_{0}|^{2}|\alpha_{0}+1|^{2}
-\frac{3}{8}|\alpha_{1}-\alpha_{0}|^{2}-\frac{3}{2}|\alpha_{0}+1|^{2}\nonumber\\
&+\frac{3}{8}|\alpha_{1}+1|^{2}+\frac{5}{2}\sum_{j=0}^{\infty}\Big(\log(1-|\alpha_{j}|^{2})+|\alpha_{j}|^{2}+\frac{1}{2}|\alpha_{j}|^{4}+\frac{1}{3}|\alpha_{j}|^{6}\Big)\nonumber\\
&-\frac{1}{8}\sum_{j=0}^{\infty}|\alpha_{j+2}-3\alpha_{j+1}+3\alpha_{j}-\alpha_{j-1}|^{2}-\frac{1}{2}\sum_{j=0}^{\infty}|\alpha_{j}|^{2}|\alpha_{j+1}-\alpha_{j-1}|^{2}\nonumber\\
&-\frac{1}{2}\sum_{j=0}^{\infty}(|\alpha_{j}|^{2}-|\alpha_{j-1}|^{2})^{2}-\frac{5}{8}\sum_{j=0}^{\infty}(|\alpha_{j}|^{2}+|\alpha_{j-1}|^{2})|\alpha_{j}-\alpha_{j-1}|^{2}\nonumber\\
&-\frac{1}{8}\sum_{j=0}^{\infty}|\alpha_{j+1}-\alpha_{j-1}|^{2}\big(|\alpha_{j}-\alpha_{j-1}|^{2}+|\alpha_{j+1}-\alpha_{j}|^{2}\big)\nonumber\\
&-\frac{1}{8}\sum_{j=0}^{\infty}|\alpha_{j+2}-\alpha_{j-1}|^{2}|\alpha_{j+1}|^{2}|\alpha_{j}|^{2}-\frac{3}{4}\sum_{j=0}^{\infty}|\alpha_{j}|^{6}\nonumber\\
&-\frac{1}{8}\sum_{j=0}^{\infty}\Big[(|\alpha_{j+1}|^{2}+2|\alpha_{j}|^{2}+|\alpha_{j-1}|^{2})|\alpha_{j+1}-\alpha_{j-1}|^{2}\nonumber\\
&+|\alpha_{j+1}|^{2}|\alpha_{j+1}-\alpha_{j}|^{2}\Big]|\alpha_{j}|^{2}\nonumber\\
&-\frac{1}{24}\sum_{j=0}^{\infty}|\alpha_{j}-\alpha_{j-1}|^{6}-\frac{1}{8}\sum_{j=0}^{\infty}\big(|\alpha_{j}|^{2}+|\alpha_{j-1}|^{2}\big)^{2}|\alpha_{j}-\alpha_{j-1}|^{2}\nonumber
\end{align}
\begin{align}
&+\frac{3}{8}\sum_{j=0}^{\infty}|\alpha_{j}-\alpha_{j-1}|^{4}+\frac{1}{8}\sum_{j=0}^{\infty}(|\alpha_{j+1}|^{2}-|\alpha_{j-1}|^{2})^{2}\nonumber\\
&+\frac{1}{8}\sum_{j=0}^{\infty}(|\alpha_{j+2}|^{2}+|\alpha_{j-1}|^{2})|\alpha_{j+1}|^{2}|\alpha_{j}|^{2}\nonumber\\
&+\frac{1}{8}\sum_{j=0}^{\infty}\Big[|\alpha_{j+1}|^{2}|\alpha_{j}|^{2}+|\alpha_{j-1}|^{4}+|\alpha_{j}|^{2}|\alpha_{j-1}|^{2}+|\alpha_{j+1}|^{4}\nonumber\\
&+|\alpha_{j+1}-\alpha_{j-1}|^{2}\big(|\alpha_{j}-\alpha_{j-1}|^{2}+|\alpha_{j+1}-\alpha_{j}|^{2}\big)\Big]|\alpha_{j}|^{2}\nonumber\\
&+\frac{1}{8}\sum_{j=0}^{\infty}|\alpha_{j+2}-\alpha_{j-1}|^{2}(|\alpha_{j+1}|^{2}+|\alpha_{j}|^{2})\nonumber\\
&+\frac{1}{8}\sum_{j=0}^{\infty}\Big[(|\alpha_{j+1}|^{2}+|\alpha_{j-1}|^{2})|\alpha_{j+1}-\alpha_{j-1}|^{2}+(|\alpha_{j}|^{2}+|\alpha_{j-1}|^{2})|\alpha_{j}-\alpha_{j-1}|^{4}\Big].
\end{align}
\end{thm}

\begin{proof}
Since
\begin{align*}
(1-\cos\theta)^{3}=&\frac{5}{2}-\frac{15}{4}\cos\theta+\frac{3}{2}\cos2\theta-\frac{1}{4}\cos3\theta\nonumber\\
=&\frac{15}{4}(1-\cos\theta)-\frac{3}{2}(1-\cos2\theta)+\frac{1}{4}(1-\cos3\theta),
\end{align*}
then
\begin{align}
Z_{3,3}(\mu)=\frac{15}{4}Z_{1}(\mu)-3Z_{2,1}(\mu)+\frac{1}{4}Z_{3,1}(\mu).
\end{align}

By (4.2), (4.12), (4.60) and (4.103), we obtain
\begin{align}
&Z_{3,3}(\mu)=\int_{0}^{2\pi}(1-\cos\theta)^{3}\log w(\theta)\frac{d\theta}{2\pi}\nonumber\\
=&\frac{15}{8}+\frac{15}{4}\sum_{j=0}^{\infty}\Big(\log(1-|\alpha_{j}|^{2})+|\alpha_{j}|^{2}\Big)-\frac{15}{8}\sum_{j=0}^{\infty}|\alpha_{j}-\alpha_{j-1}|^{2}\nonumber\\
&-\frac{9}{8}-\frac{3}{2}\sum_{j=0}^{\infty}\Big[\log(1-|\alpha_{j}|^{2})+|\alpha_{j}|^{2}+\frac{1}{2}|\alpha_{j}|^{4}\Big]\nonumber\\
&+\frac{3}{2}\sum_{j=0}^{\infty}|\alpha_{j}\alpha_{j-1}|^{2}+\frac{3}{4}\sum_{j=0}^{\infty}\rho_{j}^{2}|\alpha_{j+1}-\alpha_{j-1}|^{2}\nonumber\\
&+\frac{3}{16}\sum_{j=0}^{\infty}\left[\big(2|\alpha_{j}|^{2}-|\alpha_{j}-\alpha_{j-1}|^{2}\big)^{2}+\big(2|\alpha_{j-1}|^{2}-|\alpha_{j}-\alpha_{j-1}|^{2}\big)^{2}\right]\nonumber\\
&+\frac{1}{6}-\frac{1}{8}(|\alpha_{0}|^{2}+|\alpha_{1}|^{2})-\frac{1}{8}|\alpha_{0}|^{2}|\alpha_{0}+1|^{2}\nonumber\\
&+\frac{1}{4}\sum_{j=0}^{\infty}\Big[\log(1-|\alpha_{j}|^{2})+|\alpha_{j}|^{2}+\frac{1}{2}|\alpha_{j}|^{4}+\frac{1}{3}|\alpha_{j}|^{6}\Big]\nonumber
\end{align}
\begin{align}
&-\frac{1}{8}\sum_{j=0}^{\infty}|\alpha_{j}|^{4}-\frac{1}{8}\sum_{j=0}^{\infty}(|\alpha_{j+2}|^{2}+|\alpha_{j-1}|^{2})|\alpha_{j}|^{2}\nonumber\\
&-\frac{1}{8}\sum_{j=0}^{\infty}(|\alpha_{j+2}|^{2}+|\alpha_{j-1}|^{2})|\alpha_{j+1}|^{2}\rho_{j}^{2}-\frac{1}{8}\sum_{j=0}^{\infty}|\alpha_{j+2}-\alpha_{j-1}|^{2}\rho_{j+1}^{2}\rho_{j}^{2}\nonumber\\
&-\frac{1}{8}\sum_{j=0}^{\infty}\Big[|\alpha_{j+1}|^{2}|\alpha_{j}|^{2}+|\alpha_{j-1}|^{4}+|\alpha_{j}|^{2}|\alpha_{j-1}|^{2}+|\alpha_{j+1}|^{4}\nonumber\\
&+|\alpha_{j+1}-\alpha_{j-1}|^{2}\big(|\alpha_{j}-\alpha_{j-1}|^{2}+|\alpha_{j+1}-\alpha_{j}|^{2}\big)\Big]\rho_{j}^{2}\nonumber\\
&-\frac{1}{8}\sum_{j=0}^{\infty}\Big[(|\alpha_{j+1}|^{2}+2|\alpha_{j}|^{2}+|\alpha_{j-1}|^{2})|\alpha_{j+1}-\alpha_{j-1}|^{2}\nonumber\\
&+|\alpha_{j-1}|^{2}|\alpha_{j}-\alpha_{j-1}|^{2}+|\alpha_{j+1}|^{2}|\alpha_{j+1}-\alpha_{j}|^{2}\Big]|\alpha_{j}|^{2}\nonumber\\
&-\frac{1}{24}\sum_{j=0}^{\infty}|\alpha_{j}-\alpha_{j-1}|^{6}-\frac{1}{8}\sum_{j=0}^{\infty}\big(|\alpha_{j}|^{4}+|\alpha_{j}|^{2}|\alpha_{j-1}|^{2}+|\alpha_{j-1}|^{4})|\alpha_{j}-\alpha_{j-1}|^{2}\nonumber\\
&+\frac{1}{8}\sum_{j=0}^{\infty}\Big[(|\alpha_{j+1}|^{2}+2|\alpha_{j}|^{2}+|\alpha_{j-1}|^{2})|\alpha_{j+1}-\alpha_{j-1}|^{2}\nonumber\\
&+(|\alpha_{j}|^{2}+|\alpha_{j-1}|^{2})(|\alpha_{j}-\alpha_{j-1}|^{2}+|\alpha_{j}-\alpha_{j-1}|^{4})\Big]\nonumber\\
=&\frac{23}{12}+\frac{3}{8}|\alpha_{0}|^{2}+\frac{1}{8}|\alpha_{0}|^{2}|\alpha_{1}|^{2}-\frac{1}{8}|\alpha_{0}|^{2}|\alpha_{0}+1|^{2}
-\frac{3}{8}|\alpha_{1}-\alpha_{0}|^{2}-\frac{3}{2}|\alpha_{0}+1|^{2}\nonumber\\
&+\frac{3}{8}|\alpha_{1}+1|^{2}+\frac{5}{2}\sum_{j=0}^{\infty}\Big(\log(1-|\alpha_{j}|^{2})+|\alpha_{j}|^{2}+\frac{1}{2}|\alpha_{j}|^{4}+\frac{1}{3}|\alpha_{j}|^{6}\Big)\nonumber\\
&-\frac{1}{8}\sum_{j=0}^{\infty}|\alpha_{j+2}-3\alpha_{j+1}+3\alpha_{j}-\alpha_{j-1}|^{2}-\frac{1}{2}\sum_{j=0}^{\infty}|\alpha_{j}|^{2}|\alpha_{j+1}-\alpha_{j-1}|^{2}\nonumber\\
&-\frac{1}{2}\sum_{j=0}^{\infty}(|\alpha_{j}|^{2}-|\alpha_{j-1}|^{2})^{2}-\frac{5}{8}\sum_{j=0}^{\infty}(|\alpha_{j}|^{2}+|\alpha_{j-1}|^{2})|\alpha_{j}-\alpha_{j-1}|^{2}\nonumber\\
&-\frac{1}{8}\sum_{j=0}^{\infty}|\alpha_{j+1}-\alpha_{j-1}|^{2}\big(|\alpha_{j}-\alpha_{j-1}|^{2}+|\alpha_{j+1}-\alpha_{j}|^{2}\big)\nonumber\\
&-\frac{1}{8}\sum_{j=0}^{\infty}|\alpha_{j+2}-\alpha_{j-1}|^{2}|\alpha_{j+1}|^{2}|\alpha_{j}|^{2}-\frac{3}{4}\sum_{j=0}^{\infty}|\alpha_{j}|^{6}\nonumber\\
&-\frac{1}{8}\sum_{j=0}^{\infty}\Big[(|\alpha_{j+1}|^{2}+2|\alpha_{j}|^{2}+|\alpha_{j-1}|^{2})|\alpha_{j+1}-\alpha_{j-1}|^{2}+|\alpha_{j+1}|^{2}|\alpha_{j+1}-\alpha_{j}|^{2}\Big]|\alpha_{j}|^{2}\nonumber\\
&-\frac{1}{24}\sum_{j=0}^{\infty}|\alpha_{j}-\alpha_{j-1}|^{6}-\frac{1}{8}\sum_{j=0}^{\infty}\big(|\alpha_{j}|^{2}+|\alpha_{j-1}|^{2})^{2}|\alpha_{j}-\alpha_{j-1}|^{2}\nonumber\\
&+\frac{3}{8}\sum_{j=0}^{\infty}|\alpha_{j}-\alpha_{j-1}|^{4}+\frac{1}{8}\sum_{j=0}^{\infty}(|\alpha_{j+1}|^{2}-|\alpha_{j-1}|^{2})^{2}\nonumber\\
&+\frac{1}{8}\sum_{j=0}^{\infty}(|\alpha_{j+2}|^{2}+|\alpha_{j-1}|^{2})|\alpha_{j+1}|^{2}|\alpha_{j}|^{2}\nonumber
\end{align}
\begin{align}
&+\frac{1}{8}\sum_{j=0}^{\infty}\Big[|\alpha_{j+1}|^{2}|\alpha_{j}|^{2}+|\alpha_{j-1}|^{4}+|\alpha_{j}|^{2}|\alpha_{j-1}|^{2}+|\alpha_{j+1}|^{4}\nonumber\\
&+|\alpha_{j+1}-\alpha_{j-1}|^{2}\big(|\alpha_{j}-\alpha_{j-1}|^{2}+|\alpha_{j+1}-\alpha_{j}|^{2}\big)\Big]|\alpha_{j}|^{2}\nonumber\\
&+\frac{1}{8}\sum_{j=0}^{\infty}|\alpha_{j+2}-\alpha_{j-1}|^{2}(|\alpha_{j+1}|^{2}+|\alpha_{j}|^{2})\nonumber\\
&+\frac{1}{8}\sum_{j=0}^{\infty}\Big[(|\alpha_{j+1}|^{2}+|\alpha_{j-1}|^{2})|\alpha_{j+1}-\alpha_{j-1}|^{2}+(|\alpha_{j}|^{2}+|\alpha_{j-1}|^{2})|\alpha_{j}-\alpha_{j-1}|^{4}\Big],\nonumber
\end{align}
where the following identity is used
\begin{align}
&|\alpha_{j+2}-3\alpha_{j+1}+3\alpha_{j}-\alpha_{j-1}|^{2}\nonumber\\
=&3|\alpha_{j+2}-\alpha_{j+1}|^{2}-3|\alpha_{j+2}-\alpha_{j}|^{2}+|\alpha_{j+2}-\alpha_{j-1}|^{2}\nonumber\\
&+9|\alpha_{j+1}-\alpha_{j}|^{2}-3|\alpha_{j+1}-\alpha_{j-1}|^{2}+3|\alpha_{j}-\alpha_{j-1}|^{2}
\end{align}
which implying
\begin{align}
&\frac{1}{8}\sum_{j=0}^{\infty}|\alpha_{j+2}-3\alpha_{j+1}+3\alpha_{j}-\alpha_{j-1}|^{2}\nonumber\\
=&-\frac{3}{2}|\alpha_{0}+1|^{2}-\frac{3}{8}|\alpha_{1}-\alpha_{0}|^{2}+\frac{3}{8}|\alpha_{1}+1|^{2}\nonumber\\
&+\frac{15}{8}\sum_{j=0}^{\infty}|\alpha_{j}-\alpha_{j-1}|^{2}-\frac{3}{4}|\alpha_{j+1}-\alpha_{j-1}|^{2}+\frac{1}{8}|\alpha_{j+2}-\alpha_{j-1}|^{2}.
\end{align}
\end{proof}

As before, set
\begin{align}
\mathrm{EP}=&\frac{23}{12}+\frac{3}{8}|\alpha_{0}|^{2}+\frac{1}{8}|\alpha_{0}|^{2}|\alpha_{1}|^{2}-\frac{1}{8}|\alpha_{0}|^{2}|\alpha_{0}+1|^{2}
-\frac{3}{8}|\alpha_{1}-\alpha_{0}|^{2}-\frac{3}{2}|\alpha_{0}+1|^{2}\nonumber\\
&+\frac{3}{8}|\alpha_{1}+1|^{2}+\frac{5}{2}\sum_{j=0}^{\infty}\Big(\log(1-|\alpha_{j}|^{2})+|\alpha_{j}|^{2}+\frac{1}{2}|\alpha_{j}|^{4}+\frac{1}{3}|\alpha_{j}|^{6}\Big)\nonumber\\
&-\frac{1}{8}\sum_{j=0}^{\infty}|\alpha_{j+2}-3\alpha_{j+1}+3\alpha_{j}-\alpha_{j-1}|^{2}-\frac{1}{2}\sum_{j=0}^{\infty}|\alpha_{j}|^{2}|\alpha_{j+1}-\alpha_{j-1}|^{2}\nonumber\\
&-\frac{1}{2}\sum_{j=0}^{\infty}(|\alpha_{j}|^{2}-|\alpha_{j-1}|^{2})^{2}-\frac{5}{8}\sum_{j=0}^{\infty}(|\alpha_{j}|^{2}+|\alpha_{j-1}|^{2})|\alpha_{j}-\alpha_{j-1}|^{2}\nonumber\\
&-\frac{1}{8}\sum_{j=0}^{\infty}|\alpha_{j+1}-\alpha_{j-1}|^{2}\big(|\alpha_{j}-\alpha_{j-1}|^{2}+|\alpha_{j+1}-\alpha_{j}|^{2}\big)\nonumber\\
&-\frac{1}{8}\sum_{j=0}^{\infty}|\alpha_{j+2}-\alpha_{j-1}|^{2}|\alpha_{j+1}|^{2}|\alpha_{j}|^{2}-\frac{3}{4}\sum_{j=0}^{\infty}|\alpha_{j}|^{6}\nonumber\\
&-\frac{1}{8}\sum_{j=0}^{\infty}\Big[(|\alpha_{j+1}|^{2}+2|\alpha_{j}|^{2}+|\alpha_{j-1}|^{2})|\alpha_{j+1}-\alpha_{j-1}|^{2}+|\alpha_{j+1}|^{2}|\alpha_{j+1}-\alpha_{j}|^{2}\Big]|\alpha_{j}|^{2}\nonumber
\end{align}
\begin{align}
&-\frac{1}{24}\sum_{j=0}^{\infty}|\alpha_{j}-\alpha_{j-1}|^{6}-\frac{1}{8}\sum_{j=0}^{\infty}\big(|\alpha_{j}|^{2}+|\alpha_{j-1}|^{2}\big)^{2}|\alpha_{j}-\alpha_{j-1}|^{2}
\end{align}
and
\begin{align}
\mathrm{CP}=&\frac{3}{8}\sum_{j=0}^{\infty}|\alpha_{j}-\alpha_{j-1}|^{4}+\frac{1}{8}\sum_{j=0}^{\infty}(|\alpha_{j+1}|^{2}-|\alpha_{j-1}|^{2})^{2}\nonumber\\
&+\frac{1}{8}\sum_{j=0}^{\infty}(|\alpha_{j+2}|^{2}+|\alpha_{j-1}|^{2})|\alpha_{j+1}|^{2}|\alpha_{j}|^{2}\nonumber\\
&+\frac{1}{8}\sum_{j=0}^{\infty}\Big[|\alpha_{j+1}|^{2}|\alpha_{j}|^{2}+|\alpha_{j-1}|^{4}+|\alpha_{j}|^{2}|\alpha_{j-1}|^{2}+|\alpha_{j+1}|^{4}\nonumber\\
&+|\alpha_{j+1}-\alpha_{j-1}|^{2}\big(|\alpha_{j}-\alpha_{j-1}|^{2}+|\alpha_{j+1}-\alpha_{j}|^{2}\big)\Big]|\alpha_{j}|^{2}\nonumber\\
&+\frac{1}{8}\sum_{j=0}^{\infty}|\alpha_{j+2}-\alpha_{j-1}|^{2}(|\alpha_{j+1}|^{2}+|\alpha_{j}|^{2})\nonumber\\
&+\frac{1}{8}\sum_{j=0}^{\infty}\Big[(|\alpha_{j+1}|^{2}+|\alpha_{j-1}|^{2})|\alpha_{j+1}-\alpha_{j-1}|^{2}+(|\alpha_{j}|^{2}+|\alpha_{j-1}|^{2})|\alpha_{j}-\alpha_{j-1}|^{4}\Big].
\end{align}

By using sum rule (4.102), we have
\begin{thm}
Assume $\alpha\in \ell^{4}$, then
\begin{align}
&\int_{0}^{2\pi}(1-\cos\theta)^{3}\frac{d\theta}{2\pi}>-\infty \,\,\Longleftrightarrow\,\,(S-1)^{3}\alpha\in \ell^{2}.
\end{align}
\end{thm}

\begin{proof}
By the assumption, it is easy to get that $\mathrm{CP}<+\infty$ and almost all series in $\mathrm{EP}$ are convergent except only one, that is,
$$-\frac{1}{8}\sum_{j=0}^{\infty}|\alpha_{j+2}-3\alpha_{j+1}+3\alpha_{j}-\alpha_{j-1}|^{2}.$$
Thus $Z_{3,3}(\mu)>-\infty$ if and only if $(S-1)^{3}\alpha\in\ell^{2}$ as $\alpha\in \ell^{4}$.
\end{proof}

\begin{rem}
This is also one special case of a result  in \cite{gz} due to Golinskii and Zlato\v{s}.
\end{rem}

\begin{thm}
Assume $(S-1)\alpha\in \ell^{1}$, then
\begin{align}
&\int_{0}^{2\pi}(1-\cos\theta)^{3}\frac{d\theta}{2\pi}>-\infty \,\,\Longleftrightarrow\,\,\alpha\in \ell^{8}.
\end{align}
\end{thm}

\begin{proof}
Note that
\begin{align}
&2(a^{3}+b^{3}+c^{3})-2abc-ab(a+b)-bc(b+c)\nonumber\\
=&2(a^{3}+b^{3}+c^{3})-(a+b+c)(a+c)b\nonumber\\
=&2(a^{3}+c^{3})-(a+b+c)(a+c-2b)b-2b^{2}(a+c)\nonumber\\
=&2a(a^{2}-b^{2})+2c(c^{2}-b^{2})-(a+b+c)(a+c-2b)b
\end{align}
and
\begin{align}
\big||a|^{2}-|b|^{2}\big|\leq|a-b|(|a|+|b|),
\end{align}
where $a,b,c\in \mathbb{C}$.

Let $a=|\alpha_{j+1}|^{2}$, $b=|\alpha_{j}|^{2}$ and $c=|\alpha_{j-1}|^{2}$, by (4.110), we have
\begin{align}
&-\frac{3}{4}\sum_{j=0}^{\infty}|\alpha_{j}|^{6}+\frac{1}{8}\sum_{j=0}^{\infty}(|\alpha_{j+2}|^{2}+|\alpha_{j-1}|^{2})|\alpha_{j+1}|^{2}|\alpha_{j}|^{2}\nonumber\\
&+\frac{1}{8}\sum_{j=0}^{\infty}\Big[|\alpha_{j+1}|^{2}|\alpha_{j}|^{2}+|\alpha_{j-1}|^{4}+|\alpha_{j}|^{2}|\alpha_{j-1}|^{2}+|\alpha_{j+1}|^{4}\Big]|\alpha_{j}|^{2}\nonumber\\
=&\frac{1}{4}\big(1-|\alpha_{0}|^{6}\big)-\frac{1}{8}|\alpha_{1}|^{2}|\alpha_{0}|^{2}\nonumber\\
&-\frac{1}{4}\sum_{j=0}^{\infty}\big(|\alpha_{j+1}|^{6}+|\alpha_{j}|^{6}+|\alpha_{j-1}|^{6}\big)+\frac{1}{4}\sum_{j=0}^{\infty}|\alpha_{j+1}|^{2}|\alpha_{j}|^{2}|\alpha_{j-1}|^{2}\nonumber\\
&+\frac{1}{8}\sum_{j=0}^{\infty}\Big[|\alpha_{j+1}|^{4}|\alpha_{j}|^{2}+|\alpha_{j+1}|^{2}|\alpha_{j}|^{4}+|\alpha_{j}|^{4}|\alpha_{j-1}|^{2}+|\alpha_{j}|^{2}|\alpha_{j-1}|^{4}\Big]\nonumber\\
=&\frac{1}{4}\big(1-|\alpha_{0}|^{6}\big)-\frac{1}{8}|\alpha_{1}|^{2}|\alpha_{0}|^{2}\nonumber\\
&-\frac{1}{4}\sum_{j=0}^{\infty}|\alpha_{j+1}|^{2}\big(|\alpha_{j+1}|^{4}-|\alpha_{j}|^{4}\big)-\frac{1}{4}\sum_{j=0}^{\infty}|\alpha_{j-1}|^{2}\big(|\alpha_{j-1}|^{4}-|\alpha_{j}|^{4}\big)\nonumber\\
&+\frac{1}{8}\sum_{j=0}^{\infty}\big(|\alpha_{j+1}|^{2}+|\alpha_{j}|^{2}+|\alpha_{j-1}|^{2}\big)\big(|\alpha_{j+1}|^{2}+|\alpha_{j-1}|^{2}-2|\alpha_{j}|^{2}\big)|\alpha_{j}|^{2}\nonumber\\
=&\frac{1}{4}\big(1-|\alpha_{0}|^{6}\big)-\frac{1}{8}|\alpha_{1}|^{2}|\alpha_{0}|^{2}\nonumber\\
&-\frac{1}{4}\sum_{j=0}^{\infty}|\alpha_{j+1}|^{2}\big(|\alpha_{j+1}|-|\alpha_{j}|\big)\big(|\alpha_{j+1}|+|\alpha_{j}|\big)\big(|\alpha_{j+1}|^{2}+|\alpha_{j}|^{2}\big)\nonumber\\
&-\frac{1}{4}\sum_{j=0}^{\infty}|\alpha_{j-1}|^{2}\big(|\alpha_{j-1}|-|\alpha_{j}|\big)\big(|\alpha_{j-1}|+|\alpha_{j}|\big)\big(|\alpha_{j-1}|^{2}+|\alpha_{j}|^{2}\big)\nonumber\\
&+\frac{1}{8}\sum_{j=0}^{\infty}\big(|\alpha_{j+1}|^{2}+|\alpha_{j}|^{2}+|\alpha_{j-1}|^{2}\big)\big(|\alpha_{j+1}|-|\alpha_{j}|\big)\big(|\alpha_{j+1}|+|\alpha_{j}|\big)|\alpha_{j}|^{2}\nonumber\\
&+\frac{1}{8}\sum_{j=0}^{\infty}\big(|\alpha_{j+1}|^{2}+|\alpha_{j}|^{2}+|\alpha_{j-1}|^{2}\big)\big(|\alpha_{j-1}|-|\alpha_{j}|\big)\big(|\alpha_{j-1}|+|\alpha_{j}|\big)|\alpha_{j}|^{2}.
\end{align}

By the assumption, (4.111), (4.112) and triangle inequality, we obtain that the series in (4.102) are finite except
$$\frac{5}{2}\sum_{j=0}^{\infty}\Big(\log(1-|\alpha_{j}|^{2})+|\alpha_{j}|^{2}+\frac{1}{2}|\alpha_{j}|^{4}+\frac{1}{3}|\alpha_{j}|^{6}\Big).$$

So by Lemma 4.3, $Z_{3,3}(\mu)>-\infty$ is equivalent to $\alpha\in\ell^{8}$ as $(S-1)\alpha\in\ell^{1}$.
\end{proof}

\begin{rem}
Noting the following Young inequality,
\begin{equation}
|a-b||c|^{2}|d|^{2}\leq \frac{1}{2}|a-b|^{2}+\frac{1}{8}(|c|^{8}+|d|^{8}),
\end{equation}
where $a,b,c\in \mathbb{C}$, by (4.102), (4.112) and (4.113), it is easy to get that
\begin{align}
(S-1)\alpha\in\ell^{2}\,\,\mathrm{and}\,\,\alpha\in \ell^{8}\,\,\Longrightarrow\,\,\int_{0}^{2\pi}(1-\cos\theta)^{3}\frac{d\theta}{2\pi}>-\infty.
\end{align}
\end{rem}

\begin{rem}
For the case of $(S-1)\alpha\in\ell^{1}$, we mainly used the identity (4.110) and the inequality (4.111) to treat the intractable series in (4.112) in the proof of the above theorem. Indeed, an alternate approach is to apply the following telescoping sum
\begin{align}
A_{1}A_{2}\cdots A_{n}-B_{1}B_{2}\cdots B_{n}=\sum_{j=1}^{n}A_{1}\cdots A_{j-1}(A_{j}-B_{j})B_{j+1}\cdots B_{n}.
\end{align}
Such idea was used in \cite{lu1}.
\end{rem}

Corresponding to the above theorem and (4.114), we actually have a stronger result as follows
\begin{thm}
Assume $(S-1)\alpha\in \ell^{2}$, then
\begin{align}
&\int_{0}^{2\pi}(1-\cos\theta)^{3}\frac{d\theta}{2\pi}>-\infty\,\,\Longleftrightarrow\,\,\alpha\in \ell^{8}.
\end{align}
\end{thm}

\begin{proof}
Since (see \cite{lu1})
\begin{align}
\left|\frac{z_{1}^{3}+z_{2}^{3}+z_{3}^{3}}{3}-z_{1}z_{2}z_{3}\right|\leq 4\max_{1\leq i,j\leq 3}|z_{i}-z_{j}|^{2},
\end{align}
where $z_{1},z_{2},z_{3}\in \overline{\mathbb{D}}$, then
\begin{align}
&\left|\frac{|\alpha_{j+1}|^{6}+|\alpha_{j}|^{6}+|\alpha_{j-1}|^{6}}{3}-|\alpha_{j+1}|^{2}|\alpha_{j}|^{2}|\alpha_{j-1}|^{2}\right|\nonumber\\
\leq &16\max_{k,l\in\{-1,0,1\} }|\alpha_{j+k}-\alpha_{j+l}|^{2},
\end{align}
\begin{align}
\left|\frac{2|\alpha_{j+1}|^{6}+|\alpha_{j}|^{6}}{3}-|\alpha_{j+1}|^{4}|\alpha_{j}|^{2}\right|
\leq 16\max_{k,l\in\{0,1\} }|\alpha_{j+k}-\alpha_{j+l}|^{2},
\end{align}
\begin{align}
\left|\frac{|\alpha_{j+1}|^{6}+2|\alpha_{j}|^{6}}{3}-|\alpha_{j+1}|^{2}|\alpha_{j}|^{4}\right|
\leq 16\max_{k,l\in\{0,1\} }|\alpha_{j+k}-\alpha_{j+l}|^{2},
\end{align}
\begin{align}
\left|\frac{2|\alpha_{j}|^{6}+|\alpha_{j-1}|^{6}}{3}-|\alpha_{j}|^{4}|\alpha_{j-1}|^{2}\right|
\leq 16\max_{k,l\in\{-1,0\} }|\alpha_{j+k}-\alpha_{j+l}|^{2}
\end{align}
and
\begin{align}
\left|\frac{|\alpha_{j}|^{6}+2|\alpha_{j-1}|^{6}}{3}-|\alpha_{j}|^{2}|\alpha_{j-1}|^{4}\right|
\leq 16\max_{k,l\in\{-1,0\} }|\alpha_{j+k}-\alpha_{j+l}|^{2},
\end{align}
where we have used the following inequality
\begin{equation*}
\big||a|^{2}-|b|^{2}\big|^{2}\leq 4 |a-b|^{2}
\end{equation*}
for any $a,b\in \overline{\mathbb{D}}$.
In addition, we have
\begin{align}
\frac{3}{4}\sum_{j=0}^{\infty}|\alpha_{j}|^{6}=&\frac{5}{24}\big(|\alpha_{0}|^{6}-1\big)+\frac{1}{4}\sum_{j=0}^{\infty}\Big[\frac{|\alpha_{j+1}|^{6}+|\alpha_{j}|^{6}+|\alpha_{j-1}|^{6}}{3}\nonumber\\
&+\frac{1}{8}\Big(\frac{2|\alpha_{j+1}|^{6}+|\alpha_{j}|^{6}}{3}+\frac{|\alpha_{j+1}|^{6}+2|\alpha_{j}|^{6}}{3}\nonumber\\
&+\frac{2|\alpha_{j}|^{6}+|\alpha_{j-1}|^{6}}{3}+\frac{|\alpha_{j}|^{6}+2|\alpha_{j-1}|^{6}}{3}\Big)\Big].
\end{align}
By (4.118)-(4.123), we get that the series in (4.112) is finite as $(S-1)\alpha\in\ell^{2}$. Thus by (4.102), we obtain that $Z_{3,3}(\mu)>-\infty$ if and only if $\alpha\in\ell^{8}$ as $(S-1)\alpha\in\ell^{2}$.
\end{proof}

\begin{rem}
This is one special case of the result in \cite{lu1} due to Lukic. The idea for the argument was also used in it. In this special case, the only difference is that we have the explicit sum rule (4.102) whereas there was no such sum rule in \cite{lu1}. Such idea is based on the following inequality due to Lukic (for details, see Lemma 2.2 in \cite{lu1})
\begin{align}
\left|\frac{z_{1}^{n}+z_{2}^{n}+\cdots+z_{n}^{n}}{n}-z_{1}z_{2}\cdots z_{n}\right|\leq (n-1)^{2}\max_{1\leq i,j\leq n}|z_{i}-z_{j}|^{2},
\end{align}
where $z_{1},z_{2},\ldots, z_{n}\in \overline{\mathbb{D}}$. As $n=3$, (4.124) is just (4.117). In what follows, we will use this inequality for many times.
\end{rem}

\begin{rem}
From the arguments to Theorems 4.35 and 4.38, the roles of equivalent and conditional parts are not absolute for some assumptions (for instance, $(S-1)\alpha\in\ell^{1}$ in Theorem 4.35 and $(S-1)\alpha\in\ell^{2}$ in Theorem 4.38) because some summands in both \textrm{EP} and \textrm{CP} should be jointly considered for applying of the assumption (for example, in the above two theorems, $-\frac{3}{4}\sum_{j=0}^{\infty}|\alpha_{j}|$ is in \textrm{EP} whereas $\frac{1}{8}\sum_{j=0}^{\infty}(|\alpha_{j+2}|^{2}+|\alpha_{j-1}|^{2})|\alpha_{j+1}|^{2}|\alpha_{j}|^{2}$ and $\frac{1}{8}\sum_{j=0}^{\infty}\Big[|\alpha_{j+1}|^{2}|\alpha_{j}|^{2}+|\alpha_{j-1}|^{4}+|\alpha_{j}|^{2}|\alpha_{j-1}|^{2}+|\alpha_{j+1}|^{4}\Big]|\alpha_{j}|^{2}$ are in \textrm{CP}). Nevertheless, the roles of them are worthy of their names in most situations.
\end{rem}

\subsection{Fourth order case} Finally, as the end of this section, we establish some sum rules of fourth order concerning with $w_{j}$, $0\leq j\leq 4$.
Set
\begin{equation}
Z_{4,1}(\mu)=\int_{0}^{2\pi}(1-\cos4\theta)\log w(\theta) \frac{d\theta}{2\pi},
\end{equation}
\begin{equation}
Z_{4,2}(\mu)=\int_{0}^{2\pi}(1-\cos\theta)^{2}(1+\cos\theta)^{2}\log w(\theta) \frac{d\theta}{2\pi},
\end{equation}
\begin{equation}
Z_{4,3}(\mu)=\int_{0}^{2\pi}(1-\cos\theta)^{3}(1+\cos\theta)\log w(\theta) \frac{d\theta}{2\pi}
\end{equation}
and
\begin{equation}
Z_{4,4}(\mu)=\int_{0}^{2\pi}(1-\cos\theta)^{4}\log w(\theta) \frac{d\theta}{2\pi}.
\end{equation}

By using (4.50), we can extend (4.50) and (4.52) to more general cases.

\begin{prop}
\begin{itemize}
  \item [(1)] For any even $n\in \mathbb{N}$,
  \begin{align}
  \mathrm{Re}(A_{1}\cdots A_{n}A_{n+1}\cdots A_{2n})=&\sum_{l=n+1}^{2n}(-1)^{2n-l}\mathrm{Re}(A_{1}\cdots A_{l-1}\overline{A}_{l+1}\cdots \overline{A}_{2n})\mathrm{Re}(A_{l})\nonumber\\
  &+\mathrm{Re}(A_{1}\cdots A_{n})\mathrm{Re}(\overline{A}_{n+1}\cdots \overline{A}_{2n});
  \end{align}
  \item [(2)] For any odd $n\in \mathbb{N}$,
  \begin{align}
  \mathrm{Re}(A_{1}\cdots A_{n}A_{n+1}\cdots A_{2n})=&\sum_{l=n+2}^{2n}(-1)^{2n-l}\mathrm{Re}(A_{1}\cdots A_{l-1}\overline{A}_{l+1}\cdots \overline{A}_{2n})\mathrm{Re}(A_{l})\nonumber\\
  &+\mathrm{Re}(A_{1}\cdots A_{n+1})\mathrm{Re}(\overline{A}_{n+2}\cdots \overline{A}_{2n});
  \end{align}
  \item [(3)] For any $n\in \mathbb{N}$,
  \begin{align}
  \mathrm{Re}(A_{1}A_{2}\cdots A_{2n-1})=&\sum_{l=1}^{2n-1}(-1)^{2n-1-l}\mathrm{Re}(A_{1}\cdots A_{l-1}\overline{A}_{l+1}\cdots \overline{A}_{2n-1})\mathrm{Re}(A_{l}),
  \end{align}
  where $A_{l}\in \mathbb{C}$, $1\leq l\leq 2n$.
\end{itemize}
\end{prop}

\begin{proof}
Since the arguments are similar, it is sufficient for us to deduce the case (1).

By (4.50),
\begin{align*}
\mathrm{Re}(A_{1}\cdots A_{n}A_{n+1}\cdots A_{2n})=&2\mathrm{Re}(A_{1}\cdots A_{n}A_{n+1}\cdots A_{2n-1})\mathrm{Re}(A_{2n})\nonumber\\
&-\mathrm{Re}(A_{1}\cdots A_{n}A_{n+1}\cdots A_{2n-1}\overline{A}_{2n}),\\
\mathrm{Re}(A_{1}\cdots A_{n}A_{n+1}\cdots A_{2n-1}\overline{A}_{2n})=&2\mathrm{Re}(A_{1}\cdots A_{n}A_{n+1}\cdots A_{2n-2}\overline{A}_{2n})\mathrm{Re}(A_{2n-1})\nonumber\\
&-\mathrm{Re}(A_{1}\cdots A_{n}A_{n+1}\cdots \overline{A}_{2n-1}\overline{A}_{2n}),\\
&\vdots&\\
\mathrm{Re}(A_{1}\cdots A_{n}A_{n+1}\overline{A}_{n+2}\cdots \overline{A}_{2n})=&2\mathrm{Re}(A_{1}\cdots A_{n}\overline{A}_{n+2}\cdots \overline{A}_{2n})\mathrm{Re}(A_{n+1})\\
&-\mathrm{Re}(A_{1}\cdots A_{n}\overline{A}_{n+1}\cdots \overline{A}_{2n-1}\overline{A}_{2n})
\end{align*}
and
\begin{align*}
\mathrm{Re}(A_{1}\cdots A_{n}\overline{A}_{n+1}\cdots \overline{A}_{2n})=&2\mathrm{Re}(A_{1}\cdots A_{n})\mathrm{Re}(\overline{A}_{n+1}\cdots \overline{A}_{2n})\nonumber\\
&-\mathrm{Re}(A_{1}\cdots A_{n}A_{n+1}\cdots A_{2n}),
\end{align*}
then we get
\begin{align*}
&\mathrm{Re}(A_{1}\cdots A_{n}A_{n+1}\cdots A_{2n})\\
=&2\mathrm{Re}(A_{1}\cdots A_{n}A_{n+1}\cdots A_{2n-1})\mathrm{Re}(A_{2n})-2\mathrm{Re}(A_{1}\cdots A_{n}A_{n+1}\cdots A_{2n-2}\overline{A}_{2n})\mathrm{Re}(A_{2n-1})\nonumber\\
&+2\mathrm{Re}(A_{1}\cdots A_{n}A_{n+1}\cdots A_{2n-3}\overline{A}_{2n-1}\overline{A}_{2n})\mathrm{Re}(A_{2n-2})-\cdots\\
&+(-1)^{2n-l}2\mathrm{Re}(A_{1}\cdots A_{n}A_{n+1}\cdots A_{l-1}\overline{A}_{l+1}\cdots\overline{A}_{2n-1}\overline{A}_{2n})\mathrm{Re}(A_{l})+\cdots\\
&-2\mathrm{Re}(A_{1}\cdots A_{n}\overline{A}_{n+2}\cdots \overline{A}_{2n})\mathrm{Re}(A_{n+1})+
2\mathrm{Re}(A_{1}\cdots A_{n})\mathrm{Re}(\overline{A}_{n+1}\cdots \overline{A}_{2n})\nonumber\\
&-\mathrm{Re}(A_{1}\cdots A_{n}A_{n+1}\cdots A_{2n}).
\end{align*}
Thus (4.129) immediately holds.
\end{proof}

\begin{rem}
From (4.50), (4.52), (4.129), (4.130) and (4.131), applying recursion step by step, we know that any real part of a product of $n$ complex numbers can be explicitly expressed in terms of the real parts of one complex and products of two complex numbers.
\end{rem}

Similarly, we also have
\begin{prop}
For $A,B,C,D\in \mathbb{C}$,
\begin{align}
\mathrm{Re}(AB\overline{C}\,\overline{D})=&\mathrm{Re}(A\overline{C})\mathrm{Re}(B\overline{D})+\mathrm{Re}(A\overline{D})\mathrm{Re}(B\overline{C})-\mathrm{Re}(A\overline{B})\mathrm{Re}(C\overline{D}).
\end{align}
\end{prop}

\begin{proof}
By (4.50), we have
\begin{align*}
\mathrm{Re}(AB\overline{C}\,\overline{D})=2\mathrm{Re}(A\overline{C})\mathrm{Re}(B\overline{D})-\mathrm{Re}(A\overline{C}\,\overline{B}D),
\end{align*}
\begin{align*}
\mathrm{Re}(A\overline{C}\,\overline{B}D)=2\mathrm{Re}(A\overline{B})\mathrm{Re}(\overline{C}D)-\mathrm{Re}(A\overline{B}C\overline{D})
\end{align*}
and
\begin{align*}
\mathrm{Re}(A\overline{B}C\overline{D})=2\mathrm{Re}(A\overline{D})\mathrm{Re}(\overline{B}C)-\mathrm{Re}(AB\overline{C}\,\overline{D}),
\end{align*}
Thus (4.132) follows these facts and $\mathrm{Re}(A)=\mathrm{Re}(\overline{A})$ for any $A\in \mathbb{C}$.
\end{proof}

\begin{thm}
For any $\alpha$,
\begin{align}
&\int_{0}^{2\pi}(1-\cos4\theta)\log w(\theta) \frac{d\theta}{2\pi}\nonumber\\
=&1-\sum_{j=0}^{2}|\alpha_{j}|^{2}+\sum_{j=0}^{\infty}\Big(\log(1-|\alpha_{j}|^{2})+|\alpha_{j}|^{2}\Big)-\frac{1}{2}\sum_{j=0}^{\infty}|\alpha_{j}|^{2}(|\alpha_{j+3}|^{2}+|\alpha_{j-1}|^{2})\nonumber\\
&-\frac{1}{2}\sum_{j=0}^{\infty}\rho_{j}^{2}|\alpha_{j+1}|^{2}(|\alpha_{j+3}|^{2}+|\alpha_{j-1}|^{2})-\frac{1}{2}\sum_{j=0}^{\infty}\rho_{j+1}^{2}\rho_{j}^{2}|\alpha_{j+2}|^{2}(|\alpha_{j+3}|^{2}+|\alpha_{j-1}|^{2})\nonumber\\
&-\frac{1}{2}\sum_{j=0}^{\infty}\rho_{j+2}^{2}\rho_{j+1}^{2}\rho_{j}^{2}|\alpha_{j+3}-\alpha_{j-1}|^{2}\nonumber\\
&-\frac{1}{2}\sum_{j=0}^{\infty}\rho_{j+1}^{2}\rho_{j}^{2}\Big[|\alpha_{j+2}|^{4}+|\alpha_{j-1}|^{4}+|\alpha_{j+1}|^{2}|\alpha_{j-1}|^{2}+|\alpha_{j+2}|^{2}|\alpha_{j}|^{2}\nonumber\\
&+2|\alpha_{j+2}|^{2}|\alpha_{j+1}|^{2}+2|\alpha_{j}|^{2}|\alpha_{j-1}|^{2}+|\alpha_{j}|^{2}|\alpha_{j+2}-\alpha_{j+1}|^{2}+|\alpha_{j+1}|^{2}|\alpha_{j}-\alpha_{j-1}|^{2}\nonumber\\
&+|\alpha_{j+2}-\alpha_{j+1}|^{2}|\alpha_{j+2}-\alpha_{j-1}|^{2}+|\alpha_{j+2}-\alpha_{j-1}|^{2}|\alpha_{j}-\alpha_{j-1}|^{2}\Big]\nonumber\\
&-\frac{1}{2}\sum_{j=0}^{\infty}|\alpha_{j}|^{2}\Big[|\alpha_{j+2}|^{2}|\alpha_{j+2}-\alpha_{j+1}|^{2}+|\alpha_{j-1}|^{2}|\alpha_{j}-\alpha_{j-1}|^{2}\nonumber\\
&+\big(|\alpha_{j+2}|^{2}+|\alpha_{j-1}|^{2}\big)|\alpha_{j+1}-\alpha_{j}|^{2}+|\alpha_{j}|^{2}|\alpha_{j+1}-\alpha_{j-1}|^{2}\nonumber\\
&+|\alpha_{j+1}|^{2}|\alpha_{j+2}-\alpha_{j}|^{2}+\big(|\alpha_{j+2}|^{2}+2|\alpha_{j+1}|^{2}+2|\alpha_{j}|^{2}+|\alpha_{j-1}|^{2}\big)|\alpha_{j+2}-\alpha_{j-1}|^{2}\Big]\nonumber\\
&-\frac{1}{2}\sum_{j=0}^{\infty}\rho_{j}^{2}|\alpha_{j+1}|^{2}\Big[|\alpha_{j+2}|^{2}|\alpha_{j+2}-\alpha_{j+1}|^{2}+|\alpha_{j-1}|^{2}|\alpha_{j}-\alpha_{j-1}|^{2}\nonumber\\
&+\big(|\alpha_{j+2}|^{2}+|\alpha_{j-1}|^{2}\big)|\alpha_{j+1}-\alpha_{j}|^{2}+|\alpha_{j}|^{2}|\alpha_{j+1}-\alpha_{j-1}|^{2}\nonumber\\
&+|\alpha_{j+1}|^{2}|\alpha_{j+2}-\alpha_{j}|^{2}+\big(|\alpha_{j+2}|^{2}+2|\alpha_{j+1}|^{2}+2|\alpha_{j}|^{2}+|\alpha_{j-1}|^{2}\big)|\alpha_{j+2}-\alpha_{j-1}|^{2}\Big]\nonumber
\end{align}
\begin{align}
&-\frac{1}{2}\sum_{j=0}^{\infty}|\alpha_{j}|^{2}\Big[|\alpha_{j+1}|^{6}+|\alpha_{j-1}|^{6}+|\alpha_{j}|^{4}|\alpha_{j-1}|^{2}+|\alpha_{j+1}|^{2}|\alpha_{j}|^{4}\nonumber\\
&+\big(|\alpha_{j+1}|^{2}+|\alpha_{j-1}|^{2}\big)\big(|\alpha_{j+1}-\alpha_{j}|^{4}+|\alpha_{j}-\alpha_{j-1}|^{4}\big)\nonumber\\
&+|\alpha_{j-1}|^{4}|\alpha_{j+1}-\alpha_{j}|^{2}+|\alpha_{j+1}|^{4}|\alpha_{j}-\alpha_{j-1}|^{2}\nonumber\\
&+2\big(|\alpha_{j+1}|^{2}+|\alpha_{j}|^{2}\big)|\alpha_{j+1}-\alpha_{j-1}|^{2}|\alpha_{j+1}-\alpha_{j}|^{2}\nonumber\\
&+2\big(|\alpha_{j}|^{2}+|\alpha_{j-1}|^{2}\big)|\alpha_{j+1}-\alpha_{j-1}|^{2}|\alpha_{j}-\alpha_{j-1}|^{2}\Big]\nonumber\\
&-\frac{1}{2}\sum_{j=0}^{\infty}\rho_{j}^{2}\Big[\big(|\alpha_{j+1}|^{4}+2|\alpha_{j}|^{4}+|\alpha_{j-1}|^{4}+|\alpha_{j}|^{2}(|\alpha_{j+1}|^{2}+|\alpha_{j-1}|^{2})\big)|\alpha_{j+1}-\alpha_{j-1}|^{2}\nonumber\\
&+\big(2|\alpha_{j+1}|^{4}+|\alpha_{j+1}|^{2}|\alpha_{j}|^{2}+|\alpha_{j+1}|^{2}|\alpha_{j-1}|^{2}+|\alpha_{j}|^{2}|\alpha_{j-1}|^{2}\big)|\alpha_{j+1}-\alpha_{j}|^{2}\nonumber\\
&+\big(2|\alpha_{j-1}|^{4}+|\alpha_{j+1}|^{2}|\alpha_{j}|^{2}+|\alpha_{j+1}|^{2}|\alpha_{j-1}|^{2}+|\alpha_{j}|^{2}|\alpha_{j-1}|^{2}\big)|\alpha_{j}-\alpha_{j-1}|^{2}\nonumber\\
&+\big(|\alpha_{j+1}|^{2}+|\alpha_{j-1}|^{2}\big)|\alpha_{j+1}-\alpha_{j}|^{2}|\alpha_{j}-\alpha_{j-1}|^{2}\nonumber\\
&+\big(|\alpha_{j+1}-\alpha_{j}|^{4}+|\alpha_{j}-\alpha_{j-1}|^{4}\big)|\alpha_{j+1}-\alpha_{j-1}|^{2}\Big]\nonumber\\
&-\sum_{j=0}^{\infty}\big(1+2|\alpha_{j}|^{2}\big)\big(|\alpha_{j+1}|^{2}+|\alpha_{j-1}|^{2}\big)|\alpha_{j+1}-\alpha_{j-1}|^{2}\nonumber\\
&-\frac{1}{4}\sum_{j=0}^{\infty}\Big(|\alpha_{j+1}|^{4}+|\alpha_{j-1}|^{4}+|\alpha_{j+1}-\alpha_{j-1}|^{4}\Big)\nonumber\\
&-\frac{3}{4}\sum_{j=0}^{\infty}|\alpha_{j}|^{4}\Big(|\alpha_{j+1}|^{4}+|\alpha_{j-1}|^{4}+|\alpha_{j+1}-\alpha_{j-1}|^{4}\Big)\nonumber\\
&-\frac{1}{8}\sum_{j=0}^{\infty}\Big[|\alpha_{j}|^{8}+|\alpha_{j-1}|^{8}+|\alpha_{j}-\alpha_{j-1}|^{8}+4\big(|\alpha_{j}|^{2}+|\alpha_{j-1}|^{2}\big)^{2}|\alpha_{j}-\alpha_{j-1}|^{4}\nonumber\\
&+2\big(|\alpha_{j}|^{4}+|\alpha_{j-1}|^{4}\big)|\alpha_{j}-\alpha_{j-1}|^{4}\Big]\nonumber\\
&+\frac{1}{2}\sum_{j=0}^{\infty}\Big[|\alpha_{j+2}|^{2}|\alpha_{j+2}-\alpha_{j+1}|^{2}+|\alpha_{j-1}|^{2}|\alpha_{j}-\alpha_{j-1}|^{2}\nonumber\\
&+\big(|\alpha_{j+2}|^{2}+|\alpha_{j-1}|^{2}\big)|\alpha_{j+1}-\alpha_{j}|^{2}+|\alpha_{j}|^{2}|\alpha_{j+1}-\alpha_{j-1}|^{2}\nonumber\\
&+|\alpha_{j+1}|^{2}|\alpha_{j+2}-\alpha_{j}|^{2}+\big(|\alpha_{j+2}|^{2}+2|\alpha_{j+1}|^{2}+2|\alpha_{j}|^{2}+|\alpha_{j-1}|^{2}\big)|\alpha_{j+2}-\alpha_{j-1}|^{2}\Big]\nonumber\\
&+\frac{1}{2}\sum_{j=0}^{\infty}\Big[|\alpha_{j+1}|^{6}+|\alpha_{j-1}|^{6}+|\alpha_{j}|^{4}|\alpha_{j-1}|^{2}+|\alpha_{j+1}|^{2}|\alpha_{j}|^{4}\nonumber\\
&+\big(|\alpha_{j+1}|^{2}+|\alpha_{j-1}|^{2}\big)\big(|\alpha_{j+1}-\alpha_{j}|^{4}+|\alpha_{j}-\alpha_{j-1}|^{4}\big)\nonumber\\
&+|\alpha_{j-1}|^{4}|\alpha_{j+1}-\alpha_{j}|^{2}+|\alpha_{j+1}|^{4}|\alpha_{j}-\alpha_{j-1}|^{2}\nonumber\\
&+2\big(|\alpha_{j+1}|^{2}+|\alpha_{j}|^{2}\big)|\alpha_{j+1}-\alpha_{j-1}|^{2}|\alpha_{j+1}-\alpha_{j}|^{2}\nonumber\\
&+2\big(|\alpha_{j}|^{2}+|\alpha_{j-1}|^{2}\big)|\alpha_{j+1}-\alpha_{j-1}|^{2}|\alpha_{j}-\alpha_{j-1}|^{2}\Big]\nonumber\\
&+\sum_{j=0}^{\infty}|\alpha_{j}|^{2}\Big(|\alpha_{j+1}|^{4}+|\alpha_{j-1}|^{4}+|\alpha_{j+1}-\alpha_{j-1}|^{4}\Big)\nonumber
\end{align}
\begin{align}
&+\frac{3}{2}\sum_{j=0}^{\infty}\big(1+|\alpha_{j}|^{4}\big)\big(|\alpha_{j+1}|^{2}+|\alpha_{j-1}|^{2}\big)|\alpha_{j+1}-\alpha_{j-1}|^{2}\nonumber\\
&+\frac{1}{2}\sum_{j=0}^{\infty}\Big[\big(|\alpha_{j}|^{2}+|\alpha_{j-1}|^{2}\big)|\alpha_{j}-\alpha_{j-1}|^{6}\nonumber\\
&+\big(|\alpha_{j}|^{6}+|\alpha_{j-1}|^{6}+|\alpha_{j}|^{4}|\alpha_{j-1}|^{2}+|\alpha_{j}|^{2}|\alpha_{j-1}|^{4}\big)|\alpha_{j}-\alpha_{j-1}|^{2}\Big].
\end{align}
\end{thm}

\begin{proof}
Since
\begin{equation}
1-\cos4\theta=1-\frac{e^{i4\theta}+e^{-i4\theta}}{2},
\end{equation}
then
\begin{equation}
Z_{4,1}(\mu)=w_{0}-\mathrm{Re}(w_{4}).
\end{equation}
By (4.50)-(4.52) and (4.132), we have
\begin{align}
\mathrm{Re}(\alpha_{j+3}\overline{\alpha}_{j-1})=\frac{1}{2}(|\alpha_{j+3}|^{2}+|\alpha_{j-1}|^{2}-|\alpha_{j+3}-\alpha_{j-1}|^{2}),
\end{align}
\begin{align}
&\mathrm{Re}(\alpha_{j+2}^{2}\overline{\alpha}_{j+1}\overline{\alpha}_{j-1})=2\mathrm{Re}(\alpha_{j+2}\overline{\alpha}_{j+1})\mathrm{Re}(\alpha_{j+2}\overline{\alpha}_{j-1})
-|\alpha_{j+2}|^{2}\mathrm{Re}(\alpha_{j+1}\overline{\alpha}_{j-1})\nonumber\\
=&\frac{1}{2}\Big[|\alpha_{j+2}|^{4}+|\alpha_{j+1}|^{2}|\alpha_{j-1}|^{2}-\big(|\alpha_{j+2}|^{2}+|\alpha_{j-1}|^{2}\big)|\alpha_{j+2}-\alpha_{j+1}|^{2}\nonumber\\
&+|\alpha_{j+2}-\alpha_{j+1}|^{2}|\alpha_{j+2}-\alpha_{j-1}|^{2}+|\alpha_{j+2}|^{2}|\alpha_{j+1}-\alpha_{j-1}|^{2}\nonumber\\
&-\big(|\alpha_{j+2}|^{2}+|\alpha_{j+1}|^{2}\big)|\alpha_{j+2}-\alpha_{j-1}|^{2}\Big],
\end{align}
\begin{align}
&\mathrm{Re}(\alpha_{j+2}\alpha_{j+1}\overline{\alpha}_{j}\overline{\alpha}_{j-1})=\mathrm{Re}(\alpha_{j+2}\overline{\alpha}_{j})\mathrm{Re}(\alpha_{j+1}\overline{\alpha}_{j-1})
-\mathrm{Re}(\alpha_{j+2}\overline{\alpha}_{j+1})\mathrm{Re}(\overline{\alpha}_{j}\alpha_{j-1})\nonumber\\
&+\mathrm{Re}(\alpha_{j+2}\overline{\alpha}_{j-1})\mathrm{Re}(\overline{\alpha}_{j+1}\alpha_{j})\nonumber\\
=&\frac{1}{2}\Big(|\alpha_{j+2}|^{2}|\alpha_{j+1}|^{2}+|\alpha_{j}|^{2}|\alpha_{j-1}|^{2}\Big)-\frac{1}{4}\Big[\big(|\alpha_{j+1}|^{2}+|\alpha_{j-1}|^{2}\big)|\alpha_{j+2}-\alpha_{j}|^{2}\nonumber\\
&+\big(|\alpha_{j+2}|^{2}+|\alpha_{j}|^{2}\big)|\alpha_{j+1}-\alpha_{j-1}|^{2}-\big(|\alpha_{j}|^{2}+|\alpha_{j-1}|^{2}\big)|\alpha_{j+2}-\alpha_{j+1}|^{2}\nonumber\\
&-\big(|\alpha_{j+2}|^{2}+|\alpha_{j+1}|^{2}\big)|\alpha_{j}-\alpha_{j-1}|^{2}+\big(|\alpha_{j+1}|^{2}+|\alpha_{j}|^{2}\big)|\alpha_{j+2}-\alpha_{j-1}|^{2}\nonumber\\
&+\big(|\alpha_{j+2}|^{2}+|\alpha_{j-1}|^{2}\big)|\alpha_{j+1}-\alpha_{j}|^{2}\Big],
\end{align}

\begin{align}
&\mathrm{Re}(\alpha_{j+2}\alpha_{j}\overline{\alpha}_{j-1}^{2})=2\mathrm{Re}(\alpha_{j+2}\overline{\alpha}_{j-1})\mathrm{Re}(\alpha_{j}\overline{\alpha}_{j-1})
-|\alpha_{j-1}|^{2}\mathrm{Re}(\alpha_{j+2}\overline{\alpha}_{j})\nonumber\\
=&\frac{1}{2}\Big[|\alpha_{j-1}|^{4}+|\alpha_{j+2}|^{2}|\alpha_{j}|^{2}-\big(|\alpha_{j}|^{2}+|\alpha_{j-1}|^{2}\big)|\alpha_{j+2}-\alpha_{j-1}|^{2}\nonumber\\
&+|\alpha_{j+2}-\alpha_{j-1}|^{2}|\alpha_{j}-\alpha_{j-1}|^{2}+|\alpha_{j-1}|^{2}|\alpha_{j+2}-\alpha_{j}|^{2}\nonumber\\
&-\big(|\alpha_{j+2}|^{2}+|\alpha_{j-1}|^{2}\big)|\alpha_{j}-\alpha_{j-1}|^{2}\Big],
\end{align}
\begin{align}
&\mathrm{Re}(\alpha_{j+1}^{3}\overline{\alpha}_{j}^{2}\overline{\alpha}_{j-1})=2\mathrm{Re}(\alpha_{j+1}\overline{\alpha}_{j-1})\mathrm{Re}(\alpha_{j+1}^{2}\overline{\alpha}_{j}^{2})
-|\alpha_{j+1}|^{2}\mathrm{Re}(\alpha_{j}^{2}\overline{\alpha}_{j+1}\overline{\alpha}_{j-1})\nonumber\\
=&2\mathrm{Re}(\alpha_{j+1}\overline{\alpha}_{j-1})\Big[2\mathrm{Re}(\alpha_{j+1}\overline{\alpha}_{j})^{2}-|\alpha_{j+1}|^{2}|\alpha_{j}|^{2}\Big]\nonumber\\
&-|\alpha_{j+1}|^{2}\Big[2\mathrm{Re}(\alpha_{j}\overline{\alpha}_{j+1})\mathrm{Re}(\alpha_{j}\overline{\alpha}_{j-1})-|\alpha_{j}|^{2}\mathrm{Re}(\alpha_{j+1}\overline{\alpha}_{j-1})\Big]\nonumber
\end{align}
\begin{align}
=&\frac{1}{2}\Big[|\alpha_{j+1}|^{6}+|\alpha_{j}|^{4}|\alpha_{j-1}|^{2}\nonumber\\
&-\big(|\alpha_{j+1}|^{4}+|\alpha_{j}|^{4}\big)|\alpha_{j+1}-\alpha_{j-1}|^{2}+\big(|\alpha_{j+1}|^{2}+|\alpha_{j-1}|^{2}\big)|\alpha_{j+1}-\alpha_{j}|^{4}\nonumber\\
&-\big(2|\alpha_{j+1}|^{4}+|\alpha_{j+1}|^{2}|\alpha_{j}|^{2}+|\alpha_{j+1}|^{2}|\alpha_{j-1}|^{2}+2|\alpha_{j}|^{2}|\alpha_{j-1}|^{2}\big)|\alpha_{j+1}-\alpha_{j}|^{2}\nonumber\\
&+2\big(|\alpha_{j+1}|^{2}+|\alpha_{j}|^{2}\big)|\alpha_{j+1}-\alpha_{j-1}|^{2}|\alpha_{j+1}-\alpha_{j}|^{2}\nonumber\\
&-|\alpha_{j+1}-\alpha_{j}|^{4}|\alpha_{j+1}-\alpha_{j-1}|^{2}\Big]\nonumber\\
&-\frac{1}{2}\Big[|\alpha_{j+1}|^{2}|\alpha_{j+1}-\alpha_{j}|^{2}|\alpha_{j}-\alpha_{j-1}|^{2}+|\alpha_{j+1}|^{2}|\alpha_{j}|^{2}|\alpha_{j+1}-\alpha_{j-1}|^{2}\nonumber\\
&-(|\alpha_{j+1}|^{4}+|\alpha_{j+1}|^{2}|\alpha_{j}|^{2})|\alpha_{j}-\alpha_{j-1}|^{2}\Big],
\end{align}
\begin{align}
&\mathrm{Re}(\alpha_{j+1}\alpha_{j}^{2}\overline{\alpha}_{j-1}^{3})=2\mathrm{Re}(\alpha_{j+1}\overline{\alpha}_{j-1})\mathrm{Re}(\alpha_{j}^{2}\overline{\alpha}_{j-1}^{2})
-|\alpha_{j-1}|^{2}\mathrm{Re}(\alpha_{j}^{2}\overline{\alpha}_{j+1}\overline{\alpha}_{j-1})\nonumber\\
=&2\mathrm{Re}(\alpha_{j+1}\overline{\alpha}_{j-1})\Big[2\mathrm{Re}(\alpha_{j}\overline{\alpha}_{j-1})^{2}-|\alpha_{j}|^{2}|\alpha_{j-1}|^{2}\Big]\nonumber\\
&-|\alpha_{j-1}|^{2}\Big[2\mathrm{Re}(\alpha_{j}\overline{\alpha}_{j+1})\mathrm{Re}(\alpha_{j}\overline{\alpha}_{j-1})-|\alpha_{j}|^{2}\mathrm{Re}(\alpha_{j+1}\overline{\alpha}_{j-1})\Big]\nonumber\\
=&\frac{1}{2}\Big[|\alpha_{j-1}|^{6}+|\alpha_{j+1}|^{2}|\alpha_{j}|^{4}\nonumber\\
&-\big(|\alpha_{j}|^{4}+|\alpha_{j-1}|^{4}\big)|\alpha_{j+1}-\alpha_{j-1}|^{2}+\big(|\alpha_{j+1}|^{2}+|\alpha_{j-1}|^{2}\big)|\alpha_{j}-\alpha_{j-1}|^{4}\nonumber\\
&-\big(2|\alpha_{j-1}|^{4}+2|\alpha_{j+1}|^{2}|\alpha_{j}|^{2}+|\alpha_{j+1}|^{2}|\alpha_{j-1}|^{2}+|\alpha_{j}|^{2}|\alpha_{j-1}|^{2}\big)|\alpha_{j}-\alpha_{j-1}|^{2}\nonumber\\
&+2\big(|\alpha_{j}|^{2}+|\alpha_{j-1}|^{2}\big)|\alpha_{j+1}-\alpha_{j-1}|^{2}|\alpha_{j}-\alpha_{j-1}|^{2}\nonumber\\
&-|\alpha_{j}-\alpha_{j-1}|^{4}|\alpha_{j+1}-\alpha_{j-1}|^{2}\Big]\nonumber\\
&-\frac{1}{2}\Big[|\alpha_{j-1}|^{2}|\alpha_{j+1}-\alpha_{j}|^{2}|\alpha_{j}-\alpha_{j-1}|^{2}+|\alpha_{j}|^{2}|\alpha_{j-1}|^{2}|\alpha_{j+1}-\alpha_{j-1}|^{2}\nonumber\\
&-(|\alpha_{j-1}|^{4}+|\alpha_{j}|^{2}|\alpha_{j-1}|^{2})|\alpha_{j+1}-\alpha_{j}|^{2}\Big],
\end{align}
\begin{align}
&\mathrm{Re}(\alpha_{j+1}^{2}\overline{\alpha}_{j-1}^{2})=2\mathrm{Re}(\alpha_{j+1}\overline{\alpha}_{j-1})^{2}
-|\alpha_{j+1}|^{2}|\alpha_{j-1}|^{2}\nonumber\\
=&\frac{1}{2}(|\alpha_{j+1}|^{2}+|\alpha_{j-1}|^{2}-|\alpha_{j+1}-\alpha_{j-1}|^{2})^{2}-|\alpha_{j+1}|^{2}|\alpha_{j-1}|^{2}\nonumber\\
=&\frac{1}{2}\Big[|\alpha_{j+1}|^{4}+|\alpha_{j-1}|^{4}+|\alpha_{j+1}-\alpha_{j-1}|^{4}-2\big(|\alpha_{j+1}|^{2}+|\alpha_{j-1}|^{2}\big)|\alpha_{j+1}-\alpha_{j-1}|^{2}\Big]
\end{align}
and
\begin{align}
&\mathrm{Re}(\alpha_{j}^{4}\overline{\alpha}_{j-1}^{4})=2\mathrm{Re}(\alpha_{j}^{2}\overline{\alpha}_{j-1}^{2})^{2}
-|\alpha_{j}|^{4}|\alpha_{j-1}|^{4}\nonumber\\
=&2\Big[2\mathrm{Re}(\alpha_{j}\overline{\alpha}_{j-1})^{2}-|\alpha_{j}|^{2}|\alpha_{j-1}|^{2}\Big]^{2}-|\alpha_{j}|^{4}|\alpha_{j-1}|^{4}\nonumber\\
=&\frac{1}{2}\Big[|\alpha_{j}|^{8}+|\alpha_{j-1}|^{8}+|\alpha_{j}-\alpha_{j-1}|^{8}+4\big(|\alpha_{j}|^{2}|+\alpha_{j-1}|^{2}\big)^{2}|\alpha_{j}-\alpha_{j-1}|^{4}\nonumber\\
&+2\big(|\alpha_{j}|^{4}+|\alpha_{j-1}|^{4}\big)|\alpha_{j}-\alpha_{j-1}|^{4}-4\big(|\alpha_{j}|^{2}+|\alpha_{j-1}|^{2}\big)|\alpha_{j}-\alpha_{j-1}|^{6}\nonumber\\
&-4\big(|\alpha_{j}|^{6}+|\alpha_{j-1}|^{6}+|\alpha_{j}|^{4}|\alpha_{j-1}|^{2}+|\alpha_{j}|^{2}|\alpha_{j-1}|^{4}\big)|\alpha_{j}-\alpha_{j-1}|^{2}\Big].
\end{align}

Thus by (3.39), (4.5), (4.135)-(4.143), we obtain
\begin{align}
&Z_{4,1}(\mu)=\int_{0}^{2\pi}(1-\cos4\theta)\frac{d\theta}{2\pi}\nonumber\\
=&\sum_{j=0}^{\infty}\log(1-|\alpha_{j}|^{2})+\frac{1}{2}\sum_{j=0}^{\infty}\rho_{j+2}^{2}\rho_{j+1}^{2}\rho_{j}^{2}(|\alpha_{j+3}|^{2}+|\alpha_{j-1}|^{2}-|\alpha_{j+3}-\alpha_{j-1}|^{2})\nonumber\\
&-\frac{1}{2}\sum_{j=0}^{\infty}\rho_{j+1}^{2}\rho_{j}^{2}\Big[|\alpha_{j+2}|^{4}+|\alpha_{j+1}|^{2}|\alpha_{j-1}|^{2}-\big(|\alpha_{j+2}|^{2}+|\alpha_{j-1}|^{2}\big)|\alpha_{j+2}-\alpha_{j+1}|^{2}\nonumber\\
&+|\alpha_{j+2}-\alpha_{j+1}|^{2}|\alpha_{j+2}-\alpha_{j-1}|^{2}+|\alpha_{j+2}|^{2}|\alpha_{j+1}-\alpha_{j-1}|^{2}\nonumber\\
&-\big(|\alpha_{j+2}|^{2}+|\alpha_{j+1}|^{2}\big)|\alpha_{j+2}-\alpha_{j-1}|^{2}\Big]\nonumber\\
&-\sum_{j=0}^{\infty}\rho_{j+1}^{2}\rho_{j}^{2}\Big\{\Big(|\alpha_{j+2}|^{2}|\alpha_{j+1}|^{2}+|\alpha_{j}|^{2}|\alpha_{j-1}|^{2}\Big)-\frac{1}{2}\Big[\big(|\alpha_{j+1}|^{2}+|\alpha_{j-1}|^{2}\big)|\alpha_{j+2}-\alpha_{j}|^{2}\nonumber\\
&+\big(|\alpha_{j+2}|^{2}+|\alpha_{j}|^{2}\big)|\alpha_{j+1}-\alpha_{j-1}|^{2}-\big(|\alpha_{j}|^{2}+|\alpha_{j-1}|^{2}\big)|\alpha_{j+2}-\alpha_{j+1}|^{2}\nonumber\\
&-\big(|\alpha_{j+2}|^{2}+|\alpha_{j+1}|^{2}\big)|\alpha_{j}-\alpha_{j-1}|^{2}+\big(|\alpha_{j+1}|^{2}+|\alpha_{j}|^{2}\big)|\alpha_{j+2}-\alpha_{j-1}|^{2}\nonumber\\
&+\big(|\alpha_{j+2}|^{2}+|\alpha_{j-1}|^{2}\big)|\alpha_{j+1}-\alpha_{j}|^{2}\Big]\Big\}\nonumber\\
&-\frac{1}{2}\sum_{j=0}^{\infty}\rho_{j+1}^{2}\rho_{j}^{2}\Big[|\alpha_{j-1}|^{4}+|\alpha_{j+2}|^{2}|\alpha_{j}|^{2}-\big(|\alpha_{j}|^{2}+|\alpha_{j-1}|^{2}\big)|\alpha_{j+2}-\alpha_{j-1}|^{2}\nonumber\\
&+|\alpha_{j+2}-\alpha_{j-1}|^{2}|\alpha_{j}-\alpha_{j-1}|^{2}+|\alpha_{j-1}|^{2}|\alpha_{j+2}-\alpha_{j}|^{2}\nonumber\\
&-\big(|\alpha_{j+2}|^{2}+|\alpha_{j-1}|^{2}\big)|\alpha_{j}-\alpha_{j-1}|^{2}\Big]\nonumber\\
&+\frac{1}{2}\sum_{j=0}^{\infty}\rho_{j}^{2}\Big[|\alpha_{j+1}|^{6}+|\alpha_{j}|^{4}|\alpha_{j-1}|^{2}\nonumber\\
&-\big(|\alpha_{j+1}|^{4}+|\alpha_{j}|^{4}\big)|\alpha_{j+1}-\alpha_{j-1}|^{2}+\big(|\alpha_{j+1}|^{2}+|\alpha_{j-1}|^{2}\big)|\alpha_{j+1}-\alpha_{j}|^{4}\nonumber\\
&-\big(2|\alpha_{j+1}|^{4}+|\alpha_{j+1}|^{2}|\alpha_{j}|^{2}+|\alpha_{j+1}|^{2}|\alpha_{j-1}|^{2}+2|\alpha_{j}|^{2}|\alpha_{j-1}|^{2}\big)|\alpha_{j+1}-\alpha_{j}|^{2}\nonumber\\
&+2\big(|\alpha_{j+1}|^{2}+|\alpha_{j}|^{2}\big)|\alpha_{j+1}-\alpha_{j-1}|^{2}|\alpha_{j+1}-\alpha_{j}|^{2}\nonumber\\
&-|\alpha_{j+1}-\alpha_{j}|^{4}|\alpha_{j+1}-\alpha_{j-1}|^{2}\Big]\nonumber\\
&-\frac{1}{2}\sum_{j=0}^{\infty}\rho_{j}^{2}\Big[|\alpha_{j+1}|^{2}|\alpha_{j+1}-\alpha_{j}|^{2}|\alpha_{j}-\alpha_{j-1}|^{2}+|\alpha_{j+1}|^{2}|\alpha_{j}|^{2}|\alpha_{j+1}-\alpha_{j-1}|^{2}\nonumber\\
&-(|\alpha_{j+1}|^{4}+|\alpha_{j+1}|^{2}|\alpha_{j}|^{2})|\alpha_{j}-\alpha_{j-1}|^{2}\Big]\nonumber\\
&+\frac{1}{2}\sum_{j=0}^{\infty}\rho_{j}^{2}\Big[|\alpha_{j-1}|^{6}+|\alpha_{j+1}|^{2}|\alpha_{j}|^{4}\nonumber\\
&-\big(|\alpha_{j}|^{4}+|\alpha_{j-1}|^{4}\big)|\alpha_{j+1}-\alpha_{j-1}|^{2}+\big(|\alpha_{j+1}|^{2}+|\alpha_{j-1}|^{2}\big)|\alpha_{j}-\alpha_{j-1}|^{4}\nonumber\\
&-\big(2|\alpha_{j-1}|^{4}+2|\alpha_{j+1}|^{2}|\alpha_{j}|^{2}+|\alpha_{j+1}|^{2}|\alpha_{j-1}|^{2}+|\alpha_{j}|^{2}|\alpha_{j-1}|^{2}\big)|\alpha_{j}-\alpha_{j-1}|^{2}\nonumber\\
&+2\big(|\alpha_{j}|^{2}+|\alpha_{j-1}|^{2}\big)|\alpha_{j+1}-\alpha_{j-1}|^{2}|\alpha_{j}-\alpha_{j-1}|^{2}\nonumber\\
&-|\alpha_{j}-\alpha_{j-1}|^{4}|\alpha_{j+1}-\alpha_{j-1}|^{2}\Big]\nonumber
\end{align}
\begin{align}
&-\frac{1}{2}\sum_{j=0}^{\infty}\rho_{j}^{2}\Big[|\alpha_{j-1}|^{2}|\alpha_{j+1}-\alpha_{j}|^{2}|\alpha_{j}-\alpha_{j-1}|^{2}+|\alpha_{j}|^{2}|\alpha_{j-1}|^{2}|\alpha_{j+1}-\alpha_{j-1}|^{2}\nonumber\\
&-(|\alpha_{j-1}|^{4}+|\alpha_{j}|^{2}|\alpha_{j-1}|^{2})|\alpha_{j+1}-\alpha_{j}|^{2}\Big]\nonumber\\
&+\frac{1}{2}\sum_{j=0}^{\infty}\rho_{j}^{2}\Big[|\alpha_{j+1}|^{4}+|\alpha_{j-1}|^{4}+|\alpha_{j+1}-\alpha_{j-1}|^{4}-2\big(|\alpha_{j+1}|^{2}+|\alpha_{j-1}|^{2}\big)|\alpha_{j+1}-\alpha_{j-1}|^{2}\Big]\nonumber\\
&-\frac{3}{4}\sum_{j=0}^{\infty}\rho_{j}^{4}\Big[|\alpha_{j+1}|^{4}+|\alpha_{j-1}|^{4}+|\alpha_{j+1}-\alpha_{j-1}|^{4}-2\big(|\alpha_{j+1}|^{2}+|\alpha_{j-1}|^{2}\big)|\alpha_{j+1}-\alpha_{j-1}|^{2}\Big]\nonumber\\
&-\frac{1}{8}\sum_{j=0}^{\infty}\Big[|\alpha_{j}|^{8}+|\alpha_{j-1}|^{8}+|\alpha_{j}-\alpha_{j-1}|^{8}+4\big(|\alpha_{j}|^{2}+|\alpha_{j-1}|^{2}\big)^{2}|\alpha_{j}-\alpha_{j-1}|^{4}\nonumber\\
&+2\big(|\alpha_{j}|^{4}+|\alpha_{j-1}|^{4}\big)|\alpha_{j}-\alpha_{j-1}|^{4}-4\big(|\alpha_{j}|^{2}+|\alpha_{j-1}|^{2}\big)|\alpha_{j}-\alpha_{j-1}|^{6}\nonumber\\
&-4\big(|\alpha_{j}|^{6}+|\alpha_{j-1}|^{6}+|\alpha_{j}|^{4}|\alpha_{j-1}|^{2}+|\alpha_{j}|^{2}|\alpha_{j-1}|^{4}\big)|\alpha_{j}-\alpha_{j-1}|^{2}\Big]\nonumber\tag{4.133a}\\
=&1-\sum_{j=0}^{2}|\alpha_{j}|^{2}+\sum_{j=0}^{\infty}\Big(\log(1-|\alpha_{j}|^{2})+|\alpha_{j}|^{2}\Big)-\frac{1}{2}\sum_{j=0}^{\infty}|\alpha_{j}|^{2}(|\alpha_{j+3}|^{2}+|\alpha_{j-1}|^{2})\nonumber\\
&-\frac{1}{2}\sum_{j=0}^{\infty}\rho_{j}^{2}|\alpha_{j+1}|^{2}(|\alpha_{j+3}|^{2}+|\alpha_{j-1}|^{2})-\frac{1}{2}\sum_{j=0}^{\infty}\rho_{j+1}^{2}\rho_{j}^{2}|\alpha_{j+2}|^{2}(|\alpha_{j+3}|^{2}+|\alpha_{j-1}|^{2})\nonumber\\
&-\frac{1}{2}\sum_{j=0}^{\infty}\rho_{j+2}^{2}\rho_{j+1}^{2}\rho_{j}^{2}|\alpha_{j+3}-\alpha_{j-1}|^{2}\nonumber\\
&-\frac{1}{2}\sum_{j=0}^{\infty}\rho_{j+1}^{2}\rho_{j}^{2}\Big[|\alpha_{j+2}|^{4}+|\alpha_{j-1}|^{4}+|\alpha_{j+1}|^{2}|\alpha_{j-1}|^{2}+|\alpha_{j+2}|^{2}|\alpha_{j}|^{2}\nonumber\\
&+2|\alpha_{j+2}|^{2}|\alpha_{j+1}|^{2}+2|\alpha_{j}|^{2}|\alpha_{j-1}|^{2}+|\alpha_{j}|^{2}|\alpha_{j+2}-\alpha_{j+1}|^{2}+|\alpha_{j+1}|^{2}|\alpha_{j}-\alpha_{j-1}|^{2}\nonumber\\
&+|\alpha_{j+2}-\alpha_{j+1}|^{2}|\alpha_{j+2}-\alpha_{j-1}|^{2}+|\alpha_{j+2}-\alpha_{j-1}|^{2}|\alpha_{j}-\alpha_{j-1}|^{2}\Big]\nonumber\\
&-\frac{1}{2}\sum_{j=0}^{\infty}|\alpha_{j}|^{2}\Big[|\alpha_{j+2}|^{2}|\alpha_{j+2}-\alpha_{j+1}|^{2}+|\alpha_{j-1}|^{2}|\alpha_{j}-\alpha_{j-1}|^{2}\nonumber\\
&+\big(|\alpha_{j+2}|^{2}+|\alpha_{j-1}|^{2}\big)|\alpha_{j+1}-\alpha_{j}|^{2}+|\alpha_{j}|^{2}|\alpha_{j+1}-\alpha_{j-1}|^{2}\nonumber\\
&+|\alpha_{j+1}|^{2}|\alpha_{j+2}-\alpha_{j}|^{2}+\big(|\alpha_{j+2}|^{2}+2|\alpha_{j+1}|^{2}+2|\alpha_{j}|^{2}+|\alpha_{j-1}|^{2}\big)|\alpha_{j+2}-\alpha_{j-1}|^{2}\Big]\nonumber\\
&-\frac{1}{2}\sum_{j=0}^{\infty}\rho_{j}^{2}|\alpha_{j+1}|^{2}\Big[|\alpha_{j+2}|^{2}|\alpha_{j+2}-\alpha_{j+1}|^{2}+|\alpha_{j-1}|^{2}|\alpha_{j}-\alpha_{j-1}|^{2}\nonumber\\
&+\big(|\alpha_{j+2}|^{2}+|\alpha_{j-1}|^{2}\big)|\alpha_{j+1}-\alpha_{j}|^{2}+|\alpha_{j}|^{2}|\alpha_{j+1}-\alpha_{j-1}|^{2}\nonumber\\
&+|\alpha_{j+1}|^{2}|\alpha_{j+2}-\alpha_{j}|^{2}+\big(|\alpha_{j+2}|^{2}+2|\alpha_{j+1}|^{2}+2|\alpha_{j}|^{2}+|\alpha_{j-1}|^{2}\big)|\alpha_{j+2}-\alpha_{j-1}|^{2}\Big]\nonumber\\
&-\frac{1}{2}\sum_{j=0}^{\infty}|\alpha_{j}|^{2}\Big[|\alpha_{j+1}|^{6}+|\alpha_{j-1}|^{6}+|\alpha_{j}|^{4}|\alpha_{j-1}|^{2}+|\alpha_{j+1}|^{2}|\alpha_{j}|^{4}\nonumber\\
&+\big(|\alpha_{j+1}|^{2}+|\alpha_{j-1}|^{2}\big)\big(|\alpha_{j+1}-\alpha_{j}|^{4}+|\alpha_{j}-\alpha_{j-1}|^{4}\big)\nonumber\\
&+|\alpha_{j-1}|^{4}|\alpha_{j+1}-\alpha_{j}|^{2}+|\alpha_{j+1}|^{4}|\alpha_{j}-\alpha_{j-1}|^{2}\nonumber\\
&+2\big(|\alpha_{j+1}|^{2}+|\alpha_{j}|^{2}\big)|\alpha_{j+1}-\alpha_{j-1}|^{2}|\alpha_{j+1}-\alpha_{j}|^{2}\nonumber
\end{align}
\begin{align}
&+2\big(|\alpha_{j}|^{2}+|\alpha_{j-1}|^{2}\big)|\alpha_{j+1}-\alpha_{j-1}|^{2}|\alpha_{j}-\alpha_{j-1}|^{2}\Big]\nonumber\\
&-\frac{1}{2}\sum_{j=0}^{\infty}\rho_{j}^{2}\Big[\big(|\alpha_{j+1}|^{4}+2|\alpha_{j}|^{4}+|\alpha_{j-1}|^{4}+|\alpha_{j}|^{2}(|\alpha_{j+1}|^{2}+|\alpha_{j-1}|^{2})\big)|\alpha_{j+1}-\alpha_{j-1}|^{2}\nonumber\\
&+\big(2|\alpha_{j+1}|^{4}+|\alpha_{j+1}|^{2}|\alpha_{j}|^{2}+|\alpha_{j+1}|^{2}|\alpha_{j-1}|^{2}+|\alpha_{j}|^{2}|\alpha_{j-1}|^{2}\big)|\alpha_{j+1}-\alpha_{j}|^{2}\nonumber\\
&+\big(2|\alpha_{j-1}|^{4}+|\alpha_{j+1}|^{2}|\alpha_{j}|^{2}+|\alpha_{j+1}|^{2}|\alpha_{j-1}|^{2}+|\alpha_{j}|^{2}|\alpha_{j-1}|^{2}\big)|\alpha_{j}-\alpha_{j-1}|^{2}\nonumber\\
&+\big(|\alpha_{j+1}|^{2}+|\alpha_{j-1}|^{2}\big)|\alpha_{j+1}-\alpha_{j}|^{2}|\alpha_{j}-\alpha_{j-1}|^{2}\nonumber\\
&+\big(|\alpha_{j+1}-\alpha_{j}|^{4}+|\alpha_{j}-\alpha_{j-1}|^{4}\big)|\alpha_{j+1}-\alpha_{j-1}|^{2}\Big]\nonumber\\
&-\sum_{j=0}^{\infty}\big(1+2|\alpha_{j}|^{2}\big)\big(|\alpha_{j+1}|^{2}+|\alpha_{j-1}|^{2}\big)|\alpha_{j+1}-\alpha_{j-1}|^{2}\nonumber\\
&-\frac{1}{4}\sum_{j=0}^{\infty}\Big(|\alpha_{j+1}|^{4}+|\alpha_{j-1}|^{4}+|\alpha_{j+1}-\alpha_{j-1}|^{4}\Big)\nonumber\\
&-\frac{3}{4}\sum_{j=0}^{\infty}|\alpha_{j}|^{4}\Big(|\alpha_{j+1}|^{4}+|\alpha_{j-1}|^{4}+|\alpha_{j+1}-\alpha_{j-1}|^{4}\Big)\nonumber\\
&-\frac{1}{8}\sum_{j=0}^{\infty}\Big[|\alpha_{j}|^{8}+|\alpha_{j-1}|^{8}+|\alpha_{j}-\alpha_{j-1}|^{8}+4\big(|\alpha_{j}|^{2}+|\alpha_{j-1}|^{2}\big)^{2}|\alpha_{j}-\alpha_{j-1}|^{4}\nonumber\\
&+2\big(|\alpha_{j}|^{4}+|\alpha_{j-1}|^{4}\big)|\alpha_{j}-\alpha_{j-1}|^{4}\Big]\nonumber\\
&+\frac{1}{2}\sum_{j=0}^{\infty}\Big[|\alpha_{j+2}|^{2}|\alpha_{j+2}-\alpha_{j+1}|^{2}+|\alpha_{j-1}|^{2}|\alpha_{j}-\alpha_{j-1}|^{2}\nonumber\\
&+\big(|\alpha_{j+2}|^{2}+|\alpha_{j-1}|^{2}\big)|\alpha_{j+1}-\alpha_{j}|^{2}+|\alpha_{j}|^{2}|\alpha_{j+1}-\alpha_{j-1}|^{2}\nonumber\\
&+|\alpha_{j+1}|^{2}|\alpha_{j+2}-\alpha_{j}|^{2}+\big(|\alpha_{j+2}|^{2}+2|\alpha_{j+1}|^{2}+2|\alpha_{j}|^{2}+|\alpha_{j-1}|^{2}\big)|\alpha_{j+2}-\alpha_{j-1}|^{2}\Big]\nonumber\\
&+\frac{1}{2}\sum_{j=0}^{\infty}\Big[|\alpha_{j+1}|^{6}+|\alpha_{j-1}|^{6}+|\alpha_{j}|^{4}|\alpha_{j-1}|^{2}+|\alpha_{j+1}|^{2}|\alpha_{j}|^{4}\nonumber\\
&+\big(|\alpha_{j+1}|^{2}+|\alpha_{j-1}|^{2}\big)\big(|\alpha_{j+1}-\alpha_{j}|^{4}+|\alpha_{j}-\alpha_{j-1}|^{4}\big)\nonumber\\
&+|\alpha_{j-1}|^{4}|\alpha_{j+1}-\alpha_{j}|^{2}+|\alpha_{j+1}|^{4}|\alpha_{j}-\alpha_{j-1}|^{2}\nonumber\\
&+2\big(|\alpha_{j+1}|^{2}+|\alpha_{j}|^{2}\big)|\alpha_{j+1}-\alpha_{j-1}|^{2}|\alpha_{j+1}-\alpha_{j}|^{2}\nonumber\\
&+2\big(|\alpha_{j}|^{2}+|\alpha_{j-1}|^{2}\big)|\alpha_{j+1}-\alpha_{j-1}|^{2}|\alpha_{j}-\alpha_{j-1}|^{2}\Big]\nonumber\\
&+\sum_{j=0}^{\infty}|\alpha_{j}|^{2}\Big(|\alpha_{j+1}|^{4}+|\alpha_{j-1}|^{4}+|\alpha_{j+1}-\alpha_{j-1}|^{4}\Big)\nonumber\\
&+\frac{3}{2}\sum_{j=0}^{\infty}\big(1+|\alpha_{j}|^{4}\big)\big(|\alpha_{j+1}|^{2}+|\alpha_{j-1}|^{2}\big)|\alpha_{j+1}-\alpha_{j-1}|^{2}\nonumber\\
&+\frac{1}{2}\sum_{j=0}^{\infty}\Big[\big(|\alpha_{j}|^{2}+|\alpha_{j-1}|^{2}\big)|\alpha_{j}-\alpha_{j-1}|^{6}\nonumber\\
&+\big(|\alpha_{j}|^{6}+|\alpha_{j-1}|^{6}+|\alpha_{j}|^{4}|\alpha_{j-1}|^{2}+|\alpha_{j}|^{2}|\alpha_{j-1}|^{4}\big)|\alpha_{j}-\alpha_{j-1}|^{2}\Big].\nonumber\qedhere
\end{align}
\end{proof}

By virtue of sum rule (4.133), we can easily obtain the following higher order Szeg\H{o} theorems. Here and in some other sum rules below (i.e., (4.157), (4.167) and (4.176)), $\mathrm{EP}$ and $\mathrm{CP}$ are analogously defined as before. That is, $\mathrm{EP}$ is the sum of negative series in sum rules whereas $\mathrm{CP}$ is the sum of positive series in them.
\begin{thm}
Assume $\alpha\in\ell^{6}$ and $(S-1)\alpha\in \ell^{2}$, then
\begin{align}
&\int_{0}^{2\pi}(1-\cos4\theta)\log w(\theta)\frac{d\theta}{2\pi}>-\infty\,\,\Longleftrightarrow\,\, \alpha\in \ell^{4}.
\end{align}
\end{thm}

\begin{rem}
This is a weaker result similar to the one in \cite{lu1} due to Lukic .
\end{rem}

\begin{proof}
Since $\alpha\in\ell^{6}$ and $(S-1)\alpha\in \ell^{2}$, by Young inequality (or H\"older inequality), then $\mathrm{CP}<+\infty$. Moreover,
the series in $\mathrm{EP}$ are convergent except the following ones
\begin{align}
&\sum_{j=0}^{\infty}\Big(\log(1-|\alpha_{j}|^{2})+|\alpha_{j}|^{2}\Big),\\
&-\frac{1}{2}\sum_{j=0}^{\infty}|\alpha_{j}|^{2}(|\alpha_{j+3}|^{2}+|\alpha_{j-1}|^{2}),\\
&-\frac{1}{2}\sum_{j=0}^{\infty}\rho_{j}^{2}|\alpha_{j+1}|^{2}(|\alpha_{j+3}|^{2}+|\alpha_{j-1}|^{2}),\\
&-\frac{1}{2}\sum_{j=0}^{\infty}\rho_{j+1}^{2}\rho_{j}^{2}|\alpha_{j+2}|^{2}(|\alpha_{j+3}|^{2}+|\alpha_{j-1}|^{2}),\\
&-\frac{1}{2}\sum_{j=0}^{\infty}\rho_{j+1}^{2}\rho_{j}^{2}\Big[|\alpha_{j+2}|^{4}+|\alpha_{j-1}|^{4}+|\alpha_{j+1}|^{2}|\alpha_{j-1}|^{2}+|\alpha_{j+2}|^{2}|\alpha_{j}|^{2}\nonumber\\
&+2|\alpha_{j+2}|^{2}|\alpha_{j+1}|^{2}+2|\alpha_{j}|^{2}|\alpha_{j-1}|^{2}\Big]
\end{align}
and
\begin{equation}
-\frac{1}{4}\sum_{j=0}^{\infty}\Big(|\alpha_{j+1}|^{4}+|\alpha_{j-1}|^{4}\Big).
\end{equation}
Thus $Z_{4,1}(\mu)>-\infty$ is equivalent to all above series are convergent in (4.145)-(4.150) under the assumption of $\alpha\in\ell^{6}$ and $(S-1)\alpha\in \ell^{2}$. By Lemma 4.3 and H\"older inequality, it is easy to see that this fact means that
$Z_{4,1}(\mu)>-\infty$ if and only if $\sum_{j=0}^{\infty}|\alpha_{j}|^{4}<+\infty$ (that is, $\alpha\in\ell^{4}$) as $\alpha\in\ell^{6}$ and $(S-1)\alpha\in \ell^{2}$.
\end{proof}

\begin{thm}
Assume $\alpha\in\ell^{4}$, then
\begin{align}
&\int_{0}^{2\pi}(1-\cos4\theta)\log w(\theta)\frac{d\theta}{2\pi}>-\infty\,\,\Longleftrightarrow\,\, (S^{4}-1)\alpha\in \ell^{2}.
\end{align}
\end{thm}

\begin{rem}
This is one special case of the result in \cite{gz} due to Golinskii and Zlato\v{s}.
\end{rem}

\begin{proof}
By sum rule (4.133), as $\alpha\in\ell^{4}$, all the series in $\mathrm{CP}$ and the ones in $\mathrm{EP}$ are finite except
\begin{equation}
-\frac{1}{2}\sum_{j=0}^{\infty}\rho_{j+2}^{2}\rho_{j+1}^{2}\rho_{j}^{2}|\alpha_{j+3}-\alpha_{j-1}|^{2}.
\end{equation}
Then $Z_{4,1}(\mu)>-\infty$ if and only if $\sum_{j=0}^{\infty}|\alpha_{j+3}-\alpha_{j-1}|^{2}<+\infty$ (namely, $(S^{4}-1)\alpha\in \ell^{2}$) as $\alpha\in\ell^{4}$.
\end{proof}

\begin{thm}
Assume $\alpha\in\ell^{6}$ and $(S-1)\alpha\in\ell^{3}$, then
\begin{align}
&\int_{0}^{2\pi}(1-\cos4\theta)\log w(\theta)\frac{d\theta}{2\pi}>-\infty\,\,\Longleftrightarrow\,\, (S^{4}-1)\alpha\in \ell^{2}\,\,and\,\,\alpha\in \ell^{4}.
\end{align}
\end{thm}

\begin{rem}
This is a one-case result of the original Simon conjecture under certain conditions.
\end{rem}

\begin{proof}
Note that
\begin{align}
|a|^{2}|b|^{2}|c|^{2}\leq\frac{|a|^{6}+|b|^{6}+|c|^{6}}{3}
\end{align}
and (by Young inequality)
\begin{align}
|a|^{2}|b-c|^{2}\leq\frac{|a|^{6}}{3}+\frac{3|b-c|^{3}}{2}
\end{align}
for any $a,b,c\in \mathbb{C}$,
then by (4.133), $\mathrm{CP}<+\infty$ and the series in $\mathrm{EP}$ are finite except the ones in (4.145)-(4.150) and the following one
\begin{align}
-\frac{1}{2}\sum_{j=0}^{\infty}\rho_{j+2}^{2}\rho_{j+1}^{2}\rho_{j}^{2}|\alpha_{j+3}-\alpha_{j-1}|^{2}.
\end{align}
Thus $Z_{4,1}(\mu)>-\infty$ is equivalent to $\sum_{j=0}^{\infty}|\alpha_{j}|^{4}<+\infty$ and $\sum_{j=0}^{\infty}|\alpha_{j+3}-\alpha_{j-1}|^{2}<+\infty$ (i.e., $\alpha\in\ell^{4}$ and $(S^{4}-1)\alpha\in\ell^{2}$) as $\alpha\in\ell^{6}$ and $(S-1)\alpha\in\ell^{3}$.
\end{proof}

\begin{thm}
For any $\alpha$,
\begin{align}
&\int_{0}^{2\pi}(1-\cos\theta)^{2}(1+\cos\theta)^{2}\log w(\theta) \frac{d\theta}{2\pi}\nonumber\\
=&\frac{9}{32}+\frac{1}{16}|\alpha_{0}|^{2}-\frac{5}{32}|\alpha_{0}|^{4}-\frac{1}{8}\sum_{j=0}^{1}|\alpha_{j+1}-\alpha_{j-1}|^{2}
-\frac{3}{16}\sum_{j=0}^{2}|\alpha_{j}|^{2}|\alpha_{j-1}|^{2}\nonumber\\&+\frac{1}{8}|\alpha_{2}|^{2}|\alpha_{1}|^{2}
-\frac{1}{16}\sum_{j=0}^{1}|\alpha_{j+1}|^{2}|\alpha_{j-1}|^{2}+\frac{3}{8}\sum_{j=0}^{\infty}\Big[\log(1-|\alpha_{j}|^{2})+|\alpha_{j}|^{2}+\frac{1}{2}|\alpha_{j}|^{4}\Big]\nonumber\\
&-\frac{1}{16}\sum_{j=0}^{\infty}\big(|\alpha_{j+2}|^{2}-|\alpha_{j-1}|^{2}\big)^{2}-\frac{1}{8}\sum_{j=0}^{\infty}\big(|\alpha_{j+1}|^{2}-|\alpha_{j-1}|^{2}\big)^{2}\nonumber\\
&-\frac{1}{16}\sum_{j=0}^{\infty}|\alpha_{j+3}-2\alpha_{j+1}+\alpha_{j-1}|^{2}-\frac{1}{8}\sum_{j=0}^{\infty}|\alpha_{j}-\alpha_{j-1}|^{4}\nonumber\\
&+\frac{1}{8}\Big\{-\frac{1}{2}\sum_{j=0}^{\infty}\big(|\alpha_{j+2}|^{2}|\alpha_{j+1}|^{2}+|\alpha_{j+1}|^{2}|\alpha_{j}|^{2}+|\alpha_{j+2}|^{2}|\alpha_{j}|^{2}\big)(|\alpha_{j+3}|^{2}+|\alpha_{j-1}|^{2})\nonumber\\
&-\frac{1}{2}\sum_{j=0}^{\infty}|\alpha_{j}|^{2}|\alpha_{j+3}-\alpha_{j-1}|^{2}\nonumber
\end{align}
\begin{align}
&-\frac{1}{2}\sum_{j=0}^{\infty}\rho_{j}^{2}|\alpha_{j+1}|^{2}|\alpha_{j+3}-\alpha_{j-1}|^{2}-\frac{1}{2}\sum_{j=0}^{\infty}\rho_{j+1}^{2}\rho_{j}^{2}|\alpha_{j+2}|^{2}|\alpha_{j+3}-\alpha_{j-1}|^{2}\nonumber\\
&-\frac{1}{2}\sum_{j=0}^{\infty}\rho_{j+1}^{2}\rho_{j}^{2}\Big[|\alpha_{j+2}|^{2}|\alpha_{j+2}-\alpha_{j+1}|^{2}+|\alpha_{j-1}|^{2}|\alpha_{j}-\alpha_{j-1}|^{2}\nonumber\\
&+\big(|\alpha_{j+2}|^{2}+|\alpha_{j-1}|^{2}\big)|\alpha_{j+1}-\alpha_{j}|^{2}+|\alpha_{j}|^{2}|\alpha_{j+1}-\alpha_{j-1}|^{2}\nonumber\\
&+|\alpha_{j+1}|^{2}|\alpha_{j+2}-\alpha_{j}|^{2}+\big(|\alpha_{j+2}|^{2}+2|\alpha_{j+1}|^{2}+2|\alpha_{j}|^{2}+|\alpha_{j-1}|^{2}\big)|\alpha_{j+2}-\alpha_{j-1}|^{2}\Big]\nonumber\\
&-\frac{1}{2}\sum_{j=0}^{\infty}|\alpha_{j}|^{2}\Big[|\alpha_{j+2}|^{4}+|\alpha_{j+1}|^{2}|\alpha_{j-1}|^{2}+|\alpha_{j+2}|^{2}|\alpha_{j}|^{2}\nonumber\\
&+2|\alpha_{j+2}|^{2}|\alpha_{j+1}|^{2}+2|\alpha_{j}|^{2}|\alpha_{j-1}|^{2}+|\alpha_{j}|^{2}|\alpha_{j+2}-\alpha_{j+1}|^{2}
+|\alpha_{j+1}|^{2}|\alpha_{j}-\alpha_{j-1}|^{2}\nonumber\\
&+|\alpha_{j+2}-\alpha_{j+1}|^{2}|\alpha_{j+2}-\alpha_{j-1}|^{2}+|\alpha_{j+2}-\alpha_{j-1}|^{2}|\alpha_{j}-\alpha_{j-1}|^{2}\Big]\nonumber\\
&-\frac{1}{2}\sum_{j=0}^{\infty}\rho_{j}^{2}|\alpha_{j+1}|^{2}\Big[|\alpha_{j+2}|^{4}+|\alpha_{j-1}|^{4}+|\alpha_{j+1}|^{2}|\alpha_{j-1}|^{2}+|\alpha_{j+2}|^{2}|\alpha_{j}|^{2}\nonumber\\
&+2|\alpha_{j+2}|^{2}|\alpha_{j+1}|^{2}+2|\alpha_{j}|^{2}|\alpha_{j-1}|^{2}+|\alpha_{j}|^{2}|\alpha_{j+2}-\alpha_{j+1}|^{2}
+|\alpha_{j+1}|^{2}|\alpha_{j}-\alpha_{j-1}|^{2}\nonumber\\
&+|\alpha_{j+2}-\alpha_{j+1}|^{2}|\alpha_{j+2}-\alpha_{j-1}|^{2}+|\alpha_{j+2}-\alpha_{j-1}|^{2}|\alpha_{j}-\alpha_{j-1}|^{2}\Big]\nonumber\\
&-\frac{1}{2}\sum_{j=0}^{\infty}\rho_{j}^{2}\Big[|\alpha_{j+1}|^{6}+|\alpha_{j-1}|^{6}+|\alpha_{j}|^{4}|\alpha_{j-1}|^{2}+|\alpha_{j+1}|^{2}|\alpha_{j}|^{4}\nonumber\\
&+\big(|\alpha_{j+1}|^{2}+|\alpha_{j-1}|^{2}\big)\big(|\alpha_{j+1}-\alpha_{j}|^{4}+|\alpha_{j}-\alpha_{j-1}|^{4}\big)\nonumber\\
&+|\alpha_{j-1}|^{4}|\alpha_{j+1}-\alpha_{j}|^{2}+|\alpha_{j+1}|^{4}|\alpha_{j}-\alpha_{j-1}|^{2}\nonumber\\
&+2\big(|\alpha_{j+1}|^{2}+|\alpha_{j}|^{2}\big)|\alpha_{j+1}-\alpha_{j-1}|^{2}|\alpha_{j+1}-\alpha_{j}|^{2}\nonumber\\
&+2\big(|\alpha_{j}|^{2}+|\alpha_{j-1}|^{2}\big)|\alpha_{j+1}-\alpha_{j-1}|^{2}|\alpha_{j}-\alpha_{j-1}|^{2}\Big]\nonumber\\
&-\frac{1}{2}\sum_{j=0}^{\infty}|\alpha_{j}|^{2}\Big[
\big(|\alpha_{j+1}|^{4}+2|\alpha_{j}|^{4}+|\alpha_{j-1}|^{4}+|\alpha_{j}|^{2}(|\alpha_{j+1}|^{2}+|\alpha_{j-1}|^{2})\big)|\alpha_{j+1}-\alpha_{j-1}|^{2}\nonumber\\
&+\big(2|\alpha_{j+1}|^{4}+|\alpha_{j+1}|^{2}|\alpha_{j}|^{2}+|\alpha_{j+1}|^{2}|\alpha_{j-1}|^{2}+|\alpha_{j}|^{2}|\alpha_{j-1}|^{2}\big)|\alpha_{j+1}-\alpha_{j}|^{2}\nonumber\\
&+\big(2|\alpha_{j-1}|^{4}+|\alpha_{j+1}|^{2}|\alpha_{j}|^{2}+|\alpha_{j+1}|^{2}|\alpha_{j-1}|^{2}+|\alpha_{j}|^{2}|\alpha_{j-1}|^{2}\big)|\alpha_{j}-\alpha_{j-1}|^{2}\nonumber\\
&+\big(|\alpha_{j+1}|^{2}+|\alpha_{j-1}|^{2}\big)|\alpha_{j+1}-\alpha_{j}|^{2}|\alpha_{j}-\alpha_{j-1}|^{2}\nonumber\\
&+\big(|\alpha_{j+1}-\alpha_{j}|^{4}+|\alpha_{j}-\alpha_{j-1}|^{4}\big)|\alpha_{j+1}-\alpha_{j-1}|^{2}\Big]\nonumber\\
&-\sum_{j=0}^{\infty}|\alpha_{j}|^{2}\big(|\alpha_{j+1}|^{2}+|\alpha_{j-1}|^{2}\big)|\alpha_{j+1}-\alpha_{j-1}|^{2}\nonumber\\
&-\frac{3}{2}\sum_{j=0}^{\infty}\rho_{j}^{4}\big(|\alpha_{j+1}|^{2}+|\alpha_{j-1}|^{2}\big)|\alpha_{j+1}-\alpha_{j-1}|^{2}\nonumber\\
&-\frac{1}{2}\sum_{j=0}^{\infty}|\alpha_{j}|^{2}\Big[2|\alpha_{j+1}|^{4}+3|\alpha_{j-1}|^{4}\Big]-\sum_{j=0}^{\infty}|\alpha_{j}|^{2}|\alpha_{j+1}-\alpha_{j-1}|^{4}\nonumber\\
&-\frac{1}{2}\sum_{j=0}^{\infty}\Big[\big(|\alpha_{j}|^{2}+|\alpha_{j-1}|^{2}\big)|\alpha_{j}-\alpha_{j-1}|^{6}\nonumber\\
&+\big(|\alpha_{j}|^{6}+|\alpha_{j-1}|^{6}+|\alpha_{j}|^{4}|\alpha_{j-1}|^{2}+|\alpha_{j}|^{2}|\alpha_{j-1}|^{4}\big)|\alpha_{j}-\alpha_{j-1}|^{2}\Big]\Big\}\nonumber
\end{align}
\begin{align}
&+\frac{1}{4}\sum_{j=0}^{\infty}|\alpha_{j}|^{2}|\alpha_{j+1}-\alpha_{j-1}|^{2}+\frac{1}{4}\sum_{j=0}^{\infty}\big(|\alpha_{j}|^{2}+|\alpha_{j-1}|^{2}\big)|\alpha_{j}-\alpha_{j-1}|^{2}\nonumber\\
&+\frac{1}{16}\sum_{j=0}^{\infty}\big(|\alpha_{j}|^{2}-|\alpha_{j-1}|^{2}\big)^{2}+\frac{1}{8}\Big\{\frac{1}{2}\sum_{j=0}^{\infty}|\alpha_{j+2}|^{2}|\alpha_{j+1}|^{2}|\alpha_{j}|^{2}(|\alpha_{j+3}|^{2}+|\alpha_{j-1}|^{2})\nonumber\\
&+\frac{1}{2}\sum_{j=0}^{\infty}\Big[|\alpha_{j}|^{2}|\alpha_{j+2}-\alpha_{j+1}|^{2}
+|\alpha_{j+1}|^{2}|\alpha_{j}-\alpha_{j-1}|^{2}\nonumber\\
&+|\alpha_{j+2}-\alpha_{j+1}|^{2}|\alpha_{j+2}-\alpha_{j-1}|^{2}+|\alpha_{j+2}-\alpha_{j-1}|^{2}|\alpha_{j}-\alpha_{j-1}|^{2}\Big]\nonumber\\
&+\frac{1}{2}\sum_{j=0}^{\infty}\Big[
\big(|\alpha_{j+1}|^{4}+2|\alpha_{j}|^{4}+|\alpha_{j-1}|^{4}+|\alpha_{j+1}|^{2}|\alpha_{j}|^{2}+|\alpha_{j}|^{2}|\alpha_{j-1}|^{2}\big)|\alpha_{j+1}-\alpha_{j-1}|^{2}\nonumber\\
&+\big(2|\alpha_{j+1}|^{4}+|\alpha_{j+1}|^{2}|\alpha_{j}|^{2}+|\alpha_{j+1}|^{2}|\alpha_{j-1}|^{2}+|\alpha_{j}|^{2}|\alpha_{j-1}|^{2}\big)|\alpha_{j+1}-\alpha_{j}|^{2}\nonumber\\
&+\big(2|\alpha_{j-1}|^{4}+|\alpha_{j+1}|^{2}|\alpha_{j}|^{2}+|\alpha_{j+1}|^{2}|\alpha_{j-1}|^{2}+|\alpha_{j}|^{2}|\alpha_{j-1}|^{2}\big)|\alpha_{j}-\alpha_{j-1}|^{2}\nonumber\\
&+\big(|\alpha_{j+1}|^{2}+|\alpha_{j-1}|^{2}\big)|\alpha_{j+1}-\alpha_{j}|^{2}|\alpha_{j}-\alpha_{j-1}|^{2}\nonumber\\
&+\big(|\alpha_{j+1}-\alpha_{j}|^{4}+|\alpha_{j}-\alpha_{j-1}|^{4}\big)|\alpha_{j+1}-\alpha_{j-1}|^{2}\Big]\nonumber\\
&+\sum_{j=0}^{\infty}\big(|\alpha_{j+1}|^{2}+|\alpha_{j-1}|^{2}\big)|\alpha_{j+1}-\alpha_{j-1}|^{2}+\frac{1}{4}\sum_{j=0}^{\infty}|\alpha_{j+1}-\alpha_{j-1}|^{4}\nonumber\\
&+\frac{3}{4}\sum_{j=0}^{\infty}|\alpha_{j}|^{4}\big(|\alpha_{j+1}|^{4}+|\alpha_{j-1}|^{4}+|\alpha_{j+1}-\alpha_{j-1}|^{4}\big)\nonumber\\
&+\frac{1}{8}\sum_{j=0}^{\infty}\Big[|\alpha_{j}|^{8}+|\alpha_{j-1}|^{8}+|\alpha_{j}-\alpha_{j-1}|^{8}+4\big(|\alpha_{j}|^{2}+|\alpha_{j-1}|^{2}\big)^{2}|\alpha_{j}-\alpha_{j-1}|^{4}\nonumber\\
&+2\big(|\alpha_{j}|^{4}+|\alpha_{j-1}|^{4}\big)|\alpha_{j}-\alpha_{j-1}|^{4}\Big]\Big\}.
\end{align}
\end{thm}
\begin{proof}
Since
\begin{align}
(1-\cos\theta)^{2}(1+\cos\theta)^{2}&=\frac{3}{8}-\frac{1}{2}\cos2\theta+\frac{1}{8}\cos4\theta\nonumber\\
&=\frac{1}{2}(1-\cos2\theta)-\frac{1}{8}(1-\cos4\theta),
\end{align}
then
\begin{equation}
Z_{4,2}(\mu)=\frac{3}{8}w_{0}-\frac{1}{2}\mathrm{Re}(w_{2})+\frac{1}{8}\mathrm{Re}(w_{4})=Z_{2,1}(\mu)-\frac{1}{8}Z_{4,1}(\mu).
\end{equation}
By (4.12), (4.133a) and (4.159), we have
\begin{align}
&Z_{4,2}(\mu)=\int_{0}^{2\pi}\big(1-\cos^{2}\theta\big)^{2}\log w(\theta)\frac{d\theta}{2\pi}\nonumber
\end{align}
\begin{align}
=&\frac{3}{8}+\frac{1}{2}\sum_{j=0}^{\infty}\Big[\log(1-|\alpha_{j}|^{2})+|\alpha_{j}|^{2}+\frac{1}{2}|\alpha_{j}|^{4}\Big]\nonumber\\
&-\frac{1}{2}\sum_{j=0}^{\infty}|\alpha_{j}\alpha_{j-1}|^{2}-\frac{1}{4}\sum_{j=0}^{\infty}\rho_{j}^{2}|\alpha_{j+1}-\alpha_{j-1}|^{2}\nonumber\\
&-\frac{1}{16}\sum_{j=0}^{\infty}\left[\big(2|\alpha_{j}|^{2}-|\alpha_{j}-\alpha_{j-1}|^{2}\big)^{2}+\big(2|\alpha_{j-1}|^{2}-|\alpha_{j}-\alpha_{j-1}|^{2}\big)^{2}\right]\nonumber\\
&-\frac{1}{8}\Big\{\sum_{j=0}^{\infty}\log(1-|\alpha_{j}|^{2})+\frac{1}{2}\sum_{j=0}^{\infty}\rho_{j+2}^{2}\rho_{j+1}^{2}\rho_{j}^{2}(|\alpha_{j+3}|^{2}+|\alpha_{j-1}|^{2}-|\alpha_{j+3}-\alpha_{j-1}|^{2})\nonumber\\
&-\frac{1}{2}\sum_{j=0}^{\infty}\rho_{j+1}^{2}\rho_{j}^{2}\Big[|\alpha_{j+2}|^{4}+|\alpha_{j-1}|^{4}+|\alpha_{j+1}|^{2}|\alpha_{j-1}|^{2}+|\alpha_{j+2}|^{2}|\alpha_{j}|^{2}\nonumber\\
&+2|\alpha_{j+2}|^{2}|\alpha_{j+1}|^{2}+2|\alpha_{j}|^{2}|\alpha_{j-1}|^{2}-\big(|\alpha_{j+2}|^{2}-|\alpha_{j}|^{2}\big)|\alpha_{j+2}-\alpha_{j+1}|^{2}\nonumber\\
&-\big(|\alpha_{j+2}|^{2}+|\alpha_{j-1}|^{2}\big)|\alpha_{j+1}-\alpha_{j}|^{2}-|\alpha_{j}|^{2}|\alpha_{j+1}-\alpha_{j-1}|^{2}\nonumber\\
&-|\alpha_{j+1}|^{2}|\alpha_{j+2}-\alpha_{j}|^{2}+\big(|\alpha_{j+1}|^{2}-|\alpha_{j-1}|^{2}\big)|\alpha_{j}-\alpha_{j-1}|^{2}\nonumber\\
&-\big(|\alpha_{j+2}|^{2}+2|\alpha_{j+1}|^{2}+2|\alpha_{j}|^{2}+|\alpha_{j-1}|^{2}\big)|\alpha_{j+2}-\alpha_{j-1}|^{2}\nonumber\\
&+|\alpha_{j+2}-\alpha_{j+1}|^{2}|\alpha_{j+2}-\alpha_{j}|^{2}+|\alpha_{j+2}-\alpha_{j-1}|^{2}|\alpha_{j}-\alpha_{j-1}|^{2}\Big]\nonumber\\
&+\frac{1}{2}\sum_{j=0}^{\infty}\rho_{j}^{2}\Big[|\alpha_{j+1}|^{6}+|\alpha_{j-1}|^{6}+|\alpha_{j}|^{4}|\alpha_{j-1}|^{2}+|\alpha_{j+1}|^{2}|\alpha_{j}|^{4}\nonumber\\
&-\big(|\alpha_{j+1}|^{4}+2|\alpha_{j}|^{4}+|\alpha_{j-1}|^{4}+|\alpha_{j+1}|^{2}|\alpha_{j}|^{2}+|\alpha_{j}|^{2}|\alpha_{j-1}|^{2}\big)|\alpha_{j+1}-\alpha_{j-1}|^{2}\nonumber\\
&+\big(|\alpha_{j+1}|^{2}+|\alpha_{j-1}|^{2}\big)\big(|\alpha_{j+1}-\alpha_{j}|^{4}+|\alpha_{j}-\alpha_{j-1}|^{4}\big)\nonumber\\
&-\big(2|\alpha_{j+1}|^{4}-|\alpha_{j-1}|^{4}+|\alpha_{j+1}|^{2}|\alpha_{j}|^{2}+|\alpha_{j+1}|^{2}|\alpha_{j-1}|^{2}+|\alpha_{j}|^{2}|\alpha_{j-1}|^{2}\big)|\alpha_{j+1}-\alpha_{j}|^{2}\nonumber\\
&-\big(2|\alpha_{j-1}|^{4}-|\alpha_{j+1}|^{4}+|\alpha_{j+1}|^{2}|\alpha_{j}|^{2}+|\alpha_{j+1}|^{2}|\alpha_{j-1}|^{2}+|\alpha_{j}|^{2}|\alpha_{j-1}|^{2}\big)|\alpha_{j}-\alpha_{j-1}|^{2}\nonumber\\
&+2\big(|\alpha_{j+1}|^{2}+|\alpha_{j}|^{2}\big)|\alpha_{j+1}-\alpha_{j-1}|^{2}|\alpha_{j+1}-\alpha_{j}|^{2}\nonumber\\
&+2\big(|\alpha_{j}|^{2}+|\alpha_{j-1}|^{2}\big)|\alpha_{j+1}-\alpha_{j-1}|^{2}|\alpha_{j}-\alpha_{j-1}|^{2}\nonumber\\
&-\big(|\alpha_{j+1}|^{2}+|\alpha_{j-1}|^{2}\big)|\alpha_{j+1}-\alpha_{j}|^{2}|\alpha_{j}-\alpha_{j-1}|^{2}\nonumber\\
&-\big(|\alpha_{j+1}-\alpha_{j}|^{4}+|\alpha_{j}-\alpha_{j-1}|^{4}\big)|\alpha_{j+1}-\alpha_{j-1}|^{2}\Big]\nonumber\\
&+\frac{1}{2}\sum_{j=0}^{\infty}\rho_{j}^{2}\Big[|\alpha_{j+1}|^{4}+|\alpha_{j-1}|^{4}+|\alpha_{j+1}-\alpha_{j-1}|^{4}-2\big(|\alpha_{j+1}|^{2}+|\alpha_{j-1}|^{2}\big)|\alpha_{j+1}-\alpha_{j-1}|^{2}\Big]\nonumber\\
&-\frac{3}{4}\sum_{j=0}^{\infty}\rho_{j}^{4}\Big[|\alpha_{j+1}|^{4}+|\alpha_{j-1}|^{4}+|\alpha_{j+1}-\alpha_{j-1}|^{4}-2\big(|\alpha_{j+1}|^{2}+|\alpha_{j-1}|^{2}\big)|\alpha_{j+1}-\alpha_{j-1}|^{2}\Big]\nonumber\\
&-\frac{1}{8}\sum_{j=0}^{\infty}\Big[|\alpha_{j}|^{8}+|\alpha_{j-1}|^{8}+|\alpha_{j}-\alpha_{j-1}|^{8}+4\big(|\alpha_{j}|^{2}+|\alpha_{j-1}|^{2}\big)^{2}|\alpha_{j}-\alpha_{j-1}|^{4}\nonumber\\
&+2\big(|\alpha_{j}|^{4}+|\alpha_{j-1}|^{4}\big)|\alpha_{j}-\alpha_{j-1}|^{4}-4\big(|\alpha_{j}|^{2}+|\alpha_{j-1}|^{2}\big)|\alpha_{j}-\alpha_{j-1}|^{6}\nonumber\\
&-4\big(|\alpha_{j}|^{6}+|\alpha_{j-1}|^{6}+|\alpha_{j}|^{4}|\alpha_{j-1}|^{2}+|\alpha_{j}|^{2}|\alpha_{j-1}|^{4}\big)|\alpha_{j}-\alpha_{j-1}|^{2}\Big]\Big\}\nonumber
\end{align}
\begin{align}
=&\frac{9}{32}+\frac{1}{16}|\alpha_{0}|^{2}-\frac{5}{32}|\alpha_{0}|^{4}-\frac{1}{8}\sum_{j=0}^{1}|\alpha_{j+1}-\alpha_{j-1}|^{2}
-\frac{3}{16}\sum_{j=0}^{2}|\alpha_{j}|^{2}|\alpha_{j-1}|^{2}\nonumber\\&+\frac{1}{8}|\alpha_{2}|^{2}|\alpha_{1}|^{2}
-\frac{1}{16}\sum_{j=0}^{1}|\alpha_{j+1}|^{2}|\alpha_{j-1}|^{2}+\frac{3}{8}\sum_{j=0}^{\infty}\Big[\log(1-|\alpha_{j}|^{2})+|\alpha_{j}|^{2}+\frac{1}{2}|\alpha_{j}|^{4}\Big]\nonumber\\
&-\frac{1}{16}\sum_{j=0}^{\infty}\big(|\alpha_{j+2}|^{2}-|\alpha_{j-1}|^{2}\big)^{2}-\frac{1}{8}\sum_{j=0}^{\infty}\big(|\alpha_{j+1}|^{2}-|\alpha_{j-1}|^{2}\big)^{2}\nonumber\\
&-\frac{1}{16}\sum_{j=0}^{\infty}|\alpha_{j+3}-2\alpha_{j+1}+\alpha_{j-1}|^{2}-\frac{1}{8}\sum_{j=0}^{\infty}|\alpha_{j}-\alpha_{j-1}|^{4}\nonumber\\
&+\frac{1}{8}\Big\{-\frac{1}{2}\sum_{j=0}^{\infty}\big(|\alpha_{j+2}|^{2}|\alpha_{j+1}|^{2}+|\alpha_{j+1}|^{2}|\alpha_{j}|^{2}+|\alpha_{j+2}|^{2}|\alpha_{j}|^{2}\big)(|\alpha_{j+3}|^{2}+|\alpha_{j-1}|^{2})\nonumber\\
&-\frac{1}{2}\sum_{j=0}^{\infty}|\alpha_{j}|^{2}|\alpha_{j+3}-\alpha_{j-1}|^{2}-\frac{1}{2}\sum_{j=0}^{\infty}\rho_{j}^{2}|\alpha_{j+1}|^{2}|\alpha_{j+3}-\alpha_{j-1}|^{2}\nonumber\\
&-\frac{1}{2}\sum_{j=0}^{\infty}\rho_{j+1}^{2}\rho_{j}^{2}|\alpha_{j+2}|^{2}|\alpha_{j+3}-\alpha_{j-1}|^{2}\nonumber\\
&-\frac{1}{2}\sum_{j=0}^{\infty}\rho_{j+1}^{2}\rho_{j}^{2}\Big[|\alpha_{j+2}|^{2}|\alpha_{j+2}-\alpha_{j+1}|^{2}+|\alpha_{j-1}|^{2}|\alpha_{j}-\alpha_{j-1}|^{2}\nonumber\\
&+\big(|\alpha_{j+2}|^{2}+|\alpha_{j-1}|^{2}\big)|\alpha_{j+1}-\alpha_{j}|^{2}+|\alpha_{j}|^{2}|\alpha_{j+1}-\alpha_{j-1}|^{2}\nonumber\\
&+|\alpha_{j+1}|^{2}|\alpha_{j+2}-\alpha_{j}|^{2}+\big(|\alpha_{j+2}|^{2}+2|\alpha_{j+1}|^{2}+2|\alpha_{j}|^{2}+|\alpha_{j-1}|^{2}\big)|\alpha_{j+2}-\alpha_{j-1}|^{2}\Big]\nonumber\\
&-\frac{1}{2}\sum_{j=0}^{\infty}|\alpha_{j}|^{2}\Big[|\alpha_{j+2}|^{4}+|\alpha_{j+1}|^{2}|\alpha_{j-1}|^{2}+|\alpha_{j+2}|^{2}|\alpha_{j}|^{2}\nonumber\\
&+2|\alpha_{j+2}|^{2}|\alpha_{j+1}|^{2}+2|\alpha_{j}|^{2}|\alpha_{j-1}|^{2}+|\alpha_{j}|^{2}|\alpha_{j+2}-\alpha_{j+1}|^{2}
+|\alpha_{j+1}|^{2}|\alpha_{j}-\alpha_{j-1}|^{2}\nonumber\\
&+|\alpha_{j+2}-\alpha_{j+1}|^{2}|\alpha_{j+2}-\alpha_{j-1}|^{2}+|\alpha_{j+2}-\alpha_{j-1}|^{2}|\alpha_{j}-\alpha_{j-1}|^{2}\Big]\nonumber\\
&-\frac{1}{2}\sum_{j=0}^{\infty}\rho_{j}^{2}|\alpha_{j+1}|^{2}\Big[|\alpha_{j+2}|^{4}+|\alpha_{j-1}|^{4}+|\alpha_{j+1}|^{2}|\alpha_{j-1}|^{2}+|\alpha_{j+2}|^{2}|\alpha_{j}|^{2}\nonumber\\
&+2|\alpha_{j+2}|^{2}|\alpha_{j+1}|^{2}+2|\alpha_{j}|^{2}|\alpha_{j-1}|^{2}+|\alpha_{j}|^{2}|\alpha_{j+2}-\alpha_{j+1}|^{2}
+|\alpha_{j+1}|^{2}|\alpha_{j}-\alpha_{j-1}|^{2}\nonumber\\
&+|\alpha_{j+2}-\alpha_{j+1}|^{2}|\alpha_{j+2}-\alpha_{j-1}|^{2}+|\alpha_{j+2}-\alpha_{j-1}|^{2}|\alpha_{j}-\alpha_{j-1}|^{2}\Big]\nonumber\\
&-\frac{1}{2}\sum_{j=0}^{\infty}\rho_{j}^{2}\Big[|\alpha_{j+1}|^{6}+|\alpha_{j-1}|^{6}+|\alpha_{j}|^{4}|\alpha_{j-1}|^{2}+|\alpha_{j+1}|^{2}|\alpha_{j}|^{4}\nonumber\\
&+\big(|\alpha_{j+1}|^{2}+|\alpha_{j-1}|^{2}\big)\big(|\alpha_{j+1}-\alpha_{j}|^{4}+|\alpha_{j}-\alpha_{j-1}|^{4}\big)\nonumber\\
&+|\alpha_{j-1}|^{4}|\alpha_{j+1}-\alpha_{j}|^{2}+|\alpha_{j+1}|^{4}|\alpha_{j}-\alpha_{j-1}|^{2}\nonumber\\
&+2\big(|\alpha_{j+1}|^{2}+|\alpha_{j}|^{2}\big)|\alpha_{j+1}-\alpha_{j-1}|^{2}|\alpha_{j+1}-\alpha_{j}|^{2}\nonumber\\
&+2\big(|\alpha_{j}|^{2}+|\alpha_{j-1}|^{2}\big)|\alpha_{j+1}-\alpha_{j-1}|^{2}|\alpha_{j}-\alpha_{j-1}|^{2}\Big]\nonumber
\end{align}
\begin{align}
&-\frac{1}{2}\sum_{j=0}^{\infty}|\alpha_{j}|^{2}\Big[
\big(|\alpha_{j+1}|^{4}+2|\alpha_{j}|^{4}+|\alpha_{j-1}|^{4}+|\alpha_{j}|^{2}(|\alpha_{j+1}|^{2}+|\alpha_{j-1}|^{2})\big)|\alpha_{j+1}-\alpha_{j-1}|^{2}\nonumber\\
&+\big(2|\alpha_{j+1}|^{4}+|\alpha_{j+1}|^{2}|\alpha_{j}|^{2}+|\alpha_{j+1}|^{2}|\alpha_{j-1}|^{2}+|\alpha_{j}|^{2}|\alpha_{j-1}|^{2}\big)|\alpha_{j+1}-\alpha_{j}|^{2}\nonumber\\
&+\big(2|\alpha_{j-1}|^{4}+|\alpha_{j+1}|^{2}|\alpha_{j}|^{2}+|\alpha_{j+1}|^{2}|\alpha_{j-1}|^{2}+|\alpha_{j}|^{2}|\alpha_{j-1}|^{2}\big)|\alpha_{j}-\alpha_{j-1}|^{2}\nonumber\\
&+\big(|\alpha_{j+1}|^{2}+|\alpha_{j-1}|^{2}\big)|\alpha_{j+1}-\alpha_{j}|^{2}|\alpha_{j}-\alpha_{j-1}|^{2}\nonumber\\
&+\big(|\alpha_{j+1}-\alpha_{j}|^{4}+|\alpha_{j}-\alpha_{j-1}|^{4}\big)|\alpha_{j+1}-\alpha_{j-1}|^{2}\Big]\nonumber\\
&-\sum_{j=0}^{\infty}|\alpha_{j}|^{2}\big(|\alpha_{j+1}|^{2}+|\alpha_{j-1}|^{2}\big)|\alpha_{j+1}-\alpha_{j-1}|^{2}\nonumber\\
&-\frac{3}{2}\sum_{j=0}^{\infty}\rho_{j}^{4}\big(|\alpha_{j+1}|^{2}+|\alpha_{j-1}|^{2}\big)|\alpha_{j+1}-\alpha_{j-1}|^{2}\nonumber\\
&-\frac{1}{2}\sum_{j=0}^{\infty}|\alpha_{j}|^{2}\Big[2|\alpha_{j+1}|^{4}+3|\alpha_{j-1}|^{4}\Big]-\sum_{j=0}^{\infty}|\alpha_{j}|^{2}|\alpha_{j+1}-\alpha_{j-1}|^{4}\nonumber\\
&-\frac{1}{2}\sum_{j=0}^{\infty}\Big[\big(|\alpha_{j}|^{2}+|\alpha_{j-1}|^{2}\big)|\alpha_{j}-\alpha_{j-1}|^{6}\nonumber\\
&+\big(|\alpha_{j}|^{6}+|\alpha_{j-1}|^{6}+|\alpha_{j}|^{4}|\alpha_{j-1}|^{2}+|\alpha_{j}|^{2}|\alpha_{j-1}|^{4}\big)|\alpha_{j}-\alpha_{j-1}|^{2}\Big]\Big\}\nonumber\\
&+\frac{1}{4}\sum_{j=0}^{\infty}|\alpha_{j}|^{2}|\alpha_{j+1}-\alpha_{j-1}|^{2}+\frac{1}{4}\sum_{j=0}^{\infty}\big(|\alpha_{j}|^{2}+|\alpha_{j-1}|^{2}\big)|\alpha_{j}-\alpha_{j-1}|^{2}\nonumber\\
&+\frac{1}{16}\sum_{j=0}^{\infty}\big(|\alpha_{j}|^{2}-|\alpha_{j-1}|^{2}\big)^{2}+\frac{1}{8}\Big\{\frac{1}{2}\sum_{j=0}^{\infty}|\alpha_{j+2}|^{2}|\alpha_{j+1}|^{2}|\alpha_{j}|^{2}(|\alpha_{j+3}|^{2}+|\alpha_{j-1}|^{2})\nonumber\\
&+\frac{1}{2}\sum_{j=0}^{\infty}\Big[|\alpha_{j}|^{2}|\alpha_{j+2}-\alpha_{j+1}|^{2}
+|\alpha_{j+1}|^{2}|\alpha_{j}-\alpha_{j-1}|^{2}\nonumber\\
&+|\alpha_{j+2}-\alpha_{j+1}|^{2}|\alpha_{j+2}-\alpha_{j-1}|^{2}+|\alpha_{j+2}-\alpha_{j-1}|^{2}|\alpha_{j}-\alpha_{j-1}|^{2}\Big]\nonumber\\
&+\frac{1}{2}\sum_{j=0}^{\infty}\Big[
\big(|\alpha_{j+1}|^{4}+2|\alpha_{j}|^{4}+|\alpha_{j-1}|^{4}+|\alpha_{j+1}|^{2}|\alpha_{j}|^{2}+|\alpha_{j}|^{2}|\alpha_{j-1}|^{2}\big)|\alpha_{j+1}-\alpha_{j-1}|^{2}\nonumber\\
&+\big(2|\alpha_{j+1}|^{4}+|\alpha_{j+1}|^{2}|\alpha_{j}|^{2}+|\alpha_{j+1}|^{2}|\alpha_{j-1}|^{2}+|\alpha_{j}|^{2}|\alpha_{j-1}|^{2}\big)|\alpha_{j+1}-\alpha_{j}|^{2}\nonumber\\
&+\big(2|\alpha_{j-1}|^{4}+|\alpha_{j+1}|^{2}|\alpha_{j}|^{2}+|\alpha_{j+1}|^{2}|\alpha_{j-1}|^{2}+|\alpha_{j}|^{2}|\alpha_{j-1}|^{2}\big)|\alpha_{j}-\alpha_{j-1}|^{2}\nonumber\\
&+\big(|\alpha_{j+1}|^{2}+|\alpha_{j-1}|^{2}\big)|\alpha_{j+1}-\alpha_{j}|^{2}|\alpha_{j}-\alpha_{j-1}|^{2}\nonumber\\
&+\big(|\alpha_{j+1}-\alpha_{j}|^{4}+|\alpha_{j}-\alpha_{j-1}|^{4}\big)|\alpha_{j+1}-\alpha_{j-1}|^{2}\Big]\nonumber\\
&+\sum_{j=0}^{\infty}\big(|\alpha_{j+1}|^{2}+|\alpha_{j-1}|^{2}\big)|\alpha_{j+1}-\alpha_{j-1}|^{2}+\frac{1}{4}\sum_{j=0}^{\infty}|\alpha_{j+1}-\alpha_{j-1}|^{4}\nonumber
\end{align}
\begin{align}
&+\frac{3}{4}\sum_{j=0}^{\infty}|\alpha_{j}|^{4}\big(|\alpha_{j+1}|^{4}+|\alpha_{j-1}|^{4}+|\alpha_{j+1}-\alpha_{j-1}|^{4}\big)\nonumber\\
&+\frac{1}{8}\sum_{j=0}^{\infty}\Big[|\alpha_{j}|^{8}+|\alpha_{j-1}|^{8}+|\alpha_{j}-\alpha_{j-1}|^{8}+4\big(|\alpha_{j}|^{2}+|\alpha_{j-1}|^{2}\big)^{2}|\alpha_{j}-\alpha_{j-1}|^{4}\nonumber\\
&+2\big(|\alpha_{j}|^{4}+|\alpha_{j-1}|^{4}\big)|\alpha_{j}-\alpha_{j-1}|^{4}\Big]\Big\},\nonumber
\end{align}
where the following identity is used
\begin{align*}
|\alpha_{j+3}-2\alpha_{j+1}+\alpha_{j-1}|^{2}=2|\alpha_{j+3}-\alpha_{j+1}|^{2}-|\alpha_{j+3}-\alpha_{j-1}|^{2}+2|\alpha_{j+1}-\alpha_{j-1}|^{2}
\end{align*}
which implying
\begin{align*}
\frac{1}{16}\sum_{j=0}^{\infty}|\alpha_{j+3}-2\alpha_{j+1}+\alpha_{j-1}|^{2}=&-\frac{1}{8}\sum_{j=0}^{1}|\alpha_{j+1}-\alpha_{j-1}|^{2}+\frac{1}{4}\sum_{j=0}^{\infty}|\alpha_{j+1}-\alpha_{j-1}|^{2}\\
&-\frac{1}{16}\sum_{j=0}^{\infty}|\alpha_{j+3}-\alpha_{j+1}|^{2}.\qedhere
\end{align*}
\end{proof}
\begin{thm}
Assume $\alpha\in\ell^{6}$ and $(S-1)\alpha\in \ell^{3}$, then
\begin{align}
&\int_{0}^{2\pi}(1-\cos\theta)^{2}(1+\cos\theta)^{2}\log w(\theta)\frac{d\theta}{2\pi}>-\infty\,\,\Longleftrightarrow\,\,  (S^{2}-1)^{2}\alpha\in \ell^{2}.
\end{align}
\end{thm}

\begin{rem}
This is also a one-case result of the original Simon conjecture under certain conditions.
\end{rem}

\begin{proof}
Since
\begin{equation}
\big(|a|^{2}-|b|^{2}\big)^{2}=\big(|a|+|b|\big)^{2}\big(|a|-|b|\big)^{2}\leq 2\big(|a|^{2}+|b|^{2}\big)|a-b|^{2}
\end{equation}
for any $a,b\in \mathbb{C}$, by applying (4.154) and (4.155) to the sum rule (4.157), we have that if $\alpha\in\ell^{6}$ and $(S-1)\alpha\in \ell^{3}$, then $\mathrm{CP}<+\infty$ and the series in $\mathrm{EP}$ are convergent except the following one
\begin{equation}
-\frac{1}{16}\sum_{j=0}^{\infty}|\alpha_{j+3}-2\alpha_{j+1}+\alpha_{j-1}|^{2}.
\end{equation}
So $Z_{4,2}(\mu)>-\infty$ if and only if $\sum_{j=0}^{\infty}|\alpha_{j+3}-2\alpha_{j+1}+\alpha_{j-1}|^{2}<+\infty$ (i.e., $(S^{2}-1)^{2}\alpha\in\ell^{2}$) as $\alpha\in\ell^{6}$ and $(S-1)\alpha\in \ell^{3}$.
\end{proof}

\begin{thm}
Assume $\alpha\in\ell^{4}$, then
\begin{align}
&\int_{0}^{2\pi}(1-\cos\theta)^{2}(1+\cos\theta)^{2}\log w(\theta)\frac{d\theta}{2\pi}>-\infty\,\,\Longleftrightarrow\,\, (S^{2}-1)^{2}\alpha\in \ell^{2}.
\end{align}
\end{thm}

\begin{rem}
This is one special case of a result in \cite{gz} due to Golinskii and Zlato\v{s}.
\end{rem}

\begin{proof}
As $\alpha\in\ell^{4}$, by the sum rule (4.157) and H\"older inequality, we have that the series in it are convergent except the only one (4.162). Thus $Z_{4,2}(\mu)>-\infty$ if and only if $(S^{2}-1)^{2}\alpha\in\ell^{2}$ as $\alpha\in\ell^{4}$.
\begin{equation}
\end{equation}
\end{proof}

\begin{thm}
Assume $\alpha\in\ell^{8}$ and $(S-1)\alpha\in \ell^{2}$, then
\begin{align}
&\int_{0}^{2\pi}(1-\cos\theta)^{2}(1+\cos\theta)^{2}\log w(\theta)\frac{d\theta}{2\pi}>-\infty\,\, \Longleftrightarrow \,\, \alpha\in \ell^{6}.
\end{align}
\end{thm}

\begin{proof}
As $\alpha\in\ell^{8}$ and $(S-1)\alpha\in \ell^{2}$, by (4.161) and the sum rule (4.157), we have $\mathrm{CP}<+\infty$. So $Z_{4,2}(\mu)>-\infty$ is equivalent to $-\mathrm{EP}<+\infty$ under the assumption of $\alpha\in\ell^{8}$ and $(S-1)\alpha\in \ell^{2}$. By applying (4.154) to $\mathrm{EP}$, it is easy to know that $-\mathrm{EP}<+\infty$ is equivalent to $\alpha\in \ell^{6}$.
\end{proof}

\begin{rem}
This is a weaker result similar to the one of Lukic in \cite{lu1}.
\end{rem}

With respect to one-directional implication from $\alpha$ to $Z_{4,2}(\mu)$, we have

\begin{thm}
Assume $\alpha\in\ell^{6}$ and $(S-1)\alpha\in \ell^{2}$, then
\begin{align}
&\int_{0}^{2\pi}(1-\cos\theta)^{2}(1+\cos\theta)^{2}\log w(\theta)\frac{d\theta}{2\pi}>-\infty.
\end{align}
\end{thm}

\begin{proof}
As a direct proof, similar to Theorem 4.52, it mainly follows from (4.154) and (4.155). As an indirect proof, note that $(S-1)\alpha\in \ell^{2}$ implies $(S-1)\alpha\in \ell^{3}$ and
$(S^{2}-1)^{2}\alpha\in \ell^{2}$, it is a consequence of Theorem 4.52. More immediately, it is also a consequence of Theorem 4.56.
\end{proof}

\begin{thm}
For any $\alpha$,
\begin{align}
&\int_{0}^{2\pi}(1-\cos\theta)^{3}(1+\cos\theta)\log w(\theta) \frac{d\theta}{2\pi}\nonumber\\
=&\frac{37}{96}-\frac{1}{8}|\alpha_{0}|^{2}+\frac{1}{16}|\alpha_{1}|^{2}+\frac{1}{16}|\alpha_{2}|^{2}(|\alpha_{0}|^{2}+|\alpha_{1}|^{2})-\frac{1}{16}|\alpha_{1}|^{2}|\alpha_{0}|^{2}\nonumber\\
&-\frac{7}{32}|\alpha_{0}|^{4}+\frac{1}{8}|\alpha_{1}|^{4}+\frac{1}{8}\sum_{j=0}^{2}|\alpha_{j}-\alpha_{j-1}|^{2}+\frac{1}{4}|1+\alpha_{1}|^{2}-\frac{1}{8}|1+\alpha_{2}|^{2}\nonumber\\
&+\frac{5}{8}\sum_{j=0}^{\infty}\Big(\log(1-|\alpha_{j}|^{2})+|\alpha_{j}|^{2}+\frac{1}{2}|\alpha_{j}|^{4}\Big)\nonumber\\
&-\frac{1}{16}|\alpha_{j+3}-2\alpha_{j+2}+2\alpha_{j}-\alpha_{j-1}|^{2}-\frac{1}{8}\sum_{j=0}^{\infty}|\alpha_{j}-\alpha_{j-1}|^{4}\nonumber\\
&+\frac{1}{2}\Big\{-\frac{1}{3}\sum_{j=0}^{\infty}|\alpha_{j}|^{6}-\frac{1}{2}\sum_{j=0}^{\infty}|\alpha_{j+1}|^{2}|\alpha_{j}|^{2}(|\alpha_{j+2}|^{2}+|\alpha_{j-1}|^{2})\nonumber\\
&-\frac{1}{2}\sum_{j=0}^{\infty}\big(|\alpha_{j+1}|^{2}+|\alpha_{j}|^{2}\big)|\alpha_{j+2}-\alpha_{j-1}|^{2}\nonumber\\
&-\frac{1}{2}\sum_{j=0}^{\infty}|\alpha_{j}|^{2}\Big[|\alpha_{j+1}|^{2}|\alpha_{j}|^{2}+|\alpha_{j-1}|^{4}+|\alpha_{j}|^{2}|\alpha_{j-1}|^{2}+|\alpha_{j+1}|^{4}\nonumber\\
&+|\alpha_{j+1}-\alpha_{j-1}|^{2}\big(|\alpha_{j+1}-\alpha_{j}|^{2}+|\alpha_{j}-\alpha_{j-1}|^{2}\big)\Big]\nonumber
\end{align}
\begin{align}
&-\frac{1}{2}\sum_{j=0}^{\infty}\big(|\alpha_{j}|^{2}+|\alpha_{j-1}|^{2}\big)|\alpha_{j}-\alpha_{j-1}|^{4}\nonumber\\
&-\frac{1}{2}\sum_{j=0}^{\infty}\rho_{j}^{2}\Big[(|\alpha_{j+1}|^{2}+2|\alpha_{j}|^{2}+|\alpha_{j-1}|^{2})|\alpha_{j+1}-\alpha_{j-1}|^{2}\nonumber\\
&+|\alpha_{j-1}|^{2}|\alpha_{j}-\alpha_{j-1}|^{2}+|\alpha_{j+1}|^{2}|\alpha_{j+1}-\alpha_{j}|^{2}\Big]\Big\}\nonumber\\
&-\frac{1}{16}\sum_{j=0}^{\infty}\big(|\alpha_{j}|^{2}-|\alpha_{j-1}|^{2}\big)^{2}
-\frac{1}{8}\sum_{j=0}^{\infty}\big(|\alpha_{j+1}|^{2}-|\alpha_{j-1}|^{2}\big)^{2}\nonumber\\
&+\frac{1}{8}\Big\{-\frac{1}{2}\sum_{j=0}^{\infty}|\alpha_{j+2}|^{2}|\alpha_{j+1}|^{2}|\alpha_{j}|^{2}\big(|\alpha_{j+3}|^{2}+|\alpha_{j-1}|^{2}\big)\nonumber\\
&-\frac{1}{2}\sum_{j=0}^{\infty}\big(|\alpha_{j+2}|^{2}|\alpha_{j+1}|^{2}+|\alpha_{j+2}|^{2}|\alpha_{j}|^{2}+|\alpha_{j+1}|^{2}|\alpha_{j}|^{2}\big)|\alpha_{j+3}-\alpha_{j-1}|^{2}\nonumber\\
&-\frac{1}{2}\sum_{j=0}^{\infty}|\alpha_{j+1}|^{2}|\alpha_{j}|^{2}\Big[|\alpha_{j+2}|^{4}+|\alpha_{j-1}|^{4}+|\alpha_{j+1}|^{2}|\alpha_{j-1}|^{2}+|\alpha_{j+2}|^{2}|\alpha_{j}|^{2}\nonumber\\
&+2|\alpha_{j+2}|^{2}|\alpha_{j+1}|^{2}+2|\alpha_{j}|^{2}|\alpha_{j-1}|^{2}\Big]\nonumber\\
&-\frac{1}{2}\sum_{j=0}^{\infty}\rho_{j+1}^{2}\rho_{j}^{2}\Big[|\alpha_{j}|^{2}|\alpha_{j+2}-\alpha_{j+1}|^{2}+|\alpha_{j+1}|^{2}|\alpha_{j}-\alpha_{j-1}|^{2}\nonumber\\
&+|\alpha_{j+2}-\alpha_{j+1}|^{2}|\alpha_{j+2}-\alpha_{j-1}|^{2}+|\alpha_{j+2}-\alpha_{j-1}|^{2}|\alpha_{j}-\alpha_{j-1}|^{2}\Big]\nonumber\\
&-\frac{1}{2}\sum_{j=0}^{\infty}|\alpha_{j}|^{2}\Big[|\alpha_{j+2}|^{2}|\alpha_{j+2}-\alpha_{j+1}|^{2}+|\alpha_{j-1}|^{2}|\alpha_{j}-\alpha_{j-1}|^{2}\nonumber\\
&+\big(|\alpha_{j+2}|^{2}+|\alpha_{j-1}|^{2}\big)|\alpha_{j+1}-\alpha_{j}|^{2}+|\alpha_{j}|^{2}|\alpha_{j+1}-\alpha_{j-1}|^{2}\nonumber\\
&+|\alpha_{j+1}|^{2}|\alpha_{j+2}-\alpha_{j}|^{2}+\big(|\alpha_{j+2}|^{2}+2|\alpha_{j+1}|^{2}+2|\alpha_{j}|^{2}+|\alpha_{j-1}|^{2}\big)|\alpha_{j+2}-\alpha_{j-1}|^{2}\Big]\nonumber\\
&-\frac{1}{2}\sum_{j=0}^{\infty}\rho_{j}^{2}|\alpha_{j+1}|^{2}\Big[|\alpha_{j+2}|^{2}|\alpha_{j+2}-\alpha_{j+1}|^{2}+|\alpha_{j-1}|^{2}|\alpha_{j}-\alpha_{j-1}|^{2}\nonumber\\
&+\big(|\alpha_{j+2}|^{2}+|\alpha_{j-1}|^{2}\big)|\alpha_{j+1}-\alpha_{j}|^{2}+|\alpha_{j}|^{2}|\alpha_{j+1}-\alpha_{j-1}|^{2}\nonumber\\
&+|\alpha_{j+1}|^{2}|\alpha_{j+2}-\alpha_{j}|^{2}+\big(|\alpha_{j+2}|^{2}+2|\alpha_{j+1}|^{2}+2|\alpha_{j}|^{2}+|\alpha_{j-1}|^{2}\big)|\alpha_{j+2}-\alpha_{j-1}|^{2}\Big]\nonumber\\
&-\frac{1}{2}\sum_{j=0}^{\infty}|\alpha_{j}|^{2}\Big[|\alpha_{j+1}|^{6}+|\alpha_{j-1}|^{6}+|\alpha_{j}|^{4}|\alpha_{j-1}|^{2}+|\alpha_{j+1}|^{2}|\alpha_{j}|^{4}\nonumber\\
&+\big(|\alpha_{j+1}|^{2}+|\alpha_{j-1}|^{2}\big)\big(|\alpha_{j+1}-\alpha_{j}|^{4}+|\alpha_{j}-\alpha_{j-1}|^{4}\big)\nonumber\\
&+|\alpha_{j-1}|^{4}|\alpha_{j+1}-\alpha_{j}|^{2}+|\alpha_{j+1}|^{4}|\alpha_{j}-\alpha_{j-1}|^{2}\nonumber\\
&+2\big(|\alpha_{j+1}|^{2}+|\alpha_{j}|^{2}\big)|\alpha_{j+1}-\alpha_{j-1}|^{2}|\alpha_{j+1}-\alpha_{j}|^{2}\nonumber\\
&+2\big(|\alpha_{j}|^{2}+|\alpha_{j-1}|^{2}\big)|\alpha_{j+1}-\alpha_{j-1}|^{2}|\alpha_{j}-\alpha_{j-1}|^{2}\Big]\nonumber\\
&-\frac{1}{2}\sum_{j=0}^{\infty}\rho_{j}^{2}\Big[\big(|\alpha_{j+1}|^{4}+2|\alpha_{j}|^{4}+|\alpha_{j-1}|^{4}+|\alpha_{j}|^{2}(|\alpha_{j+1}|^{2}+|\alpha_{j-1}|^{2})\big)|\alpha_{j+1}-\alpha_{j-1}|^{2}\nonumber
\end{align}
\begin{align}
&+\big(2|\alpha_{j+1}|^{4}+|\alpha_{j+1}|^{2}|\alpha_{j}|^{2}+|\alpha_{j+1}|^{2}|\alpha_{j-1}|^{2}+|\alpha_{j}|^{2}|\alpha_{j-1}|^{2}\big)|\alpha_{j+1}-\alpha_{j}|^{2}\nonumber\\
&+\big(2|\alpha_{j-1}|^{4}+|\alpha_{j+1}|^{2}|\alpha_{j}|^{2}+|\alpha_{j+1}|^{2}|\alpha_{j-1}|^{2}+|\alpha_{j}|^{2}|\alpha_{j-1}|^{2}\big)|\alpha_{j}-\alpha_{j-1}|^{2}\nonumber\\
&+\big(|\alpha_{j+1}|^{2}+|\alpha_{j-1}|^{2}\big)|\alpha_{j+1}-\alpha_{j}|^{2}|\alpha_{j}-\alpha_{j-1}|^{2}\nonumber\\
&+\big(|\alpha_{j+1}-\alpha_{j}|^{4}+|\alpha_{j}-\alpha_{j-1}|^{4}\big)|\alpha_{j+1}-\alpha_{j-1}|^{2}\Big]\nonumber\\
&-\sum_{j=0}^{\infty}\rho_{j}^{2}\big(|\alpha_{j+1}|^{2}+|\alpha_{j-1}|^{2}\big)|\alpha_{j+1}-\alpha_{j-1}|^{2}\nonumber\\
&-\frac{1}{4}\sum_{j=0}^{\infty}|\alpha_{j+1}-\alpha_{j-1}|^{4}-\frac{3}{4}\sum_{j=0}^{\infty}|\alpha_{j}|^{4}\Big(|\alpha_{j+1}|^{4}+|\alpha_{j-1}|^{4}+|\alpha_{j+1}-\alpha_{j-1}|^{4}\Big)\nonumber\\
&-3\sum_{j=0}^{\infty}|\alpha_{j}|^{2}\big(|\alpha_{j+1}|^{2}+|\alpha_{j-1}|^{2}\big)|\alpha_{j+1}-\alpha_{j-1}|^{2}\nonumber\\
&-\frac{1}{8}\sum_{j=0}^{\infty}\Big[|\alpha_{j}|^{8}+|\alpha_{j-1}|^{8}+|\alpha_{j}-\alpha_{j-1}|^{8}+4\big(|\alpha_{j}|^{2}+|\alpha_{j-1}|^{2}\big)^{2}|\alpha_{j}-\alpha_{j-1}|^{4}\nonumber\\
&+2\big(|\alpha_{j}|^{4}+|\alpha_{j-1}|^{4}\big)|\alpha_{j}-\alpha_{j-1}|^{4}\Big]\nonumber\\
&+\frac{1}{2}\sum_{j=0}^{\infty}\Big[|\alpha_{j+2}|^{2}|\alpha_{j+2}-\alpha_{j+1}|^{2}+|\alpha_{j-1}|^{2}|\alpha_{j}-\alpha_{j-1}|^{2}\nonumber\\
&+\big(|\alpha_{j+2}|^{2}+|\alpha_{j-1}|^{2}\big)|\alpha_{j+1}-\alpha_{j}|^{2}+|\alpha_{j}|^{2}|\alpha_{j+1}-\alpha_{j-1}|^{2}\nonumber\\
&+|\alpha_{j+1}|^{2}|\alpha_{j+2}-\alpha_{j}|^{2}+\big(|\alpha_{j+2}|^{2}+2|\alpha_{j+1}|^{2}+2|\alpha_{j}|^{2}+|\alpha_{j-1}|^{2}\big)|\alpha_{j+2}-\alpha_{j-1}|^{2}\Big]\nonumber\\
&+\frac{1}{2}\sum_{j=0}^{\infty}\Big[|\alpha_{j+1}|^{6}+|\alpha_{j-1}|^{6}+|\alpha_{j}|^{4}|\alpha_{j-1}|^{2}+|\alpha_{j+1}|^{2}|\alpha_{j}|^{4}\nonumber\\
&+\big(|\alpha_{j+1}|^{2}+|\alpha_{j-1}|^{2}\big)\big(|\alpha_{j+1}-\alpha_{j}|^{4}+|\alpha_{j}-\alpha_{j-1}|^{4}\big)\nonumber\\
&+|\alpha_{j-1}|^{4}|\alpha_{j+1}-\alpha_{j}|^{2}+|\alpha_{j+1}|^{4}|\alpha_{j}-\alpha_{j-1}|^{2}\nonumber\\
&+2\big(|\alpha_{j+1}|^{2}+|\alpha_{j}|^{2}\big)|\alpha_{j+1}-\alpha_{j-1}|^{2}|\alpha_{j+1}-\alpha_{j}|^{2}\nonumber\\
&+2\big(|\alpha_{j}|^{2}+|\alpha_{j-1}|^{2}\big)|\alpha_{j+1}-\alpha_{j-1}|^{2}|\alpha_{j}-\alpha_{j-1}|^{2}\Big]\nonumber\\
&+\sum_{j=0}^{\infty}|\alpha_{j}|^{2}\Big(|\alpha_{j+1}|^{4}+|\alpha_{j-1}|^{4}+|\alpha_{j+1}-\alpha_{j-1}|^{4}\Big)\nonumber\\
&+\frac{3}{2}\sum_{j=0}^{\infty}\big(1+|\alpha_{j}|^{4}\big)\big(|\alpha_{j+1}|^{2}+|\alpha_{j-1}|^{2}\big)|\alpha_{j+1}-\alpha_{j-1}|^{2}\nonumber\\
&+\frac{1}{2}\sum_{j=0}^{\infty}\Big[\big(|\alpha_{j}|^{2}+|\alpha_{j-1}|^{2}\big)|\alpha_{j}-\alpha_{j-1}|^{6}\nonumber\\
&+\big(|\alpha_{j}|^{6}+|\alpha_{j-1}|^{6}+|\alpha_{j}|^{4}|\alpha_{j-1}|^{2}+|\alpha_{j}|^{2}|\alpha_{j-1}|^{4}\big)|\alpha_{j}-\alpha_{j-1}|^{2}\Big]\Big\}\nonumber\\
&+\frac{1}{4}\sum_{j=0}^{\infty}|\alpha_{j}|^{2}|\alpha_{j+1}-\alpha_{j-1}|^{2}+\frac{1}{16}\sum_{j=0}^{\infty}\big(|\alpha_{j+2}|^{2}|-\alpha_{j-1}|^{2}\big)^{2}\nonumber
\end{align}
\begin{align}
&+\frac{1}{4}\sum_{j=0}^{\infty}\big(|\alpha_{j}|^{2}+|\alpha_{j-1}|^{2}\big)|\alpha_{j}-\alpha_{j-1}|^{2}+\frac{1}{2}\Big\{\frac{1}{2}\sum_{j=0}^{\infty}|\alpha_{j+1}|^{2}|\alpha_{j}|^{2}|\alpha_{j+2}-\alpha_{j-1}|^{2}\nonumber\\
&+\frac{1}{2}\sum_{j=0}^{\infty}|\alpha_{j+1}-\alpha_{j-1}|^{2}\big(|\alpha_{j+1}-\alpha_{j}|^{2}+|\alpha_{j}-\alpha_{j-1}|^{2}\big)\nonumber\\
&+\frac{1}{6}\sum_{j=0}^{\infty}|\alpha_{j}-\alpha_{j-1}|^{6}+\frac{1}{2}\sum_{j=0}^{\infty}\big(|\alpha_{j}|^{4}+|\alpha_{j}|^{2}|\alpha_{j-1}|^{2}+|\alpha_{j-1}|^{4})|\alpha_{j}-\alpha_{j-1}|^{2}\Big\}\nonumber\\
&+\frac{1}{8}\Big\{\frac{1}{2}\sum_{j=0}^{\infty}|\alpha_{j}|^{2}|\alpha_{j+1}|^{2}(|\alpha_{j+3}|^{2}+|\alpha_{j-1}|^{2})\nonumber\\
&+\frac{1}{2}\sum_{j=0}^{\infty}\big(|\alpha_{j+2}|^{2}+|\alpha_{j+1}|^{2}+|\alpha_{j}|^{2}+|\alpha_{j+2}|^{2}|\alpha_{j+1}|^{2}|\alpha_{j}|^{2}\big)|\alpha_{j+3}-\alpha_{j-1}|^{2}\nonumber\\
&+\frac{1}{2}\sum_{j=0}^{\infty}|\alpha_{j+2}|^{2}(|\alpha_{j+1}|^{2}+|\alpha_{j}|^{2})(|\alpha_{j+3}|^{2}+|\alpha_{j-1}|^{2})\nonumber\\
&+\frac{1}{2}\sum_{j=0}^{\infty}\big(|\alpha_{j+1}|^{2}+|\alpha_{j}|^{2}\big)\Big[|\alpha_{j+2}|^{4}+|\alpha_{j-1}|^{4}+|\alpha_{j+1}|^{2}|\alpha_{j-1}|^{2}+|\alpha_{j+2}|^{2}|\alpha_{j}|^{2}\nonumber\\
&+2|\alpha_{j+2}|^{2}|\alpha_{j+1}|^{2}+2|\alpha_{j}|^{2}|\alpha_{j-1}|^{2}\Big]\Big\}.
\end{align}
\end{thm}

\begin{proof}
Since
\begin{align}
&(1-\cos\theta)^{3}(1+\cos\theta)
=\frac{1}{8}\left(5-4\cos\theta-4\cos2\theta+4\cos3\theta-\cos4\theta\right)\nonumber\\
=&\frac{1}{2}(1-\cos\theta)+\frac{1}{2}(1-\cos2\theta)-\frac{1}{2}(1-\cos3\theta)+\frac{1}{8}(1-\cos4\theta),
\end{align}
then
\begin{align}
&Z_{4,3}(\mu)=\frac{5}{8}w_{0}-\frac{1}{2}\mathrm{Re}(w_{1})-\frac{1}{2}\mathrm{Re}(w_{2})+\frac{1}{2}\mathrm{Re}(w_{3})-\frac{1}{8}\mathrm{Re}(w_{4})\nonumber\\
=&\frac{1}{2}\big(w_{0}-\mathrm{Re}(w_{1})\big)+\frac{1}{2}\big(w_{0}-\mathrm{Re}(w_{2})\big)-\frac{1}{2}\big(w_{0}-\mathrm{Re}(w_{3})\big)
+\frac{1}{8}\big(w_{0}-\mathrm{Re}(w_{4})\big)\nonumber\\
=&\frac{1}{2}Z_{1}(\mu)+Z_{2,1}(\mu)-\frac{1}{2}Z_{3,1}(\mu)+\frac{1}{8}Z_{4,1}(\mu).
\end{align}

Therefore, by (4.2), (4.12), (4.60), (4.133) and (4.169), we have
\begin{align}
Z_{4,3}(\mu)=&\frac{1}{4}+\frac{1}{2}\sum_{n=0}^{\infty}\Big(\log(1-|\alpha_{n}|^{2})+|\alpha_{n}|^{2}\Big)-\frac{1}{4}\sum_{n=0}^{\infty}|\alpha_{n}-\alpha_{n-1}|^{2}\nonumber\\
&+\frac{3}{8}+\frac{1}{2}\sum_{j=0}^{\infty}\Big[\log(1-|\alpha_{j}|^{2})+|\alpha_{j}|^{2}+\frac{1}{2}|\alpha_{j}|^{4}\Big]\nonumber\\
&-\frac{1}{2}\sum_{j=0}^{\infty}|\alpha_{j}\alpha_{j-1}|^{2}-\frac{1}{4}\sum_{j=0}^{\infty}\rho_{j}^{2}|\alpha_{j+1}-\alpha_{j-1}|^{2}\nonumber
\end{align}
\begin{align}
&-\frac{1}{16}\sum_{j=0}^{\infty}\left[\big(2|\alpha_{j}|^{2}-|\alpha_{j}-\alpha_{j-1}|^{2}\big)^{2}+\big(2|\alpha_{j-1}|^{2}-|\alpha_{j}-\alpha_{j-1}|^{2}\big)^{2}\right]\nonumber\\
&-\frac{1}{2}\Big\{\frac{2}{3}-\frac{1}{2}(|\alpha_{0}|^{2}+|\alpha_{1}|^{2})-\frac{1}{2}|\alpha_{0}|^{2}|\alpha_{0}+1|^{2}\nonumber\\
&+\sum_{j=0}^{\infty}\Big[\log(1-|\alpha_{j}|^{2})+|\alpha_{j}|^{2}+\frac{1}{2}|\alpha_{j}|^{4}+\frac{1}{3}|\alpha_{j}|^{6}\Big]\nonumber\\
&-\frac{1}{2}\sum_{j=0}^{\infty}|\alpha_{j}|^{4}-\frac{1}{2}\sum_{j=0}^{\infty}(|\alpha_{j+2}|^{2}+|\alpha_{j-1}|^{2})|\alpha_{j}|^{2}\nonumber\\
&-\frac{1}{2}\sum_{j=0}^{\infty}(|\alpha_{j+2}|^{2}+|\alpha_{j-1}|^{2})|\alpha_{j+1}|^{2}\rho_{j}^{2}-\frac{1}{2}\sum_{j=0}^{\infty}|\alpha_{j+2}-\alpha_{j-1}|^{2}\rho_{j+1}^{2}\rho_{j}^{2}\nonumber\\
&-\frac{1}{2}\sum_{j=0}^{\infty}\Big[|\alpha_{j+1}|^{2}|\alpha_{j}|^{2}+|\alpha_{j-1}|^{4}+|\alpha_{j}|^{2}|\alpha_{j-1}|^{2}+|\alpha_{j+1}|^{4}\nonumber\\
&+|\alpha_{j+1}-\alpha_{j-1}|^{2}\big(|\alpha_{j+1}-\alpha_{j}|^{2}+|\alpha_{j}-\alpha_{j-1}|^{2}\big)\Big]\rho_{j}^{2}\nonumber\\
&-\frac{1}{2}\sum_{j=0}^{\infty}\Big[(|\alpha_{j+1}|^{2}+2|\alpha_{j}|^{2}+|\alpha_{j-1}|^{2})|\alpha_{j+1}-\alpha_{j-1}|^{2}\nonumber\\
&+|\alpha_{j-1}|^{2}|\alpha_{j}-\alpha_{j-1}|^{2}+|\alpha_{j+1}|^{2}|\alpha_{j+1}-\alpha_{j}|^{2}\Big]|\alpha_{j}|^{2}\nonumber\\
&-\frac{1}{6}\sum_{j=0}^{\infty}|\alpha_{j}-\alpha_{j-1}|^{6}-\frac{1}{2}\sum_{j=0}^{\infty}\big(|\alpha_{j}|^{4}+|\alpha_{j}|^{2}|\alpha_{j-1}|^{2}+|\alpha_{j-1}|^{4})|\alpha_{j}-\alpha_{j-1}|^{2}\nonumber\\
&+\frac{1}{2}\sum_{j=0}^{\infty}\Big[(|\alpha_{j+1}|^{2}+2|\alpha_{j}|^{2}+|\alpha_{j-1}|^{2})|\alpha_{j+1}-\alpha_{j-1}|^{2}\nonumber\\
&+(|\alpha_{j}|^{2}+|\alpha_{j-1}|^{2})(|\alpha_{j}-\alpha_{j-1}|^{2}+|\alpha_{j}-\alpha_{j-1}|^{4})\Big]\Big\}\nonumber\\
&+\frac{1}{8}\Big\{\sum_{j=0}^{\infty}\log(1-|\alpha_{j}|^{2})+\frac{1}{2}\sum_{j=0}^{\infty}\rho_{j+2}^{2}\rho_{j+1}^{2}\rho_{j}^{2}(|\alpha_{j+3}|^{2}+|\alpha_{j-1}|^{2}-|\alpha_{j+3}-\alpha_{j-1}|^{2})\nonumber\\
&-\frac{1}{2}\sum_{j=0}^{\infty}\rho_{j+1}^{2}\rho_{j}^{2}\Big[|\alpha_{j+2}|^{4}+|\alpha_{j-1}|^{4}+|\alpha_{j+1}|^{2}|\alpha_{j-1}|^{2}+|\alpha_{j+2}|^{2}|\alpha_{j}|^{2}\nonumber\\
&+2|\alpha_{j+2}|^{2}|\alpha_{j+1}|^{2}+2|\alpha_{j}|^{2}|\alpha_{j-1}|^{2}-\big(|\alpha_{j+2}|^{2}-|\alpha_{j}|^{2}\big)|\alpha_{j+2}-\alpha_{j+1}|^{2}\nonumber\\
&-\big(|\alpha_{j+2}|^{2}+|\alpha_{j-1}|^{2}\big)|\alpha_{j+1}-\alpha_{j}|^{2}-|\alpha_{j}|^{2}|\alpha_{j+1}-\alpha_{j-1}|^{2}\nonumber\\
&-|\alpha_{j+1}|^{2}|\alpha_{j+2}-\alpha_{j}|^{2}+\big(|\alpha_{j+1}|^{2}-|\alpha_{j-1}|^{2}\big)|\alpha_{j}-\alpha_{j-1}|^{2}\nonumber\\
&-\big(|\alpha_{j+2}|^{2}+2|\alpha_{j+1}|^{2}+2|\alpha_{j}|^{2}+|\alpha_{j-1}|^{2}\big)|\alpha_{j+2}-\alpha_{j-1}|^{2}\nonumber\\
&+|\alpha_{j+2}-\alpha_{j+1}|^{2}|\alpha_{j+2}-\alpha_{j-1}|^{2}+|\alpha_{j+2}-\alpha_{j-1}|^{2}|\alpha_{j}-\alpha_{j-1}|^{2}\Big]\nonumber
\end{align}
\begin{align}
&+\frac{1}{2}\sum_{j=0}^{\infty}\rho_{j}^{2}\Big[|\alpha_{j+1}|^{6}+|\alpha_{j-1}|^{6}+|\alpha_{j}|^{4}|\alpha_{j-1}|^{2}+|\alpha_{j+1}|^{2}|\alpha_{j}|^{4}\nonumber\\
&-\big(|\alpha_{j+1}|^{4}+2|\alpha_{j}|^{4}+|\alpha_{j-1}|^{4}+|\alpha_{j+1}|^{2}|\alpha_{j}|^{2}+|\alpha_{j}|^{2}|\alpha_{j-1}|^{2}\big)|\alpha_{j+1}-\alpha_{j-1}|^{2}\nonumber\\
&+\big(|\alpha_{j+1}|^{2}+|\alpha_{j-1}|^{2}\big)\big(|\alpha_{j+1}-\alpha_{j}|^{4}+|\alpha_{j}-\alpha_{j-1}|^{4}\big)\nonumber\\
&-\big(2|\alpha_{j+1}|^{4}-|\alpha_{j-1}|^{4}+|\alpha_{j+1}|^{2}|\alpha_{j}|^{2}+|\alpha_{j+1}|^{2}|\alpha_{j-1}|^{2}+|\alpha_{j}|^{2}|\alpha_{j-1}|^{2}\big)|\alpha_{j+1}-\alpha_{j}|^{2}\nonumber\\
&-\big(2|\alpha_{j-1}|^{4}-|\alpha_{j+1}|^{4}+|\alpha_{j+1}|^{2}|\alpha_{j}|^{2}+|\alpha_{j+1}|^{2}|\alpha_{j-1}|^{2}+|\alpha_{j}|^{2}|\alpha_{j-1}|^{2}\big)|\alpha_{j}-\alpha_{j-1}|^{2}\nonumber\\
&+2\big(|\alpha_{j+1}|^{2}+|\alpha_{j}|^{2}\big)|\alpha_{j+1}-\alpha_{j-1}|^{2}|\alpha_{j+1}-\alpha_{j}|^{2}\nonumber\\
&+2\big(|\alpha_{j}|^{2}+|\alpha_{j-1}|^{2}\big)|\alpha_{j+1}-\alpha_{j-1}|^{2}|\alpha_{j}-\alpha_{j-1}|^{2}\nonumber\\
&-\big(|\alpha_{j+1}|^{2}+|\alpha_{j-1}|^{2}\big)|\alpha_{j+1}-\alpha_{j}|^{2}|\alpha_{j}-\alpha_{j-1}|^{2}\nonumber\\
&-\big(|\alpha_{j+1}-\alpha_{j}|^{4}+|\alpha_{j}-\alpha_{j-1}|^{4}\big)|\alpha_{j+1}-\alpha_{j-1}|^{2}\Big]\nonumber\\
&+\frac{1}{2}\sum_{j=0}^{\infty}\rho_{j}^{2}\Big[|\alpha_{j+1}|^{4}+|\alpha_{j-1}|^{4}+|\alpha_{j+1}-\alpha_{j-1}|^{4}-2\big(|\alpha_{j+1}|^{2}+|\alpha_{j-1}|^{2}\big)|\alpha_{j+1}-\alpha_{j-1}|^{2}\Big]\nonumber\\
&-\frac{3}{4}\sum_{j=0}^{\infty}\rho_{j}^{4}\Big[|\alpha_{j+1}|^{4}+|\alpha_{j-1}|^{4}+|\alpha_{j+1}-\alpha_{j-1}|^{4}-2\big(|\alpha_{j+1}|^{2}+|\alpha_{j-1}|^{2}\big)|\alpha_{j+1}-\alpha_{j-1}|^{2}\Big]\nonumber\\
&-\frac{1}{8}\sum_{j=0}^{\infty}\Big[|\alpha_{j}|^{8}+|\alpha_{j-1}|^{8}+|\alpha_{j}-\alpha_{j-1}|^{8}+4\big(|\alpha_{j}|^{2}+|\alpha_{j-1}|^{2}\big)^{2}|\alpha_{j}-\alpha_{j-1}|^{4}\nonumber\\
&+2\big(|\alpha_{j}|^{4}+|\alpha_{j-1}|^{4}\big)|\alpha_{j}-\alpha_{j-1}|^{4}-4\big(|\alpha_{j}|^{2}+|\alpha_{j-1}|^{2}\big)|\alpha_{j}-\alpha_{j-1}|^{6}\nonumber\\
&-4\big(|\alpha_{j}|^{6}+|\alpha_{j-1}|^{6}+|\alpha_{j}|^{4}|\alpha_{j-1}|^{2}+|\alpha_{j}|^{2}|\alpha_{j-1}|^{4}\big)|\alpha_{j}-\alpha_{j-1}|^{2}\Big]\Big\}\nonumber\\
=&\frac{37}{96}-\frac{1}{8}|\alpha_{0}|^{2}+\frac{1}{16}|\alpha_{1}|^{2}+\frac{1}{16}|\alpha_{2}|^{2}(|\alpha_{0}|^{2}+|\alpha_{1}|^{2})-\frac{1}{16}|\alpha_{1}|^{2}|\alpha_{0}|^{2}\nonumber\\
&-\frac{7}{32}|\alpha_{0}|^{4}+\frac{1}{8}|\alpha_{1}|^{4}+\frac{1}{8}\sum_{j=0}^{2}|\alpha_{j}-\alpha_{j-1}|^{2}+\frac{1}{4}|1+\alpha_{1}|^{2}-\frac{1}{8}|1+\alpha_{2}|^{2}\nonumber\\
&+\frac{5}{8}\sum_{j=0}^{\infty}\Big(\log(1-|\alpha_{j}|^{2})+|\alpha_{j}|^{2}+\frac{1}{2}|\alpha_{j}|^{4}\Big)\nonumber\\
&-\frac{1}{16}|\alpha_{j+3}-2\alpha_{j+2}+2\alpha_{j}-\alpha_{j-1}|^{2}-\frac{1}{8}\sum_{j=0}^{\infty}|\alpha_{j}-\alpha_{j-1}|^{4}\nonumber\\
&+\frac{1}{2}\Big\{-\frac{1}{3}\sum_{j=0}^{\infty}|\alpha_{j}|^{6}-\frac{1}{2}\sum_{j=0}^{\infty}|\alpha_{j+1}|^{2}|\alpha_{j}|^{2}(|\alpha_{j+2}|^{2}+|\alpha_{j-1}|^{2})\nonumber\\
&-\frac{1}{2}\sum_{j=0}^{\infty}\big(|\alpha_{j+1}|^{2}+|\alpha_{j}|^{2}\big)|\alpha_{j+2}-\alpha_{j-1}|^{2}\nonumber\\
&-\frac{1}{2}\sum_{j=0}^{\infty}|\alpha_{j}|^{2}\Big[|\alpha_{j+1}|^{2}|\alpha_{j}|^{2}+|\alpha_{j-1}|^{4}+|\alpha_{j}|^{2}|\alpha_{j-1}|^{2}+|\alpha_{j+1}|^{4}\nonumber\\
&+|\alpha_{j+1}-\alpha_{j-1}|^{2}\big(|\alpha_{j+1}-\alpha_{j}|^{2}+|\alpha_{j}-\alpha_{j-1}|^{2}\big)\Big]\nonumber
\end{align}
\begin{align}
&-\frac{1}{2}\sum_{j=0}^{\infty}\big(|\alpha_{j}|^{2}+|\alpha_{j-1}|^{2}\big)|\alpha_{j}-\alpha_{j-1}|^{4}\nonumber\\
&-\frac{1}{2}\sum_{j=0}^{\infty}\rho_{j}^{2}\Big[(|\alpha_{j+1}|^{2}+2|\alpha_{j}|^{2}+|\alpha_{j-1}|^{2})|\alpha_{j+1}-\alpha_{j-1}|^{2}\nonumber\\
&+|\alpha_{j-1}|^{2}|\alpha_{j}-\alpha_{j-1}|^{2}+|\alpha_{j+1}|^{2}|\alpha_{j+1}-\alpha_{j}|^{2}\Big]\Big\}\nonumber\\
&-\frac{1}{16}\sum_{j=0}^{\infty}\big(|\alpha_{j}|^{2}-|\alpha_{j-1}|^{2}\big)^{2}
-\frac{1}{8}\sum_{j=0}^{\infty}\big(|\alpha_{j+1}|^{2}-|\alpha_{j-1}|^{2}\big)^{2}\nonumber\\
&+\frac{1}{8}\Big\{-\frac{1}{2}\sum_{j=0}^{\infty}|\alpha_{j+2}|^{2}|\alpha_{j+1}|^{2}|\alpha_{j}|^{2}\big(|\alpha_{j+3}|^{2}+|\alpha_{j-1}|^{2}\big)\nonumber\\
&-\frac{1}{2}\sum_{j=0}^{\infty}\big(|\alpha_{j+2}|^{2}|\alpha_{j+1}|^{2}+|\alpha_{j+2}|^{2}|\alpha_{j}|^{2}+|\alpha_{j+1}|^{2}|\alpha_{j}|^{2}\big)|\alpha_{j+3}-\alpha_{j-1}|^{2}\nonumber\\
&-\frac{1}{2}\sum_{j=0}^{\infty}|\alpha_{j+1}|^{2}|\alpha_{j}|^{2}\Big[|\alpha_{j+2}|^{4}+|\alpha_{j-1}|^{4}+|\alpha_{j+1}|^{2}|\alpha_{j-1}|^{2}+|\alpha_{j+2}|^{2}|\alpha_{j}|^{2}\nonumber\\
&+2|\alpha_{j+2}|^{2}|\alpha_{j+1}|^{2}+2|\alpha_{j}|^{2}|\alpha_{j-1}|^{2}\Big]\nonumber\\
&-\frac{1}{2}\sum_{j=0}^{\infty}\rho_{j+1}^{2}\rho_{j}^{2}\Big[|\alpha_{j}|^{2}|\alpha_{j+2}-\alpha_{j+1}|^{2}+|\alpha_{j+1}|^{2}|\alpha_{j}-\alpha_{j-1}|^{2}\nonumber\\
&+|\alpha_{j+2}-\alpha_{j+1}|^{2}|\alpha_{j+2}-\alpha_{j-1}|^{2}+|\alpha_{j+2}-\alpha_{j-1}|^{2}|\alpha_{j}-\alpha_{j-1}|^{2}\Big]\nonumber\\
&-\frac{1}{2}\sum_{j=0}^{\infty}|\alpha_{j}|^{2}\Big[|\alpha_{j+2}|^{2}|\alpha_{j+2}-\alpha_{j+1}|^{2}+|\alpha_{j-1}|^{2}|\alpha_{j}-\alpha_{j-1}|^{2}\nonumber\\
&+\big(|\alpha_{j+2}|^{2}+|\alpha_{j-1}|^{2}\big)|\alpha_{j+1}-\alpha_{j}|^{2}+|\alpha_{j}|^{2}|\alpha_{j+1}-\alpha_{j-1}|^{2}\nonumber\\
&+|\alpha_{j+1}|^{2}|\alpha_{j+2}-\alpha_{j}|^{2}+\big(|\alpha_{j+2}|^{2}+2|\alpha_{j+1}|^{2}+2|\alpha_{j}|^{2}+|\alpha_{j-1}|^{2}\big)|\alpha_{j+2}-\alpha_{j-1}|^{2}\Big]\nonumber\\
&-\frac{1}{2}\sum_{j=0}^{\infty}\rho_{j}^{2}|\alpha_{j+1}|^{2}\Big[|\alpha_{j+2}|^{2}|\alpha_{j+2}-\alpha_{j+1}|^{2}+|\alpha_{j-1}|^{2}|\alpha_{j}-\alpha_{j-1}|^{2}\nonumber\\
&+\big(|\alpha_{j+2}|^{2}+|\alpha_{j-1}|^{2}\big)|\alpha_{j+1}-\alpha_{j}|^{2}+|\alpha_{j}|^{2}|\alpha_{j+1}-\alpha_{j-1}|^{2}\nonumber\\
&+|\alpha_{j+1}|^{2}|\alpha_{j+2}-\alpha_{j}|^{2}+\big(|\alpha_{j+2}|^{2}+2|\alpha_{j+1}|^{2}+2|\alpha_{j}|^{2}+|\alpha_{j-1}|^{2}\big)|\alpha_{j+2}-\alpha_{j-1}|^{2}\Big]\nonumber\\
&-\frac{1}{2}\sum_{j=0}^{\infty}|\alpha_{j}|^{2}\Big[|\alpha_{j+1}|^{6}+|\alpha_{j-1}|^{6}+|\alpha_{j}|^{4}|\alpha_{j-1}|^{2}+|\alpha_{j+1}|^{2}|\alpha_{j}|^{4}\nonumber\\
&+\big(|\alpha_{j+1}|^{2}+|\alpha_{j-1}|^{2}\big)\big(|\alpha_{j+1}-\alpha_{j}|^{4}+|\alpha_{j}-\alpha_{j-1}|^{4}\big)\nonumber\\
&+|\alpha_{j-1}|^{4}|\alpha_{j+1}-\alpha_{j}|^{2}+|\alpha_{j+1}|^{4}|\alpha_{j}-\alpha_{j-1}|^{2}\nonumber\\
&+2\big(|\alpha_{j+1}|^{2}+|\alpha_{j}|^{2}\big)|\alpha_{j+1}-\alpha_{j-1}|^{2}|\alpha_{j+1}-\alpha_{j}|^{2}\nonumber\\
&+2\big(|\alpha_{j}|^{2}+|\alpha_{j-1}|^{2}\big)|\alpha_{j+1}-\alpha_{j-1}|^{2}|\alpha_{j}-\alpha_{j-1}|^{2}\Big]\nonumber\\
&-\frac{1}{2}\sum_{j=0}^{\infty}\rho_{j}^{2}\Big[\big(|\alpha_{j+1}|^{4}+2|\alpha_{j}|^{4}+|\alpha_{j-1}|^{4}+|\alpha_{j}|^{2}(|\alpha_{j+1}|^{2}+|\alpha_{j-1}|^{2})\big)|\alpha_{j+1}-\alpha_{j-1}|^{2}\nonumber\\
&+\big(2|\alpha_{j+1}|^{4}+|\alpha_{j+1}|^{2}|\alpha_{j}|^{2}+|\alpha_{j+1}|^{2}|\alpha_{j-1}|^{2}+|\alpha_{j}|^{2}|\alpha_{j-1}|^{2}\big)|\alpha_{j+1}-\alpha_{j}|^{2}\nonumber
\end{align}
\begin{align}
&+\big(2|\alpha_{j-1}|^{4}+|\alpha_{j+1}|^{2}|\alpha_{j}|^{2}+|\alpha_{j+1}|^{2}|\alpha_{j-1}|^{2}+|\alpha_{j}|^{2}|\alpha_{j-1}|^{2}\big)|\alpha_{j}-\alpha_{j-1}|^{2}\nonumber\\
&+\big(|\alpha_{j+1}|^{2}+|\alpha_{j-1}|^{2}\big)|\alpha_{j+1}-\alpha_{j}|^{2}|\alpha_{j}-\alpha_{j-1}|^{2}\nonumber\\
&+\big(|\alpha_{j+1}-\alpha_{j}|^{4}+|\alpha_{j}-\alpha_{j-1}|^{4}\big)|\alpha_{j+1}-\alpha_{j-1}|^{2}\Big]\nonumber\\
&-\sum_{j=0}^{\infty}\rho_{j}^{2}\big(|\alpha_{j+1}|^{2}+|\alpha_{j-1}|^{2}\big)|\alpha_{j+1}-\alpha_{j-1}|^{2}\nonumber\\
&-\frac{1}{4}\sum_{j=0}^{\infty}|\alpha_{j+1}-\alpha_{j-1}|^{4}-\frac{3}{4}\sum_{j=0}^{\infty}|\alpha_{j}|^{4}\Big(|\alpha_{j+1}|^{4}+|\alpha_{j-1}|^{4}+|\alpha_{j+1}-\alpha_{j-1}|^{4}\Big)\nonumber\\
&-3\sum_{j=0}^{\infty}|\alpha_{j}|^{2}\big(|\alpha_{j+1}|^{2}+|\alpha_{j-1}|^{2}\big)|\alpha_{j+1}-\alpha_{j-1}|^{2}\nonumber\\
&-\frac{1}{8}\sum_{j=0}^{\infty}\Big[|\alpha_{j}|^{8}+|\alpha_{j-1}|^{8}+|\alpha_{j}-\alpha_{j-1}|^{8}+4\big(|\alpha_{j}|^{2}+|\alpha_{j-1}|^{2}\big)^{2}|\alpha_{j}-\alpha_{j-1}|^{4}\nonumber\\
&+2\big(|\alpha_{j}|^{4}+|\alpha_{j-1}|^{4}\big)|\alpha_{j}-\alpha_{j-1}|^{4}\Big]\nonumber\\
&+\frac{1}{2}\sum_{j=0}^{\infty}\Big[|\alpha_{j+2}|^{2}|\alpha_{j+2}-\alpha_{j+1}|^{2}+|\alpha_{j-1}|^{2}|\alpha_{j}-\alpha_{j-1}|^{2}\nonumber\\
&+\big(|\alpha_{j+2}|^{2}+|\alpha_{j-1}|^{2}\big)|\alpha_{j+1}-\alpha_{j}|^{2}+|\alpha_{j}|^{2}|\alpha_{j+1}-\alpha_{j-1}|^{2}\nonumber\\
&+|\alpha_{j+1}|^{2}|\alpha_{j+2}-\alpha_{j}|^{2}+\big(|\alpha_{j+2}|^{2}+2|\alpha_{j+1}|^{2}+2|\alpha_{j}|^{2}+|\alpha_{j-1}|^{2}\big)|\alpha_{j+2}-\alpha_{j-1}|^{2}\Big]\nonumber\\
&+\frac{1}{2}\sum_{j=0}^{\infty}\Big[|\alpha_{j+1}|^{6}+|\alpha_{j-1}|^{6}+|\alpha_{j}|^{4}|\alpha_{j-1}|^{2}+|\alpha_{j+1}|^{2}|\alpha_{j}|^{4}\nonumber\\
&+\big(|\alpha_{j+1}|^{2}+|\alpha_{j-1}|^{2}\big)\big(|\alpha_{j+1}-\alpha_{j}|^{4}+|\alpha_{j}-\alpha_{j-1}|^{4}\big)\nonumber\\
&+|\alpha_{j-1}|^{4}|\alpha_{j+1}-\alpha_{j}|^{2}+|\alpha_{j+1}|^{4}|\alpha_{j}-\alpha_{j-1}|^{2}\nonumber\\
&+2\big(|\alpha_{j+1}|^{2}+|\alpha_{j}|^{2}\big)|\alpha_{j+1}-\alpha_{j-1}|^{2}|\alpha_{j+1}-\alpha_{j}|^{2}\nonumber\\
&+2\big(|\alpha_{j}|^{2}+|\alpha_{j-1}|^{2}\big)|\alpha_{j+1}-\alpha_{j-1}|^{2}|\alpha_{j}-\alpha_{j-1}|^{2}\Big]\nonumber\\
&+\sum_{j=0}^{\infty}|\alpha_{j}|^{2}\Big(|\alpha_{j+1}|^{4}+|\alpha_{j-1}|^{4}+|\alpha_{j+1}-\alpha_{j-1}|^{4}\Big)\nonumber\\
&+\frac{3}{2}\sum_{j=0}^{\infty}\big(1+|\alpha_{j}|^{4}\big)\big(|\alpha_{j+1}|^{2}+|\alpha_{j-1}|^{2}\big)|\alpha_{j+1}-\alpha_{j-1}|^{2}\nonumber\\
&+\frac{1}{2}\sum_{j=0}^{\infty}\Big[\big(|\alpha_{j}|^{2}+|\alpha_{j-1}|^{2}\big)|\alpha_{j}-\alpha_{j-1}|^{6}\nonumber\\
&+\big(|\alpha_{j}|^{6}+|\alpha_{j-1}|^{6}+|\alpha_{j}|^{4}|\alpha_{j-1}|^{2}+|\alpha_{j}|^{2}|\alpha_{j-1}|^{4}\big)|\alpha_{j}-\alpha_{j-1}|^{2}\Big]\Big\}\nonumber
\end{align}
\begin{align}
&+\frac{1}{4}\sum_{j=0}^{\infty}|\alpha_{j}|^{2}|\alpha_{j+1}-\alpha_{j-1}|^{2}+\frac{1}{16}\sum_{j=0}^{\infty}\big(|\alpha_{j+2}|^{2}|-\alpha_{j-1}|^{2}\big)^{2}\nonumber\\
&+\frac{1}{4}\sum_{j=0}^{\infty}\big(|\alpha_{j}|^{2}+|\alpha_{j-1}|^{2}\big)|\alpha_{j}-\alpha_{j-1}|^{2}+\frac{1}{2}\Big\{\frac{1}{2}\sum_{j=0}^{\infty}|\alpha_{j+1}|^{2}|\alpha_{j}|^{2}|\alpha_{j+2}-\alpha_{j-1}|^{2}\nonumber\\
&+\frac{1}{2}\sum_{j=0}^{\infty}|\alpha_{j+1}-\alpha_{j-1}|^{2}\big(|\alpha_{j+1}-\alpha_{j}|^{2}+|\alpha_{j}-\alpha_{j-1}|^{2}\big)\nonumber\\
&+\frac{1}{6}\sum_{j=0}^{\infty}|\alpha_{j}-\alpha_{j-1}|^{6}+\frac{1}{2}\sum_{j=0}^{\infty}\big(|\alpha_{j}|^{4}+|\alpha_{j}|^{2}|\alpha_{j-1}|^{2}+|\alpha_{j-1}|^{4})|\alpha_{j}-\alpha_{j-1}|^{2}\Big\}\nonumber\\
&+\frac{1}{8}\Big\{\frac{1}{2}\sum_{j=0}^{\infty}|\alpha_{j}|^{2}|\alpha_{j+1}|^{2}(|\alpha_{j+3}|^{2}+|\alpha_{j-1}|^{2})\nonumber\\
&+\frac{1}{2}\sum_{j=0}^{\infty}\big(|\alpha_{j+2}|^{2}+|\alpha_{j+1}|^{2}+|\alpha_{j}|^{2}+|\alpha_{j+2}|^{2}|\alpha_{j+1}|^{2}|\alpha_{j}|^{2}\big)|\alpha_{j+3}-\alpha_{j-1}|^{2}\nonumber\\
&+\frac{1}{2}\sum_{j=0}^{\infty}|\alpha_{j+2}|^{2}(|\alpha_{j+1}|^{2}+|\alpha_{j}|^{2})(|\alpha_{j+3}|^{2}+|\alpha_{j-1}|^{2})\nonumber\\
&+\frac{1}{2}\sum_{j=0}^{\infty}\big(|\alpha_{j+1}|^{2}+|\alpha_{j}|^{2}\big)\Big[|\alpha_{j+2}|^{4}+|\alpha_{j-1}|^{4}+|\alpha_{j+1}|^{2}|\alpha_{j-1}|^{2}+|\alpha_{j+2}|^{2}|\alpha_{j}|^{2}\nonumber\\
&+2|\alpha_{j+2}|^{2}|\alpha_{j+1}|^{2}+2|\alpha_{j}|^{2}|\alpha_{j-1}|^{2}\Big]\Big\},\nonumber
\end{align}
where the following identity is used
\begin{align*}
&|\alpha_{j+3}-2\alpha_{j+2}+2\alpha_{j}-\alpha_{j-1}|^{2}\\
=&2|\alpha_{j+3}-\alpha_{j+2}|^{2}-2|\alpha_{j+3}-\alpha_{j}|^{2}+|\alpha_{j+3}-\alpha_{j-1}|^{2}+4|\alpha_{j+2}-\alpha_{j}|^{2}\\
&-2|\alpha_{j+2}-\alpha_{j-1}|^{2}+2|\alpha_{j}-\alpha_{j-1}|^{2}
\end{align*}
which implying
\begin{align*}
&-\frac{1}{16}\sum_{j=0}^{\infty}|\alpha_{j+3}-2\alpha_{j+2}+2\alpha_{j}-\alpha_{j-1}|^{2}\\
=&\mathrm{bdy}-\frac{1}{16}\sum_{j=0}^{\infty}|\alpha_{j+3}-\alpha_{j-1}|^{2}+\frac{1}{4}\sum_{j=0}^{\infty}|\alpha_{j+2}-\alpha_{j-1}|^{2}-\frac{1}{4}\sum_{j=0}^{\infty}|\alpha_{j+1}-\alpha_{j-1}|^{2}\\
&-\frac{1}{4}\sum_{j=0}^{\infty}|\alpha_{j+1}-\alpha_{j-1}|^{2}
\end{align*}
with boundary
\begin{align*}
\mathrm{bdy}=-\frac{1}{8}\sum_{j=0}^{2}|\alpha_{j}-\alpha_{j-1}|^{2}-\frac{1}{4}|1+\alpha_{1}|^{2}+\frac{1}{8}|1+\alpha_{2}|^{2}.\,\,\,\,\hspace{6mm}\qedhere
\end{align*}
\end{proof}

\begin{thm}
Assume $\alpha\in\ell^{6}$ and $(S-1)\alpha\in \ell^{3}$, then
\begin{align}
&\int_{0}^{2\pi}(1-\cos\theta)^{3}(1+\cos\theta)\log w(\theta)\frac{d\theta}{2\pi}>-\infty\,\,\Longleftrightarrow\,\,  (S-1)^{3}(S+1)\alpha\in \ell^{2}.
\end{align}
\end{thm}

\begin{rem}
This is still a one-case result of the original Simon conjecture under certain conditions.
\end{rem}

\begin{proof}
By applying (4.154), (4.155) and (4.161) to the sum rule (4.167), we have that if $\alpha\in\ell^{6}$ and $(S-1)\alpha\in \ell^{3}$, then $\mathrm{CP}<+\infty$ and the series in $\mathrm{EP}$ are convergent except the following one
\begin{equation}
-\frac{1}{16}|\alpha_{j+3}-2\alpha_{j+2}+2\alpha_{j}-\alpha_{j-1}|^{2}.
\end{equation}
So $Z_{4,3}(\mu)>-\infty$ if and only if $|\alpha_{j+3}-2\alpha_{j+2}+2\alpha_{j}-\alpha_{j-1}|^{2}<+\infty$ (i.e., $(S-1)^{3}(S-1)\alpha\in\ell^{2}$) as $\alpha\in\ell^{6}$ and $(S-1)\alpha\in \ell^{3}$.
\end{proof}

\begin{thm}
Assume $\alpha\in\ell^{4}$, then
\begin{align}
&\int_{0}^{2\pi}(1-\cos\theta)^{3}(1+\cos\theta)\log w(\theta)\frac{d\theta}{2\pi}>-\infty\,\,\Longleftrightarrow\,\, (S-1)^{3}(S+1)\alpha\in \ell^{2}.
\end{align}
\end{thm}

\begin{rem}
This is one special case of a result in \cite{gz} due to Golinskii and Zlato\v{s}.
\end{rem}

\begin{proof}
As $\alpha\in\ell^{4}$, by the sum rule (4.167) and H\"older inequality, we have that the series in it are convergent except the only one (4.171). Thus $Z_{4,3}(\mu)>-\infty$ if and only if $(S-1)^{3}(S+1)\alpha\in\ell^{2}$ as $\alpha\in\ell^{4}$.
\end{proof}

\begin{thm}
Assume $\alpha\in\ell^{8}$ and $(S-1)\alpha\in \ell^{2}$, then
\begin{align}
&\int_{0}^{2\pi}(1-\cos\theta)^{3}(1+\cos\theta)\log w(\theta)\frac{d\theta}{2\pi}>-\infty\,\, \Longleftrightarrow \,\, \alpha\in \ell^{6}.
\end{align}
\end{thm}

\begin{proof}
Note that
\begin{align}
&\frac{1}{2}\Big\{-\frac{1}{3}\sum_{j=0}^{\infty}|\alpha_{j}|^{6}-\frac{1}{2}\sum_{j=0}^{\infty}|\alpha_{j+1}|^{2}|\alpha_{j}|^{2}(|\alpha_{j+2}|^{2}+|\alpha_{j-1}|^{2})\nonumber\\
&-\frac{1}{2}\sum_{j=0}^{\infty}|\alpha_{j}|^{2}\Big[|\alpha_{j+1}|^{2}|\alpha_{j}|^{2}+|\alpha_{j-1}|^{4}+|\alpha_{j}|^{2}|\alpha_{j-1}|^{2}+|\alpha_{j+1}|^{4}\Big]\Big\}\nonumber
\end{align}
\begin{align}
&+\frac{1}{8}\Big\{\sum_{j=0}^{\infty}|\alpha_{j}|^{2}\big(|\alpha_{j+1}|^{4}+|\alpha_{j-1}|^{4}\big)\nonumber\\
&+\frac{1}{2}\sum_{j=0}^{\infty}|\alpha_{j}|^{2}|\alpha_{j+1}|^{2}(|\alpha_{j+3}|^{2}+|\alpha_{j-1}|^{2})\nonumber\\
&+\frac{1}{2}\sum_{j=0}^{\infty}|\alpha_{j+2}|^{2}(|\alpha_{j+1}|^{2}+|\alpha_{j}|^{2})(|\alpha_{j+3}|^{2}+|\alpha_{j-1}|^{2})\nonumber\\
&+\frac{1}{2}\sum_{j=0}^{\infty}\big(|\alpha_{j+1}|^{2}+|\alpha_{j}|^{2}\big)\Big[|\alpha_{j+2}|^{4}+|\alpha_{j-1}|^{4}+|\alpha_{j+1}|^{2}|\alpha_{j-1}|^{2}+|\alpha_{j+2}|^{2}|\alpha_{j}|^{2}\nonumber\\
&+2|\alpha_{j+2}|^{2}|\alpha_{j+1}|^{2}+2|\alpha_{j}|^{2}|\alpha_{j-1}|^{2}\Big]\Big\}\nonumber\\
=&\frac{1}{24}+\frac{1}{16}|\alpha_{1}|^{2}|\alpha_{0}|^{2}+\frac{1}{4}|\alpha_{0}|^{2}+\frac{1}{8}|\alpha_{0}|^{4}-\frac{1}{16}|\alpha_{1}|^{2}-\frac{1}{16}|\alpha_{1}|^{4}
-\frac{1}{16}|\alpha_{2}|^{2}|\alpha_{1}|^{2}\nonumber\\
&-\frac{1}{16}|\alpha_{2}|^{2}|\alpha_{0}|^{2}-\frac{1}{16}\sum_{j=0}^{1}|\alpha_{j+1}|^{2}|\alpha_{j}|^{2}|\alpha_{j-1}|^{2}-\frac{1}{16}\sum_{j=0}^{1}|\alpha_{j}|^{4}|\alpha_{j-1}|^{2}
\nonumber\\
&-\frac{1}{8}\sum_{j=0}^{1}|\alpha_{j}|^{2}|\alpha_{j-1}|^{4}-\frac{1}{24}\sum_{j=0}^{1}|\alpha_{j}|^{6}-\frac{1}{48}|\alpha_{0}|^{6}\nonumber\\
&-\frac{1}{24}\sum_{j=0}^{\infty}|\alpha_{j}|^{6}+\frac{1}{8}\sum_{j=0}^{\infty}\Big[\big(|\alpha_{j+1}|^{4}-|\alpha_{j}|^{4}\big)|\alpha_{j-1}|^{2}+\big(|\alpha_{j+1}|^{2}-|\alpha_{j}|^{2}\big)|\alpha_{j-1}|^{4}\Big]\nonumber\\
&-\frac{1}{16}\sum_{j=0}^{\infty}\big(|\alpha_{j+2}|^{2}-|\alpha_{j-1}|^{2}\big)\big(|\alpha_{j+2}|^{2}|\alpha_{j+1}|^{2}-|\alpha_{j}|^{2}|\alpha_{j-1}|^{2}\big)\nonumber\\
&-\frac{1}{16}\sum_{j=0}^{\infty}\left(\frac{|\alpha_{j+2}|^{6}+|\alpha_{j+1}|^{6}+|\alpha_{j-1}|^{6}}{3}-|\alpha_{j+2}|^{2}|\alpha_{j+1}|^{2}|\alpha_{j-1}|^{2}\right)\nonumber\\
&-\frac{1}{16}\sum_{j=0}^{\infty}\left(\frac{|\alpha_{j+2}|^{6}+|\alpha_{j}|^{6}+|\alpha_{j-1}|^{6}}{3}-|\alpha_{j+2}|^{2}|\alpha_{j}|^{2}|\alpha_{j-1}|^{2}\right),
\end{align}
then if $\alpha\in\ell^{8}$ and $(S-1)\alpha\in \ell^{2}$, by (4.113), (4.117), (4.161) and the sum rule (4.167), we have the series in $\mathrm{CP}$ and $\mathrm{EP}$ are finite except the following one
\begin{align*}
\frac{5}{8}\sum_{j=0}^{\infty}\Big(\log(1-|\alpha_{j}|^{2})+|\alpha_{j}|^{2}+\frac{1}{2}|\alpha_{j}|^{4}\Big)
\end{align*}
and $\displaystyle{-\frac{1}{24}\sum_{j=0}^{\infty}|\alpha_{j}|^{6}}$ in (4.174).  Thus, by Lemma 4.3, we get that
$Z_{4,3}(\mu)>-\infty$ is equivalent to $\alpha\in \ell^{6}$ as $\alpha\in\ell^{8}$ and $(S-1)\alpha\in \ell^{2}$.
\end{proof}

\begin{rem}
As Theorem 4.56, this is also a weaker result similar to the one  in \cite{lu1} due to Lukic.
\end{rem}

As Theorem 4.58, we have

\begin{thm}
Assume $\alpha\in\ell^{6}$ and $(S-1)\alpha\in \ell^{2}$, then
\begin{align}
&\int_{0}^{2\pi}(1-\cos\theta)^{3}(1+\cos\theta)\log w(\theta)\frac{d\theta}{2\pi}>-\infty.
\end{align}
\end{thm}

\begin{proof}
For a direct proof, similar to Theorem 4.60, it mainly follows from (4.154) and (4.155). For an indirect proof, note that $(S-1)\alpha\in \ell^{2}$ implies $(S-1)\alpha\in \ell^{3}$ and $(S-1)^{3}(S+1)\alpha\in \ell^{2}$, it is a consequence of Theorem 4.60. More immediately, it is a consequence of Theorem 4.64 since $\alpha\in\ell^{6}$ implies $\alpha\in\ell^{8}$.
\end{proof}

\begin{thm}
For any $\alpha$,
\begin{align}
&\int_{0}^{2\pi}(1-\cos\theta)^{4}\log w(\theta) \frac{d\theta}{2\pi}\nonumber\\
=&\frac{653}{192}-\frac{1}{16}|\alpha_{1}|^{2}+\frac{25}{32}|\alpha_{0}|^{4}-\frac{1}{16}|\alpha_{1}|^{4}-\frac{1}{2}|\alpha_{0}|^{2}|\alpha_{0}+1|^{2}\nonumber\\
&-\frac{1}{16}|\alpha_{2}|^{2}|\alpha_{1}|^{2}+\frac{5}{16}\sum_{j=0}^{1}|\alpha_{j}|^{2}|\alpha_{j-1}|^{2}-\frac{1}{16}\sum_{j=0}^{1}|\alpha_{j+1}|^{2}|\alpha_{j-1}|^{2}+\frac{1}{16}|\alpha_{0}|^{6}\nonumber\\
&-\frac{1}{4}\sum_{j=0}^{3}|\alpha_{j}-\alpha_{j-1}|^{2}-\frac{3}{2}\sum_{j=0}^{2}|\alpha_{j}-\alpha_{j-1}|^{2}-\frac{3}{2}\sum_{j=0}^{1}|\alpha_{j}-\alpha_{j-1}|^{2}-\frac{1}{4}|\alpha_{0}+1|^{2}\nonumber\\
&+\frac{3}{8}\sum_{j=0}^{2}|\alpha_{j+1}-\alpha_{j-1}|^{2}+\sum_{j=0}^{1}|\alpha_{j+1}-\alpha_{j-1}|^{2}+\frac{3}{8}|\alpha_{1}+1|^{2}\nonumber\\
&-\frac{1}{4}\sum_{j=0}^{1}|\alpha_{j+2}-\alpha_{j-1}|^{2}-\frac{1}{4}|\alpha_{2}+1|^{2}+\frac{1}{16}|\alpha_{3}+1|^{2}\nonumber\\
&+\frac{35}{8}\sum_{j=0}^{\infty}\Big[\log(1-|\alpha_{j}|^{2})+|\alpha_{j}|^{2}+\frac{1}{2}|\alpha_{j}|^{4}+\frac{1}{3}|\alpha_{j}|^{6}+\frac{1}{4}|\alpha_{j}|^{8}\Big]\nonumber\\
&-\frac{5}{4}\sum_{n=0}^{\infty}|\alpha_{j}|^{6}-\frac{17}{16}\sum_{n=0}^{\infty}|\alpha_{j}|^{8}-\frac{1}{16}\sum_{n=0}^{\infty}|\alpha_{j+4}-4\alpha_{j+3}+6\alpha_{j+2}-4\alpha_{j+1}+\alpha_{j}|^{2}\nonumber\\
&-7\Big[\frac{1}{4}\sum_{j=0}^{\infty}|\alpha_{j}|^{2}|\alpha_{j+1}-\alpha_{j-1}|^{2}+\frac{1}{4}\sum_{j=0}^{\infty}\big(|\alpha_{j}|^{2}+|\alpha_{j-1}|^{2}\big)|\alpha_{j}-\alpha_{j-1}|^{2}\Big]\nonumber\\
&+\Big\{-\frac{15}{16}\sum_{j=0}^{\infty}\big(|\alpha_{j}|^{2}-|\alpha_{j-1}|^{2}\big)^{2}-\frac{1}{16}\sum_{j=0}^{\infty}\big(|\alpha_{j+2}|^{2}-|\alpha_{j-1}|^{2}\big)^{2}\nonumber\\
&-\frac{1}{2}\sum_{j=0}^{\infty}|\alpha_{j+1}|^{2}|\alpha_{j}|^{2}|\alpha_{j+2}-\alpha_{j-1}|^{2}\nonumber\\
&-\frac{1}{2}\sum_{j=0}^{\infty}\rho_{j}^{2}|\alpha_{j+1}-\alpha_{j-1}|^{2}\big(|\alpha_{j+1}-\alpha_{j}|^{2}+|\alpha_{j}-\alpha_{j-1}|^{2}\big)\nonumber
\end{align}
\begin{align}
&-\frac{1}{2}\sum_{j=0}^{\infty}\Big[(|\alpha_{j+1}|^{2}+2|\alpha_{j}|^{2}+|\alpha_{j-1}|^{2})|\alpha_{j+1}-\alpha_{j-1}|^{2}\nonumber\\
&+|\alpha_{j-1}|^{2}|\alpha_{j}-\alpha_{j-1}|^{2}+|\alpha_{j+1}|^{2}|\alpha_{j+1}-\alpha_{j}|^{2}\Big]|\alpha_{j}|^{2}\nonumber\\
&-\frac{1}{6}\sum_{j=0}^{\infty}|\alpha_{j}-\alpha_{j-1}|^{6}-\frac{1}{2}\sum_{j=0}^{\infty}\big(|\alpha_{j}|^{4}+|\alpha_{j}|^{2}|\alpha_{j-1}|^{2}+|\alpha_{j-1}|^{4})|\alpha_{j}-\alpha_{j-1}|^{2}
\Big\}\nonumber\\
&-\frac{1}{8}\Big\{\frac{1}{2}\sum_{j=0}^{\infty}\big(|\alpha_{j+2}|^{2}|\alpha_{j+1}|^{2}+|\alpha_{j+2}|^{2}|\alpha_{j}|^{2}+|\alpha_{j+1}|^{2}|\alpha_{j}|^{2}\big)(|\alpha_{j+3}|^{2}+|\alpha_{j-1}|^{2})\nonumber\\
&+\frac{1}{2}\sum_{j=0}^{\infty}\big(|\alpha_{j+2}|^{2}+|\alpha_{j+1}|^{2}+|\alpha_{j}|^{2}+|\alpha_{j+2}|^{2}|\alpha_{j+1}|^{2}|\alpha_{j}|^{2}\big)|\alpha_{j+3}-\alpha_{j-1}|^{2}\nonumber\\
&+\frac{1}{2}\sum_{j=0}^{\infty}\big(|\alpha_{j+1}|^{2}+|\alpha_{j}|^{2}\big)\Big[|\alpha_{j+2}|^{4}+|\alpha_{j-1}|^{4}+|\alpha_{j+1}|^{2}|\alpha_{j-1}|^{2}+|\alpha_{j+2}|^{2}|\alpha_{j}|^{2}\nonumber\\
&+2|\alpha_{j+2}|^{2}|\alpha_{j+1}|^{2}+2|\alpha_{j}|^{2}|\alpha_{j-1}|^{2}+|\alpha_{j}|^{2}|\alpha_{j+2}-\alpha_{j+1}|^{2}+|\alpha_{j+1}|^{2}|\alpha_{j}-\alpha_{j-1}|^{2}\nonumber\\
&+|\alpha_{j+2}-\alpha_{j+1}|^{2}|\alpha_{j+2}-\alpha_{j-1}|^{2}+|\alpha_{j+2}-\alpha_{j-1}|^{2}|\alpha_{j}-\alpha_{j-1}|^{2}\Big]\nonumber\\
&+\frac{1}{2}\sum_{j=0}^{\infty}\rho_{j+1}^{2}\rho_{j}^{2}\Big[|\alpha_{j+2}|^{2}|\alpha_{j+2}-\alpha_{j+1}|^{2}+|\alpha_{j-1}|^{2}|\alpha_{j}-\alpha_{j-1}|^{2}\nonumber\\
&+\big(|\alpha_{j+2}|^{2}+|\alpha_{j-1}|^{2}\big)|\alpha_{j+1}-\alpha_{j}|^{2}+|\alpha_{j}|^{2}|\alpha_{j+1}-\alpha_{j-1}|^{2}\nonumber\\
&+|\alpha_{j+1}|^{2}|\alpha_{j+2}-\alpha_{j}|^{2}+\big(|\alpha_{j+2}|^{2}+2|\alpha_{j+1}|^{2}+2|\alpha_{j}|^{2}+|\alpha_{j-1}|^{2}\big)|\alpha_{j+2}-\alpha_{j-1}|^{2}\Big]\nonumber\\
&+\frac{1}{2}\sum_{j=0}^{\infty}\Big[|\alpha_{j}|^{4}|\alpha_{j-1}|^{2}+|\alpha_{j+1}|^{2}|\alpha_{j}|^{4}\Big]\nonumber\\
&+\frac{1}{2}\sum_{j=0}^{\infty}\rho_{j}^{2}\Big[\big(|\alpha_{j+1}|^{2}+|\alpha_{j-1}|^{2}\big)\big(|\alpha_{j+1}-\alpha_{j}|^{4}+|\alpha_{j}-\alpha_{j-1}|^{4}\big)\nonumber\\
&+|\alpha_{j-1}|^{4}|\alpha_{j+1}-\alpha_{j}|^{2}+|\alpha_{j+1}|^{4}|\alpha_{j}-\alpha_{j-1}|^{2}\nonumber\\
&+2\big(|\alpha_{j+1}|^{2}+|\alpha_{j}|^{2}\big)|\alpha_{j+1}-\alpha_{j-1}|^{2}|\alpha_{j+1}-\alpha_{j}|^{2}\nonumber\\
&+2\big(|\alpha_{j}|^{2}+|\alpha_{j-1}|^{2}\big)|\alpha_{j+1}-\alpha_{j-1}|^{2}|\alpha_{j}-\alpha_{j-1}|^{2}\Big]\nonumber\\
&+\frac{1}{2}\sum_{j=0}^{\infty}\rho_{j}^{2}|\alpha_{j+1}-\alpha_{j-1}|^{4}+\sum_{j=0}^{\infty}|\alpha_{j}|^{2}\Big[|\alpha_{j+1}|^{4}+|\alpha_{j-1}|^{4}\Big]\nonumber\\
&+\frac{3}{2}\sum_{j=0}^{\infty}\rho_{j}^{4}\big(|\alpha_{j+1}|^{2}+|\alpha_{j-1}|^{2}\big)|\alpha_{j+1}-\alpha_{j-1}|^{2}\nonumber\\
&+\frac{1}{2}\sum_{j=0}^{\infty}\Big[\big(|\alpha_{j}|^{2}+|\alpha_{j-1}|^{2}\big)|\alpha_{j}-\alpha_{j-1}|^{6}\nonumber\\
&+\big(|\alpha_{j}|^{6}+|\alpha_{j-1}|^{6}+|\alpha_{j}|^{4}|\alpha_{j-1}|^{2}+|\alpha_{j}|^{2}|\alpha_{j-1}|^{4}\big)|\alpha_{j}-\alpha_{j-1}|^{2}\Big]\Big\}\nonumber
\end{align}
\begin{align}
&+\frac{3}{8}\sum_{j=0}^{\infty}\big(|\alpha_{j+1}|^{2}-\alpha_{j-1}|^{2}\big)^{2}+
\frac{7}{8}\sum_{j=0}^{\infty}|\alpha_{j}-\alpha_{j-1}|^{4}\nonumber\\
&+\frac{1}{2}\sum_{j=0}^{\infty}(|\alpha_{j+2}|^{2}+|\alpha_{j-1}|^{2})|\alpha_{j+1}|^{2}|\alpha_{j}|^{2}+\frac{1}{2}\sum_{j=0}^{\infty}\big(|\alpha_{j+1}|^{2}+|\alpha_{j}|^{2}\big)|\alpha_{j+2}-\alpha_{j-1}|^{2}\nonumber\\
&+\frac{1}{2}\sum_{j=0}^{\infty}\Big[|\alpha_{j+1}|^{2}|\alpha_{j}|^{4}+|\alpha_{j}|^{2}|\alpha_{j-1}|^{4}+|\alpha_{j}|^{4}|\alpha_{j-1}|^{2}+|\alpha_{j}|^{2}|\alpha_{j+1}|^{4}\Big]\nonumber\\
&+\frac{1}{2}\sum_{j=0}^{\infty}\Big[(|\alpha_{j+1}|^{2}+2|\alpha_{j}|^{2}+|\alpha_{j-1}|^{2})|\alpha_{j+1}-\alpha_{j-1}|^{2}\nonumber\\
&+(|\alpha_{j}|^{2}+|\alpha_{j-1}|^{2})(|\alpha_{j}-\alpha_{j-1}|^{2}+|\alpha_{j}-\alpha_{j-1}|^{4})\Big]\nonumber\\
&+\frac{1}{8}\Big\{\frac{1}{2}\sum_{j=0}^{\infty}|\alpha_{j+2}|^{2}|\alpha_{j+1}|^{2}|\alpha_{j}|^{2}\big(|\alpha_{j+3}|^{2}+|\alpha_{j-1}|^{2}\big)\nonumber\\
&+\frac{1}{2}\sum_{j=0}^{\infty}\big(|\alpha_{j+2}|^{2}|\alpha_{j+1}|^{2}+|\alpha_{j+2}|^{2}|\alpha_{j}|^{2}+|\alpha_{j+1}|^{2}|\alpha_{j}|^{2}\big)|\alpha_{j+3}-\alpha_{j-1}|^{2}\nonumber\\
&+\frac{1}{2}\sum_{j=0}^{\infty}|\alpha_{j}|^{2}\Big[|\alpha_{j+1}|^{6}+|\alpha_{j-1}|^{6}+|\alpha_{j}|^{4}|\alpha_{j-1}|^{2}+|\alpha_{j+1}|^{2}|\alpha_{j}|^{4}\Big]\nonumber\\
&+\frac{1}{2}\sum_{j=0}^{\infty}\Big[|\alpha_{j}|^{2}|\alpha_{j+2}-\alpha_{j+1}|^{2}+|\alpha_{j+1}|^{2}|\alpha_{j}-\alpha_{j-1}|^{2}\nonumber\\
&+|\alpha_{j+2}-\alpha_{j+1}|^{2}|\alpha_{j+2}-\alpha_{j-1}|^{2}+|\alpha_{j+2}-\alpha_{j-1}|^{2}|\alpha_{j}-\alpha_{j-1}|^{2}\Big]\nonumber\\
&+\frac{1}{2}\sum_{j=0}^{\infty}|\alpha_{j+1}|^{2}|\alpha_{j}|^{2}\Big[|\alpha_{j+2}|^{4}+|\alpha_{j-1}|^{4}+|\alpha_{j+1}|^{2}|\alpha_{j-1}|^{2}+|\alpha_{j+2}|^{2}|\alpha_{j}|^{2}\nonumber\\
&+2|\alpha_{j+2}|^{2}|\alpha_{j+1}|^{2}+2|\alpha_{j}|^{2}|\alpha_{j-1}|^{2}+|\alpha_{j}|^{2}|\alpha_{j+2}-\alpha_{j+1}|^{2}+|\alpha_{j+1}|^{2}|\alpha_{j}-\alpha_{j-1}|^{2}\nonumber\\
&+|\alpha_{j+2}-\alpha_{j+1}|^{2}|\alpha_{j+2}-\alpha_{j-1}|^{2}+|\alpha_{j+2}-\alpha_{j-1}|^{2}|\alpha_{j}-\alpha_{j-1}|^{2}\Big]\nonumber\\
&+\frac{1}{2}\sum_{j=0}^{\infty}\rho_{j}^{2}\Big[\big(|\alpha_{j+1}|^{4}+2|\alpha_{j}|^{4}+|\alpha_{j-1}|^{4}+|\alpha_{j}|^{2}(|\alpha_{j+1}|^{2}+|\alpha_{j-1}|^{2})\big)|\alpha_{j+1}-\alpha_{j-1}|^{2}\nonumber\\
&+\big(2|\alpha_{j+1}|^{4}+|\alpha_{j+1}|^{2}|\alpha_{j}|^{2}+|\alpha_{j+1}|^{2}|\alpha_{j-1}|^{2}+|\alpha_{j}|^{2}|\alpha_{j-1}|^{2}\big)|\alpha_{j+1}-\alpha_{j}|^{2}\nonumber\\
&+\big(2|\alpha_{j-1}|^{4}+|\alpha_{j+1}|^{2}|\alpha_{j}|^{2}+|\alpha_{j+1}|^{2}|\alpha_{j-1}|^{2}+|\alpha_{j}|^{2}|\alpha_{j-1}|^{2}\big)|\alpha_{j}-\alpha_{j-1}|^{2}\nonumber\\
&+\big(|\alpha_{j+1}|^{2}+|\alpha_{j-1}|^{2}\big)|\alpha_{j+1}-\alpha_{j}|^{2}|\alpha_{j}-\alpha_{j-1}|^{2}\nonumber\\
&+\big(|\alpha_{j+1}-\alpha_{j}|^{4}+|\alpha_{j}-\alpha_{j-1}|^{4}\big)|\alpha_{j+1}-\alpha_{j-1}|^{2}\Big]\nonumber\\
&+\sum_{j=0}^{\infty}\rho_{j}^{2}\big(|\alpha_{j+1}|^{2}+|\alpha_{j-1}|^{2}\big)|\alpha_{j+1}-\alpha_{j-1}|^{2}\nonumber\\
&+\frac{3}{4}\sum_{j=0}^{\infty}|\alpha_{j}|^{4}\Big[|\alpha_{j+1}|^{4}+|\alpha_{j-1}|^{4}\Big]+\frac{3}{4}\rho_{j}^{4}\sum_{j=0}^{\infty}|\alpha_{j+1}-\alpha_{j-1}|^{4}\nonumber\\
&+\frac{1}{8}\sum_{j=0}^{\infty}\Big[|\alpha_{j}-\alpha_{j-1}|^{8}+4\big(|\alpha_{j}|^{4}+|\alpha_{j-1}|^{4}+2|\alpha_{j}|^{2}|\alpha_{j-1}|^{2}\big)|\alpha_{j}-\alpha_{j-1}|^{4}\nonumber\\
&+2\big(|\alpha_{j}|^{4}+|\alpha_{j-1}|^{4}\big)|\alpha_{j}-\alpha_{j-1}|^{4}\Big]\Big\}.
\end{align}
\end{thm}

\begin{proof}
Since
\begin{align}
&(1-\cos\theta)^{4}=(1-2\cos\theta+\cos^{2}\theta)^{2}\nonumber\\
=&\frac{35}{8}-7\cos\theta+\frac{7}{2}\cos2\theta-\cos3\theta+\frac{1}{8}\cos4\theta\nonumber\\
=&7(1-\cos\theta)-\frac{7}{2}(1-\cos2\theta)+(1-\cos3\theta)-\frac{1}{8}(1-\cos4\theta),
\end{align}
then
\begin{align}
&Z_{4,4}(\mu)=\frac{35}{8}w_{0}-7\mathrm{Re}(w_{1})+\frac{7}{2}\mathrm{Re}(w_{2})-\mathrm{Re}(w_{3})+\frac{1}{8}\mathrm{Re}(w_{4})\nonumber\\
=&7(w_{0}-\mathrm{Re}(w_{1}))-\frac{7}{2}(w_{0}-\mathrm{Re}(w_{2}))+(w_{0}-\mathrm{Re}(w_{3}))
-\frac{1}{8}(w_{0}-\mathrm{Re}(w_{4}))\nonumber\\
=&7Z_{1}(\mu)-7Z_{2,1}(\mu)+Z_{3,1}(\mu)-\frac{1}{8}Z_{4,1}(\mu).
\end{align}

Therefore, by (4.2), (4.12), (4.60), (4.133a) and (4.178), we have
\begin{align}
Z_{4,4}(\mu)=&\frac{7}{2}+7\sum_{n=0}^{\infty}\Big(\log(1-|\alpha_{n}|^{2})+|\alpha_{n}|^{2}\Big)-\frac{7}{2}\sum_{n=0}^{\infty}|\alpha_{n}-\alpha_{n-1}|^{2}\nonumber\\
&-7\Big\{\frac{3}{8}+\frac{1}{2}\sum_{j=0}^{\infty}\Big[\log(1-|\alpha_{j}|^{2})+|\alpha_{j}|^{2}+\frac{1}{2}|\alpha_{j}|^{4}\Big]\nonumber\\
&-\frac{1}{2}\sum_{j=0}^{\infty}|\alpha_{j}\alpha_{j-1}|^{2}-\frac{1}{4}\sum_{j=0}^{\infty}\rho_{j}^{2}|\alpha_{j+1}-\alpha_{j-1}|^{2}\nonumber\\
&-\frac{1}{16}\sum_{j=0}^{\infty}\left[\big(2|\alpha_{j}|^{2}-|\alpha_{j}-\alpha_{j-1}|^{2}\big)^{2}+\big(2|\alpha_{j-1}|^{2}-|\alpha_{j}-\alpha_{j-1}|^{2}\big)^{2}\right]\Big\}\nonumber\\
&+\Big\{\frac{2}{3}-\frac{1}{2}(|\alpha_{0}|^{2}+|\alpha_{1}|^{2})-\frac{1}{2}|\alpha_{0}|^{2}|\alpha_{0}+1|^{2}\nonumber\\
&+\sum_{j=0}^{\infty}\Big[\log(1-|\alpha_{j}|^{2})+|\alpha_{j}|^{2}+\frac{1}{2}|\alpha_{j}|^{4}+\frac{1}{3}|\alpha_{j}|^{6}\Big]\nonumber\\
&-\frac{1}{2}\sum_{j=0}^{\infty}|\alpha_{j}|^{4}-\frac{1}{2}\sum_{j=0}^{\infty}(|\alpha_{j+2}|^{2}+|\alpha_{j-1}|^{2})|\alpha_{j}|^{2}\nonumber\\
&-\frac{1}{2}\sum_{j=0}^{\infty}(|\alpha_{j+2}|^{2}+|\alpha_{j-1}|^{2})|\alpha_{j+1}|^{2}\rho_{j}^{2}-\frac{1}{2}\sum_{j=0}^{\infty}|\alpha_{j+2}-\alpha_{j-1}|^{2}\rho_{j+1}^{2}\rho_{j}^{2}\nonumber\\
&-\frac{1}{2}\sum_{j=0}^{\infty}\Big[|\alpha_{j+1}|^{2}|\alpha_{j}|^{2}+|\alpha_{j-1}|^{4}+|\alpha_{j}|^{2}|\alpha_{j-1}|^{2}+|\alpha_{j+1}|^{4}\nonumber\\
&+|\alpha_{j+1}-\alpha_{j-1}|^{2}\big(|\alpha_{j+1}-\alpha_{j}|^{2}+|\alpha_{j}-\alpha_{j-1}|^{2}\big)\Big]\rho_{j}^{2}\nonumber\\
&-\frac{1}{2}\sum_{j=0}^{\infty}\Big[(|\alpha_{j+1}|^{2}+2|\alpha_{j}|^{2}+|\alpha_{j-1}|^{2})|\alpha_{j+1}-\alpha_{j-1}|^{2}\nonumber\\
&+|\alpha_{j-1}|^{2}|\alpha_{j}-\alpha_{j-1}|^{2}+|\alpha_{j+1}|^{2}|\alpha_{j+1}-\alpha_{j}|^{2}\Big]|\alpha_{j}|^{2}\nonumber
\end{align}
\begin{align}
&-\frac{1}{6}\sum_{j=0}^{\infty}|\alpha_{j}-\alpha_{j-1}|^{6}-\frac{1}{2}\sum_{j=0}^{\infty}\big(|\alpha_{j}|^{4}+|\alpha_{j}|^{2}|\alpha_{j-1}|^{2}+|\alpha_{j-1}|^{4})|\alpha_{j}-\alpha_{j-1}|^{2}\nonumber\\
&+\frac{1}{2}\sum_{j=0}^{\infty}\Big[(|\alpha_{j+1}|^{2}+2|\alpha_{j}|^{2}+|\alpha_{j-1}|^{2})|\alpha_{j+1}-\alpha_{j-1}|^{2}\nonumber\\
&+(|\alpha_{j}|^{2}+|\alpha_{j-1}|^{2})(|\alpha_{j}-\alpha_{j-1}|^{2}+|\alpha_{j}-\alpha_{j-1}|^{4})\Big]\Big\}\nonumber\\
&-\frac{1}{8}\Big\{\sum_{j=0}^{\infty}\log(1-|\alpha_{j}|^{2})+\frac{1}{2}\sum_{j=0}^{\infty}\rho_{j+2}^{2}\rho_{j+1}^{2}\rho_{j}^{2}(|\alpha_{j+3}|^{2}+|\alpha_{j-1}|^{2}-|\alpha_{j+3}-\alpha_{j-1}|^{2})\nonumber\\
&-\frac{1}{2}\sum_{j=0}^{\infty}\rho_{j+1}^{2}\rho_{j}^{2}\Big[|\alpha_{j+2}|^{4}+|\alpha_{j-1}|^{4}+|\alpha_{j+1}|^{2}|\alpha_{j-1}|^{2}+|\alpha_{j+2}|^{2}|\alpha_{j}|^{2}\nonumber\\
&+2|\alpha_{j+2}|^{2}|\alpha_{j+1}|^{2}+2|\alpha_{j}|^{2}|\alpha_{j-1}|^{2}-\big(|\alpha_{j+2}|^{2}-|\alpha_{j}|^{2}\big)|\alpha_{j+2}-\alpha_{j+1}|^{2}\nonumber\\
&-\big(|\alpha_{j+2}|^{2}+|\alpha_{j-1}|^{2}\big)|\alpha_{j+1}-\alpha_{j}|^{2}-|\alpha_{j}|^{2}|\alpha_{j+1}-\alpha_{j-1}|^{2}\nonumber\\
&-|\alpha_{j+1}|^{2}|\alpha_{j+2}-\alpha_{j}|^{2}+\big(|\alpha_{j+1}|^{2}-|\alpha_{j-1}|^{2}\big)|\alpha_{j}-\alpha_{j-1}|^{2}\nonumber\\
&-\big(|\alpha_{j+2}|^{2}+2|\alpha_{j+1}|^{2}+2|\alpha_{j}|^{2}+|\alpha_{j-1}|^{2}\big)|\alpha_{j+2}-\alpha_{j-1}|^{2}\nonumber\\
&+|\alpha_{j+2}-\alpha_{j+1}|^{2}|\alpha_{j+2}-\alpha_{j-1}|^{2}+|\alpha_{j+2}-\alpha_{j-1}|^{2}|\alpha_{j}-\alpha_{j-1}|^{2}\Big]\nonumber\\
&+\frac{1}{2}\sum_{j=0}^{\infty}\rho_{j}^{2}\Big[|\alpha_{j+1}|^{6}+|\alpha_{j-1}|^{6}+|\alpha_{j}|^{4}|\alpha_{j-1}|^{2}+|\alpha_{j+1}|^{2}|\alpha_{j}|^{4}\nonumber\\
&-\big(|\alpha_{j+1}|^{4}+2|\alpha_{j}|^{4}+|\alpha_{j-1}|^{4}+|\alpha_{j+1}|^{2}|\alpha_{j}|^{2}+|\alpha_{j}|^{2}|\alpha_{j-1}|^{2}\big)|\alpha_{j+1}-\alpha_{j-1}|^{2}\nonumber\\
&+\big(|\alpha_{j+1}|^{2}+|\alpha_{j-1}|^{2}\big)\big(|\alpha_{j+1}-\alpha_{j}|^{4}+|\alpha_{j}-\alpha_{j-1}|^{4}\big)\nonumber\\
&-\big(2|\alpha_{j+1}|^{4}-|\alpha_{j-1}|^{4}+|\alpha_{j+1}|^{2}|\alpha_{j}|^{2}+|\alpha_{j+1}|^{2}|\alpha_{j-1}|^{2}+|\alpha_{j}|^{2}|\alpha_{j-1}|^{2}\big)|\alpha_{j+1}-\alpha_{j}|^{2}\nonumber\\
&-\big(2|\alpha_{j-1}|^{4}-|\alpha_{j+1}|^{4}+|\alpha_{j+1}|^{2}|\alpha_{j}|^{2}+|\alpha_{j+1}|^{2}|\alpha_{j-1}|^{2}+|\alpha_{j}|^{2}|\alpha_{j-1}|^{2}\big)|\alpha_{j}-\alpha_{j-1}|^{2}\nonumber\\
&+2\big(|\alpha_{j+1}|^{2}+|\alpha_{j}|^{2}\big)|\alpha_{j+1}-\alpha_{j-1}|^{2}|\alpha_{j+1}-\alpha_{j}|^{2}\nonumber\\
&+2\big(|\alpha_{j}|^{2}+|\alpha_{j-1}|^{2}\big)|\alpha_{j+1}-\alpha_{j-1}|^{2}|\alpha_{j}-\alpha_{j-1}|^{2}\nonumber\\
&-\big(|\alpha_{j+1}|^{2}+|\alpha_{j-1}|^{2}\big)|\alpha_{j+1}-\alpha_{j}|^{2}|\alpha_{j}-\alpha_{j-1}|^{2}\nonumber\\
&-\big(|\alpha_{j+1}-\alpha_{j}|^{4}+|\alpha_{j}-\alpha_{j-1}|^{4}\big)|\alpha_{j+1}-\alpha_{j-1}|^{2}\Big]\nonumber\\
&+\frac{1}{2}\sum_{j=0}^{\infty}\rho_{j}^{2}\Big[|\alpha_{j+1}|^{4}+|\alpha_{j-1}|^{4}+|\alpha_{j+1}-\alpha_{j-1}|^{4}-2\big(|\alpha_{j+1}|^{2}+|\alpha_{j-1}|^{2}\big)|\alpha_{j+1}-\alpha_{j-1}|^{2}\Big]\nonumber\\
&-\frac{3}{4}\sum_{j=0}^{\infty}\rho_{j}^{4}\Big[|\alpha_{j+1}|^{4}+|\alpha_{j-1}|^{4}+|\alpha_{j+1}-\alpha_{j-1}|^{4}-2\big(|\alpha_{j+1}|^{2}+|\alpha_{j-1}|^{2}\big)|\alpha_{j+1}-\alpha_{j-1}|^{2}\Big]\nonumber\\
&-\frac{1}{8}\sum_{j=0}^{\infty}\Big[|\alpha_{j}|^{8}+|\alpha_{j-1}|^{8}+|\alpha_{j}-\alpha_{j-1}|^{8}+4\big(|\alpha_{j}|^{2}+|\alpha_{j-1}|^{2}\big)^{2}|\alpha_{j}-\alpha_{j-1}|^{4}\nonumber\\
&+2\big(|\alpha_{j}|^{4}+|\alpha_{j-1}|^{4}\big)|\alpha_{j}-\alpha_{j-1}|^{4}-4\big(|\alpha_{j}|^{2}+|\alpha_{j-1}|^{2}\big)|\alpha_{j}-\alpha_{j-1}|^{6}\nonumber\\
&-4\big(|\alpha_{j}|^{6}+|\alpha_{j-1}|^{6}+|\alpha_{j}|^{4}|\alpha_{j-1}|^{2}+|\alpha_{j}|^{2}|\alpha_{j-1}|^{4}\big)|\alpha_{j}-\alpha_{j-1}|^{2}\Big]\Big\}\nonumber
\end{align}
\begin{align}
=&\frac{653}{192}-\frac{1}{16}|\alpha_{1}|^{2}+\frac{25}{32}|\alpha_{0}|^{4}-\frac{1}{16}|\alpha_{1}|^{4}-\frac{1}{2}|\alpha_{0}|^{2}|\alpha_{0}+1|^{2}\nonumber\\
&-\frac{1}{16}|\alpha_{2}|^{2}|\alpha_{1}|^{2}+\frac{5}{16}\sum_{j=0}^{1}|\alpha_{j}|^{2}|\alpha_{j-1}|^{2}-\frac{1}{16}\sum_{j=0}^{1}|\alpha_{j+1}|^{2}|\alpha_{j-1}|^{2}+\frac{1}{16}|\alpha_{0}|^{6}\nonumber\\
&-\frac{1}{4}\sum_{j=0}^{3}|\alpha_{j}-\alpha_{j-1}|^{2}-\frac{3}{2}\sum_{j=0}^{2}|\alpha_{j}-\alpha_{j-1}|^{2}-\frac{3}{2}\sum_{j=0}^{1}|\alpha_{j}-\alpha_{j-1}|^{2}-\frac{1}{4}|\alpha_{0}+1|^{2}\nonumber\\
&+\frac{3}{8}\sum_{j=0}^{2}|\alpha_{j+1}-\alpha_{j-1}|^{2}+\sum_{j=0}^{1}|\alpha_{j+1}-\alpha_{j-1}|^{2}+\frac{3}{8}|\alpha_{1}+1|^{2}\nonumber\\
&-\frac{1}{4}\sum_{j=0}^{1}|\alpha_{j+2}-\alpha_{j-1}|^{2}-\frac{1}{4}|\alpha_{2}+1|^{2}+\frac{1}{16}|\alpha_{3}+1|^{2}\nonumber\\
&+\frac{35}{8}\sum_{j=0}^{\infty}\Big[\log(1-|\alpha_{j}|^{2})+|\alpha_{j}|^{2}+\frac{1}{2}|\alpha_{j}|^{4}+\frac{1}{3}|\alpha_{j}|^{6}+\frac{1}{4}|\alpha_{j}|^{8}\Big]\nonumber\\
&-\frac{5}{4}\sum_{n=0}^{\infty}|\alpha_{j}|^{6}-\frac{17}{16}\sum_{n=0}^{\infty}|\alpha_{j}|^{8}-\frac{1}{16}\sum_{n=0}^{\infty}|\alpha_{j+4}-4\alpha_{j+3}+6\alpha_{j+2}-4\alpha_{j+1}+\alpha_{j}|^{2}\nonumber\\
&-7\Big[\frac{1}{4}\sum_{j=0}^{\infty}|\alpha_{j}|^{2}|\alpha_{j+1}-\alpha_{j-1}|^{2}+\frac{1}{4}\sum_{j=0}^{\infty}\big(|\alpha_{j}|^{2}+|\alpha_{j-1}|^{2}\big)|\alpha_{j}-\alpha_{j-1}|^{2}\Big]\nonumber\\
&+\Big\{-\frac{15}{16}\sum_{j=0}^{\infty}\big(|\alpha_{j}|^{2}-|\alpha_{j-1}|^{2}\big)^{2}-\frac{1}{16}\sum_{j=0}^{\infty}\big(|\alpha_{j+2}|^{2}-|\alpha_{j-1}|^{2}\big)^{2}\nonumber\\
&-\frac{1}{2}\sum_{j=0}^{\infty}|\alpha_{j+1}|^{2}|\alpha_{j}|^{2}|\alpha_{j+2}-\alpha_{j-1}|^{2}\nonumber\\
&-\frac{1}{2}\sum_{j=0}^{\infty}\rho_{j}^{2}|\alpha_{j+1}-\alpha_{j-1}|^{2}\big(|\alpha_{j+1}-\alpha_{j}|^{2}+|\alpha_{j}-\alpha_{j-1}|^{2}\big)\nonumber\\
&-\frac{1}{2}\sum_{j=0}^{\infty}\Big[(|\alpha_{j+1}|^{2}+2|\alpha_{j}|^{2}+|\alpha_{j-1}|^{2})|\alpha_{j+1}-\alpha_{j-1}|^{2}\nonumber\\
&+|\alpha_{j-1}|^{2}|\alpha_{j}-\alpha_{j-1}|^{2}+|\alpha_{j+1}|^{2}|\alpha_{j+1}-\alpha_{j}|^{2}\Big]|\alpha_{j}|^{2}\nonumber\\
&-\frac{1}{6}\sum_{j=0}^{\infty}|\alpha_{j}-\alpha_{j-1}|^{6}-\frac{1}{2}\sum_{j=0}^{\infty}\big(|\alpha_{j}|^{4}+|\alpha_{j}|^{2}|\alpha_{j-1}|^{2}+|\alpha_{j-1}|^{4})|\alpha_{j}-\alpha_{j-1}|^{2}
\Big\}\nonumber
\end{align}
\begin{align}
&-\frac{1}{8}\Big\{\frac{1}{2}\sum_{j=0}^{\infty}\big(|\alpha_{j+2}|^{2}|\alpha_{j+1}|^{2}+|\alpha_{j+2}|^{2}|\alpha_{j}|^{2}+|\alpha_{j+1}|^{2}|\alpha_{j}|^{2}\big)(|\alpha_{j+3}|^{2}+|\alpha_{j-1}|^{2})\nonumber\\
&+\frac{1}{2}\sum_{j=0}^{\infty}\big(|\alpha_{j+2}|^{2}+|\alpha_{j+1}|^{2}+|\alpha_{j}|^{2}+|\alpha_{j+2}|^{2}|\alpha_{j+1}|^{2}|\alpha_{j}|^{2}\big)|\alpha_{j+3}-\alpha_{j-1}|^{2}\nonumber\\
&+\frac{1}{2}\sum_{j=0}^{\infty}\big(|\alpha_{j+1}|^{2}+|\alpha_{j}|^{2}\big)\Big[|\alpha_{j+2}|^{4}+|\alpha_{j-1}|^{4}+|\alpha_{j+1}|^{2}|\alpha_{j-1}|^{2}+|\alpha_{j+2}|^{2}|\alpha_{j}|^{2}\nonumber\\
&+2|\alpha_{j+2}|^{2}|\alpha_{j+1}|^{2}+2|\alpha_{j}|^{2}|\alpha_{j-1}|^{2}+|\alpha_{j}|^{2}|\alpha_{j+2}-\alpha_{j+1}|^{2}+|\alpha_{j+1}|^{2}|\alpha_{j}-\alpha_{j-1}|^{2}\nonumber\\
&+|\alpha_{j+2}-\alpha_{j+1}|^{2}|\alpha_{j+2}-\alpha_{j-1}|^{2}+|\alpha_{j+2}-\alpha_{j-1}|^{2}|\alpha_{j}-\alpha_{j-1}|^{2}\Big]\nonumber\\
&+\frac{1}{2}\sum_{j=0}^{\infty}\rho_{j+1}^{2}\rho_{j}^{2}\Big[|\alpha_{j+2}|^{2}|\alpha_{j+2}-\alpha_{j+1}|^{2}+|\alpha_{j-1}|^{2}|\alpha_{j}-\alpha_{j-1}|^{2}\nonumber\\
&+\big(|\alpha_{j+2}|^{2}+|\alpha_{j-1}|^{2}\big)|\alpha_{j+1}-\alpha_{j}|^{2}+|\alpha_{j}|^{2}|\alpha_{j+1}-\alpha_{j-1}|^{2}\nonumber\\
&+|\alpha_{j+1}|^{2}|\alpha_{j+2}-\alpha_{j}|^{2}+\big(|\alpha_{j+2}|^{2}+2|\alpha_{j+1}|^{2}+2|\alpha_{j}|^{2}+|\alpha_{j-1}|^{2}\big)|\alpha_{j+2}-\alpha_{j-1}|^{2}\Big]\nonumber\\
&+\frac{1}{2}\sum_{j=0}^{\infty}\Big[|\alpha_{j}|^{4}|\alpha_{j-1}|^{2}+|\alpha_{j+1}|^{2}|\alpha_{j}|^{4}\Big]\nonumber\\
&+\frac{1}{2}\sum_{j=0}^{\infty}\rho_{j}^{2}\Big[\big(|\alpha_{j+1}|^{2}+|\alpha_{j-1}|^{2}\big)\big(|\alpha_{j+1}-\alpha_{j}|^{4}+|\alpha_{j}-\alpha_{j-1}|^{4}\big)\nonumber\\
&+|\alpha_{j-1}|^{4}|\alpha_{j+1}-\alpha_{j}|^{2}+|\alpha_{j+1}|^{4}|\alpha_{j}-\alpha_{j-1}|^{2}\nonumber\\
&+2\big(|\alpha_{j+1}|^{2}+|\alpha_{j}|^{2}\big)|\alpha_{j+1}-\alpha_{j-1}|^{2}|\alpha_{j+1}-\alpha_{j}|^{2}\nonumber\\
&+2\big(|\alpha_{j}|^{2}+|\alpha_{j-1}|^{2}\big)|\alpha_{j+1}-\alpha_{j-1}|^{2}|\alpha_{j}-\alpha_{j-1}|^{2}\Big]\nonumber\\
&+\frac{1}{2}\sum_{j=0}^{\infty}\rho_{j}^{2}|\alpha_{j+1}-\alpha_{j-1}|^{4}+\sum_{j=0}^{\infty}|\alpha_{j}|^{2}\Big[|\alpha_{j+1}|^{4}+|\alpha_{j-1}|^{4}\Big]\nonumber\\
&+\frac{3}{2}\sum_{j=0}^{\infty}\rho_{j}^{4}\big(|\alpha_{j+1}|^{2}+|\alpha_{j-1}|^{2}\big)|\alpha_{j+1}-\alpha_{j-1}|^{2}\nonumber\\
&+\frac{1}{2}\sum_{j=0}^{\infty}\Big[\big(|\alpha_{j}|^{2}+|\alpha_{j-1}|^{2}\big)|\alpha_{j}-\alpha_{j-1}|^{6}\nonumber\\
&+\big(|\alpha_{j}|^{6}+|\alpha_{j-1}|^{6}+|\alpha_{j}|^{4}|\alpha_{j-1}|^{2}+|\alpha_{j}|^{2}|\alpha_{j-1}|^{4}\big)|\alpha_{j}-\alpha_{j-1}|^{2}\Big]\Big\}\nonumber\\
&+\frac{3}{8}\sum_{j=0}^{\infty}\big(|\alpha_{j+1}|^{2}-\alpha_{j-1}|^{2}\big)^{2}+
\frac{7}{8}\sum_{j=0}^{\infty}|\alpha_{j}-\alpha_{j-1}|^{4}\nonumber\\
&+\frac{1}{2}\sum_{j=0}^{\infty}(|\alpha_{j+2}|^{2}+|\alpha_{j-1}|^{2})|\alpha_{j+1}|^{2}|\alpha_{j}|^{2}+\frac{1}{2}\sum_{j=0}^{\infty}\big(|\alpha_{j+1}|^{2}+|\alpha_{j}|^{2}\big)|\alpha_{j+2}-\alpha_{j-1}|^{2}\nonumber
\end{align}
\begin{align}
&+\frac{1}{2}\sum_{j=0}^{\infty}\Big[|\alpha_{j+1}|^{2}|\alpha_{j}|^{4}+|\alpha_{j}|^{2}|\alpha_{j-1}|^{4}+|\alpha_{j}|^{4}|\alpha_{j-1}|^{2}+|\alpha_{j}|^{2}|\alpha_{j+1}|^{4}\Big]\nonumber\\
&+\frac{1}{2}\sum_{j=0}^{\infty}\Big[(|\alpha_{j+1}|^{2}+2|\alpha_{j}|^{2}+|\alpha_{j-1}|^{2})|\alpha_{j+1}-\alpha_{j-1}|^{2}\nonumber\\
&+(|\alpha_{j}|^{2}+|\alpha_{j-1}|^{2})(|\alpha_{j}-\alpha_{j-1}|^{2}+|\alpha_{j}-\alpha_{j-1}|^{4})\Big]\nonumber\\
&+\frac{1}{8}\Big\{\frac{1}{2}\sum_{j=0}^{\infty}|\alpha_{j+2}|^{2}|\alpha_{j+1}|^{2}|\alpha_{j}|^{2}\big(|\alpha_{j+3}|^{2}+|\alpha_{j-1}|^{2}\big)\nonumber\\
&+\frac{1}{2}\sum_{j=0}^{\infty}\big(|\alpha_{j+2}|^{2}|\alpha_{j+1}|^{2}+|\alpha_{j+2}|^{2}|\alpha_{j}|^{2}+|\alpha_{j+1}|^{2}|\alpha_{j}|^{2}\big)|\alpha_{j+3}-\alpha_{j-1}|^{2}\nonumber\\
&+\frac{1}{2}\sum_{j=0}^{\infty}|\alpha_{j}|^{2}\Big[|\alpha_{j+1}|^{6}+|\alpha_{j-1}|^{6}+|\alpha_{j}|^{4}|\alpha_{j-1}|^{2}+|\alpha_{j+1}|^{2}|\alpha_{j}|^{4}\Big]\nonumber\\
&+\frac{1}{2}\sum_{j=0}^{\infty}\Big[|\alpha_{j}|^{2}|\alpha_{j+2}-\alpha_{j+1}|^{2}+|\alpha_{j+1}|^{2}|\alpha_{j}-\alpha_{j-1}|^{2}\nonumber\\
&+|\alpha_{j+2}-\alpha_{j+1}|^{2}|\alpha_{j+2}-\alpha_{j-1}|^{2}+|\alpha_{j+2}-\alpha_{j-1}|^{2}|\alpha_{j}-\alpha_{j-1}|^{2}\Big]\nonumber\\
&+\frac{1}{2}\sum_{j=0}^{\infty}|\alpha_{j+1}|^{2}|\alpha_{j}|^{2}\Big[|\alpha_{j+2}|^{4}+|\alpha_{j-1}|^{4}+|\alpha_{j+1}|^{2}|\alpha_{j-1}|^{2}+|\alpha_{j+2}|^{2}|\alpha_{j}|^{2}\nonumber\\
&+2|\alpha_{j+2}|^{2}|\alpha_{j+1}|^{2}+2|\alpha_{j}|^{2}|\alpha_{j-1}|^{2}+|\alpha_{j}|^{2}|\alpha_{j+2}-\alpha_{j+1}|^{2}+|\alpha_{j+1}|^{2}|\alpha_{j}-\alpha_{j-1}|^{2}\nonumber\\
&+|\alpha_{j+2}-\alpha_{j+1}|^{2}|\alpha_{j+2}-\alpha_{j-1}|^{2}+|\alpha_{j+2}-\alpha_{j-1}|^{2}|\alpha_{j}-\alpha_{j-1}|^{2}\Big]\nonumber\\
&+\frac{1}{2}\sum_{j=0}^{\infty}\rho_{j}^{2}\Big[\big(|\alpha_{j+1}|^{4}+2|\alpha_{j}|^{4}+|\alpha_{j-1}|^{4}+|\alpha_{j}|^{2}(|\alpha_{j+1}|^{2}+|\alpha_{j-1}|^{2})\big)|\alpha_{j+1}-\alpha_{j-1}|^{2}\nonumber\\
&+\big(2|\alpha_{j+1}|^{4}+|\alpha_{j+1}|^{2}|\alpha_{j}|^{2}+|\alpha_{j+1}|^{2}|\alpha_{j-1}|^{2}+|\alpha_{j}|^{2}|\alpha_{j-1}|^{2}\big)|\alpha_{j+1}-\alpha_{j}|^{2}\nonumber\\
&+\big(2|\alpha_{j-1}|^{4}+|\alpha_{j+1}|^{2}|\alpha_{j}|^{2}+|\alpha_{j+1}|^{2}|\alpha_{j-1}|^{2}+|\alpha_{j}|^{2}|\alpha_{j-1}|^{2}\big)|\alpha_{j}-\alpha_{j-1}|^{2}\nonumber\\
&+\big(|\alpha_{j+1}|^{2}+|\alpha_{j-1}|^{2}\big)|\alpha_{j+1}-\alpha_{j}|^{2}|\alpha_{j}-\alpha_{j-1}|^{2}\nonumber\\
&+\big(|\alpha_{j+1}-\alpha_{j}|^{4}+|\alpha_{j}-\alpha_{j-1}|^{4}\big)|\alpha_{j+1}-\alpha_{j-1}|^{2}\Big]\nonumber\\
&+\sum_{j=0}^{\infty}\rho_{j}^{2}\big(|\alpha_{j+1}|^{2}+|\alpha_{j-1}|^{2}\big)|\alpha_{j+1}-\alpha_{j-1}|^{2}\nonumber\\
&+\frac{3}{4}\sum_{j=0}^{\infty}|\alpha_{j}|^{4}\Big[|\alpha_{j+1}|^{4}+|\alpha_{j-1}|^{4}\Big]+\frac{3}{4}\rho_{j}^{4}\sum_{j=0}^{\infty}|\alpha_{j+1}-\alpha_{j-1}|^{4}\nonumber\\
&+\frac{1}{8}\sum_{j=0}^{\infty}\Big[|\alpha_{j}-\alpha_{j-1}|^{8}+4\big(|\alpha_{j}|^{4}+|\alpha_{j-1}|^{4}+2|\alpha_{j}|^{2}|\alpha_{j-1}|^{2}\big)|\alpha_{j}-\alpha_{j-1}|^{4}\nonumber\\
&+2\big(|\alpha_{j}|^{4}+|\alpha_{j-1}|^{4}\big)|\alpha_{j}-\alpha_{j-1}|^{4}\Big]\Big\},\nonumber
\end{align}
where the following identity is used
\begin{align*}
&|\alpha_{j+4}-4\alpha_{j+3}+6\alpha_{j+2}-4\alpha_{j+1}+\alpha_{j}|^{2}\nonumber\\
=&4|\alpha_{j+4}-\alpha_{j+3}|^{2}-6|\alpha_{j+4}-\alpha_{j+2}|^{2}+4|\alpha_{j+4}-\alpha_{j+1}|^{2}-|\alpha_{j+4}-\alpha_{j}|^{2}\nonumber\\
&+24|\alpha_{j+3}-\alpha_{j+2}|^{2}-16|\alpha_{j+3}-\alpha_{j+1}|^{2}+4|\alpha_{j+3}-\alpha_{j}|^{2}\nonumber\\
&+24|\alpha_{j+2}-\alpha_{j+1}|^{2}-6|\alpha_{j+2}-\alpha_{j}|^{2}+4|\alpha_{j+1}-\alpha_{j}|^{2}
\end{align*}
which implying
\begin{align*}
&-\frac{1}{16}\sum_{j=0}^{\infty}|\alpha_{j+4}-4\alpha_{j+3}+6\alpha_{j+2}-4\alpha_{j+1}+\alpha_{j}|^{2}\nonumber\\
=&\mathrm{bdy}-\frac{7}{2}\sum_{j=0}^{\infty}|\alpha_{j}-\alpha_{j-1}|^{2}+\frac{7}{4}\sum_{j=0}^{\infty}|\alpha_{j+1}-\alpha_{j-1}|^{2}-\frac{1}{2}\sum_{j=0}^{\infty}|\alpha_{j+2}-\alpha_{j-1}|^{2}\nonumber\\
&+\frac{1}{16}\sum_{j=0}^{\infty}|\alpha_{j+3}-\alpha_{j-1}|^{2}
\end{align*}
with boundary
\begin{align*}
\mathrm{bdy}=&\frac{1}{4}\sum_{j=0}^{3}|\alpha_{j}-\alpha_{j-1}|^{2}+\frac{3}{2}\sum_{j=0}^{2}|\alpha_{j}-\alpha_{j-1}|^{2}+\frac{3}{2}\sum_{j=0}^{1}|\alpha_{j}-\alpha_{j-1}|^{2}+\frac{1}{4}|\alpha_{0}+1|^{2}\nonumber\\
&-\frac{3}{8}\sum_{j=0}^{2}|\alpha_{j+1}-\alpha_{j-1}|^{2}-\sum_{j=0}^{1}|\alpha_{j+1}-\alpha_{j-1}|^{2}-\frac{3}{8}|\alpha_{1}+1|^{2}\nonumber\\
&+\frac{1}{4}\sum_{j=0}^{1}|\alpha_{j+2}-\alpha_{j-1}|^{2}+\frac{1}{4}|\alpha_{2}+1|^{2}-\frac{1}{16}|\alpha_{3}+1|^{2}.
\end{align*}
\end{proof}

\begin{thm}
Assume $\alpha\in\ell^{6}$ and $(S-1)\alpha\in \ell^{3}$, then
\begin{align}
&\int_{0}^{2\pi}(1-\cos\theta)^{4}\log w(\theta)\frac{d\theta}{2\pi}>-\infty\,\,\Longleftrightarrow\,\,  (S-1)^{4}\alpha\in \ell^{2}.
\end{align}
\end{thm}

\begin{rem}
This is again a one-case result of the original Simon conjecture under certain conditions.
\end{rem}

\begin{proof}
By Lemma 4.3 and applying (4.154), (4.155) and (4.161) to the sum rule (4.176), we have that if $\alpha\in\ell^{6}$ and $(S-1)\alpha\in \ell^{3}$, then $\mathrm{CP}<+\infty$ and the series in $\mathrm{EP}$ are convergent except the following one
\begin{equation}
-\frac{1}{16}\sum_{j=0}^{\infty}|\alpha_{j+3}-4\alpha_{j+2}+6\alpha_{j+1}-4\alpha_{j}+\alpha_{j-1}|^{2}.
\end{equation}
So $Z_{4,4}(\mu)>-\infty$ if and only if $\sum_{j=0}^{\infty}|\alpha_{j+3}-4\alpha_{j+2}+6\alpha_{j+1}-4\alpha_{j}+\alpha_{j-1}|^{2}<+\infty$ (i.e., $(S-1)^{4}\alpha\in\ell^{2}$) as $\alpha\in\ell^{6}$ and $(S-1)\alpha\in \ell^{3}$.
\end{proof}

\begin{thm}
Assume $\alpha\in\ell^{4}$, then
\begin{align}
&\int_{0}^{2\pi}(1-\cos\theta)^{4}\log w(\theta)\frac{d\theta}{2\pi}>-\infty\,\,\Longleftrightarrow\,\, (S-1)^{4}\alpha\in \ell^{2}.
\end{align}
\end{thm}

\begin{rem}
This is again one special case of a result in \cite{gz} due to Golinskii and Zlato\v{s}.
\end{rem}

\begin{proof}
As $\alpha\in\ell^{4}$, by the sum rule (4.176) and H\"older inequality, we have that the series in it are convergent except the only one (4.180). Thus $Z_{4,4}(\mu)>-\infty$ if and only if $(S-1)^{4}\alpha\in\ell^{2}$ as $\alpha\in\ell^{4}$.
\end{proof}

\begin{thm}
Assume $(S-1)\alpha\in \ell^{2}$, then
\begin{align}
&\int_{0}^{2\pi}(1-\cos\theta)^{4}\log w(\theta)\frac{d\theta}{2\pi}>-\infty\,\, \Longleftrightarrow \,\, \alpha\in \ell^{10}.
\end{align}
\end{thm}

\begin{rem}
This is one special case of the result in \cite{lu1} due to Lukic.
\end{rem}

\begin{proof}
Note that
\begin{align}
&-\frac{5}{4}\sum_{j=0}^{\infty}|\alpha_{j}|^{6}
-\frac{1}{8}\Big\{\frac{1}{2}\sum_{j=0}^{\infty}\big(|\alpha_{j+2}|^{2}(|\alpha_{j+1}|^{2}+|\alpha_{j}|^{2})+|\alpha_{j+1}|^{2}|\alpha_{j}|^{2}\big)(|\alpha_{j+3}|^{2}+|\alpha_{j-1}|^{2})\nonumber\\
&+\frac{1}{2}\sum_{j=0}^{\infty}\big(|\alpha_{j+1}|^{2}+|\alpha_{j}|^{2}\big)\Big[|\alpha_{j+2}|^{4}+|\alpha_{j-1}|^{4}+|\alpha_{j+1}|^{2}|\alpha_{j-1}|^{2}+|\alpha_{j+2}|^{2}|\alpha_{j}|^{2}\nonumber\\
&+2|\alpha_{j+2}|^{2}|\alpha_{j+1}|^{2}+2|\alpha_{j}|^{2}|\alpha_{j-1}|^{2}\Big]+\frac{1}{2}\sum_{j=0}^{\infty}\Big[|\alpha_{j}|^{4}|\alpha_{j-1}|^{2}+|\alpha_{j+1}|^{2}|\alpha_{j}|^{4}\Big]\nonumber\\
&+\sum_{j=0}^{\infty}|\alpha_{j}|^{2}\big(|\alpha_{j+1}|^{4}+|\alpha_{j-1}|^{4}\big)\Big\}\nonumber\\
&+\frac{1}{2}\sum_{j=0}^{\infty}(|\alpha_{j+2}|^{2}+|\alpha_{j-1}|^{2})|\alpha_{j+1}|^{2}|\alpha_{j}|^{2}\nonumber\\
&+\frac{1}{2}\sum_{j=0}^{\infty}\Big[|\alpha_{j+1}|^{2}|\alpha_{j}|^{4}+|\alpha_{j}|^{2}|\alpha_{j-1}|^{4}+|\alpha_{j}|^{4}|\alpha_{j-1}|^{2}+|\alpha_{j}|^{2}|\alpha_{j+1}|^{4}\Big]\nonumber\\
=&\frac{7}{12}-\frac{7}{16}|\alpha_{0}|^{2}-\frac{3}{8}|\alpha_{0}|^{4}+\frac{1}{16}|\alpha_{1}|^{2}+\frac{1}{16}|\alpha_{1}|^{4}-\frac{5}{16}|\alpha_{1}|^{2}|\alpha_{0}|^{2}
\nonumber\\
&+\frac{1}{16}|\alpha_{2}|^{2}|\alpha_{1}|^{2}+\frac{1}{16}|\alpha_{2}|^{2}|\alpha_{0}|^{2}+\frac{1}{16}\sum_{j=0}^{1}|\alpha_{j+1}|^{2}|\alpha_{j}|^{2}|\alpha_{j-1}|^{2}\nonumber\\
&+\frac{1}{16}\sum_{j=0}^{1}|\alpha_{j}|^{4}|\alpha_{j-1}|^{2}+\frac{1}{8}\sum_{j=0}^{1}|\alpha_{j}|^{2}|\alpha_{j-1}|^{4}+\frac{1}{12}\sum_{j=0}^{1}|\alpha_{j}|^{6}\nonumber\\
&-\frac{5}{8}\sum_{j=0}^{\infty}\big(\frac{2|\alpha_{j}|^{6}+|\alpha_{j-1}|^{6}}{3}-|\alpha_{j}|^{4}|\alpha_{j-1}|^{2}\big)\nonumber\\
&-\frac{5}{8}\sum_{j=0}^{\infty}\big(\frac{|\alpha_{j}|^{6}+2|\alpha_{j-1}|^{6}}{3}-|\alpha_{j}|^{2}|\alpha_{j-1}|^{4}\big)\nonumber\\
&+\frac{1}{8}\sum_{j=0}^{\infty}\big(\frac{2|\alpha_{j+1}|^{6}+|\alpha_{j-1}|^{6}}{3}-|\alpha_{j+1}|^{4}|\alpha_{j-1}|^{2}\big)\nonumber
\end{align}
\begin{align}
&+\frac{1}{8}\sum_{j=0}^{\infty}\big(\frac{|\alpha_{j+1}|^{6}+2|\alpha_{j-1}|^{6}}{3}-|\alpha_{j+1}|^{2}|\alpha_{j-1}|^{4}\big)\nonumber\\
&+\frac{1}{8}\sum_{j=0}^{\infty}\big(\frac{|\alpha_{j+2}|^{6}+|\alpha_{j+1}|^{6}+|\alpha_{j-1}|^{6}}{3}-|\alpha_{j+2}|^{2}|\alpha_{j+1}|^{2}|\alpha_{j-1}|^{2}\big)\nonumber\\
&+\frac{1}{8}\sum_{j=0}^{\infty}\big(\frac{|\alpha_{j+2}|^{6}+|\alpha_{j}|^{6}+|\alpha_{j-1}|^{6}}{3}-|\alpha_{j+2}|^{2}|\alpha_{j}|^{2}|\alpha_{j-1}|^{2}\big)\nonumber\\
&-\frac{1}{2}\sum_{j=0}^{\infty}\big(\frac{|\alpha_{j+1}|^{6}+|\alpha_{j}|^{6}+|\alpha_{j-1}|^{6}}{3}-|\alpha_{j+1}|^{2}|\alpha_{j}|^{2}|\alpha_{j-1}|^{2}\big)
\end{align}
and
\begin{align}
&-\frac{17}{16}\sum_{j=0}^{\infty}|\alpha_{j}|^{8}+\frac{1}{8}\Big\{\frac{1}{2}\sum_{j=0}^{\infty}|\alpha_{j+2}|^{2}|\alpha_{j+1}|^{2}|\alpha_{j}|^{2}\big(|\alpha_{j+3}|^{2}+|\alpha_{j-1}|^{2}\big)\nonumber\\
&+\frac{1}{2}\sum_{j=0}^{\infty}|\alpha_{j}|^{2}\Big[|\alpha_{j+1}|^{6}+|\alpha_{j-1}|^{6}+|\alpha_{j}|^{4}|\alpha_{j-1}|^{2}+|\alpha_{j+1}|^{2}|\alpha_{j}|^{4}\Big]\nonumber\\
&+\frac{1}{2}\sum_{j=0}^{\infty}|\alpha_{j+1}|^{2}|\alpha_{j}|^{2}\Big[|\alpha_{j+2}|^{4}+|\alpha_{j-1}|^{4}+|\alpha_{j+1}|^{2}|\alpha_{j-1}|^{2}+|\alpha_{j+2}|^{2}|\alpha_{j}|^{2}\nonumber\\
&+2|\alpha_{j+2}|^{2}|\alpha_{j+1}|^{2}+2|\alpha_{j}|^{2}|\alpha_{j-1}|^{2}\Big]+\frac{3}{4}\sum_{j=0}^{\infty}|\alpha_{j}|^{4}\Big[|\alpha_{j+1}|^{4}+|\alpha_{j-1}|^{4}\Big]\Big\}\nonumber\\
=&\frac{13}{32}-\frac{1}{16}|\alpha_{0}|^{2}-\frac{3}{32}|\alpha_{0}|^{4}-\frac{1}{16}|\alpha_{0}|^{6}-\frac{1}{16}|\alpha_{1}|^{2}|\alpha_{0}|^{2}-\frac{1}{16}|\alpha_{1}|^{4}|\alpha_{0}|^{2}\nonumber\\
&-\frac{1}{8}|\alpha_{1}|^{2}|\alpha_{0}|^{4}-\frac{1}{16}|\alpha_{2}|^{2}|\alpha_{1}|^{2}|\alpha_{0}|^{2}-\frac{1}{32}\sum_{j=0}^{1}|\alpha_{j}|^{8}-\frac{3}{16}|\alpha_{0}|^{8}\nonumber\\
&-\frac{1}{8}\sum_{j=0}^{\infty}\big(\frac{|\alpha_{j+2}|^{8}+|\alpha_{j+1}|^{8}+|\alpha_{j}|^{8}+|\alpha_{j-1}|^{8}}{4}-|\alpha_{j+2}|^{2}|\alpha_{j+1}|^{2}|\alpha_{j}|^{2}|\alpha_{j-1}|^{2}\big)\nonumber\\
&-\frac{1}{8}\sum_{j=0}^{\infty}\big(\frac{3|\alpha_{j}|^{8}+|\alpha_{j-1}|^{8}}{4}-|\alpha_{j}|^{6}|\alpha_{j-1}|^{2}\big)\nonumber\\
&-\frac{1}{8}\sum_{j=0}^{\infty}\big(\frac{|\alpha_{j}|^{8}+3|\alpha_{j-1}|^{8}}{4}-|\alpha_{j}|^{2}|\alpha_{j-1}|^{6}\big)\nonumber\\
&-\frac{1}{4}\sum_{j=0}^{\infty}\big(\frac{|\alpha_{j+1}|^{8}+2|\alpha_{j}|^{8}+|\alpha_{j-1}|^{8}}{4}-|\alpha_{j+1}|^{2}|\alpha_{j}|^{4}|\alpha_{j-1}|^{2}\big)\nonumber\\
&-\frac{1}{8}\sum_{j=0}^{\infty}\big(\frac{2|\alpha_{j+1}|^{8}+|\alpha_{j}|^{8}+|\alpha_{j-1}|^{8}}{4}-|\alpha_{j+1}|^{4}|\alpha_{j}|^{2}|\alpha_{j-1}|^{2}\big)\nonumber\\
&-\frac{1}{8}\sum_{j=0}^{\infty}\big(\frac{|\alpha_{j+1}|^{8}+|\alpha_{j}|^{8}+2|\alpha_{j-1}|^{8}}{4}-|\alpha_{j+1}|^{2}|\alpha_{j}|^{2}|\alpha_{j-1}|^{4}\big)\nonumber
\end{align}
\begin{align}
&-\frac{3}{16}\sum_{j=0}^{\infty}\big(\frac{|\alpha_{j}|^{8}+|\alpha_{j-1}|^{8}}{2}-|\alpha_{j}|^{4}|\alpha_{j-1}|^{4}\big),
\end{align}
then if $(S-1)\alpha\in \ell^{2}$, by (4.124), (4.161) and the sum rule (4.176), we have the series in $\mathrm{CP}$ and $\mathrm{EP}$ are finite except the following one
\begin{align*}
\frac{35}{8}\sum_{j=0}^{\infty}\Big(\log(1-|\alpha_{j}|^{2})+|\alpha_{j}|^{2}+\frac{1}{2}|\alpha_{j}|^{4}+\frac{1}{3}|\alpha_{j}|^{6}+\frac{1}{4}|\alpha_{j}|^{8}\Big).
\end{align*}
Thus, by Lemma 4.3, we get that
$Z_{4,4}(\mu)>-\infty$ is equivalent to $\alpha\in \ell^{10}$ as $(S-1)\alpha\in \ell^{2}$.
\end{proof}

As Theorems 4.58 and 4.66, with respect to one-directional implication from $\alpha$ to $Z_{4,4}(\mu)$, we have

\begin{thm}
Assume $\alpha\in\ell^{6}$ and $(S-1)\alpha\in \ell^{2}$, then
\begin{align}
&\int_{0}^{2\pi}(1-\cos\theta)^{4}\log w(\theta)\frac{d\theta}{2\pi}>-\infty.
\end{align}
\end{thm}

\begin{proof}
For a direct proof, similar to Theorem 4.68, it mainly follows from (4.154) and (4.155). For an indirect proof, note that $(S-1)\alpha\in \ell^{2}$ implies $(S-1)\alpha\in \ell^{3}$ and $(S-1)^{4}\alpha\in \ell^{2}$, it is a consequence of Theorem 4.68. More immediately, it is a consequence of Theorem 4.72 since $\alpha\in\ell^{6}$ implies $\alpha\in\ell^{10}$.
\end{proof}

\section{Sum rules and higher order Szeg\H{o} theorems: General cases}

In the last section, a few of sum rules and higher order Szeg\H{o} theorems have been established in some concrete cases. One can find that there exist some common characteristics in these known results (for instance, see Theorems 4.49, 4.52, 4.60 and 4.68). In fact, these common characteristics also appear in higher order Szeg\H{o} theorems of all other cases. In this section, we will develop general sum rules and higher order Szeg\H{o} theorems. More importantly, an expression of $w_{m}$ with a single infinite index is given for any $m\in \mathbb{N}$.

At first, we generalize (4.87), (4.105) and their analogues in higher order Szeg\H{o} theorems in forth order case. These generalized results can be used in the calculations for some analogous higher order Szeg\H{o} theorem of all orders as the ones in the last section.
\begin{thm}
Let $\alpha=\{\alpha_{j}\}_{0}^{\infty}\in\ell^{2}$ and $S$ be the left shift operator on sequences satisfying $(S\alpha)_{j}=\alpha_{j+1}$, then
\begin{align}
\|(S-1)^{k}\alpha\|_{2}^{2}=\sum_{0\leq l<l^{\prime}\leq k}(-1)^{l+l^{\prime}-1}C_{k}^{l}C_{k}^{l^{\prime}}\|(S^{k-l}-S^{k-l^{\prime}})\alpha\|_{2}^{2}
\end{align}
and
\begin{align}
\|(S+1)^{k}\alpha\|_{2}^{2}=2^{k}\sum_{l=0}^{k}C_{k}^{l}\|S^{k-l}\alpha\|^{2}-\sum_{0\leq l<l^{\prime}\leq k}C_{k}^{l}C_{k}^{l^{\prime}}\|(S^{k-l}-S^{k-l^{\prime}})\alpha\|_{2}^{2},
\end{align}
where $C_{k}^{l}$ is the binomial coefficient.
More generally,
\begin{align}
\|P_{k}(S)\alpha\|_{2}^{2}=P_{k}(1)\sum_{l=0}^{k}a_{l}\|S^{l}\alpha\|^{2}-\sum_{0\leq l<l^{\prime}\leq k}a_{l}a_{l^{\prime}}\|(S^{l}-S^{l^{\prime}})\alpha\|_{2}^{2}
\end{align}
where $P_{k}(x)=a_{0}+a_{1}x+\cdots+a_{k}x^k$ with $a_{l}\in \mathbb{R}$, $0\leq l\leq k$.
\end{thm}

\begin{proof}
It is enough to prove (5.1). Similar is to (5.2) and (5.3). Note that
\begin{align}
&\big((S-1)^{k}\alpha\big)_{j}=\Big[\Big(\sum_{l=0}^{k}(-1)^{l}C_{k}^{l}S^{k-l}\Big)\alpha\Big]_{j}=\sum_{l=0}^{k}(-1)^{l}C_{k}^{l}\big(S^{k-l}\alpha\big)_{j}\nonumber\\
=&\sum_{l=0}^{k}(-1)^{l}C_{k}^{l}\alpha_{j+k-l},
\end{align}
then
\begin{align}
&\big|\big((S-1)^{k}\alpha\big)_{j}\big|^{2}=\left(\sum_{l=0}^{k}(-1)^{l}C_{k}^{l}\alpha_{j+k-l}\right)\left(\sum_{l=0}^{k}(-1)^{l}C_{k}^{l}\overline{\alpha}_{j+k-l}\right)\nonumber\\
=&\sum_{l=0}^{k}\sum_{l^{\prime}=0}^{k}(-1)^{l+l^{\prime}}C_{k}^{l}C_{k}^{l^{\prime}}\alpha_{j+k-l}\overline{\alpha}_{j+k-l^{\prime}}\nonumber\\
=&\sum_{l=0}^{k}\big(C_{k}^{l}\big)^{2}|\alpha_{j+k-l}|^{2}+\sum_{0\leq l<l^{\prime}\leq k}(-1)^{l+l^{\prime}}C_{k}^{l}C_{k}^{l^{\prime}}\big(\alpha_{j+k-l}\overline{\alpha}_{j+k-l^{\prime}}+\overline{\alpha}_{j+k-l}\alpha_{j+k-l^{\prime}}\big)\nonumber\\
=&\sum_{l=0}^{k}\big(C_{k}^{l}\big)^{2}|\alpha_{j+k-l}|^{2}+\sum_{0\leq l<l^{\prime}\leq k}(-1)^{l+l^{\prime}}C_{k}^{l}C_{k}^{l^{\prime}}\big(|\alpha_{j+k-l}|^{2}+|\alpha_{j+k-l^{\prime}}|^{2}\big)\nonumber\\
&+\sum_{0\leq l<l^{\prime}\leq k}(-1)^{l+l^{\prime}-1}C_{k}^{l}C_{k}^{l^{\prime}}|\alpha_{j+k-l}-\alpha_{j+k-l^{\prime}}|^{2}\nonumber\\
=&\sum_{l=0}^{k}\big(C_{k}^{l}\big)^{2}|\alpha_{j+k-l}|^{2}+\sum_{0\leq l\neq l^{\prime}\leq k}(-1)^{l+l^{\prime}}C_{k}^{l}C_{k}^{l^{\prime}}|\alpha_{j+k-l}|^{2}\nonumber\\
&+\sum_{0\leq l<l^{\prime}\leq k}(-1)^{l+l^{\prime}-1}C_{k}^{l}C_{k}^{l^{\prime}}|\alpha_{j+k-l}-\alpha_{j+k-l^{\prime}}|^{2}\nonumber\\
=&\sum_{l=0}^{k}\sum_{l^{\prime}=0}^{k}(-1)^{l+l^{\prime}}C_{k}^{l}C_{k}^{l^{\prime}}|\alpha_{j+k-l}|^{2}+\sum_{0\leq l<l^{\prime}\leq k}(-1)^{l+l^{\prime}-1}C_{k}^{l}C_{k}^{l^{\prime}}|\alpha_{j+k-l}-\alpha_{j+k-l^{\prime}}|^{2}\nonumber\\
=&\sum_{l=0}^{k}(-1)^{l}C_{k}^{l}|\alpha_{j+k-l}|^{2}\sum_{l^{\prime}=0}^{k}(-1)^{l^{\prime}}C_{k}^{l^{\prime}}+\sum_{0\leq l<l^{\prime}\leq k}(-1)^{l+l^{\prime}-1}C_{k}^{l}C_{k}^{l^{\prime}}|\alpha_{j+k-l}-\alpha_{j+k-l^{\prime}}|^{2}\nonumber\\
=&\sum_{0\leq l<l^{\prime}\leq k}(-1)^{l+l^{\prime}-1}C_{k}^{l}C_{k}^{l^{\prime}}|\alpha_{j+k-l}-\alpha_{j+k-l^{\prime}}|^{2}.
\end{align}
Thus (5.1) immediately follows from (5.5).
\end{proof}

\begin{thm}
Let $\alpha=\{\alpha_{j}\}_{0}^{\infty}\in\ell^{2}$ and $S$ be the left shift operator, then
\begin{align}
&\|(S-1)^{m+k}(S+1)^{k}\alpha\|_{2}^{2}\nonumber\\
=&\sum_{0\leq p<p^{\prime}\leq m}\,\sum_{0\leq l<l^{\prime}\leq k}(-1)^{p+p^{\prime}+l+l^{\prime}}C_{m}^{p}C_{m}^{p^{\prime}}C_{k}^{l}C_{k}^{l^{\prime}}\|\big(S^{m-p}-S^{m-p^{\prime}}\big)\big(S^{2(k-l)}-S^{2(k-l^{\prime})}\big)\alpha\|_{2}^{2}
\end{align}
and
\begin{align}
&\|(S+1)^{m+k}(S-1)^{k}\alpha\|_{2}^{2}\nonumber\\
=&2^{m}\sum_{p=0}^{m}\sum_{0\leq l<l^{\prime}\leq k}(-1)^{l+l^{\prime}-1}C_{m}^{p}C_{k}^{l}C_{k}^{l^{\prime}}\|S^{m-p}(S^{2(k-l)}-S^{2(k-l^{\prime})})\alpha\|_{2}^{2}\nonumber\\
&+\sum_{0\leq p<p^{\prime}\leq m}\,\sum_{0\leq l<l^{\prime}\leq k}(-1)^{l+l^{\prime}}C_{m}^{p}C_{m}^{p^{\prime}}C_{k}^{l}C_{k}^{l^{\prime}}\|\big(S^{m-p}-S^{m-p^{\prime}}\big)\big(S^{2(k-l)}-S^{2(k-l^{\prime})}\big)\alpha\|_{2}^{2}.
\end{align}
\end{thm}

\begin{proof}
Since
\begin{align}
(S-1)^{k}(S+1)^{k}=(S^{2}-1)^{2}=\sum_{l=0}^{k}(-1)^{l}C_{k}^{l}S^{2(k-l)},
\end{align}
by a similar argument to (5.1), we have
\begin{align}
\|(S^{2}-1)^{k}\beta\|_{2}^{2}=\sum_{0\leq l<l^{\prime}\leq k}(-1)^{l+l^{\prime}-1}C_{k}^{l}C_{k}^{l^{\prime}}\|(S^{2(k-l)}-S^{2(k-l^{\prime})})\beta\|_{2}^{2}.
\end{align}
Let $\beta=(S-1)^{m}\alpha$ and $\gamma=(S^{2(k-l)}-S^{2(k-l^{\prime})})\alpha$, then \begin{equation}(S^{2}-1)^{k}\beta=(S-1)^{m+k}(S+1)^{k}\alpha\end{equation} and \begin{equation}(S^{2(k-l)}-S^{2(k-l^{\prime})})\beta=(S-1)^{m}(S^{2(k-l)}-S^{2(k-l^{\prime})})\alpha=(S-1)^{m}\gamma.\end{equation}

By (5.1),
\begin{align}
&\|(S^{2(k-l)}-S^{2(k-l^{\prime})})\beta\|_{2}^{2}=\|(S-1)^{m}\gamma\|_{2}^{2}\nonumber\\
=&\sum_{0\leq p<p^{\prime}\leq m}(-1)^{p+p^{\prime}-1}C_{m}^{p}C_{m}^{p^{\prime}}\|(S^{m-p}-S^{m-p^{\prime}})\gamma\|_{2}^{2}\nonumber\\
=&\sum_{0\leq p<p^{\prime}\leq m}(-1)^{p+p^{\prime}-1}C_{m}^{p}C_{m}^{p^{\prime}}\|(S^{m-p}-S^{m-p^{\prime}})(S^{2(k-l)}-S^{2(k-l^{\prime})})\alpha\|_{2}^{2}.
\end{align}
Thus (5.6) follows from (5.9), (5.10) and (5.12).

Similarly, let $\zeta=(S+1)^{m}\alpha$, then
\begin{equation}(S^{2}-1)^{k}\zeta=(S+1)^{m+k}(S-1)^{k}\alpha\end{equation}
and \begin{equation}(S^{2(k-l)}-S^{2(k-l^{\prime})})\zeta=(S+1)^{m}(S^{2(k-l)}-S^{2(k-l^{\prime})})\alpha=(S+1)^{m}\gamma.\end{equation}
By (5.2),
\begin{align}
&\|(S^{2(k-l)}-S^{2(k-l^{\prime})})\zeta\|_{2}^{2}=\|(S+1)^{m}\gamma\|_{2}^{2}\nonumber\\
=&2^{m}\sum_{p=0}^{m}C_{m}^{p}\|S^{m-p}\gamma\|^{2}-\sum_{0\leq p<p^{\prime}\leq m}C_{m}^{p}C_{m}^{p^{\prime}}\|(S^{m-p}-S^{m-p^{\prime}})\gamma\|_{2}^{2}\nonumber\\
=&2^{m}\sum_{p=0}^{m}C_{m}^{p}\|S^{m-p}(S^{2(k-l)}-S^{2(k-l^{\prime})})\alpha\|^{2}\nonumber\\
&-\sum_{0\leq p<p^{\prime}\leq m}C_{m}^{p}C_{m}^{p^{\prime}}\|(S^{m-p}-S^{m-p^{\prime}})(S^{2(k-l)}-S^{2(k-l^{\prime})})\alpha\|_{2}^{2}.
\end{align}
Therefore, (5.7) holds by (5.13), (5.15) and (5.9) in which $\zeta$ in place of $\beta$.
\end{proof}

\begin{thm}
\begin{align}
&\|(S^{m}-S^{n})(S^{p}-S^{q})\alpha\|_{2}^{2}=\|S^{m}(S^{p}-S^{q})\alpha\|_{2}^{2}+\|S^{n}(S^{p}-S^{q})\alpha\|_{2}^{2}\nonumber\\
&+\|(S^{m}-S^{n})S^{p}\alpha\|_{2}^{2}+\|(S^{m}-S^{n})S^{q}\alpha\|_{2}^{2}-\|(S^{m+p}-S^{n+q})\alpha\|_{2}^{2}\nonumber\\
&-\|(S^{m+q}-S^{n+p})\alpha\|_{2}^{2}.
\end{align}
\end{thm}

\begin{proof}
Note that
\begin{align}
&[(S^{m}-S^{n})(S^{p}-S^{q})\alpha]_{j}=[(S^{m+p}-S^{m+q}-S^{n+p}+S^{n+q})\alpha]_{j}\nonumber\\
=&\alpha_{j+m+p}-\alpha_{j+m+q}-\alpha_{j+n+p}+\alpha_{j+n+q},
\end{align}
then (5.16) holds by Proposition 4.16.
\end{proof}

More generally, we have
\begin{thm}
Let $P_{k}(\theta)=a_{0}+a_{1}\cos\theta+\cdots +a_{k}\cos k\theta$ with $a_{l}\in \mathbb{R}$, $0\leq l\leq k$, $P_{k}(0)=0$ and $P_{k}(\theta)\geq 0$ for $\theta\in [0, 2\pi)$, then their exist $Q_{k}(z)=b_{0}+b_{1}z+\cdots+b_{k}z^{k}$ with $b_{l}\in \mathbb{R}$, $0\leq l\leq k$, such that
\begin{align}
\|Q_{k}(S)\alpha\|_{2}^{2}=\mathrm{bdy}-\frac{1}{2}\sum_{l=1}^{k}a_{l}\|(S^{l}-1)\alpha\|_{2}^{2},
\end{align}
for any sequence $\alpha=\{\alpha_{j}\}_{0}^{\infty}\in \ell^{2}$, where
\begin{equation}
\mathrm{bdy}=\sum_{0\leq l< l^{\prime}\leq k}b_{l}b_{l^{\prime}}\left(\sum_{j=0}^{l-1}|\alpha_{j+l^{\prime}-l}-\alpha_{j}|^{2}-\sum_{j=l}^{l^{\prime}-1}|\alpha_{j}|^{2}\right)
\end{equation}
in which $\rm{bdy}$ is boundary for short and means finite number terms with small indices (cf. \cite{bbz1}), and $S$ is the left shift operator.
\end{thm}

\begin{proof}
Let $L_{k}(z)=\frac{1}{2}\sum_{l=0}^{k}a_{l}(z^{-l}+z^{l})$,  then $L_{k}(e^{i\theta})=P_{k}(\theta)\geq 0$ for any $\theta\in [0, 2\pi]$. That is, $L_{k}(z)$ is a Laurent polynomials fulfilling that its restriction on the unit circle is nonnegative. By Fej\'{e}{r}-Riesz theorem, there exists a polynomial $Q_{k}(z)=\sum_{l=0}^{k}b_{l}z^{l}$ such that
\begin{equation}
L_{k}(e^{i\theta})=P_{k}(\theta)=|Q_{k}(e^{i\theta})|^{2},\,\, \theta\in [0, 2\pi].
\end{equation}
Expanding (5.20), we have
\begin{align}
\frac{1}{2}\sum_{j=0}^{k}a_{j}(e^{-ij\theta}+e^{ij\theta})=\sum_{l=0}^{k}|b_{l}|^{2}+\sum_{0\leq l< l^{\prime}\leq k}\big(b_{l}\overline{b}_{l^{\prime}}e^{i(l-l^{\prime}\theta)}+\overline{b}_{l}b_{l^{\prime}}e^{-i(l-l^{\prime})\theta}\big)
\end{align}
for any $\theta\in [0, 2\pi]$. Comparing coefficients in both sides, we obtain
\begin{equation}
\begin{cases}
a_{0}=\sum_{l=0}^{k}|b_{l}|^{2},\vspace{1mm}\\
a_{j}=2\sum_{0\leq l< l^{\prime}\leq k,\,l^{\prime}-l=j}\overline{b}_{l}b_{l^{\prime}}=2\sum_{0\leq l< l^{\prime}\leq k,\,l^{\prime}-l=j}b_{l}\overline{b}_{l^{\prime}}, 1\leq j\leq k.
\end{cases}
\end{equation}

Since $a_{j}\in \mathbb{R}$, $0\leq j\leq k$, by (5.22), we can take $b_{l}\in \mathbb{R}$, $0\leq l\leq k$. That is,
\begin{equation}
\begin{cases}
a_{0}=\sum_{l=0}^{k}b_{l}^{2},\vspace{1mm}\\
a_{j}=2\sum_{0\leq l\leq k-j}b_{l}b_{l+j}, 1\leq j\leq k,
\end{cases}
\end{equation}
in which $a_{l}, b_{l}\in \mathbb{R}$, $0\leq l\leq k$. Furthermore, we have
\begin{align}
\sum_{l=0}^{k}b_{l}^{2}+2\sum_{j=1}^{k}\sum_{0\leq l\leq k-j}b_{l}b_{l+j}=0
\end{align}
or in another form as follows
\begin{align}
\sum_{l=0}^{k}b_{l}^{2}+2\sum_{0\leq l<l^{\prime}\leq k}b_{l}b_{l^{\prime}}=0
\end{align}
since $P_{k}(0)=0$ implying $\sum_{l=0}^{k}a_{l}=0$.

Note that
\begin{align}
&\Big|\sum_{l=0}^{k}b_{l}\alpha_{j+l}\Big|^{2}=\sum_{l=0}^{k}b_{l}^{2}|\alpha_{j+l}|^{2}+\sum_{0\leq l< l^{\prime}\leq k}
b_{l}b_{l^{\prime}}\big(\alpha_{j+l}\overline{\alpha}_{j+l^{\prime}}+\overline{\alpha}_{j+l}\alpha_{j+l^{\prime}}\big)\nonumber\\
=&\sum_{l=0}^{k}b_{l}^{2}|\alpha_{j+l}|^{2}+\sum_{0\leq l< l^{\prime}\leq k}b_{l}b_{l^{\prime}}(|\alpha_{j+l}|^{2}+|\alpha_{j+l^{\prime}}|^{2})
-\sum_{0\leq l< l^{\prime}\leq k}b_{l}b_{l^{\prime}}|\alpha_{j+l}-\alpha_{j+l^{\prime}}|^{2}
\end{align}
then by (5.23), (5.25) and (5.26), for $\alpha=\{\alpha_{j}\}_{0}^{\infty}\in \ell^{2}$,
\begin{align}
&\|Q_{k}(S)\alpha\|_{2}^{2}=\sum_{l=0}^{k}b_{l}^{2}\|S^{l}\alpha\|_{2}^{2}
+\sum_{0\leq l< l^{\prime}\leq k}b_{l}b_{l^{\prime}}(\|S^{l}\alpha\|_{2}^{2}+\|S^{l^{\prime}}\alpha\|_{2}^{2})\nonumber\\
&-\sum_{0\leq l< l^{\prime}\leq k}b_{l}b_{l^{\prime}}\|(S^{l^{\prime}}-S^{l})\alpha\|_{2}^{2}\nonumber\\
=&\sum_{l=0}^{k}b_{l}^{2}\|\alpha\|_{2}^{2}-\sum_{l=0}^{k}b_{l}^{2}\sum_{j=0}^{l-1}|\alpha_{j}|^{2}+2\sum_{0\leq l< l^{\prime}\leq k}b_{l}b_{l^{\prime}}\|\alpha\|_{2}^{2}-2\sum_{0\leq l< l^{\prime}\leq k}b_{l}b_{l^{\prime}}\sum_{j=0}^{l-1}|\alpha_{j}|^{2}\nonumber\\
&-\sum_{0\leq l< l^{\prime}\leq k}b_{l}b_{l^{\prime}}\sum_{j=l}^{l^{\prime}-1}|\alpha_{j}|^{2}-\sum_{0\leq l< l^{\prime}\leq k}b_{l}b_{l^{\prime}}\|(S^{l^{\prime}-l}-1)\alpha\|_{2}^{2}\nonumber\\
&+\sum_{0\leq l< l^{\prime}\leq k}b_{l}b_{l^{\prime}}\sum_{j=0}^{l-1}|\alpha_{j+l^{\prime}-l}-\alpha_{j}|^{2}\nonumber\\
=&\sum_{0\leq l< l^{\prime}\leq k}b_{l}b_{l^{\prime}}\left(\sum_{j=0}^{l-1}|\alpha_{j+l^{\prime}-l}-\alpha_{j}|^{2}-\sum_{j=l}^{l^{\prime}-1}|\alpha_{j}|^{2}\right)-\sum_{0\leq l< l^{\prime}\leq k}b_{l}b_{l^{\prime}}\|(S^{l^{\prime}-l}-1)\alpha\|_{2}^{2}\nonumber\\
=&\sum_{0\leq l< l^{\prime}\leq k}b_{l}b_{l^{\prime}}\left(\sum_{j=0}^{l-1}|\alpha_{j+l^{\prime}-l}-\alpha_{j}|^{2}-\sum_{j=l}^{l^{\prime}-1}|\alpha_{j}|^{2}\right)-\sum_{j=1}^{k}\sum_{0\leq l\leq k-j}b_{l}b_{l+j}\|(S^{j}-1)\alpha\|_{2}^{2}\nonumber\\
=&\sum_{0\leq l< l^{\prime}\leq k}b_{l}b_{l^{\prime}}\left(\sum_{j=0}^{l-1}|\alpha_{j+l^{\prime}-l}-\alpha_{j}|^{2}-\sum_{j=l}^{l^{\prime}-1}|\alpha_{j}|^{2}\right)-\frac{1}{2}\sum_{j=1}^{k}a_{j}\|(S^{j}-1)\alpha\|_{2}^{2}.\nonumber\qedhere
\end{align}
\end{proof}

\begin{rem}
As above stated, we have used some special cases of (5.18) to get some concrete sum rules in the last section (for example, in Theorems 4.18, 4.32, 4.51, 4.59 and so on).
\end{rem}
Next, we give an explicit expression of $w_{m}$ for any $m\in \mathbb{N}$. As in Section 4, it is important to know the expression of $w_{m}$ for establishing sum rules and higher order Szeg\H{o} theorems.

In order to do so, we introduce the following notions and notations.

For any formal ordered sum $a_{1}+a_{2}+\cdots+a_{n}$, $n\in \mathbb{N}$, taking out $j-1$ plus signs from its all $n-1$ plus signs, the formal sum will be partitioned to $j$ different segments. Any result of such operation is called a $j$-partition of the ordered sum $a_{1}+a_{2}+\cdots+a_{n}$ and denoted by $\{\mbox {sgm}_{1}, \mbox {sgm}_2, \ldots, \mbox {sgm}_j\}$ in which $\mbox{sgm}_{l}$ is the $l$th segment, $1\leq l\leq j$. If one of such $j$-partitions further satisfies that the numbers of plus signs in its segments are non-increasing from left to right, then this $j$-partition is called a good $j$-partition of the ordered sum $a_{1}+a_{2}+\cdots+a_{n}$. For example, $\{a_{1}+a_{2}+a_{3}, a_{4}, a_{5}\}$, $\{a_{1}+a_{2}, a_{3}+a_{4}, a_{5}\}$ and $\{a_{1}, a_{2}+a_{3}, a_{4}, a_{5}\}$ are $3$-partitions of $a_{1}+a_{2}+a_{3}+a_{4}+a_{5}$. But the former two are good and the latter one is not good. Sometimes we also call $j$-partitions of ordered $n$-tuple $(a_{1}, a_{2}, \ldots, a_{n})$ since the mapping $a_{1}+a_{2}+\cdots+a_{n}\rightarrow(a_{1}, a_{2}, \ldots, a_{n})$ is one to one.

For any formal ordered sum $r_{1}+r_{2}+\cdots+r_{n}$, denote
\begin{align}
&\alpha^{(r_{1}+r_{2}+\cdots+r_{n})}
=\sum_{k=0}^{\infty}\sum_{l_{1}=1}^{r_{1}}\left(\prod_{\nu=2}^{n-1}\sum_{l_{\nu}=1}^{l_{\nu-1}+r_{\nu}}\right)\alpha_{k+r_{1}-1}\overline{\alpha}_{k-1}\prod_{\nu=2}^{n}\alpha_{k+l_{\nu-1}+r_{\nu}-1}\overline{\alpha}_{k+l_{\nu-1}-1}.
\end{align}

In what follows, $\mathcal{D}_{m}^{s}=\{\langle r_{1},r_{2},\ldots,r_{s}\rangle:\, \sum_{l=0}^{s}r_{l}=m, r_{l}\in \mathbb{N}, 1\leq l\leq s\}$, where $\langle r_{1},r_{2},\ldots,r_{s}\rangle$ is an unordered $s$-tuple consisting of $s$ elements $r_{l}$, $1\leq l\leq s$. That is, $\mathcal{D}_{m}^{s}$ is the set of all unordered $s$-decompositions of $m$.

\begin{lem}
For any $n$-tuple $\sigma_{n}^{r}=(r_{1},r_{2},\ldots,r_{n})$ and any sequence $A=\{A_{j}\}_{0}^{\infty}\in \ell^{1}$,
\begin{align}
&\sum_{k_{1}=0}^{\infty}A_{k_{1}+r_{1}}\sum_{k_{2}=0}^{k_{1}+r_{1}}A_{k_{2}+r_{2}}\cdots
\sum_{k_{n}=0}^{k_{n-1}+r_{n-1}}A_{k_{n}+r_{n}}\nonumber\\
=&\sum_{j=1}^{n}\sum_{\xi\in\Sigma_{\sigma^{r}_{n}}^{n-j}}\sum_{k_{1}=0}^{\infty}A_{k_{1}+\mathrm{seg}^{\xi}_{1}(\sigma^{r}_{n})}
\sum_{k_{2}=0}^{k_{1}}A_{k_{2}+\mathrm{seg}^{\xi}_{2}(\sigma^{r}_{n})}\cdots\sum_{k_{j}=0}^{k_{j-1}}A_{k_{j}+\mathrm{seg}^{\xi}_{j}(\sigma^{r}_{n})},
\end{align}
where $\xi=\{\mathrm{seg}^{\xi}_{1}(\sigma^{r}), \mathrm{seg}^{\xi}_{2}(\sigma^{r}), \ldots, \mathrm{seg}^{\xi}_{j}(\sigma^{r})\}$ is a $j$-partition of $\sigma^{r}$,  $\Sigma_{\sigma^{r}}^{n-j}$ is the set of all $j$-partitions of $\sigma^{r}$ and
\begin{align*}
A_{k_{\nu}+\mathrm{seg}^{\xi}_{\nu}(\sigma^{r})}=\sum_{l_{1}=1}^{r_{n_{\nu-1}+1}}\prod_{t_{\nu}=2}^{n_{\nu}-n_{\nu-1}-1}\sum_{l_{t_{\nu}}=1}^{l_{t_{\nu}-1}+r_{n_{\nu-1}+t_{\nu}}}A_{k_{\nu}+r_{n_{\nu-1}+1}}
\prod_{t_{\nu}=2}^{n_{\nu}-n_{\nu-1}}A_{k_{\nu}+l_{t_{\nu}-1}+r_{n_{\nu-1}+t_{\nu}}}
\end{align*}
in which $n_{v}$ is the biggest subscript of summands, $r_{\cdot}$, in $\mathrm{seg}^{\xi}_{\nu}(\sigma^{r})$, $1\leq \nu \leq j$. As a convention, $n_{0}=0$.
\end{lem}

\begin{proof}
Take the method of induction.

It is trivial as $n=1$. Assume that (5.28) holds for $n=p$, in the following, we will prove (5.28) also holds as $n=p+1$.

For $\sigma_{p+1}^{r}=(r_{1},r_{2},\ldots,r_{p+1})$ and $\sigma_{p}^{r}=(r_{1},r_{2},\ldots,r_{p})$, set
\begin{equation}
\begin{cases}
A^{\prime}_{s+r_{l}}=A_{s+r_{l}},\,\, 1\leq l\leq p-1,\vspace{2mm}\\
A^{\prime}_{s+r_{p}}=A_{s+r_{p}}\sum_{t=0}^{s+r_{p}}A_{t+r_{p+1}}
\end{cases}
\end{equation}
for any $s,t\in \mathbb{N}_{0}$, then by the assumption of induction,
\begin{align}
&\sum_{k_{1}=0}^{\infty}A_{k_{1}+r_{1}}\sum_{k_{2}=0}^{k_{1}+r_{1}}A_{k_{2}+r_{2}}\cdots
\sum_{k_{p}=0}^{k_{p-1}+r_{p-1}}A_{k_{p}+r_{p}}\sum_{k_{p+1}=0}^{k_{p}+r_{p}}A_{k_{p+1}+r_{p+1}}\nonumber\\
=&\sum_{k_{1}=0}^{\infty}A^{\prime}_{k_{1}+r_{1}}\sum_{k_{2}=0}^{k_{1}+r_{1}}A^{\prime}_{k_{2}+r_{2}}\cdots
\sum_{k_{p}=0}^{k_{p-1}+r_{p-1}}A^{\prime}_{k_{p}+r_{p}}\nonumber\\
=&\sum_{j=1}^{p}\sum_{\xi^{\prime}\in\Sigma_{\sigma^{r}_{p}}^{p-j}}\sum_{k_{1}=0}^{\infty}A^{\prime}_{k_{1}+\mathrm{seg}^{\xi^{\prime}}_{1}(\sigma^{r}_{p})}
\sum_{k_{2}=0}^{k_{1}}A^{\prime}_{k_{2}+\mathrm{seg}^{\xi^{\prime}}_{2}(\sigma^{r}_{p})}\cdots\sum_{k_{j}=0}^{k_{j-1}}A^{\prime}_{k_{j}+\mathrm{seg}^{\xi^{\prime}}_{j}(\sigma^{r}_{p})},
\end{align}
where $\xi^{\prime}=\{\mathrm{seg}^{\xi^{\prime}}_{1}(\sigma^{r}), \mathrm{seg}^{\xi^{\prime}}_{2}(\sigma^{r}), \ldots, \mathrm{seg}^{\xi^{\prime}}_{j}(\sigma^{r})\}$ is a $j$-partition of $\sigma^{r}$ and
\begin{align}
A^{\prime}_{k_{\nu}+\mathrm{seg}^{\xi^{\prime}}_{\nu}(\sigma^{r})}=\sum_{l_{1}=1}^{r_{n_{\nu-1}+1}}\prod_{t_{\nu}=2}^{n_{\nu}-n_{\nu-1}-1}\sum_{l_{t_{\nu}}=1}^{l_{t_{\nu}-1}+r_{n_{\nu-1}+t_{\nu}}}A^{\prime}_{k_{\nu}+r_{n_{\nu-1}+1}}
\prod_{t_{\nu}=2}^{n_{\nu}-n_{\nu-1}}A^{\prime}_{k_{\nu}+l_{t_{\nu}-1}+r_{n_{\nu-1}+t_{\nu}}}
\end{align}
in which $n_{v}$ is the biggest subscript of summands, $r_{\cdot}$, in $\mathrm{seg}^{\xi^{\prime}}_{\nu}(\sigma^{r})$, $1\leq \nu \leq j$.
In particular, since $n_{j}=p$, we have (by (5.29))
\begin{align}
&A^{\prime}_{k_{j}+\mathrm{seg}^{\xi^{\prime}}_{j}(\sigma^{r})}=\sum_{l_{1}=1}^{r_{n_{j-1}+1}}\prod_{t_{j}=2}^{p-n_{j-1}-1}\sum_{l_{t_{j}}=1}^{l_{t_{j}-1}+r_{n_{j-1}+t_{j}}}A^{\prime}_{k_{j}+r_{n_{j-1}+1}}
\prod_{t_{j}=2}^{p-n_{j-1}}A^{\prime}_{k_{j}+l_{t_{j-1}}+r_{n_{j-1}+t_{j}}}\nonumber\\
=&\sum_{l_{1}=1}^{r_{n_{j-1}+1}}\prod_{t_{j}=2}^{p-n_{j-1}-2}\sum_{l_{t_{j}}=1}^{l_{t_{j}-1}+r_{n_{j-1}+t_{j}}}A_{k_{j}+r_{n_{j-1}+1}}
\prod_{t_{j}=2}^{p-n_{j-1}-1}A_{k_{j}+l_{t_{j}-1}+r_{n_{j-1}+t_{j}}}\nonumber\\
&\sum_{l_{p-n_{j-1}-1}=1}^{l_{p-n_{j-1}-2}+r_{p-1}}A_{k_{j}+l_{p-n_{j-1}-1}+r_{p}}\sum_{k_{p+1}=0}^{k_{j}+l_{p-n_{j-1}-1}+r_{p}}A_{k_{p+1}+r_{p+1}}\nonumber\\
=&\sum_{l_{1}=1}^{r_{n_{j-1}+1}}\prod_{t_{j}=2}^{p-n_{j-1}-2}\sum_{l_{t_{j}}=1}^{l_{t_{j}-1}+r_{n_{j-1}+t_{j}}}A_{k_{j}+r_{n_{j-1}+1}}
\prod_{t_{j}=2}^{p-n_{j-1}-1}A_{k_{j}+l_{t_{j}-1}+r_{n_{j-1}+t_{j}}}\nonumber\\
&\sum_{l_{p-n_{j-1}-1}=1}^{l_{p-n_{j-1}-2}+r_{p-1}}A_{k_{j}+l_{p-n_{j-1}-1}+r_{p}}\sum_{k_{p+1}=0}^{k_{j}}A_{k_{p+1}+r_{p+1}}\nonumber\\
&+\sum_{l_{1}=1}^{r_{n_{j-1}}+1}\prod_{t_{j}=2}^{p-n_{j-1}-2}\sum_{l_{t_{j}}=1}^{l_{t_{j}-1}+r_{n_{j-1}+t_{j}}}A_{k_{j}+r_{n_{j-1}+1}}
\prod_{t_{j}=2}^{p-n_{j-1}-1}A_{k_{j}+l_{t_{j}-1}+r_{n_{j-1}+t_{j}}}\nonumber\\
&\sum_{l_{p-n_{j-1}-1}=1}^{l_{p-n_{j-1}-2}+r_{p-1}}A_{k_{j}+l_{p-n_{j-1}-1}+r_{p}}\sum_{k_{p+1}=k_{j}+1}^{k_{j}+l_{p-n_{j-1}-1}+r_{p}}A_{k_{p+1}+r_{p+1}}\nonumber
\end{align}
\begin{align}
=&\sum_{l_{1}=1}^{r_{n_{j-1}}+1}\prod_{t_{j}=2}^{p-n_{j-1}-2}\sum_{l_{t_{j}}=1}^{l_{t_{j}-1}+r_{n_{j-1}+t_{j}}}A_{k_{j}+r_{n_{j-1}+1}}
\prod_{t_{j}=2}^{p-n_{j-1}-1}A_{k_{j}+l_{t_{j}-1}+r_{n_{j-1}+t_{j}}}\nonumber\\
&\sum_{l_{p-n_{j-1}-1}=1}^{l_{p-n_{j-1}-2}+r_{p-1}}A_{k_{j}+l_{p-n_{j-1}-1}+r_{p}}\sum_{k_{j+1}=0}^{k_{j}}A_{k_{j+1}+r_{n_{j}+1}}\nonumber\\
&+\sum_{l_{1}=1}^{r_{n_{j-1}}+1}\prod_{t_{j}=2}^{p-n_{j-1}-2}\sum_{l_{t_{j}}=1}^{l_{t_{j}-1}+r_{n_{j-1}+t_{j}}}A_{k_{j}+r_{n_{j-1}+1}}
\prod_{t_{j}=2}^{p-n_{j-1}-1}A_{k_{j}+l_{t_{j}-1}+r_{n_{j-1}+t_{j}}}\nonumber\\
&\sum_{l_{p-n_{j-1}-1}=1}^{l_{p-n_{j-1}-2}+r_{p-1}}\sum_{l_{p-n_{j-1}}=1}^{l_{p-n_{j-1}-1}+r_{p}}A_{k_{j}+l_{p-n_{j-1}-1}+r_{p}}A_{k_{j}+l_{p-n_{j-1}}+r_{p+1}}\nonumber\\
=&\sum_{l_{1}=1}^{r_{n_{j-1}}+1}\prod_{t_{j}=2}^{p-n_{j-1}-1}\sum_{l_{t_{j}}=1}^{l_{t_{j}-1}+r_{n_{j-1}+t_{j}}}A_{k_{j}+r_{n_{j-1}+1}}
\prod_{t_{j}=2}^{p-n_{j-1}}A_{k_{j}+l_{t_{j}-1}+r_{n_{j-1}+t_{j}}}\sum_{k_{j+1}=0}^{k_{j}}A_{k_{j+1}+r_{n_{j}+1}}\nonumber\\
&+\sum_{l_{1}=1}^{r_{n_{j-1}}+1}\prod_{t_{j}=2}^{p-n_{j-1}}\sum_{l_{t_{j}}=1}^{l_{t_{j}-1}+r_{n_{j-1}+t_{j}}}A_{k_{j}+r_{n_{j-1}+1}}
\prod_{t_{j}=2}^{p-n_{j-1}+1}A_{k_{j}+l_{t_{j}-1}+r_{n_{j-1}+t_{j}}}\nonumber\\
=&\sum_{l_{1}=1}^{r_{n_{j-1}}+1}\prod_{t_{j}=2}^{n_{j}-n_{j-1}-1}\sum_{l_{t_{j}}=1}^{l_{t_{j}-1}+r_{n_{j-1}+t_{j}}}A_{k_{j}+r_{n_{j-1}+1}}
\prod_{t_{j}=2}^{n_{j}-n_{j-1}}A_{k_{j}+l_{t_{j}-1}+r_{n_{j-1}+t_{j}}}\sum_{k_{j+1}=0}^{k_{j}}A_{k_{j+1}+r_{n_{j}+1}}\nonumber\\
&+\sum_{l_{1}=1}^{r_{n_{j-1}}+1}\prod_{t_{j}=2}^{\widetilde{n}_{j}-n_{j-1}-1}\sum_{l_{t_{j}}=1}^{l_{t_{j}-1}+r_{n_{j-1}+t_{j}}}A_{k_{j}+r_{n_{j-1}+1}}
\prod_{t_{j}=2}^{\widetilde{n}_{j}-n_{j-1}}A_{k_{j}+l_{t_{j}-1}+r_{n_{j-1}+t_{j}}},
\end{align}
where $\widetilde{n}_{j}=n_{j}+1=p+1$.

For any $n\in \mathbb{N}$, the partitions of $\sigma^{r}_{n}=(r_{1},r_{2},\ldots,r_{n})$ and $\sigma_{n+1}^{r}=(r_{1},r_{2},\ldots,r_{n},r_{n+1})$ have the following properties:
\begin{itemize}
  \item [(1)] The total number of all partitions of $\sigma^{r}_{n}$ is $2^{n-1}$; \vspace{1mm}
  \item [(2)] As in (5.32), any one of $j$-partitions of $\sigma^{r}_{n}$ corresponds to two partitions of $\sigma^{r}_{n+1}$: one is $j+1$-partition and the other is $j$-partition. (For example, as $n=3$ and $j=2$, the $2$-partition, $\{r_{1}, r_{2}+r_{3}\}$, of $\sigma^{r}_{3}=(r_{1},r_{2},r_{3})$ corresponds to the $3$-partition, $\{r_{1}, r_{2}+r_{3}, r_{4}\}$, and the $2$-partition, $\{r_{1}, r_{2}+r_{3}+r_{4}\}$ of $\sigma^{r}_{4}=(r_{1},r_{2},r_{3},r_{4})$.
\end{itemize}

By the above properties, we can get all partitions of $\sigma_{n+1}^{r}$ from the ones of $\sigma_{n}^{r}$ (just as showing in (5.32)). Thus (5.28) also holds for $n=p+1$ by (5.30), (5.32) and the assumption of induction for $n=p$.
\end{proof}

For any unordered $n$-tuple $r=\langle r_{1},r_{2},\ldots,r_{s}\rangle$, denote $S_{s}^{r}$ be the set of all permutations of the elements (viz. $r_{1},r_{2},\ldots,r_{s}$) of $r$ and its element by $\sigma^{r}=(\sigma^{r}_{1},\sigma^{r}_{2},\ldots,\sigma^{r}_{s})$. For instance, if $\sigma^{r}_{l}=r_{l}$, $1\leq l\leq s$, then $\sigma^{r}=(r_{1},r_{2},\ldots,r_{s})$; if $\sigma^{r}_{l}=r_{l+1}$, $1\leq l\leq s-1$ and $\sigma^{r}_{s}=r_{1}$, then $\sigma^{r}=(r_{2},r_{3},\ldots,r_{s},r_{1})$.

\begin{thm} Let $\mathcal{D}_{m}^{s}$, $\Sigma_{\sigma^{r}}^{s-j}$, $S^{r}_{s}$ and $\xi^{g}$ be as above, then for any $m\in \mathbb{N}$,
\begin{align}
w_{m}=\sum_{s=1}^{m}\sum_{r\in \mathcal{D}_{m}^{s}}\sum_{\sigma^{r}\in S_{s}^{r}}\sum_{j=1}^{s-1}\sum_{\xi^{g}\in\Sigma_{\sigma^{r}}^{s-j}}(-1)^{s}c_{s,r}c_{j,\xi^{g}}(j-1)!\prod_{\nu=1}^{j}\odot_{j}\left(\alpha^{\big(\mathrm{sgm}_{\nu}^{\xi^{g}}(\sigma^{r})\big)}\right),
\end{align}
where
$\displaystyle c_{s,r}=\frac{1}{t_{1}!t_{2}!\cdots t_{p}!}$ with $t_{1}+t_{2}+\cdots+t_{p}=s$ in which $t_{l}$, $1\leq l\leq p$, are the multiplicities of $p$ different elements in $r$ (that is, $r$ has $p$ different elements, $r_{1},r_{2},\ldots,r_{p}$, whose multiplicities are $t_{1},t_{2},\ldots,t_{p}$ respectively); $c_{j,\xi^{g}}=\begin{cases}
\frac{1}{l_{1}!l_{2}!\cdots l_{q}!},\,\,2\leq j\leq s-1,\vspace{0.5mm}\\
1,\,\,j=1\end{cases}$ in which $l_{\nu}$, $1\leq\nu\leq q$, are the multiplicities of different segments of $\xi^{g}$ (that is, $\xi^{g}$ has $q$ different segments, $\xi^{g}_{1}(\sigma^{r}), \xi^{g}_{2}(\sigma^{r}), \ldots, \xi^{g}_{q}(\sigma^{r})$, whose multiplicities are $l_{1},l_{2},\ldots,l_{q}$ respectively), and $\alpha^{\big(\mathrm{seg}^{\xi^{g}}_{\nu}(\sigma^{r})\big)}$ are given as in (5.27)
\end{thm}

\begin{proof}
For $1\leq s\leq m$ and $r\in \mathcal{D}_{m}^{s}$, we consider two different cases.

Firstly, $r$ is assumed to have distinct elements. That is, $r=\langle r_{1},r_{2},\ldots,r_{s}\rangle$ with $r_{k}\neq r_{l}$ as $k\neq l$ and $r_{1}+r_{2}+\cdots r_{s}=m$.

By Lemma 5.6, we have
\begin{align}
&\sum_{\sigma^{r}\in S^{r}_{s}}\sum_{k_{1}=0}^{\infty}\alpha_{k_{1}+\sigma^{r}_{1}-1}\overline{\alpha}_{k_{1}-1}\sum_{k_{2}=0}^{k_{1}+\sigma^{r}_{1}}\alpha_{k_{2}+\sigma^{r}_{2}-1}\overline{\alpha}_{k_{2}-1}\cdots
\sum_{k_{s}=0}^{k_{s-1}+\sigma^{r}_{s-1}}\alpha_{k_{s}+\sigma^{r}_{s}-1}\overline{\alpha}_{k_{s}-1}\nonumber\\
=&\sum_{\sigma^{r}\in S^{r}_{s}}\sum_{j=1}^{s}\sum_{\xi\in\Sigma_{\sigma^{r}}^{s-j}}\sum_{k_{1}=0}^{\infty}\alpha_{k_{1}+\mathrm{seg}^{\xi}_{1}(\sigma^{r})}
\sum_{k_{2}=0}^{k_{1}}\alpha_{k_{2}+\mathrm{seg}^{\xi}_{2}(\sigma^{r})}\cdots\sum_{k_{j}=0}^{k_{j-1}}\alpha_{k_{j}+\mathrm{seg}^{\xi}_{j}(\sigma^{r})}\nonumber\\
=&\sum_{\sigma^{r}\in S^{r}_{s}}\sum_{j=1}^{s}\sum_{\tau_{j}\in S_{j}}\sum_{\xi^{g}\in\Sigma_{\sigma^{r}}^{s-j}}\sum_{k_{\tau_{j,1}}=0}^{\infty}\alpha_{k_{\tau_{j,1}}+\mathrm{seg}^{\xi^{g}}_{\tau_{j,1}}(\sigma^{r})}
\sum_{k_{\tau_{j,2}}=0}^{k_{\tau_{j,1}}}\alpha_{k_{\tau_{j,2}}+\mathrm{seg}^{\xi^{g}}_{\tau_{j,2}}(\sigma^{r})}\nonumber\\
&\cdots\sum_{k_{\tau_{j,j}}=0}^{k_{\tau_{j,j-1}}}\alpha_{k_{\tau_{j,j}}+\mathrm{seg}^{\xi^{g}}_{\tau_{j,j}}(\sigma^{r})}\nonumber\\
\triangleq&\sum_{\sigma^{r}\in S^{r}_{s}}\sum_{j=1}^{s}\sum_{\tau_{j}\in S_{j}}\sum_{\xi^{g}\in\Sigma_{\sigma^{r}}^{s-j}}\sum_{k_{\tau_{j,1}}=0}^{\infty}\alpha_{k_{\tau_{j,1}}+\mathrm{seg}^{\xi^{g}}_{\tau_{j,1}}(\sigma^{r})}
\prod_{\nu=2}^{j}\sum_{k_{\tau_{j,\nu}}=0}^{k_{\tau_{j,\nu-1}}}\alpha_{k_{\tau_{j,\nu}}+\mathrm{seg}^{\xi^{g}}_{\tau_{j,\nu}}(\sigma^{r})},
\end{align}
where $\xi=\{\mathrm{seg}^{\xi}_{1}(\sigma^{r}), \mathrm{seg}^{\xi}_{2}(\sigma^{r}), \ldots, \mathrm{seg}^{\xi}_{j}(\sigma^{r})\}$ are $j$-partitions of $\sigma^{r}$, $\Sigma_{\sigma^{r}}^{s-j}$ is the set of all $j$-partitions of $\sigma^{r}$,
\begin{align}
\alpha_{k_{\nu}+\mathrm{seg}^{\xi}_{\nu}(\sigma^{r})}=&\sum_{l_{1}=1}^{r_{n_{\nu-1}}}\prod_{t_{\nu}=2}^{n_{\nu}-n_{\nu-1}}\sum_{l_{t_{\nu}}=1}^{l_{t_{\nu}-1}+r_{n_{\nu-1}+t_{\nu}-1}}\alpha_{k_{\nu}+r_{n_{\nu-1}+1}-1}\overline{\alpha}_{k_{\nu}-1}\nonumber\\
&\prod_{t_{\nu}=2}^{n_{\nu}-n_{\nu-1}}\alpha_{k_{\nu}+l_{t_{\nu}}+r_{n_{\nu-1}+t_{\nu}}-1}\overline{\alpha}_{k_{\nu}+l_{t_{\nu}}-1}
\end{align}
in which $n_{v}$ is the biggest number of summands, $r_{\cdot}$, in $\mathrm{seg}^{\xi}_{\nu}(\sigma^{r})$, $1\leq \nu \leq j$,
$\tau_{j}=(\tau_{j,1},\tau_{j,2},\ldots,\tau_{j,j})$, $S_{j}$ is the classical symmetric group of order $j$ (i.e., the group of all permutations of $\{1,2,\ldots,j\}$) and $\xi^{g}$ are good $j$-partitions of $\sigma^{r}$.

The justification for the second equality in (5.34) is that any summand in LHS is a one in RHS vice versa and all summands in each hand side are distinct.

Thus by Theorem 3.11, the sum with single infinite index in (5.34) is
\begin{align}
\sum_{\sigma^{r}\in S^{r}_{s}}\sum_{j=1}^{s}\sum_{\xi^{g}\in\Sigma_{\sigma^{r}}^{s-j}}(j-1)!\alpha^{\big(\mathrm{seg}^{\xi^{g}}_{1}(\sigma^{r})\big)}
\odot\alpha^{\big(\mathrm{seg}^{\xi^{g}}_{2}(\sigma^{r})\big)}\odot\cdots\odot\alpha^{\big(\mathrm{seg}^{\xi^{g}}_{j}(\sigma^{r})\big)},
\end{align}
where $\alpha^{\big(\mathrm{seg}^{\xi^{g}}_{\nu}(\sigma^{r})\big)}$, $1\leq\nu\leq j$, are given as in (5.27).

Next, $r$ is assumed to have $p$ distinct elements, $r_{1}, r_{2}, \ldots, r_{p}$, and their multiplicities are $t_{1}, t_{2}, \ldots, t_{p}$ respectively. Namely,
$$r=\langle \underbrace{r_{1},r_{1},\cdots,r_{1}}_{t_{1}},\underbrace{r_{2},r_{2},\cdots,r_{2}}_{t_{2}},\cdots,\underbrace{r_{p},r_{p},\cdots, r_{p}}_{t_{p}}\rangle$$
with $t_{1}r_{1}+t_{2}r_{2}+\cdots+t_{p}r_{p}=m$.
Denote $S^{r}_{s,p}$ be the set of all distinct permutations of the elements of $r$. Since their exist the same elements, the number of all distinct permutations in $S^{r}_{s,p}$ are $\frac{s!}{t_{1}!t_{2}!\cdots t_{p}!}$. Therefore, repeating any one of the permutations in $S^{r}_{s,p}$ for $t_{1}!t_{2}!\cdots t_{p}!$ times, we can obtain all permutations of the elements of $r$ (with permission of appearing the same permutations) liking the former case. Thus
\begin{align}
&\sum_{\sigma^{r}\in S^{r}_{s,p}}\sum_{k_{1}=0}^{\infty}\alpha_{k_{1}+\sigma^{r}_{1}-1}\overline{\alpha}_{k_{1}-1}\sum_{k_{2}=0}^{k_{1}+\sigma^{r}_{1}}\alpha_{k_{2}+\sigma^{r}_{2}-1}\overline{\alpha}_{k_{2}-1}\cdots
\sum_{k_{s}=0}^{k_{s-1}+\sigma^{r}_{s-1}}\alpha_{k_{s}+\sigma^{r}_{s}-1}\overline{\alpha}_{k_{s}-1}\nonumber\\
=&\frac{1}{t_{1}!t_{2}!\cdots t_{p}!}\sum_{\sigma^{r}\in S^{r}_{s}}\sum_{k_{1}=0}^{\infty}\alpha_{k_{1}+\sigma^{r}_{1}-1}\overline{\alpha}_{k_{1}-1}\sum_{k_{2}=0}^{k_{1}+\sigma^{r}_{1}}\alpha_{k_{2}+\sigma^{r}_{2}-1}\overline{\alpha}_{k_{2}-1}\cdots
\sum_{k_{s}=0}^{k_{s-1}+\sigma^{r}_{s-1}}\alpha_{k_{s}+\sigma^{r}_{s}-1}\overline{\alpha}_{k_{s}-1}.
\end{align}
By Corollary 3.14, the sum with single infinite index in (5.37) is
\begin{align}
\frac{1}{t_{1}!t_{2}!\cdots t_{p}!}\frac{1}{l_{1}!l_{2}!\cdots l_{q}!}\sum_{\sigma^{r}\in S^{r}_{s}}\sum_{j=1}^{s}\sum_{\xi^{g}\in\Sigma_{\sigma^{r}}^{s-j}}(j-1)!\alpha^{\big(\mathrm{seg}^{\xi^{g}}_{1}(\sigma^{r})\big)}
\odot\alpha^{\big(\mathrm{seg}^{\xi^{g}}_{2}(\sigma^{r})\big)}\odot\cdots\odot\alpha^{\big(\mathrm{seg}^{\xi^{g}}_{j}(\sigma^{r})\big)}.
\end{align}
As the statement in Remark 3.10, similar to $w_{l}$, $1\leq l\leq 4$, by Theorem 3.1 and Golinskii-Zlato\v{s} single index theorem (more precisely, Proposition 3.7), (5.33) follows from (3.12), (5.36) and (5.38).
\end{proof}

\begin{rem}
In fact, by (5.33) and real parts of products of several complex variables (viz. Propositions 4.10, 4.41 and 4.43), we can obtain some general sum rules for $n\in \mathbb{R}$ as in the last section for $1\leq n\leq4$.
\end{rem}

Finally, with the above preliminaries, we obtain the following general higher order Szeg\H{o} theorems.
\begin{thm}
Assume that $\alpha\in\ell^{6}$ and $(S-1)\alpha\in\ell^{3}$, then
\begin{align}
Z_{P_{n}}(\mu)=\int_{0}^{2\pi}P_{n}(\theta)\log w(\theta)\frac{d\theta}{2\pi}>-\infty\,\,\Longleftrightarrow\,\,Q_{n}(S)\alpha \in\ell^{2},
\end{align}
where $P_{n}, Q_{n}$ are given as in Theorem 5.4, and $P_{n}$ also satisfies
\begin{align}
\frac{a_{0}}{2}=\begin{cases}\displaystyle 2\sum_{l=0}^{\left[\frac{n}{2}\right]}l(l+1)a_{2l+1}+\frac{1}{2}\sum_{l=1}^{\left[\frac{n}{2}\right]}(2l-1)(2l+1)a_{2l},\,\,\mbox{as n is odd},\vspace{2mm}\\
\displaystyle2\sum_{l=1}^{\left[\frac{n}{2}\right]}l(l+1)a_{2l-1}+\frac{1}{2}\sum_{l=1}^{\left[\frac{n}{2}\right]}(2l-1)(2l+1)a_{2l},\,\,\mbox{as n is even}.
\end{cases}
\end{align}
\end{thm}

\begin{proof}Since
\begin{align}
P_{n}(\theta)=a_{0}+a_{1}\cos\theta+\cdots+a_{n}\cos n\theta=-\sum_{l=1}^{n}a_{l}(1-\cos l\theta)
\end{align}
in terms of $P_{k}(0)=0$ (precisely, $\sum_{l=0}^{n}a_{l}=0$), then
\begin{align}
Z_{P_{n}}(\mu)=a_{0}\mathrm{Re}(w_{0})+a_{1}\mathrm{Re}(w_{1})+\cdots+a_{n}\mathrm{Re}(w_{n})=-\sum_{l=1}^{n}a_{l}Z_{n,1}(\mu)
\end{align}
in which $Z_{n,1}(\mu)=\int_{0}^{2\pi}(1-\cos n\theta)\log w(\theta)\frac{d\theta}{2\pi}$.

At first, we consider the series whose general term has only two factors consisting of some elements of $\alpha$ and $\overline{\alpha}$ in $w_{m}$.  For any $m\in \mathbb{N}$, by Theorem 5.7, such series in $w_{m}$ is the following
\begin{equation}
-\sum_{j=0}^{\infty}\alpha_{j+m-1}\overline{\alpha}_{j-1}.
\end{equation}
By (4.51), its real part is
\begin{align}
\mathrm{bdy}-\|\alpha\|^{2}_{2}+\frac{1}{2}\|(S^{m}-1)\alpha\|^{2}_{2}.
\end{align}

Next, turn to the series which have four factors consisting of some elements of $\alpha$ and $\overline{\alpha}$ in $w_{m}$. In what follows, we consider them from two different cases.

Case I: $m$ is odd. In this case, by Theorem 5.7, such series in $w_{m}$ are
\begin{equation}
\sum_{j=0}^{\infty}\alpha_{j+p-1}\overline{\alpha}_{j-1}\alpha_{j+s+q-1}\overline{\alpha}_{j+s-1}\,\,\,\mathrm{and}\,\,\,
\sum_{j=0}^{\infty}\alpha_{j+q-1}\overline{\alpha}_{j-1}\alpha_{j+t+p-1}\overline{\alpha}_{j+t-1}
\end{equation}
where $p+q=m$, $1\leq q\leq [\frac{m}{2}]$, $0\leq s\leq p$ and $1\leq t\leq q$.

Case II: $m$ is even. By Theorem 5.7, such series in this case are
\begin{equation}
\sum_{j=0}^{\infty}\alpha_{j+p-1}\overline{\alpha}_{j-1}\alpha_{j+s+q-1}\overline{\alpha}_{j+s-1}\,\,\,\mathrm{and}\,\,\,
\sum_{j=0}^{\infty}\alpha_{j+q-1}\overline{\alpha}_{j-1}\alpha_{j+t+p-1}\overline{\alpha}_{j+t-1},
\end{equation}
\begin{equation}
\sum_{j=0}^{\infty}\alpha_{j+\frac{m}{2}-1}\overline{\alpha}_{j-1}\alpha_{j+l+\frac{m}{2}-1}\overline{\alpha}_{j+l-1}
\end{equation}
and
\begin{equation}
\frac{1}{2}\sum_{j=0}^{\infty}\alpha_{j+\frac{m}{2}-1}^{2}\overline{\alpha}_{j-1}^{2},
\end{equation}
where $p+q=m$, $1\leq q\leq \frac{m}{2}-1$, $0\leq s\leq p$, $1\leq t\leq q$ and $1\leq l\leq \frac{m}{2}$.

By (4.51) and Theorem 4.43, we know that the real part of general term, $\alpha_{j+m-l-1}\overline{\alpha}_{j-1}\alpha_{j+k+l-1}\overline{\alpha}_{j+k-1}$, in two cases is a sum whose summands are of three different types as follows
\begin{align}
|\alpha_{j+p}|^{2}|\alpha_{j+q}|^{2},\,\,|\alpha_{j+p}|^{2}|\alpha_{j+q}-\alpha_{j+q^{\prime}}|^{2}\,\,\,\mathrm{and}\,\,\,|\alpha_{j+p}-\alpha_{j+p^{\prime}}|^{2}|\alpha_{j+q}-\alpha_{j+q^{\prime}}|^{2}.
\end{align}
By the assumption and using H\"older inequality, the series with general terms given by the latter two types in (5.49) are convergent. By simple calculations, the number of the series with general terms given by the former type in (5.49) is
\begin{equation}
\begin{cases}
[\frac{m}{2}](m+1)=[\frac{m}{2}]\big((p+1)+q\big),\,\,\mbox{if \emph{m} is odd},\vspace{1mm}\\
\frac{1}{2}(m^{2}-1)=\frac{m-1}{2}\big((p+1)+q\big)+\frac{m}{2}+\frac{1}{2},\,\,\mbox{if \emph{m} is even}.
\end{cases}
\end{equation}
Since
\begin{equation}
\sum_{j=0}^{\infty}|\alpha_{j}|^{4}-\sum_{j=0}^{\infty}|\alpha_{j+p}|^{2}|\alpha_{j+q}|^{2}=\mathrm{bdy}+\frac{1}{2}\sum_{j=0}^{\infty}\big(|\alpha_{j+p}|^{2}-|\alpha_{j+q}|^{2}\big)^{2},
\end{equation}
we get that $\sum_{j=0}^{\infty}|\alpha_{j}|^{4}-\sum_{j=0}^{\infty}|\alpha_{j+p}|^{2}|\alpha_{j+q}|^{2}$ is convergent by the assumption.

By (5.40), (5.42), (5.44) and Theorem 5.4, we have
\begin{align}
Z_{P_{n}}(\mu)=\mathrm{bdy}+a_{0}\sum_{j=0}^{\infty}\big[\log(1-|\alpha_{j}|^{2})+|\alpha_{j}|^{2}+\frac{1}{2}|\alpha_{j}|^{4}\big]-\|Q_{n}(S)\alpha\|_{2}^{2}+R_{n}(\alpha),
\end{align}
where the remainder $R_{n}(\alpha)$ is the sum of some series whose general terms are the ones in (5.51), the latter two types in (5.49) and at least six factors consisting of some elements of $\alpha$ and $\overline{\alpha}$  in $w_{m}$ (the last case is for $3\leq m \leq n$).

By the properties of real part of products of several complex variables, it is easy to get that $R_{n}(\alpha)$ is finite under the assumption of $\alpha\in\ell^{6}$ and $(S-1)\alpha\in\ell^{3}$. Applying this fact, by Lemma 4.3, (5.39) immediately follows from (5.52) as $\alpha\in\ell^{6}$ and $(S-1)\alpha\in\ell^{3}$.
\end{proof}

\begin{thm}
Assume that $\alpha\in\ell^{4}$, then
\begin{align}
Z_{P_{n}}(\mu)>-\infty\,\,\Longleftrightarrow\,\,Q_{n}(S)\alpha \in\ell^{2},
\end{align}
where $P_{n}, Q_{n}$ are given as in Theorem 5.4.
\end{thm}

\begin{proof}
It immediately follows from (5.52) under the assumption of $\alpha\in\ell^{4}$.
\end{proof}

\begin{rem}
This is a general result in \cite{gz} due to Golinskii-Zlato\v{s}.
\end{rem}

\begin{thm}
Let $P_{n}(z)=1-\cos n\theta$, $\alpha\in\ell^{6}$ and $(S-1)\alpha\in\ell^{3}$, then
\begin{align}
Z_{P_{n}}(\mu)>-\infty\,\,\Longleftrightarrow\,\,\alpha \in\ell^{4}\,\,and\,\,(S^{n}-1)\alpha \in\ell^{2}.
\end{align}
\end{thm}

\begin{proof}
Similar to Theorem 5.9.
\end{proof}

\begin{prop}
The number of all series in $\alpha^{(r_{1}+r_{2}+\cdots+r_{n})}$ is
\begin{equation}
\mathcal{N}_{r_{1}+r_{2}+\cdots+r_{n}}=\begin{cases}
1,\,\,\,n=1,\vspace{2mm}\\
r_{1},\,\,\,n=2,\vspace{2mm}\\
\displaystyle\sum_{k=1}^{r_{1}}\sum_{l=1}^{k}(l+r_{2}),\,\,\,n=3,\vspace{2mm}\\
\displaystyle\sum_{k=1}^{r_{1}}\sum_{l=1}^{k+r_{2}+r_{3}+\cdots+r_{n-2}}(l+r_{n-1}),\,\,\,n\geq 4.
\end{cases}
\end{equation}
\end{prop}

\begin{proof}
By (5.27), it is trivial for $n=1$ and $n=2$.

For $n=3$, by (5.27), the number of all series of $\alpha^{(r_{1}+r_{2}+r_{3})}$ is
\begin{align}
&\big(1+2+\cdots+(1+r_{2})\big)+\big(1+2+\cdots+(2+r_{2})\big)+\cdots+\big(1+2+\cdots+(r_{1}+r_{2})\big)\nonumber\\
=&\sum_{k=1}^{r_{1}}\sum_{l=1}^{k}(l+r_{2}).
\end{align}

For $n=4$, noting the following inclusion relation among $l_{\nu}$, $1\leq\nu\leq3$, given in (5.27),
\begin{equation}
l_{1}\begin{cases}
1\rightarrow l_{2}\begin{cases}
1\rightarrow l_{3}\begin{cases}
1\\
2\\
\vdots\\
1+r_{3}
\end{cases}\\
\vdots\\
1+r_{2}\rightarrow l_{3}\begin{cases}
1\\
2\\
\vdots\\
1+r_{2}+r_{3}
\end{cases}
\end{cases}\\
\vdots\\
r_{1}\rightarrow l_{2}\begin{cases}
1\rightarrow l_{3}\begin{cases}
1\\
2\\
\vdots\\
1+r_{3}
\end{cases}\\
\vdots\\
r_{1}+r_{2}\rightarrow l_{3}\begin{cases}
1\\
2\\
\vdots\\
r_{1}+r_{2}+r_{3}
\end{cases}
\end{cases}
\end{cases}
\end{equation}
the number of all series in $\alpha^{(r_{1}+r_{2}+r_{3}+r_{4})}$ is
\begin{align}
&\big[(1+r_{3})+(2+r_{3})+\cdots+(1+r_{2}+r_{3})\big]+\big[(1+r_{3})+(2+r_{3})+\cdots+(2+r_{2}+r_{3})\big]\nonumber\\
&+\cdots+\big[(1+r_{3})+(2+r_{3})+\cdots+(r_{1}+r_{2}+r_{3})\big]=\sum_{k=1}^{r_{1}}\sum_{l=1}^{k+r_{2}}(l+r_{3}).
\end{align}
For $n\geq 5$, the argument to get (5.55) is similar to (5.58).
\end{proof}

\begin{thm}
For $m\in \mathbb{N}$ and $1\leq s\leq m$, let $\mathcal{C}_{m,s}$ be the sum of coefficients of all series which have $2s$ factors consisting of $\alpha$ and $\overline{\alpha}$ in $w_{m}$, then
\begin{itemize}
  \item [(1)] $\mathcal{C}_{m,1}=1;$\vspace{2mm}
  \item [(2)] $\mathcal{C}_{m,2}=\begin{cases}[\frac{m}{2}](m+1),\,\,\mbox{m is odd},\vspace{1mm}\\
  \frac{1}{2}(m^{2}-1),\,\,\mbox{m is even};\end{cases}$ \vspace{2mm}
  \item [(3)] For $s\geq 3$, $$\mathcal{C}_{m,s}=\sum_{r\in \mathcal{D}_{m}^{s}}\sum_{\sigma^{r}\in S_{s}^{r}}\sum_{j=1}^{s-1}\sum_{\xi^{g}\in\Sigma_{\sigma^{r}}^{s-j}}(-1)^{s}c_{s,r}c_{j,\xi^{g}}(j-1)!\prod_{\nu=1}^{j}\mathcal{N}_{\mathrm{sgm}_{\nu}^{\xi^{g}}(\sigma^{r})},$$
      where $\mathcal{N}_{\mathrm{sgm}_{\nu}^{\xi^{g}}(\sigma^{r})}$ is given as in (5.55).
\end{itemize}
\end{thm}

\begin{proof}
By Proposition 5.13 and Theorem 5.7.
\end{proof}

\begin{thm}
Let $P_{n}(\theta)=\sum_{l=0}^{n}a_{l}\cos l\theta$ satisfy
\begin{equation}
\frac{1}{s}a_{0}=-\sum_{l=s}^{n}a_{l}\mathcal{C}_{l,s}\,\,\,1\leq s \leq n,
\end{equation}
where $\mathcal{C}_{l,s}$ is given in Theorem 5.14, and assume $(S-1)\alpha\in\ell^{2}$, then
\begin{align}
Z_{P_{n}}(\mu)>-\infty\,\,\Longleftrightarrow\,\,\alpha \in\ell^{2n+2}.
\end{align}
\end{thm}

\begin{proof}
Similar to Theorem 5.9.
\end{proof}

\begin{prop}
For any $n\in \mathbb{N}$,
\begin{equation}
\cos^{2n-1}\theta=\sum_{l=1}^{n}a_{2n-1,2l-1}\cos(2l-1)\theta
\end{equation}
and
\begin{equation}
\cos^{2n}\theta=\sum_{l=0}^{n}a_{2n,2l}\cos2l\theta
\end{equation}
where the coefficients $a_{k,l}$ satisfy the following recursion relations
\begin{align}
a_{2n,0}=\frac{1}{2}a_{2n-1,1},\,\, a_{2n,2l}=\frac{1}{2}\big(a_{2n-1,2l-1}+a_{2n-1,2l+1}\big),\,\, a_{2n,2n}=\frac{1}{2}a_{2n-1,2n-1}
\end{align}
in which $1\leq l\leq n-1$, and
\begin{align}
&a_{2n-1,1}=a_{2(n-1),0}+\frac{1}{2}a_{2(n-1),2},\,\, a_{2n-1,2l-1}=\frac{1}{2}\big(a_{2(n-1),2(l-1)}+a_{2(n-1),2l}\big),\nonumber\\ &a_{2n-1,2n-1}=\frac{1}{2}a_{2(n-1),2(n-1)}
\end{align}
in which $2\leq l\leq n-1$, as well as $a_{0,0}=1$.
\end{prop}

\begin{proof}
By direct calculations and the following trigonometric identities
\begin{equation}
\cos\theta\cos n\theta=\frac{1}{2}\big[\cos(n-1)\theta+\cos(n+1)\theta\big]
\end{equation}
for any $n\in \mathbb{N}_{0}$.
\end{proof}

\begin{rem}
More intuitively, these coefficients can be figured as follows

  $$\begin{array}{ccccccccc}
   &  &  &  & 1 &  &  &  &\\\vspace{1mm}
   &  &  & \frac{1}{2} &  & \frac{1}{2} &  &  &\\\vspace{1mm}
   &  &  & \frac{3}{4} &  & \frac{1}{4} &  &  &\\\vspace{1mm}
   &  & \frac{3}{8} &  & \frac{1}{2} &  & \frac{1}{8} &  &\\\vspace{1mm}
   &  & \frac{5}{8} &  & \frac{5}{16} &  & \frac{1}{16} & &\\\vspace{1mm}
   &\frac{5}{16}  &  & \frac{15}{32} &  & \frac{3}{16} &  & \frac{1}{32} &\\\vspace{1mm}
   &\frac{35}{64}  &  & \frac{21}{64} &  & \frac{7}{64} &  & \frac{1}{64} &\\\vspace{1mm}
  \vdots &   & \vdots &  & \vdots &  & \vdots &  & \vdots
  \end{array}$$
\end{rem}

In \cite{lu1}, Lukic obtained the following

\begin{thm}[Lukic]
Let $P_{n}(z)=(1-\cos\theta)^{n}$ and $(S-1)\alpha\in\ell^{2}$, then
\begin{align}
Z_{P_{n}}(\mu)>-\infty\,\,\Longleftrightarrow\,\,\alpha \in\ell^{2n+2}.
\end{align}
\end{thm}

When $n=1,2,3,4$, by simple calculations, we can easily find that $P_{n}(\theta)$ satisfies (5.59) in all these cases. So the results of Lukic in these cases follow from Theorem 5.15. Based on this observation, we conjecture that $P_{n}(\theta)=(1-\cos\theta)^{n}$ is also subject to (5.59) for any $n\in \mathbb{N}$. That is,

\begin{conj}
Let $P_{n}(z)=(1-\cos \theta)^{n}=\sum_{l=0}^{n}(-1)^{l}C_{n}^{l}\cos^{l}\theta$, then
\begin{itemize}
  \item [(1)] As $n=2m$, \vspace{2mm}
  \begin{align}
  \frac{1}{s}\sum_{l=0}^{m}C_{2m}^{2l}a_{2l,0}=-\sum_{\nu=s}^{m}\sum_{l=\nu}^{m}C_{2m}^{2l}a_{2l,2\nu}\mathcal{C}_{2\nu,s}
  +\sum_{\nu=s}^{m}\sum_{l=\nu}^{m}C_{2m}^{2l-1}a_{2l-1,2\nu-1}\mathcal{C}_{2\nu-1,s}
  \end{align}
  in which $1\leq s\leq 2m$ and $m\in \mathbb{N}$;
  \item [(2)]  As $n=2m-1$,
  \begin{align}
  \frac{1}{s}\sum_{l=0}^{m-1}C_{2m-1}^{2l}a_{2l,0}=-\sum_{\nu=s}^{m}\sum_{l=\nu}^{m-1}C_{2m-1}^{2l}a_{2l,2\nu}\mathcal{C}_{2\nu,s}
  +\sum_{\nu=s}^{m}\sum_{l=\nu}^{m}C_{2m-1}^{2l-1}a_{2l-1,2\nu-1}\mathcal{C}_{2\nu-1,s}
  \end{align}
   in which $1\leq s\leq 2m-1$ and $m\in \mathbb{N}$,
   where $\mathcal{C}_{\cdot,\cdot}$ and $a_{\cdot,\cdot}$ are given as in Theorem 5.14 and Proposition 5.16 respectively.
\end{itemize}
\end{conj}

In addition, we also conjecture the following
\begin{conj}
$P_{n}(\theta)=(1-\cos \theta)^{n}$ is the only trigonometric polynomials of order $n$ satisfying (5.59).
\end{conj}

If such two conjectures are exact, the essence of Lukic result in \cite{lu1} rests on (5.59).

% ----------------------------------------------------------------
\bibliographystyle{amsplain}

\end{document}